\documentclass[a4paper]{article}

\usepackage[all]{xy}\usepackage[latin1]{inputenc}        %accents
\usepackage[dvips]{graphics,graphicx}
\usepackage{amsfonts,amssymb,amsmath,xcolor,mathrsfs, amstext}
\usepackage{amsbsy, amsopn, amscd, amsxtra, amsthm,authblk, enumerate}
\usepackage{upref}
\usepackage{geometry}
\usepackage[displaymath]{lineno}
\usepackage{float}

\usepackage[colorlinks,
            linkcolor=blue,
            anchorcolor=green,
            citecolor=blue
            ]{hyperref}

\numberwithin{equation}{section}

\def\e{\varepsilon}
\def\epsilon{\varepsilon}
\def\eps{\varepsilon}

\newcommand{\ol}{\overline}
\newcommand{\wt}{\widetilde}
\def\alb#1\ale{\begin{align*}#1\end{align*}}
\newcommand{\eqb}{\begin{equation}}
\newcommand{\eqe}{\end{equation}}

\newcommand{\bbE}{\mathbb{E}}
\newcommand{\bbH}{\mathbb{H}}

\newcommand{\bbR}{\mathbb{R}}

\newcommand{\cD}{\mathcal{D}}
\newcommand{\cF}{\mathcal{F}}
\newcommand{\cL}{\mathcal{L}}

\newcommand{\QT}{\mathrm{QT}}
\newcommand{\LF}{\mathrm{LF}}
\newcommand{\SLE}{\mathrm{SLE}}
\newcommand{\Wd}{\mathrm{Weld}}
\newcommand{\Md}{{\mathcal{M}}^\mathrm{disk}}
\newcommand{\tildeMd}{\widetilde{\mathcal{M}}^\mathrm{disk}}

\newcommand{\xin}[1]{{\color{blue}{#1}}}

\newtheorem{theorem}{Theorem}[section]
\newtheorem{lemma}[theorem]{Lemma}
\newtheorem{proposition}[theorem]{Proposition}
\newtheorem*{proposition*}{Proposition}

\newtheorem*{corollary*}{Corollary}
\newtheorem{definition}[theorem]{Definition}
\newtheorem*{definitions*}{Definitions}

\newtheorem*{example*}{\bf Example}
\theoremstyle{remark}
\newtheorem{remark}[theorem]{\bf Remark}

\numberwithin{equation}{section}

\title{%Natural coupling of imaginary and 
Quantum triangles and imaginary geometry flow lines}
\author{Morris Ang  \qquad\quad    Xin Sun  \qquad\quad   Pu Yu}

\date{}
\begin{document}

\maketitle

\begin{abstract}
We define a three-parameter family of random surfaces in Liouville quantum gravity (LQG) which can be viewed as the quantum version of triangles. These quantum triangles are natural in two senses. First, by our definition they produce the boundary three-point correlation functions of Liouville conformal field theory on the disk. Second, it turns out that the laws of the triangles bounded by flow lines in imaginary geometry coupled with LQG are given by these quantum triangles. In this paper we demonstrate the second point for boundary flow lines on a quantum disk. Our method has the potential to prove general conformal welding results with quantum triangles glued in an arbitrary way. Quantum triangles play a basic role in understanding the integrability of SLE and LQG  via conformal welding. In this paper, we deduce integrability results for chordal SLE with three force points, using the conformal welding of a quantum triangle and a two-pointed quantum disk. 
{Further applications will be explored in  subsequent works.}
%we will explore their applications to  the mating-of-trees framework of LQG, including the exact evaluation of the expected proportion of inversions in skew Brownian permutons.
\end{abstract}

\setcounter{tocdepth}{1}

 \tableofcontents
%%%%%%%%%%%%%%%%%%%%%%%%%%%%%%%%%%%%%%%%%%%%%%
%%%% Main text entry area:

\section{Introduction}

Schramm-Loewner evolution (SLE) and Liouville quantum gravity (LQG)  are central subjects in random conformal geometry as canonical theories for random curves and surfaces, respectively. Starting from~\cite{She16a}, a key tool to study SLE and LQG is their coupling, where SLE curves arise as the interfaces of LQG surfaces under conformal welding. This leads to the mating-of-trees theory~\cite{DMS14}, which is fundamental in connecting LQG and the scaling limits of random planar maps decorated with statistical physics models; see the textbook~\cite{BP21} and the survey~\cite{GHS19}.  More recently, conformal welding  was used to study the integrability of SLE and LQG~\cite{AHS21,AGS21,AS21}. 

In most conformal welding results established so far, the SLE curves cut the LQG surfaces into smaller surfaces with two boundary marked points. The infinite-area version of these two-pointed marked surfaces are called quantum wedges, while the finite-area variants are  called two-pointed quantum disks. As shown in~\cite{DMS14,AHS20},  when these surfaces are welded together, the law of the SLE interfaces are a collection of flow lines in the sense of imaginary geometry~\cite{MS16a,MS17}, which is a canonical framework to couple multiple SLE curves. Two-pointed quantum disks also plays a basic role in the Liouville conformal field theory (LCFT) as they determine the  reflection coefficient for LCFT on the disk~\cite{HRV-disk,RZ20b,AHS21}.

In this paper we define a three-parameter family of LQG surfaces with three boundary marked points, which we call quantum triangles. They are defined to produce the boundary three-point correlation functions of LCFT on the disk. When two of the parameters are equal, they reduce to a two-parameter family of quantum surfaces defined in~\cite{AHS21}. The main goal of our paper is   to demonstrate that  the law of the triangular surfaces cut out by imaginary geometry flow lines on
a LQG  disk with multiple boundary marked points are given by quantum triangles; see Theorem~\ref{thm:disk}. Based on our work, a  general result with quantum triangles conformally welded in an arbitrary way will be proved by the first and the third authors in a subsequent work. Quantum triangles enrich the applications of conformal welding to SLE and LQG. In this paper, we deduce integrablity results for chordal SLE with three force points. {Further applications will be discussed in Section~\ref{subsec:intro-outlook}.}

We will give a brief description of quantum triangles in Section~\ref{subsec:intro-def} with the precise definition postponed to Section~\ref{sec-pre}. Then in Section~\ref{subsec:intro-2pt} we state a key result (Theorem \ref{thm:M+QTp-rw}) saying that the conformal welding of a quantum triangle and a two-pointed quantum disk gives another quantum triangle, which is proved in Sections~\ref{sec:WW2}---\ref{sec-6-thin-qt}.  The proof includes several novel techniques for proving  general conformal welding results. In particular, we give a Markovian characterization of the Liouville fields defining quantum triangles, which explains their ubiquity. 
As a corollary of Theorem \ref{thm:M+QTp-rw}, we state the aforementioned Theorem~\ref{thm:disk}  in Section~\ref{subsec:intro-disk} with more details  on imaginary geometry provided in Section~\ref{sec-pre-ig}. We present some applications of Theorem \ref{thm:M+QTp-rw} to SLE in Section~\ref{subsec-SLE-weighted}, whose proofs are given in Section~\ref{sec:application}. In Section~\ref{subsec:intro-outlook}, we discuss some perspectives  and related works.

\subsection{Definition of the quantum triangle}\label{subsec:intro-def}

Fix $\gamma\in (0,2)$. A quantum surface in $\gamma$-LQG is a surface with an area measure and a   metric structure  induced by a variant of Gaussian free field (GFF). The area is defined in~\cite{DS11}  and the metric is defined in~\cite{DDDF19,GM19metric}. A quantum surface with the disk topology can be represented as a pair $(D,h)$ where $D$ is a simply connected domain and $h$ is a variant of GFF. For such surfaces there is also a notion of $\gamma$-LQG length measure on the disk boundary~\cite{DS11}.   Two pairs $(D,h)$  and $(D',h')$ represent the same quantum surface if there is a conformal map between $D$ and $D'$ preserving the geometry. A particular pair $(D,h)$ is called a (conformal) embedding of the quantum surface. 

For $W>0$, the two-pointed quantum disk of weight $W$ is a quantum surface with two boundary marked points introduced in~\cite{DMS14,AHS20}, which has finite quantum area and length. It has two regimes: thick (i.e.\ $W\ge \frac{\gamma^2}{2}$) and thin (i.e.\ $W\in (0,\frac{\gamma^2}2)$).
For $W\ge \gamma^2/2$, the two-pointed quantum disk has the disk topology with two boundary  marked  points. The field near the two marked points has a $\beta$-log singularity where $\beta$ and $W$ are related by  
\begin{equation}\label{eq:betaW}
	\beta = \gamma+\frac{2-W}{\gamma},\quad \textrm{i.e.}\quad W=\gamma (Q+\frac{\gamma}2-\beta).
\end{equation}
For $W\in (0,\gamma^2/2)$, the weight-$W$ two-pointed quantum disk has the topology of an ordered collection of disks, each of which has two boundary  marked  points. There is a canonical law $\mathcal{M}_2^{\textup{disk}}(W)$ for the weight-$W$ two-pointed quantum disk, which has no constraint on the total area and boundary lengths. Other variants with fixed area and/or length can be obtained from $\mathcal{M}_2^{\textup{disk}}(W)$ by conditioning.
We also write $\mathcal{M}_2^{\textup{disk}}(2)$ as $\mathrm{QD}_{0,2}$. A sample from $\mathrm{QD}_{0,2}$ is known as the quantum disk with two typical boundary  points, because in this case the two marked points are simply distributed according to the $\gamma$-LQG boundary length measure. 
This special case arises naturally as scaling limits of random planar maps. For example,  when $\gamma=\sqrt{8/3}$, $\mathrm{QD}_{0,2}$ is the law of the LQG realization  of the Brownian disk with two boundary marked points, with free area and boundary length~\cite{lqg-tbm-1,lqg-tbm-2}. This is the scaling limit of  triangulation or quadrangulations  sampled from the critical Boltzmann measure~\cite{BM15browniandisk,GM19browniandisk}. 
In general, $\mathcal{M}_2^{\textup{disk}}(W)$ is an infinite measure. For $W\in (0,\frac{\gamma^2}{2})$, the ordered collections of disks in $\mathcal{M}_2^{\textup{disk}}(W)$
can be obtained from an initial segment of the Poisson point process with intensity measure $\mathcal{M}_2^{\textup{disk}}(\gamma^2-W)$.
We will recall the precise definition of $\mathcal{M}_2^{\textup{disk}}(W)$ in  Section~\ref{sec-pre}.

Two-pointed quantum disks are intimately related to Liouville conformal field theory on the disk~\cite{HRV-disk}. This relation is most transparent when we parameterize a quantum disk by a strip. Let $\mathcal S$ be the horizontal  strip $\mathbb R\times (0,\pi)$. For $W>\frac{\gamma^2}{2}$, let $(\mathcal{S},\phi, +\infty, -\infty)$ be an embedding of a sample from $\mathcal{M}_2^{\textup{disk}}(W)$. Let $\beta = \gamma+\frac{2-W}{\gamma}<Q$ as in~\eqref{eq:betaW}. By~\cite{AHS21},
if we independently sample $T$ from the Lebesgue measure on  $\mathbb{R}$, then the law of the field $\tilde{\phi}:=\phi(\cdot+T)$  is  $\frac{\gamma}{2(Q-\beta)^2}\textup{LF}_\mathcal{S}^{
	(\beta,\pm\infty)}$, where $\textup{LF}_\mathcal{S}^{(\beta,\pm\infty)}$ is the Liouville field on $\mathcal S$ with $\beta$ insertions at $\pm \infty$.
See Section~\ref{sec-pre} for the definition of Liouville fields with insertions.

We now describe our main quantum surfaces of interest, the quantum triangles. We first recall a special case that is already considered in~\cite{AHS21} and played a crucial rule there. 
For $\beta, \beta_3<Q$, the Liouville field measure  $\textup{LF}_\mathcal{S}^{
	(\beta,\pm\infty), (0,\beta_3)}$ 
is formally defined by
$\textup{LF}_\mathcal{S}^{
	(\beta,\pm\infty), (\beta_3,0)}(d\phi )=e^{\beta_3\phi(0)}\textup{LF}_\mathcal{S}^{
	(\beta,\pm\infty)} (d\phi)$, and can be made rigorous by regularization. Let $W,W_3>\frac{\gamma^2}{2}$ be determined by $\beta,\beta_3$ as in~\eqref{eq:betaW}, respectively. Sample $\phi$ from $\frac{1}{(Q-\beta)^2(Q-\beta_3)}\textup{LF}_\mathcal{S}^{
	(\beta,\pm\infty),(0,\beta_3)}$ and let $\QT(W,W,W_3)$ be the law of the three-pointed quantum surface $(\mathcal S, \phi, \pm \infty, 0)$. We call a sample from $\QT(W,W,W_3)$ a quantum triangle of weight $(W,W,W_3)$. Up to a multiplicative constant, the measure $\QT(W,W,W_3)$ agrees with $\mathcal{M}_{2, \bullet}^{\textup{disk}}(W;\beta_3)$ defined in~\cite{AHS21}; also see Definition~\ref{def-m2dot-alpha}. 
For $W=\frac{\gamma^2}2$, we define $\QT(\frac{\gamma^2}{2},\frac{\gamma^2}2,W_3)$ as the 
$W\downarrow \frac{\gamma^2}2$ limit of $\QT(W,W,W_3)$. For $W\in (0,\frac{\gamma^2}{2})$, 
following the definition of $\mathcal{M}_{2, \bullet}^{\textup{disk}}(W;\beta_3)$ in~\cite{AHS21}, we let $\QT(W,W,W_3)$ be the law of the three-pointed surface obtained by attaching an independent weight-$W$ 
two pointed disk at a quantum triangle of weight $(\gamma^2-W,\gamma^2-W, W_3)$.

For $W_1,W_2,W_3>0$, we define $\QT(W_1,W_2,W_3)$ as follows. For $W_1,W_2,W_3>\frac{\gamma^2}{2}$, set $\beta_i = \gamma+\frac{2-W_i}{\gamma}<Q$ and let $\textup{LF}_{\mathcal{S}}^{(\beta_1, +\infty), (\beta_2, -\infty), (\beta_3, 0)}$ be the Liouville field on $\mathcal S$ 
with insertion $\beta_1,\beta_2,\beta_3$ at $+\infty,-\infty$ and $0$, respectively. 
Sample $\phi$ from 
\[
\frac{1}{(Q-\beta_1)(Q-\beta_2)(Q-\beta_3)}\textup{LF}_{\mathcal{S}}^{(\beta_1, +\infty), (\beta_2, -\infty), (\beta_3, 0)}.
\]
We define  $\textup{QT}(W_1, W_2, W_3)$ to be the law of the 3-pointed quantum surface $(\mathcal{S}, \phi, +\infty, -\infty, 0)$. We call a sample from {$\QT(W_1,W_2,W_3)$} a quantum triangle of weight $(W_1,W_2,W_3)$.
Taking the limit $W_i\downarrow \frac{\gamma^2}{2}$, we can extend the definition of $\textup{QT}(W_1, W_2, W_3)$  to  $W_1,W_2,W_3\ge \frac{\gamma^2}{2}$; see Section~\ref{subsec:critical}. In this regime a quantum triangle has the disk topology. When $W_1\in (0,\frac{\gamma^2}{2})$ and $W_2,W_3\ge \frac{\gamma^2}2$, we define $\QT(W_1,W_2,W_3)$ by   attaching an independent weight-$W_1$ 
two pointed disk at a quantum triangle of weight $(\gamma^2-W_1,W_2, W_3)$. Using this method we extend the definition of $\textup{QT}(W_1, W_2, W_3)$  to  $W_1,W_2,W_3>0$.
We call the three marked points vertices of a quantum triangle and $W_i$ ($i=1,2,3$) is called the weight of the corresponding vertex. Given a sample of  $\QT(W_1,W_2,W_3)$, the geometry near the vertex of weight $W_i$ looks like the neighborhood of a marked point on a weight-$W_i$ quantum disk. We say a vertex is thick if its weight $W\ge \frac{\gamma^2}2$. We call it thin if $W\in (0,\frac{\gamma^2}2)$. See Figure~\ref{fig-qt} for an illustration.

\begin{figure}[ht]
	\centering
	\includegraphics[scale=0.43]{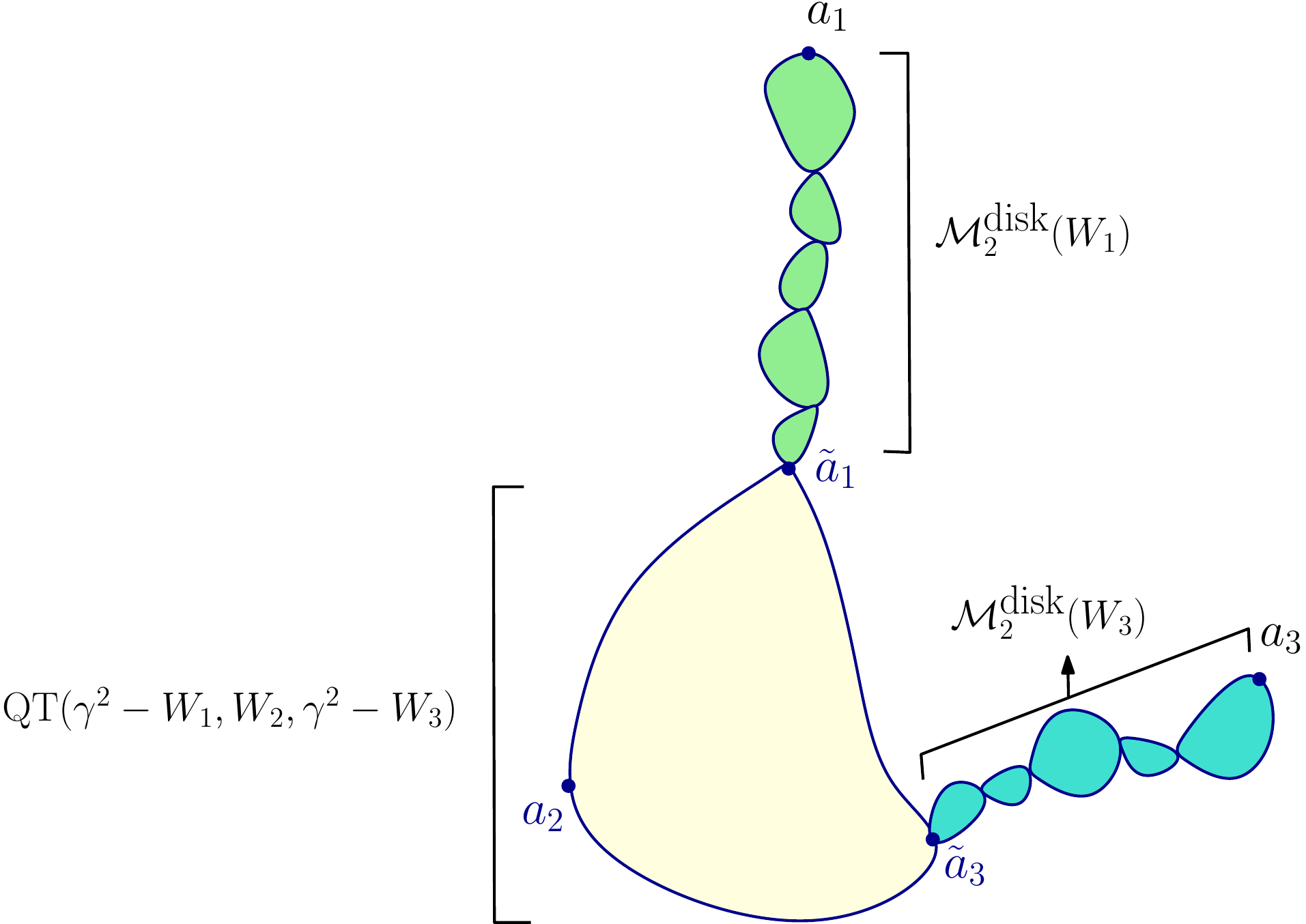}
	\caption{A sample of $\QT(W_1,W_2,W_3)$ with a thick vertex $a_2$   and  two thin vertices $a_1,a_3$, i.e.\ $W_2\geq\frac{\gamma^2}{2}$ and $W_1,W_3<\frac{\gamma^2}{2}$. The yellow surface is a quantum triangle with thick vertices $\tilde{a}_1,a_2,\tilde{a}_3$.  The two thin two-pointed quantum disks (colored green) are concatenated with the  yellow triangle at  $\tilde{a}_1$ and $\tilde{a}_3$.} \label{fig-qt}
\end{figure}

\subsection{Conformal welding of a quantum triangle and a  2-pointed quantum disk}\label{subsec:intro-2pt}
We first recall the conformal welding result for two-pointed quantum disk 
proved in~\cite{AHS20} based on its infinite-area variant in~\cite{DMS14}. For $W>0$, define $\mathcal{M}_2^{\textup{disk}}(W;\ell,r)$ via the disintegration 
$\mathcal{M}_2^{\textup{disk}}(W)=\iint_0^\infty \mathcal{M}_2^{\textup{disk}}(W;\ell,r) d\ell d r$, where $\mathcal{M}_2^{\textup{disk}}(W;\ell,r)$ is supported on surfaces with left boundary length $\ell$ and right boundary length $r$. Given a pair of quantum surfaces sampled from $\mathcal{M}_2^{\textup{disk}}(W_1;\ell_1,\ell)\times \mathcal{M}_2^{\textup{disk}}(W_2;\ell,\ell_2)$, we can conformally weld them together along the boundary with length $\ell$ to obtain a quantum surface decorated with a curve. We denote its law by $\mathrm{Weld}(\mathcal{M}_2^{\textup{disk}}(W_1;\ell_1,\ell),\mathcal{M}_2^{\textup{disk}}(W_2;\ell,\ell_2))$. 
For $\kappa>0$, $\rho_->-2$ and $\rho_+>-2$, chordal  $\SLE_\kappa(\rho_-;\rho_+)$ is a classical variant of SLE$_\kappa$ curve on simply connected domain between two boundary points, which will be recalled in Section~\ref{subsec:ig}.
Fix $W_1,W_2>0$, the conformal welding result for $\mathcal{M}_2^{\textup{disk}}(W_1)$ and $\mathcal{M}_2^{\textup{disk}}(W_2)$ says the following.
Let $(D,h,a,b)$ be an embedding of a two-pointed quantum disk sampled from $\mathcal{M}_2^{\textup{disk}}(W_1+W_2)$ with $a,b$ being the two boundary marked points. 
Let $\eta$ be a $\SLE_\kappa(\rho^-;\rho^+)$ curve on $D$ from $a$ to $b$ independent of $h$, where
\begin{equation}\label{eq:kapp}
	\kappa=\gamma^2\in (0,4); \quad \textrm{and}\quad \rho_-=W_1-2>-2; \quad \textrm{and} \quad \rho_+=W_2-2>-2.
\end{equation}
We write  $\mathcal{M}_2^{\textup{disk}}(W_1+W_2)\otimes  {\SLE}_\kappa(W_1-2;W_2-2)$ as the law of the curve-decorated surface $(D,h,\eta, a,b)$.
Then there is a constant $c>0$ such that
\begin{equation}\label{eq:2pt-weld}
	\mathcal{M}_2^{\textup{disk}}(W_1+W_2)\otimes  {\SLE}_\kappa(W_1-2;W_2-2)=c\textup{Weld} (\mathcal{M}_2^{\textup{disk}}(W_1),\mathcal{M}_2^{\textup{disk}}(W_2)),
\end{equation}
where $\textup{Weld} (\mathcal{M}_2^{\textup{disk}}(W_1),\mathcal{M}_2^{\textup{disk}}(W_2)) := \iiint_0^\infty \textup{Weld}( \mathcal{M}_2^{\textup{disk}}(W_1;\ell,\ell_1),\mathcal{M}_2^{\textup{disk}}(W_2;\ell_1,\ell_2))d\ell d\ell_1d\ell_{2}$ is called the conformal welding of $\mathcal{M}_2^{\textup{disk}}(W_1)$ and $\mathcal{M}_2^{\textup{disk}}(W_2)$.

The bulk of our paper is devoted to proving  that the conformal welding of a quantum triangle and a two-pointed quantum disk gives another quantum triangle with an SLE curve whose law is explicit. Similarly as in~\eqref{eq:2pt-weld}, we define $\QT(W_1,W_2,W_3;\ell_1,\ell_2,\ell_3)$ via the disintegration 
$\QT(W_1,W_2,W_3)=\int\QT(W_1,W_2,W_3;\ell_1,\ell_2,\ell_3) d\ell_1d\ell_2d\ell_3$. Here $\ell_i$ is the length between the weight-$W_i$ and weight-$W_{i+1}$ vertices where $i=1,2,3$ and  $3+1$  is identified with  $1$. Fix $W,W_1,W_2,W_3>0$, given a pair of quantum surfaces sampled from $\mathcal{M}_2^{\textup{disk}}(W;\ell_1,\ell)\times \QT(W_1,W_2,W_3;\ell,\ell_2,\ell_3)$, we conformally weld them together along the boundary with length $\ell$ to obtain a quantum surface decorated with a curve and three marked points, whose law is denoted by $\mathrm{Weld}(\mathcal{M}_2^{\textup{disk}}(W;\ell,\ell_1),\QT(W_1,W_2,W_3;\ell,\ell_2,\ell_3))$. We define the conformal welding of $\mathcal{M}_2^{\textup{disk}}(W)$ and $\QT(W_1,W_2,W_3)$ by
\begin{equation}\label{eq:def-weld}
	\textup{Weld} (\mathcal{M}_2^{\textup{disk}}(W),\QT(W_1,W_2,W_3)) :=
	\iiiint_0^\infty \Wd(\Md_2(W;\ell_1,\ell),\QT(W_1,W_2,W_3;\ell,\ell_2,\ell_3))d\ell d\ell_1d\ell_2 d\ell_3.
\end{equation} 
Similar to~\eqref{eq:2pt-weld}, the law of the three pointed quantum surface for $\textup{Weld} (\mathcal{M}_2^{\textup{disk}}(W),\QT(W_1,W_2,W_3))$ is proportional to $\QT(W+W_1, W+W_2, W_3)$. To describe the law of the SLE interface, we need chordal $\SLE_\kappa$ with multiple boundary forces points, which is a more general variant of chordal $\SLE$ that arises in imaginary geometry~\cite{MS16a}. We let $\SLE_\kappa(\rho_-;\rho_+,\rho_1)$  be the law of a 
chordal $\SLE_\kappa $ on the upper half plane $\bbH$ from $0$ to $\infty$ with forces points at $0^-,0^+,1$, whose weight are $\rho_-,\rho_+,\rho_1$ respectively.
We will recall its definition in Section~\ref{sec-pre-ig}, for now it is sufficient to know that it is a random simple curve on $\bbH$ from $0$ to $\infty$, with an additional boundary marked points $0^-,0^+,1$ called force points, each of which is labeled by a number called weight. (This is not to be confused with the weight for a vertex of a quantum triangle).   Our previous notion of chordal $\SLE_\kappa(\rho_-;\rho_+)$ on $\bbH$ from $0$ to $\infty$ is the special case where $\rho_1=0$.

Our first welding result (Theorem~\ref{thm:M+QT}) says that when $W_1,W_2,W_3$ satisfies $W_1+2=W_2+W_3$,  the interface in $\textup{Weld} (\mathcal{M}_2^{\textup{disk}}(W),\QT(W_1,W_2,W_3))$ is a  chordal $\SLE_\kappa(W-2;W_2-2,W_1-W_2)$ curve if  $W+W_1, W+W_2, W_3$  are all thick weights (namely $\ge\frac{\gamma^2}{2}$); and if some of $W+W_1, W+W_2, W_3$ are thin, the analogous result holds after natural modifications.
Let us first assume $W+W_1, W+W_2, W_3$ are all thick so that 
a sample from  $\QT(W+W_1, W+W_2, W_3)$
can be embedded as {$(\bbH,h,\infty,0, 1)$, where the points $\infty,0,1$ correspond to the weight $W+W_1,W+W_2,W_3$ vertices.} Sample $\eta$ from ${\SLE}_\kappa(W-2;W_2-2,W_1-W_2)$  independently from $h$. We write  $ \QT(W+W_1, W+W_2, W_3)\otimes {\SLE}_\kappa(W-2;W_2-2,W_1-W_2)$ as the law of the curve-decorated surface {$(\bbH,h,\eta,\infty,0,1)$}. 
Now if $W_3\in (0,\frac{\gamma^2}2)$ instead, then a sample of $\QT(W+W_1, W+W_2, W_3)$ can be obtained by attaching a weight $W_3$ two-pointed quantum disk to a quantum triangle of weight $(W+W_1, W+W_2, \gamma^2-W_3)$  at the weight $(\gamma^2-W_3)$ vertex. We now embed the
weight $(W+W_1, W+W_2, \gamma^2-W_3)$ triangle to $(\bbH, 0,\infty,1)$ and run an independent ${\SLE}_\kappa(W-2;W_2-2,W_1-W_2)$ curve from $0$ to $\infty$. We still write  $ \QT(W+W_1, W+W_2, W_3)\otimes {\SLE}_\kappa(W-2;W_2-2,W_1-W_2)$ as the law of the resulting curve-decorated surface with the two-pointed quantum disk attached.
We will give the precise definition of this law for the case when $W+W_1 $ or $W+W_2$ is thin
in Section~\ref{sec-6-thin-qt}.  See Figure~\ref{fig-qt-weld} for illustrations of various cases. 
%{In the statement below we also temporarily skip the case if any of the weights equal $\frac{\gamma^2}{2}$. This is because for technical reasons as explained in  Section~\ref{subsec:critical}, in this paper we only define the measure $\QT(W_1,W_2,W_3)$ in a certain range of $(W_1,W_2,W_3)$ when $\frac{\gamma^2}{2}\in \{W_1,W_2,W_3\}$. We shall present the corresponding welding result for this case in Theorem~\ref{thm:g2/2}.}

%\xin{[Explain why excluding $\gamma^2/2$. Refer to Section~\ref{subsec:critical}. Definition~\ref{def-g22-QT} only cover a range.  }
\begin{theorem}\label{thm:M+QT}
	Suppose $W,W_1,W_2,W_3>0$  with 
	$W_1+2=W_2+W_3$. 
	Then there exists some constant $c = c_{W,W_1,W_2}\in (0,\infty)$ such that
	\begin{equation}\label{eq:M+QT}
		\QT(W+W_1, W+W_2, W_3)\otimes {\SLE}_\kappa(W-2;W_2-2,W_1-W_2)= c \textup{Weld} (\mathcal{M}_2^{\textup{disk}}(W),\QT(W_1,W_2,W_3)).
	\end{equation}	 
\end{theorem}

\begin{figure}[ht]
	\centering
	\begin{tabular}{ccc} 
		\includegraphics[scale=0.56]{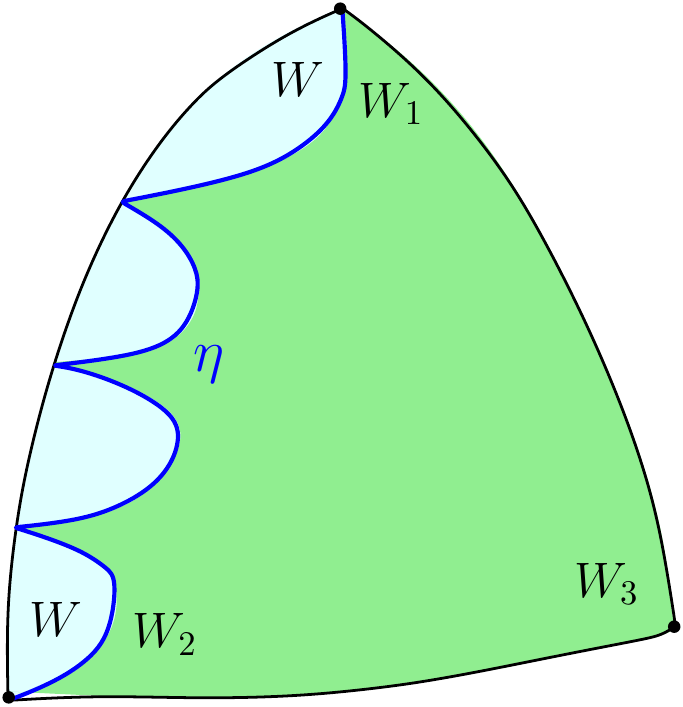}
		& & \ \ 
		\includegraphics[scale=0.53]{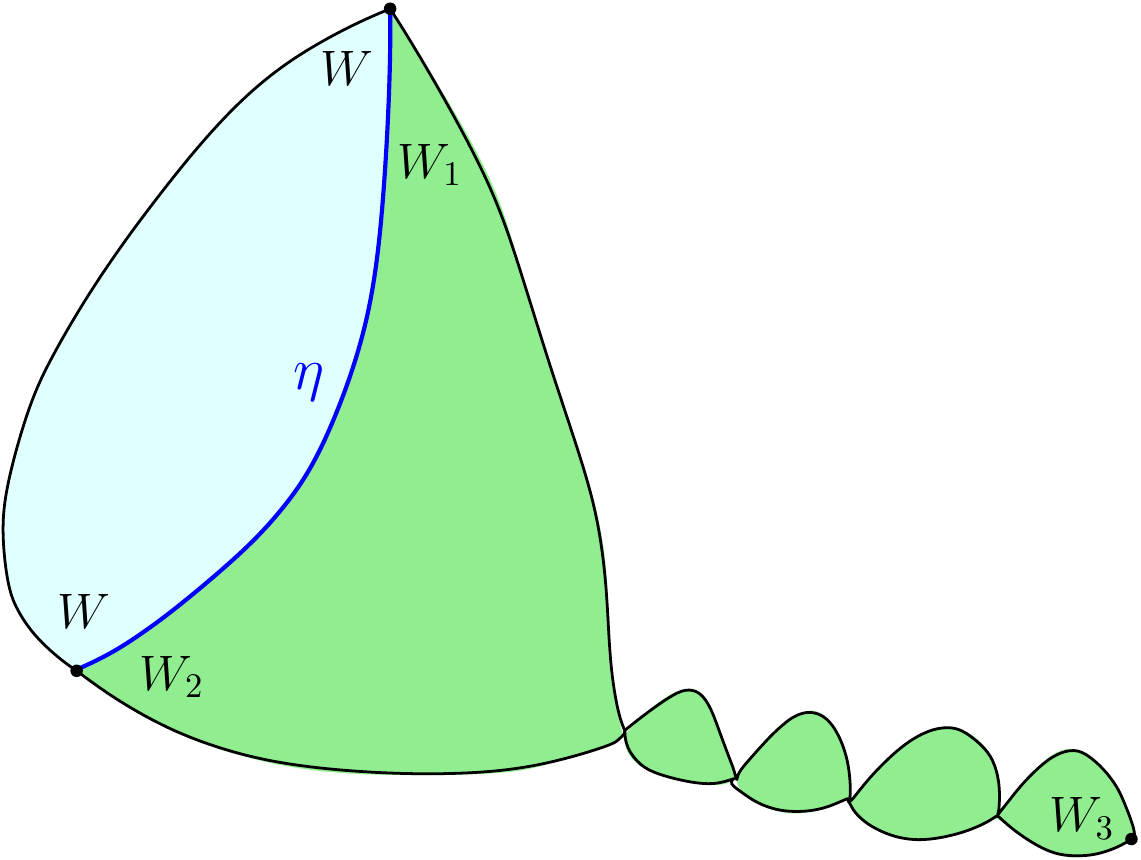}\\
		\includegraphics[scale=0.56]{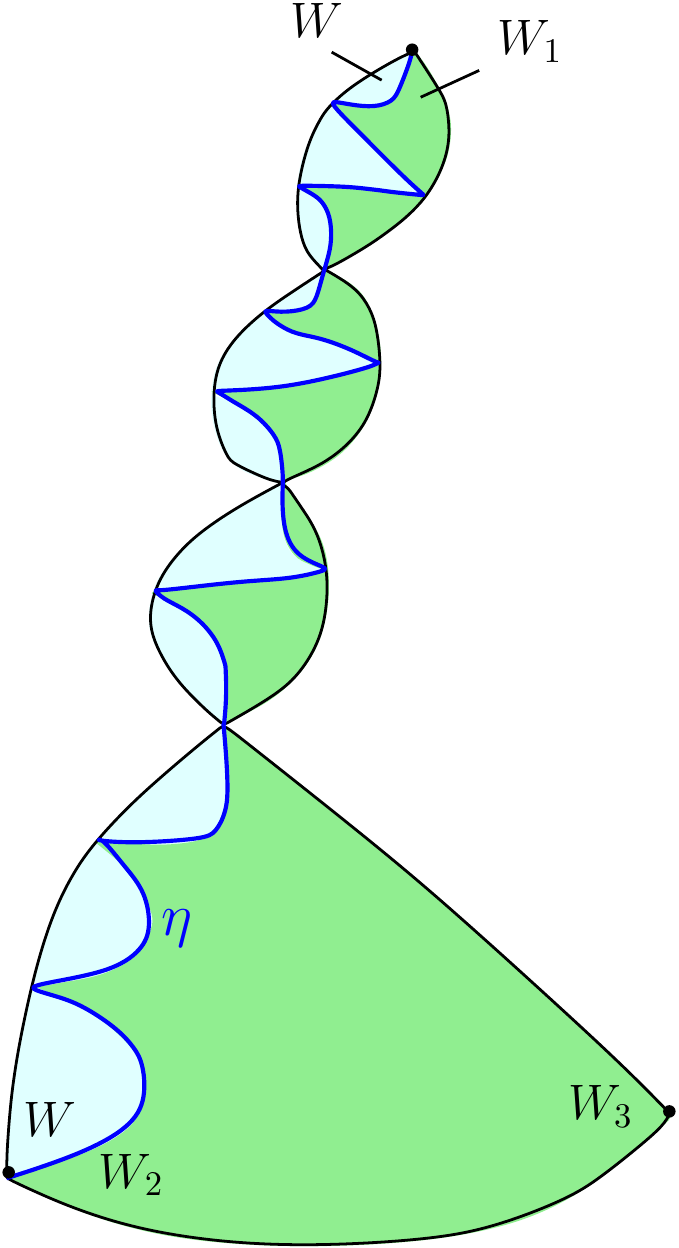}
		& &
		\includegraphics[scale=0.56]{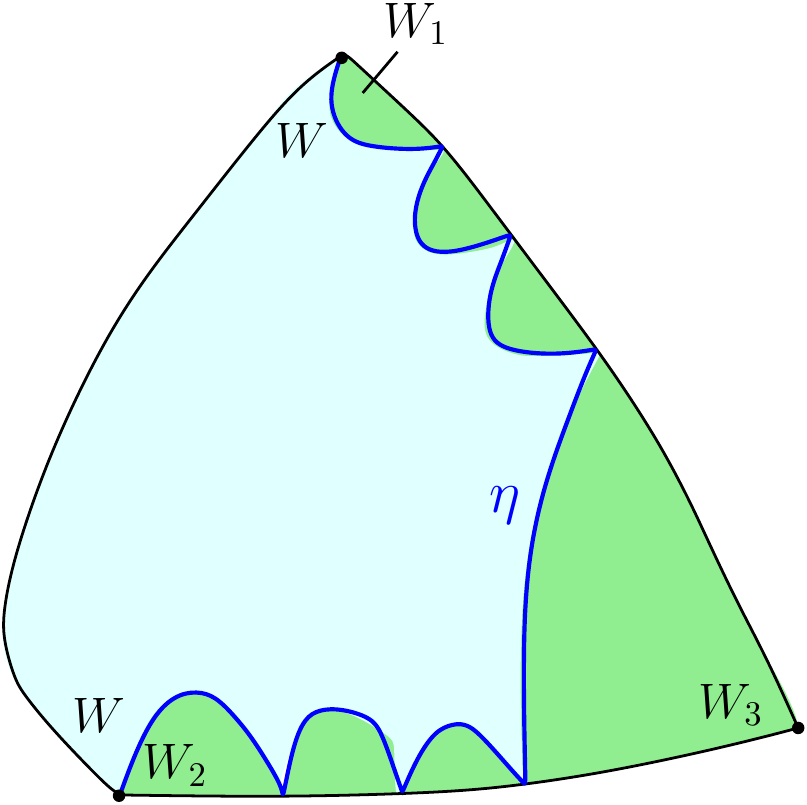}
	\end{tabular}
	\caption{Illustration of some topological scenarios in Theorem~\ref{thm:M+QT}. (a) $W\in (0,\frac{\gamma^2}2)$ and $W_1,W_2,W_3\geq \frac{\gamma^2}{2}$; (b) $W,W_1,W_2\geq \frac{\gamma^2}{2}$ and $W_3\in (0,\frac{\gamma^2}2)$; (c) $W+W_1\in (0,\frac{\gamma^2}{2})$ and $W_2,W_3\geq \frac{\gamma^2}{2}$; (d) $W,W_3\geq \frac{\gamma^2}{2}$ and $W_1,W_2\in (0,\frac{\gamma^2}{2})$. }\label{fig-qt-weld} 
\end{figure}

As we will see in Theorem~\ref{thm:disk}, quantum triangles whose weight satisfy $W_1-W_2=W_3-2$ are those that will appear naturally in imaginary geometry on quantum disk with boundary typical points.
The conformal welding result for $W_1-W_2\neq W_3-2$ can be easily deduced from Theorem~\ref{thm:M+QT} following arguments in~\cite{AHS21}.  Suppose $\eta$ is a curve from 0 to $\infty$ on $\mathbb{H}$ that does not touch $1$. Let $D_\eta$ be the component of $\mathbb{H}\backslash \eta$ containing $1$, and $\psi_\eta$ is the unique conformal map from the component $D_\eta$ to $\mathbb{H}$ fixing 1 and sending the first (resp.\ last) point on $\partial D_\eta$ hit by $\eta$ to 0 (resp.\ $\infty$). Define the measure $\widetilde{\SLE}_\kappa(\rho_-;\rho_+,\rho_1;\alpha)$ on curves from 0 to $\infty$ on $\mathbb{H}$ as follows.
\begin{equation}\label{eqn-sle-CR}
	\frac{d\widetilde{\SLE}_\kappa(\rho_-;\rho_+,\rho_1;\alpha)}{d{\SLE}_\kappa(\rho_-;\rho_+,\rho_1)}(\eta) = \psi'_{\eta}(1)^\alpha.
\end{equation}
Then we have the following extension  of Theorem~\ref{thm:M+QT}.
\begin{theorem}\label{thm:M+QTp-rw}
	Suppose $W,W_1,W_2,W_3>0$. 
	Set \begin{equation}\label{eqn-alpha}
		\alpha = \frac{W_3+W_2-W_1-2}{4\kappa}(W_3+W_1+2-W_2-\kappa).    
	\end{equation} 
	Then with the same constant  $c = c_{W,W_1,W_2}\in (0,\infty)$ as  in Theorem~\ref{thm:M+QT}, we have 
	\begin{equation}\label{eq:M+QT2}
		\QT(W+W_1, W+W_2, W_3)\otimes  \wt{\SLE}_\kappa(W-2;W_2-2,W_1-W_2;\alpha)= c \textup{Weld} (\mathcal{M}_2^{\textup{disk}}(W),\QT(W_1,W_2,W_3)).
	\end{equation}
\end{theorem} 
We will give the precise definition of $\QT(W+W_1, W+W_2, W_3)\otimes  \wt{\SLE}_\kappa(W-2;W_2-2,W_1-W_2;\alpha)$ in Section~\ref{sec-6-thin-qt}, which again requires a proper interpretation when  some of  $W+W_1, W+W_2, W_3$ are thin. 

The proof of Theorems~\ref{thm:M+QT} and~\ref{thm:M+QTp-rw} is divided into three steps that are carried  out in Sections~\ref{sec:WW2}---\ref{sec-6-thin-qt}, respectively. The first step (Proposition~\ref{thm-w-2u}) intuitively says the following. Suppose $W_1>\gamma^2/2$ and $W_2\in (0,\gamma^2/2)$ in the welding equation~\eqref{eq:2pt-weld} for $\mathcal{M}_2^{\textup{disk}}(W_1)$ and $\mathcal{M}_2^{\textup{disk}}(W_2)$, if a cut point is added to the weight-$W_2$ disk so that it is split into two independent copies of $\mathcal{M}_2^{\textup{disk}}(W_2)$, then the addition of the third point create a quantum triangles with weights compatible with Theorem~\ref{thm:M+QTp-rw}. Proposition~\ref{thm-w-2u} is proved via a limiting procedure based on results from~\cite{AHS21}. Quantum triangles that we are able to identify in this step all have two vertices of equal weight.

The second and third steps require essential new techniques for proving welding results. First of all, there is no existing mechanism to identify the law of a quantum surface obtained from welding that has  three boundary marked points of three different log singularities. In Step 2 (Section~\ref{sec:resampling}), we provide a Markovian characterizations of the three-pointed Liouville field that allows us to identify the law of quantum triangles after welding $\mathcal{M}_2^{\textup{disk}}(W)$ and $\QT(W_1,W_2,W_3)$ as in Theorem~\ref{thm:M+QT}. This proves Theorems~\ref{thm:M+QT} and~\ref{thm:M+QTp-rw}  in a restricted range of weights. The range constraint is removed in Step 3 (Section~\ref{sec-6-thin-qt}). For this purpose it is crucial to work under the setting of Theorem~\ref{thm:M+QTp-rw}, because we need the freedom to perform conformal welding along different edges of  the same quantum triangle where the condition $W_1+2=W_2+W_3$ in Theorem~\ref{thm:M+QT} cannot be satisfied simultaneously for every welding. Techniques in Sections~\ref{sec:resampling} and~\ref{sec-6-thin-qt} are quite robust and will play a crucial role in the subsequent work~\cite{AY22} proving more general welding results for quantum triangles; see Section~\ref{subsec:intro-outlook}.

\subsection{Imaginary geometry on a quantum disk with multiple boundary points}\label{subsec:intro-disk}

When welding multiple two-pointed quantum disks, the interfaces are a set of flow lines in imaginary geometry. We briefly recall the flow line construction from~\cite{MS16a}. For $\kappa\in (0,4)$ and $\rho_-,\rho_+>-2$, set 
\begin{equation}\label{eq:IG}
	\chi=\frac{2}{\sqrt{\kappa}} -\frac{\sqrt{\kappa}}{2}  \quad \textrm{and}\quad \lambda=\frac{\pi}{\sqrt{\kappa}};\quad \quad   \quad \lambda_-=-\lambda(1+ \rho_-) \quad \textrm{and}\quad  \lambda_+=\lambda(1+\rho_+).
\end{equation}
Let $\mathfrak h$ be a GFF on the upper half plane $\bbH$  with Dirichlet boundary condition such that the  boundary value  is  $\lambda_-$ between $(-\infty,0)$ and   $\lambda_+$ between $(0,\infty)$. 
Then there exists a coupling between $\mathfrak h$ and an $\SLE_\kappa(\rho_-;\rho_+)$ curve $\eta$ on $\bbH$ from 0 to $\infty$, under which $\eta$ is determined by $\mathfrak h$. Although $\mathfrak h$ is only a generalized function, the curve $\eta$ can be interpreted as the flow line from 0 to $\infty$ of the random vector field $e^{\frac{i\mathfrak h}{\chi}}$. For $\theta\in (-\frac{\lambda+\lambda_+}{\chi},\frac{\lambda-\lambda_-}{\chi})$,  we can also consider the flow line of $e^{\frac{i\mathfrak h}{\chi}+\theta}$ which is the chordal  $\SLE_\kappa(- \frac{\lambda_-+\chi \theta}{\lambda}-1; \frac{\lambda_++\chi\theta}{\lambda}-1)$ determined by $\mathfrak h +\chi \theta$ via~\eqref{eq:IG} with $(\lambda_-,\lambda_+)$ replaced by $(\lambda_-+\chi\theta,\lambda_++\chi\theta)$. Varying $\theta$, we have  multiple SLE curves between $0$ and $\infty$ with force points at $0^-$ and $0^+$coupled together. As generalization of~\eqref{eq:2pt-weld}, the conformal welding result from~\cite{AHS20} for multiple two-pointed quantum disks can be stated as follows. Fix $\gamma\in (0,2)$ and $\kappa=\gamma^2$. Consider $W=\sum_{i=0}^{n}W_i$ with $W_i>0$. Let $(\bbH, h, 0,\infty)$ be an embedding of a sample from $\mathcal M^{\mathrm{disk}}_2(W)$. Let $\mathfrak h$ be a GFF with boundary condition $\lambda_-=-(W-2)\lambda$ and $\lambda_+=0$.
Let $\theta_1>\theta_2>\cdots >\theta_n$ be defined by
\begin{equation}\label{eq:theta}
	W_i=    \frac{(\theta_i-\theta_{i+1}) }{\lambda} \quad \textrm{for }i<n \quad \textrm{and}\quad W_n=1+\frac{\theta_n \chi}{\lambda}.
\end{equation}
For $1\le i\le n$, let $\eta_i$ be the flow line of $e^{\frac{i \mathfrak h }{\chi} +\theta_i}$ from $0$ to $\infty$. Then the law of the decorated quantum surface  $(\bbH, h, \eta_1,\cdots, \eta_n)$ is given by the conformal welding of $\mathcal{M}_2^{\textup{disk}}(W_1),
\cdots, \mathcal{M}_2^{\textup{disk}}(W_n)$ in that order.

\begin{figure}[ht]
	\centering
	\includegraphics[scale=0.6]{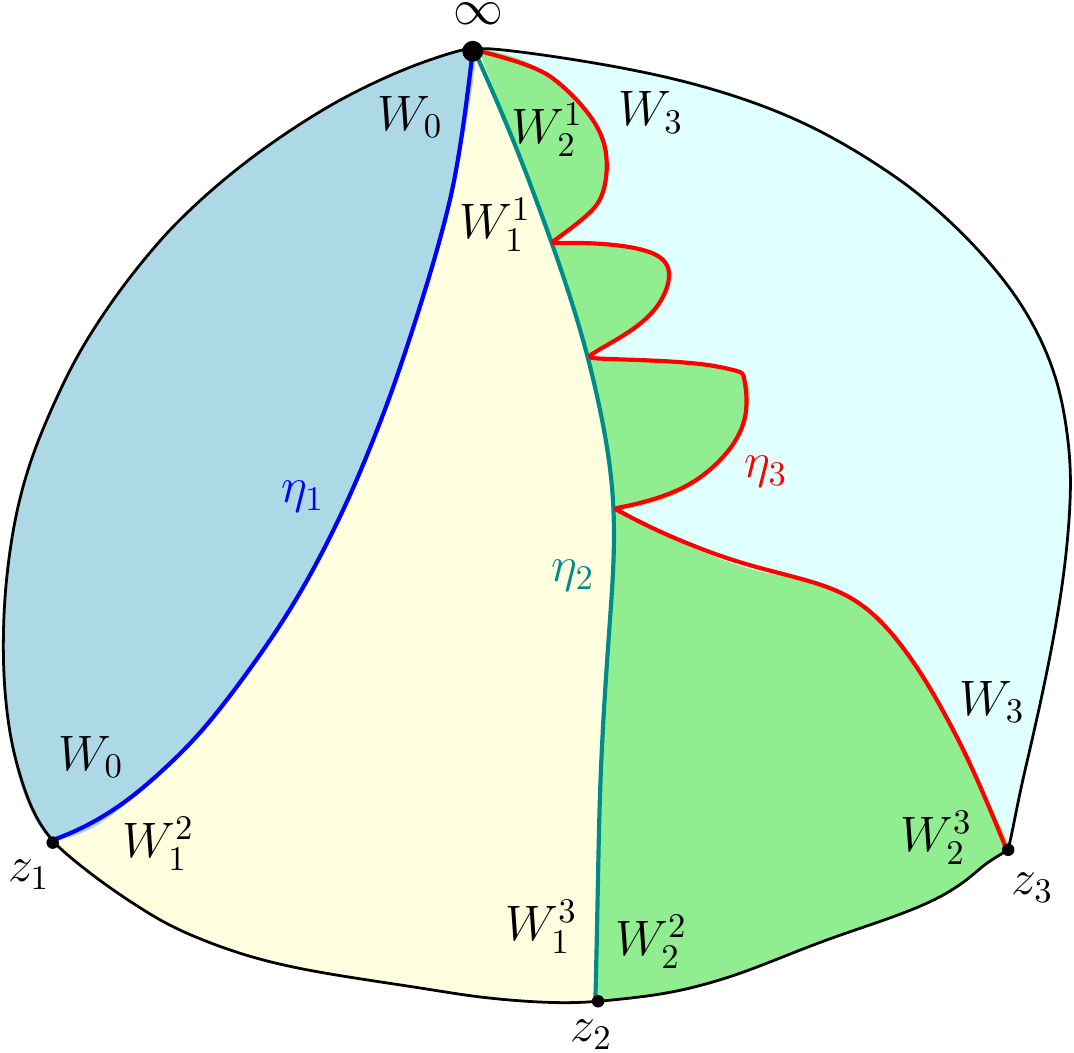}
	\caption{An illustration of Theorem~\ref{thm:disk} where $n=3$. The three flow lines cut the quantum disk $\mathrm{QD}_{0,4}$ into four parts: a weight $W_0$ and a weight $W_3$ quantum disk, a weight $(W^1_1,W_1^2,W_1^3)$ thick quantum triange and a weight $(W_{2}^1,W_{2}^2,W_{2}^3)$ thin triangle.}\label{fig-thm-disk}
\end{figure}

In the framework of imaginary geometry in~\cite{MS16a}, it is possible to emanate flow lines from
different boundary points with the same target point. Unfortunately the conformal welding result for two-pointed disk falls short of producing this rich picture. Thanks to the introduction of quantum triangles, this can now be achieved as in Theorem~\ref{thm:disk}. Although our result can be stated more generally, we restrict ourselves to the following neat setting to make the point. We consider a quantum disk with more than two  typical boundary marked points. For $n\ge 2$, a quantum disk with $n+1$ {quantum typical points} can sampled as follows. First sample a two-pointed quantum disk from the tilted measure $L^{n-1}\mathrm{QD}_{0,2}$ where $L$ is the total quantum length of a sample from  $\mathrm{QD}_{0,2}$; then sample $n-1$ additional marked points independently according to the boundary length measure. 
Moreover, we consider the imaginary geometry whose field has zero boundary condition, namely on $\bbH$ the boundary value of $\mathfrak h$ is $\lambda_-=\lambda_+=0$ on $\mathbb R$, hence the origin is not special anymore.
\begin{theorem}\label{thm:disk} 
	Fix $\gamma\in (0,2)$, $\kappa=\gamma^2$, $\lambda = \frac{\pi}{\gamma}$ and $\chi=\frac{2}{\gamma}-\frac{\gamma}{2}$. 
	For $n\ge 1$, let $(\bbH,h,\infty,z_1,\cdots, z_n)$ be an embedding of a quantum disk with $n+1$ boundary marked points and $z_1<...<z_n$. Let $\mathfrak h$ be a zero-boundary Gaussian free field  on $\bbH$ independent of $h,z_1,\cdots, z_n$.
	Fix  $\frac{\lambda}{\chi}>\theta_1 \xin{>} \theta_2>\cdots \xin{>}  \theta_n>-\frac{\lambda}{\chi}$. For $1\le i\le n$, let $\eta_i$ be the flow line of $e^{\frac{i \mathfrak h }{\chi} +\theta_i}$ starting from $z_i$. Then the law of the decorated quantum surface  $(\bbH, h, \eta_1,\cdots, \eta_n)$ is given by the conformal welding of 
	$$\Md_2(W_0), \QT(W^1_1,W_1^2,W_1^3), \cdots,\QT(W_{n-1}^1,W_{n-1}^2,W_{n-1}^3),\Md_2(W_n)\quad \textrm{in that order},$$
	with $W_0 = 1-\frac{\theta_1\chi}{\lambda}$, $W_i^1 = \frac{(\theta_i-\theta_{i+1})\chi}{\lambda}$, $W_i^2 = 1+\frac{\theta_i\chi}{\lambda}$, $W_i^3 = 1-\frac{\theta_{i+1}\chi}{\lambda}$ for $i = 1, ..., n-1$ and ${W}_n = 1+\frac{\theta_n\chi}{\lambda}$. 
\end{theorem}
%The range of $\theta$ is chosen such that $\eta_1$ can be arbitrarily close left boundary and $\eta_n$ can reach right boundary.

Every quantum triangle appearing in Theorem~\ref{thm:disk} satisfies the weight constraint in Theorem~\ref{thm:M+QT}. As we will show in Section~\ref{sec-pfthm1.1}, Theorem~\ref{thm:disk} is an easy consequence of Theorem~\ref{thm:M+QT}. By a limiting argument, it is also possible to allow $\theta_i=\theta_{i+1}$, in which case $\eta_i$ and $\eta_{i+1}$ will merge before hitting the target. The result can also be refined by allowing $z_i=z_{i+1}$. We will not carry out these extensions explicitly.

\subsection{Applications of Theorem~\ref{thm:M+QTp-rw} to $\SLE_\kappa(\rho^-; \rho^+; \rho_1)$} \label{subsec-SLE-weighted}
As demonstrated in  \cite{AHS21}, conformal welding results such as Theorem~\ref{thm:M+QTp-rw} can be used to derive the law of the conformal derivative  $\psi'(1)$ in~\eqref{eqn-sle-CR}, which is Theorem~\ref{thm:sle-conf-radius} below. Define the function 
\begin{equation}\label{eqn-f-function}
	F(x,\kappa,\rho_-,\rho_+,\rho_1) := \frac{\Gamma_{\frac{\sqrt{\kappa}}{2}}(\frac{2}{\sqrt{\kappa}}-\frac{\sqrt{\kappa}}{2} + \frac{\rho_+}{\sqrt{\kappa}}+\frac{x}{2} ) \Gamma_{\frac{\sqrt{\kappa}}{2}} (\frac{4}{\sqrt{\kappa}} +\frac{\rho_++\rho_1}{\sqrt{\kappa}} -\frac{x}{2}) }{\Gamma_{\frac{\sqrt{\kappa}}{2}}(\frac{4}{\sqrt{\kappa}}-\frac{\sqrt{\kappa}}{2}+\frac{\rho_++\rho_-}{\sqrt{\kappa}}+\frac{x}{2} ) \Gamma_{\frac{\sqrt{\kappa}}{2}} (\frac{6}{\sqrt{\kappa}}+\frac{\rho_-+\rho_++\rho_1}{\sqrt{\kappa}}-\frac{x}{2})   }.
\end{equation}where
$\Gamma_{b}(z)$ is the double gamma function that appears frequently in LCFT; see \eqref{eqn-double-gamma} for the definition. 
\begin{theorem}\label{thm:sle-conf-radius}
	Fix $\kappa\in (0,4)$, $\rho_-, \rho_+>-2$ and $\rho_1>-2-\rho_+$. % {$|\rho_1|< 4-\frac{\kappa}{2}$ with $\rho_-+\rho_+\le 2$, $\rho_-+\rho_++\rho_1\le 2$}.
	Let $\alpha_0 = \frac{1}{\kappa}(\rho_++2)(\rho_++\rho_1+4-\frac{\kappa}{2})$. For any $\alpha<\alpha_0$, let $\beta$ be a solution to \begin{equation}\label{eqn-alpha-beta-relation}
		\frac{\sqrt{\kappa}(\sqrt{\kappa}-\beta)-\rho_1}{4\kappa}\big(4-\rho_1-\sqrt{\kappa}\beta\big)=\alpha .
	\end{equation}
	Let $\eta$ be an SLE$_\kappa(\rho_-;\rho_+,\rho_1)$ on $\bbH$ from 0 to $\infty$ with force points at $0^-,0^+,1$, and $\psi$ be as in~\eqref{eqn-sle-CR}. Then
	\begin{equation}\label{eqn-sle-conf-radius}
		\bbE[\psi_{\eta}'(1)^\alpha] = \frac{F(\beta+\frac{\rho_1}{\sqrt{\kappa}}, \kappa,\rho_-, \rho_+, \rho_1 )}{F(\sqrt{\kappa}, \kappa,\rho_-, \rho_+, \rho_1)  }\quad \textrm{for }\alpha<\alpha_0.
	\end{equation}
	Moreover, if $\alpha \ge \alpha_0$ then $\bbE[\psi_{\eta}'(1)^\alpha] =\infty$.
\end{theorem}
Theorem~\ref{thm:sle-conf-radius} generalizes the main result in~\cite{AHS21}, which corresponds to the case $\rho_1=0$. The result in~\cite{AHS21} is stated for all $\kappa>0$. Our Theorem~\ref{thm:sle-conf-radius} can also be extended similarly using the same argument based on SLE duality. 
By the definition of the measure $\widetilde{\SLE}_\kappa(\rho_-;\rho_+,\rho_1;\alpha)$ in~\eqref{eqn-sle-conf-radius}, $\bbE[\psi_{\eta}'(1)^\alpha]$ equals its total mass, which can be computed from the conformal welding identity~\eqref{eq:M+QT} combined with the integrability of boundary LCFT~\cite{RZ20b}. 
See Section~\ref{sec:application-sle-conf-radius} for its proof. See the introduction of~\cite{AHS21} for a literature review of integrability results for SLE.
Theorem~\ref{thm:M+QTp-rw} also makes the following reversibility of  $\SLE_\kappa(\rho_-;\rho_+,\rho_1)$ transparent.  
\begin{theorem}\label{thm-sle-reverse}
	Fix $\rho_+>-2$, $\rho_->-2$, and $\rho_1>-2-\rho_+$. Let $\eta$ be an $\SLE_\kappa(\rho_-;\rho_+,\rho_1)$ curve in $\bbH$ from $0$ to $\infty$ with force located at $0^-,0^+$ and 1. Let $\bar\eta$ be the image of the time reversal of  $\eta$ under  $z\mapsto \frac{1}{\bar z}$.
	Then the law of $\bar \eta$ is  the probability measure  proportional to $\widetilde{\SLE}_\kappa(\rho_-;\rho_++\rho_1,-\rho_1;\frac{\rho_1(4-\kappa)}{2\kappa})$. 
\end{theorem}
For $\rho_-=0$, Theorem~\ref{thm-sle-reverse} follows from the main result in~\cite{Zhan19}. Based on this we prove the $\rho_- \neq 0$ case in Section~\ref{subsec:ig} using imaginary geometry. Although the proof does not use LQG, 
we first guessed the statement of Theorems~\ref{thm:M+QT} and~\ref{thm:M+QTp-rw} and then use them to guess the statement  Theorem~\ref{thm-sle-reverse} before proving it.  Indeed, if $\SLE_\kappa(\rho_-;\rho_+,\rho_1)$ is the interface of a sample from $\mathrm{Weld}(\Md_2(W),\QT(W_1,W_2,W_3))$ from the weight $W+W_1$ vertex to the weight $W+W_2$ vertex as in Theorem~\ref{thm:M+QT}, then by Theorem~\ref{thm:M+QTp-rw}, the law of the interface from  the weight $W+W_2$ vertex to the weight $W+W_1$ vertex in $\mathrm{Weld}(\Md_2(W),\QT(W_2,W_1,W_3))$ is
$\widetilde{\SLE}_\kappa(\rho_-;\rho_++\rho_1,-\rho_1;\frac{\rho_1(4-\kappa)}{2\kappa})$ with $\alpha=\frac{\rho_1(4-\kappa)}{2\kappa}$. Once proved,
Theorem~\ref{thm-sle-reverse} is in turn used  as a tool to prove Theorems~\ref{thm:M+QT} and~\ref{thm:M+QTp-rw} in the full range of parameters. 

As another application of  Theorem~\ref{thm:M+QTp-rw}, let $(\eta_1,\eta_2)$ be the two interfaces in the conformal welding of  a two-pointed quantum disk, a quantum triangle, another two-pointed quantum disk, in that order. Then for $i=1,2$ the marginal law of $\eta_i$ and the conditional law of $\eta_{3-i}$ are $\widetilde{\SLE}_\kappa(\rho_-;\rho_+,\rho_1;\alpha)$ curves with various parameters. This is an instance of commutation  relation for  in the spirit of ~\cite{dubedat2007commutation,zhan2008duality}. 
See Section~\ref{sec:application} for the precise statement and its proof.

\subsection{Perspectives and related work}\label{subsec:intro-outlook}
We describe a few  subsequent  works and future directions concerning quantum triangles and their various applications. 
\begin{itemize}
	\item (Integrability of quantum triangles.) Let $A$ be the area of a sample of $\QT(W_1,W_2,W_3)$ and $L_1,L_2,L_3$ be the three boundary lengths. Then  $(\mu,\mu_1,\mu_2,\mu_3)\mapsto \QT(W_1,W_2,W_3)[e^{-\mu A-\sum_{i=1}^3\mu_i L_i}]$ gives the boundary three-point structure constant of Liouville conformal field theory. For  $\mu=0$ an exact formula was obtained in~\cite{RZ20b} and is used in our proof of Theorem~\ref{thm:sle-conf-radius}. 
	With Remy and Zhu, the first and the second authors of this paper will prove the conjecture of Ponsot and Teschner~\cite{Ponsot-Teschner} that the exact expression for $\mu>0$ is given by the Virasoro fusion kernel. 
	\item (Integrability of imaginary geometry coupled with LQG.) The aforementioned integrability of quantum triangles, and the welding results in this paper, and the mating of trees theory~\cite{DMS14} can together be used to study the integrablity of  imaginary geometry coupled with LQG. For example, a class of permutons (i.e. scaling limit of permutation)  called the skew Brownian permutons were recently introduced in~\cite{Borga-skew}, with the Baxter permuton~\cite{Borga-baxter} as a special case. As  shown in \cite[Proposition 1.14]{BHSY22},   the expected portion of inversions for these permutons is related to a natural quantity in imaginary geometry coupled with LQG. In a subsequent work we will derive an exact expression for this quantity. See~\cite{BGS-meander}  for other applications of  SLE/LQG to  permutons.

	\item (Reversibility of SLE.) As explained in Section~\ref{subsec-SLE-weighted} conformal welding of quantum triangles is closely related to the  reversal property of  $\SLE_\kappa(\rho^-; \rho^+; \rho_1)$. {Recently Zhan~\cite{Zhan19} and the third named author~\cite{Yu22} gave a description of the law of the time reversal of chordal SLE curves with multiple force points. 
		We believe that the conformal welding of multiple quantum disks and quantum triangles can provide an alternative and more robust approach to such results, which can extend to other cases such as the time reversal of radial $\SLE$ with multiple force points.}

	\item (Extensions to $\kappa\in (4,8)$ and integrability of non-simple CLE.) Our conformal welding results have nontrivial extension to $\SLE_\kappa$ curves with $\kappa\in(4,8)$, which corresponds to counter flow lines in imaginary geometry~\cite{MS16a}. These results will be used to study the integrability of conformal loop ensemble (CLE) with $\kappa\in (4,8)$ where the loops are non-simple. In particular, we aim at extending results in~\cite{AS21,ARS22} for simple CLE, and deriving exact results specific to the non-simple regime such as the probability that an outermost loop of a CLE on the disk touches the boundary.
	\item (Interior flow lines.) Imaginary geometry with interior flow lines was developed in~\cite{MS17}. The first and the third named authors will prove the counterpart of Theorems~\ref{thm:M+QT}---\ref{thm:disk} in that setting and plan to use them to study properties of radial and whole plane SLE. Both results and techniques in this paper will play a crucial role.
	\item (Quantum triangulation) Given a triangulation of any surface, we can conformally weld quantum triangles following the topological prescription.  The first and third named authors will prove that conditioning on the conformal structure of the resulting Riemann surface, the field is a Liouville field on that surface. If the resulting surface is non-simple, then the conformal structure (i.e. modulus) of surface itself is random. It is an interesting challenge to understand the random  moduli and the SLE interfaces in this setting.
\end{itemize}

\medskip
\noindent\textbf{Acknowledgements.} 
We thank Nina Holden and Scott Sheffield for helpful discussions at the early stage of the project. We thank two anonymous referees for their careful reading and many helpful comments. We thank Dapeng Zhan for explaining his work~\cite{Zhan19}. M.A. and P.Y. were partially supported by NSF grant DMS-1712862. M.A.\ was partially supported by the Simons Foundation as a Junior Fellow at the Simons Society of Fellows. X.S.\ was partially supported by the NSF grant DMS-2027986, the NSF Career award 2046514, and by a fellowship from the Institute for Advanced Study (IAS) during 2022-2023. P.Y. thanks IAS for hosting his visit during Fall 2022.

\section{Quantum triangles: definition and  basic properties}\label{sec-pre}

In this section we recall some preliminaries.  In Section \ref{sec-pre-gff}, we start with the definition of the Gaussian free field (GFF) and review the definition of quantum surfaces. In Section \ref{sec-pre-qt} and {Section \ref{sec-pre-thin}}, we relate marked quantum disks and Liouville CFT and establish the precise definition of the quantum triangle. {In Section \ref{sec-pre-length}, we consider the quantum triangles with fixed boundary lengths. Finally in Section \ref{subsec:critical}, we define quantum triangles with weight $\frac{\gamma^2}{2}$ vertices by a limiting procedure.}  %In Section \ref{sec-pre-ig}, we recall the imaginary geometry coupling of GFF and SLE \cite{MS16a}, and the resampling properties of \cite[Section 4]{MS16b}. Finally in Section \ref{sec-pre-conf}, we discuss the disintegration of quantum triangles over its boundary length and review the notion of conformal welding.

In this paper we work with non-probability measures and extend the terminology of ordinary probability to this setting. For a finite or $\sigma$-finite  measure space $(\Omega, \mathcal{F}, M)$, we say $X$ is a random variable if $X$ is an $\mathcal{F}$-measurable function with its \textit{law} defined via the push-forward measure $M_X=X_*M$. In this case, we say $X$ is \textit{sampled} from $M_X$ and write $M_X[f]$ for $\int f(x)M_X(dx)$. \textit{Weighting} the law of $X$ by $f(X)$ corresponds to working with the measure $d\tilde{M}_X$ with Radon-Nikodym derivative $\frac{d\tilde{M}_X}{dM_X} = f$, and \textit{conditioning} on some event $E\in\mathcal{F}$ (with $0<M[E]<\infty$) refers to the probability measure $\frac{M[E\cap \cdot]}{M[E]} $  over the space $(E, \mathcal{F}_E)$ with $\mathcal{F}_E = \{A\cap E: A\in\mathcal{F}\}$. {For a finite measure $M$ we write $M^\# = M/|M|$ for the probability measure proportional to $M$.}  We also fix the notation $|z|_+ := \max\{|z|, 1\}$.

\subsection{The Gaussian free field and quantum surfaces}\label{sec-pre-gff}

{Let $D\subset \mathbb{C}$ be a domain with $\partial D =  \partial^D\cup\partial^F$, $\partial^D\cap\partial^F=\emptyset$. We construct the GFF on $D$ with \textit{Dirichlet} \textit{boundary conditions} on $\partial^D$ and \textit{free boundary conditions} on $\partial^F$ as follows. Consider the space of smooth functions on $D$ with finite Dirichlet energy and zero value near $\partial^D$, and let $H(D)$ be its closure with respect to the inner product $(f,g)_\nabla={(2\pi)^{-1}}\int_D(\nabla f\cdot\nabla g)\ dx\ dy$. Then our GFF is defined by
	\begin{equation}\label{eqn-def-gff}
		h = \sum_{n=1}^\infty \xi_nf_n
	\end{equation}
	where $(\xi_n)_{n\ge 1}$ is a collection of i.i.d. standard Gaussians and $(f_n)_{n\ge 1}$ is an orthonormal basis of $H(D)$. One can show that the sum \eqref{eqn-def-gff} a.s.\ converges to a random distribution independent of the choice of the basis $(f_n)_{n\ge 1}$. Note that if $\partial^D$ is harmonically trivial, then elements in $H(D)$ should be understood as smooth functions modulo global additive constants, and the resulting $h$ is a distribution modulo an additive (random) global constant. For $D = \mathcal{S}$, the horizontal strip $\mathbb{R}\times (0, \pi)$, we fix the constant by requiring every function in $H(\mathcal{S})$ has mean value zero on $\{0\}\times [0, i\pi]$, while for $D=\mathbb{H}$, the upper half plane $\{z:\text{Im} z>0\}$, every function in $H(\mathbb{H})$ should have zero average value on the semicircle $\{e^{i\theta}:\theta\in(0, \pi)\}$, and we denote the corresponding laws of $h$ by $P_{\mathcal{S}}$ and $P_{\mathbb{H}}$, and the samples from $P_{\mathcal{S}}$ and $P_\bbH$ are referred as $h_{\mathcal{S}}$ and $h_{\bbH}$. See \cite[Section 4.1.4]{DMS14} for more details. }

For $h_{\mathcal S}$ and $h_{\mathbb H}$, the covariance kernels $G_D(z,w) := \mathbb E[h_D(z) h_D(w)]$ are given by 
\begin{equation}\label{eqn-gff-cov}
	\begin{split}
		&G_{\mathcal{S}}(z,w) = -\log|e^z-e^w|-\log|e^z-e^{\bar{w}}| +2\max\{\text{Re}z, 0\}+2\max\{\text{Re}w, 0\},\\
		&G_{\mathbb{H}}(z,w) = G_{\mathcal{S}}(e^z,e^w) = -\log|z-w|-\log|z-\bar{w}|+2\log|z|_++2\log|w|_+. 
	\end{split}
\end{equation}
{The first two terms in~\eqref{eqn-gff-cov} correspond to the Green's function for Laplacian with free boundary conditions, while the last two terms comes from our normalisation that $h$ has average zero on the segment $\{0\}\times (0,\pi)$ or in the unit semicircle $\{z\in\bbH:|z|=1\}$.}
Note that the notion $h_D(z)$ is defined by first taking the circle average $h_{D,\epsilon}(z)$ of $h_D$ over $\partial B(z, \epsilon)$ and then sending $\epsilon\to 0$.

One important fact is the radial-lateral decomposition of $h_{\mathcal{S}}$. Consider the subspace $H_1(\mathcal{S})\subset H(\mathcal{S})$ (resp. $H_2(\mathcal{S})\subset H(\mathcal{S})$) of functions with constant value (resp. mean zero) on $[t, t+i\pi]:=\{t\}\times (0, \pi)$ for every $t>0$. Then we have the orthogonal decomposition $H(\mathcal{S}) =  H_1(\mathcal{S})\oplus H_2(\mathcal{S})$, and we can write
\begin{equation}\label{eqn-gff-decom}
	h_\mathcal{S} = h_\mathcal{S}^1+h_\mathcal{S}^2
\end{equation}
by gathering the corresponding orthonormal bases of $H_1(\mathcal{S})$ and $H_2(\mathcal{S})$. Moreover, the common values $\{h_\mathcal{S}^1(t)\}_{t\in\mathbb{R}}$ agrees with the law of $\{B_{2t}\}_{t\in\mathbb{R}}$ where $\{B_t\}_{t\in\mathbb{R}}$ is the standard two-sided Brownian motion, while $ h_\mathcal{S}^1,h_\mathcal{S}^2$ are independent. See \cite[Section 4.1.6]{DMS14} for more details. 

Another important result is the Markov property of the GFF, which we state below.
\begin{proposition}[Markov Property of GFF] \label{prop:Markov}
	Let $D\subset \mathbb{C}$ be a domain with $\partial D =  \partial^D\cup\partial^F$, $\partial^D\cap\partial^F=\emptyset$, and $U\subset D$ open. Let $h$ be the GFF on $D$ with  \textit{Dirichlet} (resp. free) \textit{boundary conditions} on $\partial^D$ (resp. $\partial^F$). Then we can write $h=h_1+h_2$ where: 
	\begin{enumerate}
		\item $h_1$ and $h_2$ are independent;
		\item $h_1$ is a GFF on $U$ with Dirichlet boundary condition on $\partial U\backslash \partial^F$ and free on $\partial U\cap \partial^F$;
		\item $h_2$ is the same as $h$ outside $U$ and harmonic inside $U$.
	\end{enumerate}
\end{proposition}
Note that if $\partial^D=\emptyset$ (i.e., $h$ is free) then $h_2$ is defined modulo constant. See \cite[Section 4.1.5]{DMS14} for more details.  {The above property can also be extended to random sets. We say that a (random) closed set $A\subset D$ containing $\partial D$ is \textit{local}, if one can find a law on pairs $(A, h_2)$ such that $h_2|_{D\backslash A}$ is harmonic, while given $(A, h_2)$, we have $h = h_1+h_2$ where $h_1$ is an instance of zero boundary GFF on $D\backslash A$.}

Now we turn to Liouville quantum gravity and the quantum surfaces. Throughout this paper, we fix the LQG coupling constant $\gamma\in(0,2)$ and set
\begin{equation*}
	Q = \frac{2}{\gamma}+\frac{\gamma}{2}, \qquad \qquad \kappa = \gamma^2.
\end{equation*}

For two tuples $(D,h,z_1, ..., z_m)$ and $(\tilde{D}, \tilde{h}, \tilde{z}_1, ..., \tilde{z}_m)$, where $D$ and $\tilde{D}$ are {simply connected} domains on $\mathbb{C}$ with $(z_1, ..., z_m)$ and $(\tilde{z}_1, ..., \tilde{z}_m)$ being $m$ marked points on the bulk and the boundary of $D$ and $\tilde{D}$, and $h$ (resp. $\tilde{h}$) a distribution on $D$ (resp. $\tilde{D}$), we say  % {[This sentence has too many resp. in it.]}
\begin{equation}\label{eqn-qs-relation}
	(D, h, z_1, ..., z_m)\sim_\gamma (\tilde{D}, \tilde{h}, \tilde{z}_1, ..., \tilde{z}_m)
\end{equation}
if one can find a conformal mapping $\psi:{D}\to \tilde D$ such that $\psi({z}_j) = \tilde z_j$ for each $j$ and $\tilde{h} = \psi\bullet_\gamma h:= h\circ \psi^{-1}+Q\log|(\psi^{-1}) '|$, and we call each tuple $(D, h, z_1, ..., z_m)$ modulo the equivalence relation $\bullet_\gamma$ a $\gamma$-\textit{quantum surface}. 

For a $\gamma$-quantum surface $(D, h, z_1, ..., z_m)$, its \textit{quantum area measure} $\mu_h$ is defined by taking the weak limit $\epsilon\to 0$ of $\mu_{h_\epsilon}:=\epsilon^{\frac{\gamma^2}{2}}e^{\gamma h_\epsilon(z)}d^2z$, where $d^2z$ is the Lebesgue area and $h_\epsilon(z)$ is the circle average of $h$ over $\partial B(z, \epsilon)$. When $D=\mathbb{H}$, we can also define the  \textit{quantum boundary length measure} $\nu_h:=\lim_{\epsilon\to 0}\epsilon^{\frac{\gamma^2}{4}}e^{\frac{\gamma}{2} h_\epsilon(x)}dx$ where $h_\epsilon (x)$ is the average of $h$ over the semicircle $\{x+\epsilon e^{i\theta}:\theta\in(0,\pi)\}$. It has been shown in \cite{DS11, SW16} that all these weak limits are well-defined  for the GFF and its variants we are considering in this paper, while   $\mu_{h}$ and $\nu_{h}$ could be conformally extended to other domains using the relation $\bullet_\gamma$. % {[drop ``at least''. The word ``notion'' is inappropriate here.]}

Next we present the definition of \textit{weight $W$ (thick) quantum disk},  introduced in \cite[Section 4.5]{DMS14}. 

\begin{definition}\label{def-thick-disk}
	Fix $W\ge\frac{\gamma^2}{2}$ and let {$\beta = \gamma+ \frac{2-W}{\gamma}\le Q$}. Sample independent distributions $\psi_1, \psi_2$ such that: 
	\begin{itemize}
		\item $\psi_1$ has the same law as 
		\begin{equation}
			X_t:=\left\{ \begin{array}{rcl} 
				B_{2t}-(Q-\beta) t & \mbox{for} & t\ge 0\\
				\tilde{B}_{-2t} +(Q-\beta) t & \mbox{for} & t<0
			\end{array} 
			\right.
		\end{equation}
		where $(B_t)_{t\ge 0}$ and $(\tilde{B}_t)_{t\ge 0}$ are standard Brownian motions conditioned on  $B_{2t}-(Q-\beta)t<0$ and  $ \tilde{B}_{2t} - (Q-\beta)t<0$ for all $t>0$;  
		\item $\psi_2$ has the same law as $h_{\mathcal{S}}^2$ described in \eqref{eqn-gff-decom}.
	\end{itemize}
	Let $\hat \psi = \psi_1 + \psi_2$. Independently sample $\mathbf{c}$ from $\frac{\gamma}{2}e^{(\beta-Q)c}dc$, and let $\psi = \hat \psi + \mathbf c$.
	Let  $\mathcal{M}^{\textup{disk}}_2(W)$ be infinite measure describing the law of $(\mathcal{S}, \psi, -\infty, +\infty)/{\sim_\gamma}$. We call a sample from $\mathcal{M}^{\textup{disk}}_2(W)$ a (two-pointed) \emph{quantum disk of weight $W$}.
\end{definition}
When $0<W<\frac{\gamma^2}{2}$, we can also define the \textit{thin} quantum disk as a concatenation of weight $\gamma^2-W$ (two-pointed) thick disks as in \cite[Section 2]{AHS20}.

\begin{definition}\label{def-thin-disk}
	For $W\in(0, \frac{\gamma^2}{2})$, the infinite measure $\mathcal{M}_2^{\textup{disk}}(W)$ on two-pointed beaded surfaces is defined as follows. First sample $T$ from $(1-\frac{2}{\gamma^2}W)^{-2}\textup{Leb}_{\mathbb{R}_+}$, then sample a Poisson point process $\{(u, \mathcal{D}_u)\}$ from the intensity measure $\textbf{1}_{t\in [0,T]}dt\times \mathcal{M}_2^{\textup{disk}}(\gamma^2-W)$ and finally concatenate the disks $\{\mathcal{D}_u\}$ according to the ordering induced by $u$. The total sum of the left (resp. right) boundary lengths of all the $\mathcal{D}_u$'s is referred as the left (resp. right) boundary length of the thin quantum disk.
\end{definition}

We introduce the notion of embedding a thin quantum disk in the plane. Although not mathematically essential for our arguments, it simplifies exposition by letting us talk concretely about points and curves in the plane rather than abstractly on quantum surfaces. We follow the treatment of  \cite{DMS14}.

A \emph{(beaded) quantum surface} is a tuple $(D, h, z_1, \dots, z_m)$ modulo the equivalence relation~\eqref{eqn-qs-relation}, except that $D\subset \mathbb C$ is a closed set such that each component of its interior together with its prime-end boundary is homeomorphic to the closed disk, $h$ is defined as a distribution on each such component, and $\psi: D \to \tilde D$ is any homeomorphism which is conformal on each component of the interior of $D$ and sends $\psi (z_i) = \tilde z_i$ for each $i$. An \emph{embedding} of a beaded quantum surface is any choice of representative $(D, h, z_1, \dots, z_m)$. It is easy to see that a thin quantum disk is a beaded quantum surface.

\subsection{Liouville conformal field theory and thick quantum triangles}\label{sec-pre-qt}

In this section we review the theory of Liouville CFT and its relation with quantum disks as established in \cite[Section 2]{AHS21}. We will recap the notion of $\mathcal{M}_{2, \bullet}^{\textup{disk}}(W)$, the three-pointed quantum disks and then give the definition of quantum triangles in terms of LCFT.

We start from the  LCFT on the upper half plane. Recall that $P_\mathcal{S}$ and $P_\mathbb{H}$ are the probability measure induced by the GFF as in \eqref{eqn-def-gff} with our normalization.

\begin{definition}\label{def-lf-H}
	Let $(h, \mathbf{c})$ be sampled from $P_\mathbb{H}\times [e^{-Qc}dc]$ and take $\phi = h - 2Q\log|z|_++\mathbf{c}$. We say $\phi$ is a \textup{Liouville field} on $\mathbb{H}$ and let $\textup{LF}_{\mathbb{H}}$ be its law.
\end{definition}

\begin{definition}[Liouville field with boundary insertions]\label{def-lf-H-bdry}
	Let $\beta_i\in\mathbb{R}$  and $s_i\in \partial\mathbb{H}\cup\{\infty\}$ for $i = 1, ..., m,$ where $m\ge 1$ and all the $s_i$'s are distinct. Also assume $s_i\neq\infty$ for $i\ge 2$. We say $\phi$ is a \textup{Liouville Field on $\mathbb{H}$ with insertions $\{(\beta_i, s_i)\}_{1\le i\le m}$} if $\phi$ can be produced as follows by  first sampling $(h, \mathbf{c})$ from $C_{\mathbb{H}}^{(\beta_i, s_i)_i}P_\mathbb{H}\times [e^{(\frac{1}{2}\sum_{i=1}^m\beta_i - Q)c}dc]$ with
	$$C_{\mathbb{H}}^{(\beta_i, s_i)_i} =
	\left\{ \begin{array}{rcl} 
		\prod_{i=1}^m  |s_i|_+^{-\beta_i(Q-\frac{\beta_i}{2})} \exp(\frac{1}{4}\sum_{j=i+1}^{m}\beta_i\beta_j G_\mathbb{H}(s_i, s_j)) & \mbox{if} & s_1\neq \infty\\
		\prod_{i=2}^m  |s_i|_+^{-\beta_i(Q-\frac{\beta_i}{2}-\frac{\beta_1}{2})}\exp(\frac{1}{4}\sum_{j=i+1}^{m}\beta_i\beta_j G_\mathbb{H}(s_i, s_j)) & \mbox{if} & s_1= \infty
	\end{array} 
	\right.
	$$
	and then taking
	\begin{equation}\label{eqn-def-lf-H}
		\phi(z) = h(z) - 2Q\log|z|_++\frac{1}{2}\sum_{i=1}^m\beta_i G_\mathbb{H}(s_i, z)+\mathbf{c}
	\end{equation}
	with the convention $G_\mathbb{H}(\infty, z) = 2\log|z|_+$. We write $\textup{LF}_{\mathbb{H}}^{(\beta_i, s_i)_i}$ for the law of $\phi$.
\end{definition}

The following lemma explains that adding a $\beta$-insertion point at $s\in\partial\mathbb{H}$  is equal to weighting the law of Liouville field $\phi$ by $e^{\frac{\beta}{2}\phi(s)}$  in some sense. %The proof is identical to \cite[Lemma 2.6]{AHS21} and is omitted.

\begin{lemma}[Lemma 2.6 of \cite{AHS21}]\label{lmm-lf-insertion-bdry}
	For $\beta, s\in\mathbb{R}$ such that $s\notin\{s_1, ..., s_m\}$, in the sense of vague convergence of measures,
	\begin{equation}
		\lim_{\epsilon\to 0}\epsilon^{\frac{\beta^2}{4}}e^{\frac{\beta}{2}\phi_\epsilon(s)}\textup{LF}_{\mathbb{H}}^{(\beta_i, s_i)_i} = \textup{LF}_{\mathbb{H}}^{(\beta_i, s_i)_i, (\beta, s)}.
	\end{equation}
\end{lemma}

On the other hand, insertions at the infinity can also be handled via the following approximation. For $\beta\in\mathbb{R}$, we use the shorthand 
\begin{equation}\label{eqn-delta-alpha}
	\Delta_\beta:=\frac{\beta}{2}(Q-\frac{\beta}{2}).
\end{equation}
\begin{lemma}[Lemma 2.9 of \cite{AHS21}]\label{lmm-lf-insertion-infty}
	With the same notation as Lemma \ref{lmm-lf-insertion-bdry}, in the topology of vague convergence of measures,
	\begin{equation}
		\lim_{r\to+\infty}r^{2\Delta_\beta}\textup{LF}_\bbH^{(\beta,r),(\beta_i,s_i)_i} = \textup{LF}_{\mathbb{H}}^{(\beta, \infty),(\beta_i, s_i)_i}.
	\end{equation}
\end{lemma}

Sometimes it is also natural to work on Liouville fields on the strip $\mathcal{S}$ with insertions at $\pm\infty$. 

\begin{definition}\label{def-lf-strip}
	Let $(h, \mathbf{c})$ be sampled from $C_{\mathcal{S}}^{(\beta_1, +\infty), (\beta_2, -\infty), (\beta_3, s_3)}P_{\mathcal{S}}\times[e^{(\frac{\beta_1+\beta_2+\beta_3}{2}-Q)c}dc]$ with $\beta_1, \beta_2, \beta_3\in\mathbb{R}$,  $s_3\in\partial\mathcal{S}$ and
	{$$C_{\mathcal{S}}^{(\beta_1, +\infty), (\beta_2, -\infty), (\beta_3, s_3)} = e^{(-\Delta_{\beta_3}+\frac{(\beta_1+\beta_2)\beta_3}{4})|\textup{Re}s_3|+\frac{(\beta_1-\beta_2)\beta_3}{4}\textup{Re}s_3}.$$}
	Let $\phi(z) = h(z)+\frac{\beta_1+\beta_2-2Q}{2}|\textup{Re}z|+\frac{\beta_1-\beta_2}{2}\textup{Re}z+\frac{\beta_3}{2}G_{\mathcal{S}}(z, s_3)+\mathbf{c}$. We write $\textup{LF}_{\mathcal{S}}^{(\beta_1, +\infty), (\beta_2, -\infty), (\beta_3, s_3)}$ for the law of $\phi$.
\end{definition}

In general, the Liouville fields has nice compatibility with the notion of quantum surfaces. To be more precise, for a measure $M$ on the space of distributions on a domain $D$ and a conformal map $\psi:D\to \tilde{D}$, if we let $\psi_* M$ be the push-forward of $M$ under the mapping $\phi\mapsto\phi\circ\psi^{-1}+Q\log|(\psi^{-1})'|$. Then under this push-forward, the corresponding Liouville field measures only differs a multiple constant. For instance, 

\begin{lemma}\label{lmm-lcft-H-strip}
	For $\beta_1, \beta_2, \beta_3\in\mathbb{R}$ and $s_3\in\partial\mathcal{S}$, we have 
	\begin{equation}\label{eqn-lcft-H-strip}
		\textup{LF}_{\mathbb{H}}^{(\beta_1, \infty), (\beta_2, 0), (\beta_3, e^{s_3})}= e^{-\Delta_{\beta_3}\textup{Re}s_3}\exp_*\textup{LF}_{\mathcal{S}}^{(\beta_1, +\infty), (\beta_2, -\infty), (\beta_3, s_3)}.
	\end{equation} 
\end{lemma}
For a proof, one can directly compare the expressions of the corresponding multiplicative constants and invoke the conformal invariance of the GFF and the Green's function (with the mapping $z\mapsto e^z$). We also have the following
\begin{lemma}[Proposition 2.7 of \cite{AHS21}]\label{lmm-lcft-H-conf}
	Fix $\beta_i, s_i\in\mathbb{R}$ for $i=1, ..., m$ with $s_i$'s being distinct.  Suppose $\psi:\mathbb{H}\to\mathbb{H}$ is conformal such that $\psi(s_i)\neq\infty$ for each $i$. Then $\textup{LF}_{\mathbb{H}} = \psi_*\textup{LF}_{\mathbb{H}}$, and 
	\begin{equation}
		\textup{LF}_{\mathbb{H}}^{(\beta_i, \psi(s_i))_i} = \prod_{i=1}^m|\psi'(s_i)|^{-\Delta_{\beta_i}}\psi_*\textup{LF}_{\mathbb{H}}^{(\beta_i, s_i)_i}.
	\end{equation} 
\end{lemma}

Using Lemma \ref{lmm-lf-insertion-infty}, the above result can also be extended to Liouville fields with insertions at infinity.

\begin{lemma}\label{lmm-lcft-H-conf-infty}
	Suppose $\beta_1,\beta_2,\beta_3\in\mathbb{R}$ and $\psi:\mathbb{H}\to\mathbb{H}$ being conformal with $\psi(0) = 1$, $\psi(1)=\infty$ and $\psi(\infty) = 0$. Then 
	\begin{equation}
		\textup{LF}_{\mathbb{H}}^{(\beta_1, 0), (\beta_2, 1), (\beta_3, \infty)} = \psi_*\textup{LF}_{\mathbb{H}}^{(\beta_1, \infty), (\beta_2, 0), (\beta_3, 1)}.
	\end{equation} 
\end{lemma}

\begin{proof}
	{The proof is almost identical to that of \cite[Lemma 2.11]{AHS21}.} Note $\psi(z) = \frac{1}{1-z}$, and for $r>0$ set $\psi_r(z):=\frac{z-r}{(r-1)z-r}$.  Now $|\psi_r'(0)|=1+o_r(1)$, $|\psi_r'(1)|=(1+o_r(1))r^2$ and $|\psi_r'(r)|=(1+o_r(1))r^{-2}$, by Lemma \ref{lmm-lcft-H-conf}, as $r\to\infty$,
	\begin{equation}\label{eqn-lcft-H-conf-infty}
		\textup{LF}_{\mathbb{H}}^{(\beta_1, 0), (\beta_2, 1), (\beta_3, r)} = (1+o_r(1))r^{-2\Delta_{\beta_3}+2\Delta_{\beta_1}}(\psi_r)_*\textup{LF}_{\mathbb{H}}^{(\beta_1, r), (\beta_2, 0), (\beta_3, 1)}.
	\end{equation}
	Since $\psi_r\to\psi$ in the topology of uniform convergence of analytic functions and their derivatives on compact sets, we are done by multiplying both sides of \eqref{eqn-lcft-H-conf-infty} by $r^{2\Delta_{\beta_3}}$ and applying Lemma \ref{lmm-lf-insertion-infty}.
\end{proof}

The \emph{uniform embedding} of two-pointed quantum disk in the strip gives a Liouville field:

\begin{theorem}[Theorem 2.22 of \cite{AHS21}]
	For $W>\frac{\gamma^2}{2}$ and $\beta = \gamma+\frac{2-W}{\gamma}$, if we independently sample $T$ from $\textup{Leb}_\mathbb{R}$ and $(\mathcal{S},\phi, +\infty, -\infty)$ from $\mathcal{M}_2^{\textup{disk}}(W)$, then the law of $\tilde{\phi}:=\phi(\cdot+T)$ is  $\frac{\gamma}{2(Q-\beta)^2}\textup{LF}_\mathcal{S}^{(\beta,\pm\infty)}$. 
\end{theorem}
This result also leads to the notion of three-pointed quantum disks, where we may first sample a surface from the quantum disk measure reweighted by the left/right boundary length, and then sample a third marked point on $\mathbb{R}$ from the quantum length measure. 

\begin{definition}\label{three-pointed-disk}
	Fix {$W\ge\frac{\gamma^2}{2}$}. First sample  $(\mathcal{S},\phi, +\infty, -\infty)$ from $\nu_{\phi}(\mathbb{R})\mathcal{M}_2^{\textup{disk}}(W)[d\phi]$ and then sample $s\in\mathbb{R}$ according to the probability measure proportional to $\nu_{\phi}|_{\mathbb{R}}$. We denote the law of the surface  $(\mathcal{S},\phi, +\infty, -\infty, s)/{\sim_\gamma}$ by $\mathcal{M}_{2, \bullet}^{\textup{disk}}(W)$. 
\end{definition}

The definition above can be naturally extended to the case with the marked point added on $\mathbb{R}+i\pi$. And we have the following relation between $\mathcal{M}_{2, \bullet}^{\textup{disk}}(W)$ and Liouville fields.  

\begin{proposition}[Proposition 2.18 of \cite{AHS21}]\label{prop-m2dot}
	For $W>\frac{\gamma^2}{2}$ and $\beta = \gamma+\frac{2-W}{\gamma}$, let $\phi$ be sampled from $\frac{\gamma}{2(Q-\beta)^2}\textup{LF}_{\mathcal{S}}^{(\beta,\pm\infty), (\gamma, 0)}$. Then $(\mathcal{S}, \phi, +\infty, -\infty, 0)/{\sim_\gamma}$ has the same law as $\mathcal{M}_{2, \bullet}^{\textup{disk}}(W)$.
\end{proposition}

This third added point, which is sampled from the quantum length measure, is usually referred as \textit{quantum typical point}, and results in a $\gamma$-insertion to the Liouville field. This gives rise to the quantum disks with general third insertion points, which could be defined via three-pointed Liouville fields.

\begin{definition}\label{def-m2dot-alpha}
	Fix $W>\frac{\gamma^2}{2}$ and let $\alpha\in\mathbb{R}$. Set $\mathcal{M}_{2, \bullet}^{\textup{disk}}(W;\alpha)$ to be the law of $(\mathcal{S}, \phi, +\infty, -\infty, 0)/\sim_\gamma$ with $\phi$ sampled from $\frac{\gamma}{2(Q-\beta)^2}\textup{LF}_{\mathcal{S}}^{(\beta,\pm\infty), (\alpha, 0)}$. We call the boundary arc between the two $\beta$-singularities with (resp. not containing) the $\alpha$-singularity the marked (resp. unmarked) boundary arc.
\end{definition}

One can also add a third boundary marked point for thin disks  {and extend the definition of $\mathcal{M}_{2, \bullet}^{\textup{disk}}(W)$ to $W\in(0,\frac{\gamma^2}{2})$}. Recall in \cite[Proposition 4.4]{AHS20},  {one can equivalently define $\mathcal{M}_{2, \bullet}^{\textup{disk}}(W)$ with $W\in(0,\frac{\gamma^2}{2})$} by starting from first sampling a thick disk from  $\mathcal{M}_{2, \bullet}^{\textup{disk}}(\gamma^2-W)$ and then concatenating another two independent weight $W$ thin disks to the two endpoints. Therefore this leads to 
\begin{definition}\label{def-m2dot-thin}
	For $W\in(0, \frac{\gamma^2}{2})$ and $\alpha\in\mathbb{R}$, suppose $(S_1, S_2, S_3)$ is sampled from 
	$$(1-\frac{2}{\gamma^2}W)^2\mathcal{M}_{2}^{\textup{disk}}(W) \times\mathcal{M}_{2, \bullet}^{\textup{disk}}(\gamma^2-W;\alpha)\times \mathcal{M}_{2}^{\textup{disk}}(W) $$
	and $S$ is the concatenation of the three surfaces. Then we define the infinite measure $\mathcal{M}_{2, \bullet}^{\textup{disk}}(W;\alpha)$ to be the law of $S$. 
\end{definition} 

So far we have studied three-pointed quantum surfaces in terms of LCFT whenever two of the insertion points have the same $\alpha$ value. Indeed this relation can be extended to three-pointed Liouville fields with different insertion values, from which arises the notion of \textit{quantum triangles}.     

\begin{definition}[Thick quantum triangles]\label{def-qt-thick}
	Fix {$W_1, W_2, W_3>\frac{\gamma^2}{2}$}. Set $\beta_i = \gamma+\frac{2-W_i}{\gamma}<Q$ for $i=1,2,3$, and {let $\phi$ be sampled from $\frac{1}{(Q-\beta_1)(Q-\beta_2)(Q-\beta_3)}\textup{LF}_{\mathcal{S}}^{(\beta_1, +\infty), (\beta_2, -\infty), (\beta_3, 0)}$}. Then we define the infinite measure $\textup{QT}(W_1, W_2, W_3)$ to be the law of $(\mathcal{S}, \phi, +\infty, -\infty, 0)/\sim_\gamma$.
\end{definition}

We note that by Lemma \ref{lmm-lcft-H-strip} (where $s_3=0$) and Lemma \ref{lmm-lcft-H-conf-infty}, the measure $\textup{LF}_{\mathcal{S}}^{(\beta_1, +\infty), (\beta_2, -\infty), (\beta_3, 0)}$ has invariance under the conformal mappings $\mathcal{S}\to\mathcal{S}$ rearranging $\{+\infty, -\infty, 0\}$ and hence compatible with the relation $\sim_\gamma$.  To embed our quantum triangles onto other domains, we can further apply the conformal transforms and use $\sim_\gamma$. Also as will be explained in Section \ref{subsec:critical}, the choice of the constant $\frac{1}{(Q-\beta_1)(Q-\beta_2)(Q-\beta_3)}$ shall allow us to extend the definition when some of the $W_i$ is the same as $\frac{\gamma^2}{2}$, and the boundary length law is some sort of analytic.% We also remark that by definition, the measure $\Md_{2,\bullet}(W;\alpha)$ is equivalent to $\QT(W,W,\gamma^2+2-\gamma\alpha)$ for $\alpha<Q$.

\subsection{Quantum triangles with thin vertices}\label{sec-pre-thin}

Again recall that we can define thin quantum disks of weight $W\in(0,\frac{\gamma^2}{2})$ via concatenation of weight $\gamma^2-W$ thick disks (Definition \ref{def-thin-disk}), and a triply-marked thin disk could also be constructed by concatenating three-pointed weight $\gamma^2-W$ disks with weight $W$ thin disks. We shall apply the same idea to construct quantum triangles with thin vertices. See Figure \ref{fig-qt} for an illustration.

\begin{definition}\label{def-qt-thin}
	Fix {$W_1, W_2, W_3\in (0,\frac{\gamma^2}{2})\cup(\frac{\gamma^2}{2}, \infty)$}. Let $I:=\{i\in\{1,2,3\}:W_i<\frac{\gamma^2}{2}\}$.
	{Let $\tilde{W}_i = W_i$ if $i\notin I$, and $\tilde{W}_i = \gamma^2-W_i$ if $i\in I$. Sample 
		$(S_0, (S_i)_{i\in I})$ from 
		\[\textup{QT}(\tilde{W}_1, \tilde{W}_2, \tilde{W}_3)\times \prod_{i\in I} (1-\frac{2W_i}{\gamma^2})\mathcal{M}_2^{\textup{disk}}(W_i). \]
		For $i\in I$, concatenate $S_i$  with $S_0$ at the vertex of $S_0$ of weight $\tilde W_i$.} { Let $\textup{QT}(W_1, W_2, W_3)$ be the law of the resulting quantum surface. 
	}

	%	First sample the surface $S_0 := (D, \phi, \tilde{a}_1, \tilde{a}_2, \tilde{a}_3)$ from $\textup{QT}(\tilde{W}_1, \tilde{W}_2, \tilde{W}_3)$ where $\tilde{W}_i = W_i$ if $i\notin I$, and $\tilde{W}_i = \gamma^2-W_i$ if $i\in I$ such that each point $a_i$ corresponds to the insertion as given by $\tilde{\beta}_i = \gamma+ \frac{2-\tilde{W}_i}{\gamma}$. For each $i\in I$, {independently sample a surface $S_i$ from $(1-\frac{2W_i}{\gamma^2})\mathcal{M}_2^{\textup{disk}}(W_i)$} and concatenate it with $S_0$ at point $\tilde{a}_i$, and let $S$ be the output.
	
	%	 {[I would prefer to write this in terms of product measures, since the notion of independence requires definition for non-probability measures]}
\end{definition} 

\begin{remark}\label{remark-qt-3-disk}
	When $W_3>\frac{\gamma^2}{2}$ with $\beta_3=\gamma+\frac{2-W_3}{\gamma}$, by Definitions \ref{def-m2dot-alpha}, \ref{def-m2dot-thin} and \ref{def-qt-thin}, the measure $\Md_{2,\bullet}(W;\beta_3)$ is some multiple constant of the measure $\QT(W,W,W_3)$.  {We use the notation $\Md_{2,\bullet}(W;\beta_3)$ for compatibility with~\cite{AHS21,AGS21} since we will draw on results obtained there.} 
\end{remark}

\begin{definition}
	For a quantum triangle with thin vertices as in Definition~\ref{def-qt-thin}, we call $S_0$ its \emph{core}, {and we call each $S_i$ an \emph{arm} of weight $W_i$.}
\end{definition}

Since the thin quantum triangle is a concatenation of a thick quantum triangle with one to three independent thin quantum disks, we embed the surface as  $(D, \phi, a_1, a_2, a_3)$ where $D$ is not simply connected; see the discussion after Definition~\ref{def-thin-disk}. The vertices $a_1,a_2,a_3$ correspond to the weight $W_1,W_2,W_3$ vertices respectively. 
To simplify the notations, we shall call the boundary arc between the points with weights $W_1$ and $W_2$ the \emph{left} boundary arc,  the boundary arc between the points with weights $W_2$ and $W_3$ the  \emph{bottom} boundary arc, and the points with weights $W_3$ and $W_1$ the  \emph{right} boundary arc, as depicted in Figure \ref{fig-qt}.

In the remaining  of this section, we will work on the boundary length law of quantum triangles. We begin with the integrability of boundary LQG measure as obtained in \cite{RZ20a, RZ20b}. To state the results we will need several functions. The functions $\bar{R}$ and $\bar{H}$ are introduced for more general parameters (see \cite[Page 6-8]{RZ20b}) but for simplicity we only the ones which will appear later.  For $b>0$, recall the double-gamma function, the meromorphic function $\Gamma_b(z)$ in $\mathbb{C}$ such that for $\textup{Re}z>0$,
\begin{equation}\label{eqn-double-gamma}
	\log\Gamma_b(z) = \int_0^\infty \frac{1}{t}\bigg(\frac{e^{-zt}-e^{\frac{(b^2+1)t}{2b}}}{(1-e^{-bt})(1-e^{-\frac{1}{b}t)}} -\frac{(b^2+1-2bz)^2}{2b^2}e^{-t}+\frac{2bz-b^2-1}{bt} \bigg)dt
\end{equation}
and it satisfies the shift equations 
\begin{equation}\label{eqn-gamma-shift}
	\frac{\Gamma_b(z)}{\Gamma_b(z+b)} = \frac{1}{\sqrt{2\pi}}\Gamma(bz)b^{-bz+\frac{1}{2}}, \ \ \frac{\Gamma_b(z)}{\Gamma_b(z+b^{-1})} = \frac{1}{\sqrt{2\pi}}\Gamma(b^{-1}z)b^{\frac{z}{b}-\frac{1}{2}}.
\end{equation}
For $\mu>0$, let
\begin{equation}
	\bar{R}(\beta,\mu,0) := \bar{R}(\beta,0,\mu) = \mu^{\frac{2(Q-\beta)}{\gamma}}\frac{(2\pi)^{\frac{2(Q-\beta)}{\gamma}-\frac{1}{2}}(\frac{2}{\gamma})^{\frac{\gamma(Q-\beta)}{2}-\frac{1}{2}}   }{ (Q-\beta)\Gamma(1-\frac{\gamma^2}{4})^{\frac{2(Q-\beta)}{\gamma}} } \frac{\Gamma_{\frac{\gamma}{2}}(\beta-\frac{\gamma}{2})}{\Gamma_{\frac{\gamma}{2}}(Q-\beta)}.
\end{equation}
Finally set $\bar{\beta} = \beta_1+\beta_2+\beta_3$ and
\begin{equation}\label{eq-H}
	\bar{H}_{(0,1,0)}^{(\beta_1, \beta_2, \beta_3)} := \frac{ (2\pi)^{\frac{2Q-\bar{\beta}+\gamma}{\gamma} (\frac{2}{\gamma})^{(\frac{\gamma}{2}-\frac{2}{\gamma})(Q-\frac{\bar{\beta}}{2})}-1  } }{ \Gamma(1-\frac{\gamma^2}{4})^{\frac{2Q-\bar{\beta}}{\gamma}} \Gamma(\frac{\bar{\beta}-2Q}{\gamma}) }\frac{\Gamma_{\frac{\gamma}{2}}(\frac{\bar{\beta}}{2}-Q) \Gamma_{\frac{\gamma}{2}}(\frac{\bar{\beta}-2\beta_2}{2}) \Gamma_{\frac{\gamma}{2}}(\frac{\bar{\beta}-2\beta_1}{2} ) \Gamma_{\frac{\gamma}{2}} (Q-\frac{\bar{\beta}-2\beta_3}{2} )  }{ \Gamma_{\frac{\gamma}{2}}(Q)\Gamma_{\frac{\gamma}{2}}(Q-\beta_1)\Gamma_{\frac{\gamma}{2}}(Q-\beta_2)\Gamma_{\frac{\gamma}{2}}(\beta_3)   }.
\end{equation}

\begin{proposition}[Theorem 1.1 of \cite{RZ20a}; also see Section 3.3.4 of \cite{RZ20b}]\label{prop-rz20a} Fix $\beta_1, \beta_2, \beta_3\in\mathbb{R}$ and set $\bar{\beta}=\beta_1+\beta_2+\beta_3$. Let $h$ be sampled from $P_{\mathbb{H}}$ and let $\phi(z) = h(z) - \beta_1\log|z|-\beta_2\log |1-z|$. Then
	\begin{equation}\label{eqn-rz20a}
		\bar{H}_{(0,1,0)}^{(\beta_1, \beta_2, \beta_3)} = \mathbb{E}\big[\nu_{\phi}([0,1])^{\frac{2Q-\bar{\beta}}{\gamma}}\big]
	\end{equation}
	if $\beta_1, \beta_2, \beta_3$ satisfies the constraints
	\begin{equation}\label{eqn-seiberg}
		\beta_1,\beta_2<Q,  \ |\beta_1-\beta_2|<\beta_3,  \text{and} \ \bar{\beta}>\gamma . 
	\end{equation}	
	If \eqref{eqn-seiberg} is not jointly satisfied, then the right hand side of \eqref{eqn-rz20a} is infinite.	
\end{proposition}

The boundary length law of quantum disks can also be expressed in terms of $\bar{R}$.

\begin{proposition}[{Propositions 3.3 and 3.6 of \cite{AHS21}}]\label{prop-disk-bdry-law}
	For $W<\gamma Q$, $\beta = \gamma+\frac{2-W}{\gamma}$, the left (or right) boundary of a sample from $\mathcal{M}_{2}^{\textup{disk}}(W)$ has law
	\begin{equation}\label{eqn-disk-bdry-law}
		\textup{\textbf{1}}_{\ell>0}\bar{R}(\beta; 1, 0)\ell^{-\frac{2}{\gamma^2}W}d\ell.
	\end{equation}
	When $W\ge\gamma Q$, for any subinterval $I$ of $(0,\infty)$, the event $\{\text{left boundary length}\in I  \}$ has infinite $\mathcal{M}_{2}^{\textup{disk}}(W)$ measure.
\end{proposition}

Now we are ready to find the boundary length law for our quantum triangles. For a sample from $\textup{QT}(W_1,W_2,W_3)$, let $L_{12}$ be the quantum length of the boundary arc between the $\beta_1$ and $\beta_2$ singularities.

% {[Does the below law have the factors $(Q-\beta)^{-1}$, or do they go away with our current definition of $\QT$. ]}
\begin{proposition}\label{prop-qt-bdry-thick}
	Suppose $W_1, W_2, W_3>\frac{\gamma^2}{2}$ and let $\beta_i = \gamma+\frac{2-W_i}{\gamma}$ for $i=1,2,3$. Set $\bar{\beta} = \beta_1+\beta_2+\beta_3$. Suppose $(\beta_i)$ satisfies the bounds \eqref{eqn-seiberg}. Then for a sample from $\textup{QT}(W_1,W_2,W_3)$,  $L_{12}$ has law 
	\begin{equation}\label{eqn-qt-bdry-thick}
		\textup{\textbf{1}}_{\ell>0}\frac{2}{\gamma(Q-\beta_1)(Q-\beta_2)(Q-\beta_3)}\bar{H}_{(0,1,0)}^{(
			\beta_1,\beta_2,\beta_3)} \ell^{\frac{\bar{\beta}-2Q}{\gamma}-1}d\ell.
	\end{equation}
\end{proposition}

\begin{proof}
	By Definition \ref{def-qt-thick}, we can sample our quantum triangle by sampling $\phi$ from $\textup{LF}_\bbH^{(\beta_1,0), (\beta_2,1), (\beta_3,\infty)}$ and outputting $(\mathbb{H}, \phi, 0, 1, \infty)/\sim_\gamma$. Then one can check that our $\phi$ has expression
	\begin{equation}\label{eqn-pf-qt-bdry-thick-1}
		\phi(z) = h(z) +(\bar{\beta}-2Q)\log|z|_+-\beta_1\log|z|-\beta_2\log|z-1|+\mathbf{c}:=\phi_0(z)+\mathbf{c}
	\end{equation}
	where $(h, \mathbf{c})$ is sampled from $P_{\mathbb{H}}\times e^{\frac{\bar{\beta}-2Q}{2}c}dc$. Now for $b>a>0$, we have
	\begin{equation}
		\begin{split}
			\textup{QT}(W_1,W_2,W_3)\big[\textbf{1}_{\nu_{\phi}([0,1])\in (a,b)}\big]&=\frac{1}{(Q-\beta_1)(Q-\beta_2)(Q-\beta_3)} \mathbb{E}_{P_{\mathbb{H}}}\bigg[\int_0^\infty \textbf{1}_{e^{\frac{\gamma}{2}c  }\nu_{\phi_0}([0,1])\in (a,b)}e^{\frac{\bar{\beta}-2Q}{2}c}dc \bigg]\\
			&=\frac{2}{\gamma(Q-\beta_1)(Q-\beta_2)(Q-\beta_3)}\int_a^b\mathbb{E}_{P_{\mathbb{H}}}\big[ \big(\nu_{\phi_0}([0,1])\big)^{{\frac{2Q-\bar{\beta}}{\gamma}}} \big]\ell^{\frac{\bar{\beta}-2Q}{\gamma}-1} d\ell
		\end{split}	
	\end{equation} 
	where we applied the substitution $\ell = e^{\frac{\gamma}{2}c  }\nu_{\phi_0}([0,1])$ and Fubini's theorem. We conclude the proof by noticing that our $\phi_0$ in \eqref{eqn-pf-qt-bdry-thick-1} coincides with the $\phi$ in Proposition \ref{prop-rz20a} on the interval $[0,1]$ and applying \eqref{eqn-rz20a}.
\end{proof}

We can infer from \eqref{eqn-qt-bdry-thick} that 
\begin{equation}\label{eqn-qt-thick-bdry-lp}
	\textup{QT}(W_1,W_2,W_3)[e^{-\mu L_{12}}] = \frac{2}{\gamma(Q-\beta_1)(Q-\beta_2)(Q-\beta_3)}\bar{H}_{(0,1,0)}^{(
		\beta_1,\beta_2,\beta_3)}\Gamma(\frac{\bar{\beta}-2Q}{\gamma}) \mu^{{\frac{2Q-\bar{\beta}}{\gamma}}}.
\end{equation} 
Therefore we can further use the Laplace transform to compute boundary length laws for thin quantum triangles.

\begin{proposition}\label{prop-qt-bdry-thin}
	Fix {$W_1, W_2\in (0,\frac{\gamma^2}{2})\cup(\frac{\gamma^2}{2}, \infty)$} and $W_3>\frac{\gamma^2}{2}$. For $i=1,2,3$ again let $\beta_i = \gamma+\frac{2-W_i}{\gamma}$, and $\tilde{\beta}_i$ be equal to $\beta_i$ (resp. $2Q-\beta_i$) if $W_i>\frac{\gamma^2}{2}$ (resp. $W_i<\frac{\gamma^2}{2}$). Suppose $(\tilde{\beta}_i)$ satisfies the bounds \eqref{eqn-seiberg}. Then for a sample from $\textup{QT}(W_1,W_2,W_3)$,  $L_{12}$ has law 
	\begin{equation}\label{eqn-qt-bdry-thin}
		\textup{\textbf{1}}_{\ell>0}\frac{2}{\gamma(Q-\beta_1)(Q-\beta_2)(Q-\beta_3)}{\bar{H}_{(0,1,0)}^{(
				\beta_1,\beta_2,\beta_3)}} \ell^{\frac{\bar{\beta}-2Q}{\gamma}-1}d\ell.
	\end{equation}
\end{proposition}

\begin{proof}
	We first assume that $W_1>\frac{\gamma^2}{2}$ and $W_2<\frac{\gamma^2}{2}$.  Let $L_2$ be the left boundary length of a weight $W_2$ disk, then by \eqref{eqn-disk-bdry-law}, 
	\begin{equation}\label{eqn-disk-bdry-lp}
		\mathcal{M}_2^{\textup{disk}}(W_2)[e^{-\mu L_2}] = \bar{R}(\beta_2;1,0)\Gamma(1-\frac{2W_2}{\gamma^2})\mu^{\frac{2W_2}{\gamma^2}-1}.
	\end{equation}
	By definition of $\textup{QT}(W_1,W_2,W_3)$, if we independently sample a triangle from $\textup{QT}(W_1,\gamma^2-W_2,W_3)$ and let $\tilde{L}_{12}$ be the corresponding edge length, then $L_{12}$ has the same law as $L_2+ \tilde{L}_{12}$. Therefore by combining  \eqref{eqn-qt-thick-bdry-lp} (where $\beta_2$ is replaced by $2Q-\beta_2$) with \eqref{eqn-disk-bdry-lp} , 
	\begin{equation}\label{eqn-qt-bdry-thin-1}
		\begin{split}
			\textup{QT}(W_1,W_2,W_3)&[e^{-\mu L_{12}}] = \frac{2}{\gamma(Q-\beta_1)(\beta_2-Q)(Q-\beta_3)}\times\\
			&\bar{H}_{(0,1,0)}^{(
				\beta_1,2Q-\beta_2,\beta_3)}(1-\frac{2W_2}{\gamma^2})\Gamma(\frac{\beta_1+\beta_3-\beta_2}{\gamma})\Gamma(\frac{2}{\gamma}(\beta_2-Q))\bar{R}(\beta_2;1,0) \mu^{{\frac{2Q-\bar{\beta}}{\gamma}}}.
		\end{split}
	\end{equation}
	On the other hand, by \cite[Lemma 3.4]{RZ20b}, we have
	\begin{equation}\label{eqn-qt-bdry-thin-2}
		\bar{H}_{(0,1,0)}^{(
			\beta_1,2Q-\beta_2,\beta_3)} = -\frac{\Gamma(\frac{2}{\gamma}(2Q-\beta_2-\frac{2}{\gamma})) \Gamma(\frac{\bar{\beta}-2Q}{\gamma}) }{\Gamma(\frac{\beta_1+\beta_3-\beta_2}{\gamma} )   }\bar{R}(2Q-\beta_2; 1,0)	\bar{H}_{(0,1,0)}^{(
			\beta_1,\beta_2,\beta_3)},
	\end{equation}
	\begin{equation}\label{eqn-qt-bdry-thin-3}
		\bar{R}(\beta_2;1,0)\bar{R}(2Q-\beta_2;1,0) = \frac{1}{\Gamma(1-\frac{2(Q-\beta_2)}{\gamma}) \Gamma(1+\frac{2(Q-\beta_2)}{\gamma})}
	\end{equation}
	Combining the equations \eqref{eqn-qt-bdry-thin-1}, \eqref{eqn-qt-bdry-thin-2} and \eqref{eqn-qt-bdry-thin-3} implies
	\begin{equation}
		\textup{QT}(W_1,W_2,W_3)[e^{-\mu L_{12}}] = -\frac{2}{\gamma(Q-\beta_1)(\beta_2-Q)(Q-\beta_3)}\bar{H}_{(0,1,0)}^{(
			\beta_1,\beta_2,\beta_3)}\Gamma(\frac{\bar{\beta}-2Q}{\gamma}){\mu^{{\frac{2Q-\bar{\beta}}{\gamma}}}}
	\end{equation}
	which further implies \eqref{eqn-qt-bdry-thin}. For the case when both $W_1$ and $W_2$ are smaller than $\frac{\gamma^2}{2}$, we can start from independent samples of $\textup{QT}(\gamma^2-W_1,\gamma^2-W_2,W_3)$, $	\mathcal{M}_2^{\textup{disk}}(W_1)$ and $	\mathcal{M}_2^{\textup{disk}}(W_2)$. We omit the details.
\end{proof}

{
	The above result gives the law of a quantum triangle boundary arc length. In fact, for some range of parameters, we can identify the joint law of boundary arc lengths and quantum area. Suppose $\sum \beta_i > 2Q$,  $\beta_1, \beta_2, \beta_3 < Q$, and $\mu_1, \mu_2, \mu_3 > 0$, then \cite[Theorem 1.1]{ARS22} gives an explicit description of 
	\begin{equation}\label{eq-H-bulk}
		H^{(\beta_1, \beta_2, \beta_3)}_{(\mu_1, \mu_2, \mu_3)} := \LF_{\bbH}^{(\beta_1, 0), (\beta_2, 1), (\beta_3, \infty)} [\exp(- \mu_\phi(\bbH) - \mu_1 \nu_\phi(-\infty, 0) - \mu_2 \nu_\phi(0,1) - \mu_3 \nu_\phi(1,\infty))],
	\end{equation}
	where $\phi \sim \LF_{\bbH}^{(\beta_1, 0), (\beta_2, 1), (\beta_3, \infty)}$ is the Liouville field. 
}

%{We note that the function $(Q-\beta_1)^{-1}(Q - \beta_2)^{-1}\bar{H}_{(0,1,0)}^{(	\beta_1,\beta_2,\beta_3)}$ is continuous and positive for $(\beta_1, \beta_2, \beta_3)$ satisfying  $\beta_1, \beta_2 \leq Q$,  $|\beta_1 - \beta_2| < \beta_3$ and $\sum \beta_i > \gamma$. Indeed, the prefactor $(Q-\beta_1)^{-1}(Q-\beta_2)^{-1}$ cancels the roots $\beta_1 = Q$ and $\beta_2 = Q$ of  $\bar{H}_{(0,1,0)}^{(	\beta_1,\beta_2,\beta_3)}$. This will be useful in Section~\ref{subsec:critical} when we use a limiting argument to define the quantum triangle when one or more of the $W_i$ equals $\frac{\gamma^2}2$ (i.e.\ $\beta_i = Q$). }

\subsection{Quantum triangles with fixed boundary lengths}\label{sec-pre-length}

We start by proving that, the quantum triangles we defined a.s. has positive finite length.

\begin{lemma}\label{lmm-qt-finite-length}
	For any weights $W_1, W_2, W_3>0$, the $\textup{QT}(W_1, W_2, W_3)$ measure of quantum triangles with edges having zero or infinite quantum  length is 0.
\end{lemma}

\begin{proof}
	We begin with the thick quantum triangles. Sample $\phi$ from $\textup{LF}_{\mathbb{H}}^{(\beta_1,0), (\beta_2,1), (\beta_3,\infty)}$ with $\beta_i = \gamma+\frac{2-W_i}{\gamma}<Q$, so our quantum triangle is $(\mathbb{H}, \phi, 0, 1, \infty)/{\sim_\gamma}$. Using the expression \eqref{eqn-pf-qt-bdry-thick-1} for $\phi$, it suffices to check that under $P_{\mathbb{H}}$, $\nu_{\phi_0}([0,1])$ is a.s. finite. Since $\beta_i<Q$, we can pick $p>0$ such that $p<\frac{4}{\gamma^2}\wedge\frac{2}{\gamma}(Q-\beta_1)\wedge\frac{2}{\gamma}(Q-\beta_2)$. By \cite[Theorem 1.1]{RZ20a}, $\mathbb{E}_{P_{\mathbb{H}}}\big[\nu_{\phi_0}([0,1])^p\big]<\infty$, which justifies our claim. The remaining case follows by noticing that thin triangles are produced by concatenating  independent samples of thick triangles with thin quantum disks, while  both of them have   finite length almost surely.    
\end{proof}

We are now ready to disintegrate $\textup{QT}(W_1, W_2, W_3)$ over its boundary length. Basically this is simply \textit{conditioning} on edge length. Recall that for any two-pointed disks, by \cite[Section 2.6]{AHS20}, one can construct the family of measures $\{\mathcal{M}_2^{\text{disk}}(W;\ell_1 ):\ell_1>0\}$ and $\{\mathcal{M}_2^{\text{disk}}(W;\ell_1, \ell_2):\ell_1,\ell_2>0\}$ for  such that 
\begin{equation}
	\mathcal{M}_2^{\text{disk}}(W) = \int_0^\infty\int_0^\infty \mathcal{M}_2^{\text{disk}}(W;\ell_1, \ell_2)d\ell_1d\ell_2; \ \mathcal{M}_2^{\text{disk}}(W,\ell_1) =  \int_0^\infty \mathcal{M}_2^{\text{disk}}(W;\ell_1, \ell_2)d\ell_2.
\end{equation}
Each sample from $\mathcal{M}_2^{\text{disk}}(W;\ell_1)$ has left (or right) boundary length $\ell_1$, and each sample from $\mathcal{M}_2^{\text{disk}}(W;\ell_1, \ell_2)$ has boundary lengths $\ell_1$ and $\ell_2$. And the same disintegration can be applied for $\mathcal{M}_{2,\cdot}(W;\alpha)$ over the length of unmarked boundary \cite{AHS21}.

We formally state the definition below and again start with thick triangles. 
\begin{definition}\label{def-qt-disintegration-thick}
	Suppose $W_1, W_2, W_3>\frac{\gamma^2}{2}$.  Let $\beta_i = \gamma+\frac{2-W_i}{\gamma}$ and $\bar{\beta}=\beta_1+\beta_2+\beta_3$. Sample $h$ from $P_{\mathbb{H}}$ and set
	$$\tilde{h}(z) = h(z) +(\bar{\beta}-2Q)\log|z|_+-\beta_1\log|z|-\beta_2\log|z-1|.$$ (i.e., The Liouville field $\textup{LF}_{\mathbb{H}}^{(\beta_1,0), (\beta_2, 1), (\beta_3, \infty)}$ but without the constant $\mathbf{c}$.) Fix $\ell>0$. Let $L_{12} = \nu_{\tilde{h}}([0,1])$ and we define the measure $\textup{QT}(W_1, W_2, W_3;\ell)$, the quantum triangles of weight $W_1,W_2,W_3$ with \emph{left boundary length} $\ell$, to be the law of $\tilde{h}+\frac{2}{\gamma}\log\frac{\ell}{L_{12}}$ under the reweighted measure $\frac{2}{\gamma(Q-\beta_1)(Q-\beta_2)(Q-\beta_3)}\frac{\ell^{\frac{1}{\gamma}(\bar{\beta}-2Q)-1}}{L_{12}^{\frac{1}{\gamma}(\bar{\beta}-2Q)}}P_{\mathbb{H}}(dh)$. 
\end{definition}

The above definition can be repeated for right or bottom boundary length. (i.e., $L_{12} = \nu_{\tilde{h}}([0,1])$ replaced by $L_{13} = \nu_{\tilde{h}}((-\infty, 0])$ or $L_{23} = \nu_{\tilde{h}}([1,+\infty))$.) The following lemma justifies our disintegration.
\begin{lemma}\label{lem-qt-disintegration-thick}
	In the setting of Definition \ref{def-qt-disintegration-thick}, samples from $\textup{QT}(W_1, W_2, W_3;\ell)$ has left boundary length $\ell$, and we have
	\begin{equation}\label{eqn-qt-disintegration-thick}
		\textup{QT}(W_1, W_2, W_3) = \int_0^\infty \textup{QT}(W_1, W_2, W_3;\ell)d\ell.
	\end{equation}
	Furthermore, if $(\beta_i)$ satisfies the Seiberg bounds \eqref{eqn-seiberg}, then $$\big|\textup{QT}(W_1, W_2, W_3;\ell)\big|=\frac{2}{\gamma(Q-\beta_1)(Q-\beta_2)(Q-\beta_3)}{\bar{H}_{(0,1,0)}^{(
			\beta_1,\beta_2,\beta_3)}} \ell^{\frac{\bar{\beta}-2Q}{\gamma}-1}.$$ 
\end{lemma}

\begin{proof}
	The proof is almost identical to that of \cite[Lemma 4.2]{AHS21} but we include it here for completeness. The first claim is trivial as $\nu_{\tilde{h}+\frac{2}{\gamma}\log\frac{\ell}{L_{12}}}([0,1]) =\frac{\ell}{L_{12}}\nu_{\tilde{h}}([0,1]) =\ell$.
	
	Now for any nonnegative measurable function $F$ on $H^{-1}(\mathbb{H})$ we have
	\begin{equation}
		\int_0^\infty\int F(\tilde{h}+\frac{2}{\gamma}\log\frac{\ell}{L_{12}})\frac{2}{\gamma}\frac{\ell^{\frac{1}{\gamma}(\bar{\beta}-2Q)-1}}{L_{12}^{\frac{1}{\gamma}(\bar{\beta}-2Q)}}P_{\mathbb{H}}(dh)d\ell = \int_{\mathbb{R}}\int F(\tilde{h}+c)e^{(\frac{1}{2}\bar{\beta}-Q)c}P_\bbH(dh)dc
	\end{equation}
	using Fubini's theorem and the change of variables $c = \frac{2}{\gamma}\log\frac{\ell}{L_{12}}$. Therefore by definition, \eqref{eqn-qt-disintegration-thick} holds. The last statement follows directly from Proposition \ref{prop-qt-bdry-thick}.
\end{proof}

Indeed if $W_3<\frac{\gamma^2}{2}$, the same disintegration applies by starting from $\textup{QT}(W_1, W_2, \gamma^2-W_3;\ell)$ and then concatenating an independent weight $W_3$ disk (which does not affect the left boundary length).
If $W_1<\frac{\gamma^2}{2}$ and $W_2>\frac{\gamma^2}{2}$, we can still define our disintegration over left boundary length via
\begin{equation}\label{eqn-qt-disintegration-thin-1}
	\textup{QT}(W_1, W_2, W_3;\ell) = (1-\frac{2W_1}{\gamma^2})\int_0^\ell  \mathcal{M}_2^{\text{disk}}(W_1;\ell-x)\times \textup{QT}(\gamma^2-W_1, W_2, W_3;x)dx.
\end{equation}  
Similarly, if $W_1<\frac{\gamma^2}{2}$ and $W_2<\frac{\gamma^2}{2}$, we can also define
\begin{equation}\label{eqn-qt-disintegration-thin-2}
	\begin{split}
		&\textup{QT}(W_1, W_2, W_3;\ell) = (1-\frac{2W_1}{\gamma^2})(1-\frac{2W_2}{\gamma^2})\\&\int_0^\ell\int_0^{\ell-x}  \mathcal{M}_2^{\text{disk}}(W_1;y)\times \textup{QT}(\gamma^2-W_1, \gamma^2-W_2, W_3;x)\times \mathcal{M}_2^{\text{disk}}(W_2;\ell-x-y) dydx.
	\end{split}
\end{equation}
One can directly verify that \eqref{eqn-qt-disintegration-thick} holds for our definition of $\textup{QT}(W_1, W_2, W_3;\ell)$ via \eqref{eqn-qt-disintegration-thin-1} and \eqref{eqn-qt-disintegration-thin-2}, and each sample from  $\textup{QT}(W_1, W_2, W_3;\ell)$ has left boundary length $\ell$.

We have defined disintegration \eqref{eqn-qt-disintegration-thick} over a single boundary arc length. This can naturally be extended to multiple edges, that is,
\begin{equation}
	\textup{QT}(W_1, W_2, W_3) = \iiint_{\mathbb{R}_+^3} \mathrm{QT}(W_1, W_2, W_3; \ell_1, \ell_2, \ell_3)d\ell_1d\ell_2d\ell_3.
\end{equation} 
See \cite[Section 2.6]{AHS20} for more details.

\subsection{Vertices with weight 
	$\frac{\gamma^2}2$}\label{subsec:critical}

In this section we define quantum triangles where one or more vertices have weight $\frac{\gamma^2}2$. 
Plugging in the relation $\beta=Q+\frac{\gamma}{2}-\frac{W}{\gamma}$ gives $\beta=Q$,
but the Liouville field with boundary insertion $Q$ a.s.\ has infinite boundary length near the insertion, so this does not give the correct definition. The correct definition is obtained from  the $\beta \uparrow Q$ limit of the previously defined Liouville field, which we call the Liouville field with insertion of size $\beta = Q^-$.
%We first define the Liouville field with insertions of size $\beta = Q^-$, show that it arises as a $\beta \uparrow Q$ limit of the previously defined Liouville field, then use this Liouville field to define the quantum triangle with weight $\frac{\gamma^2}2$ vertices. 

We first define an infinite measure $M^{Q^-}$ as follows. 
For $a>0$, let $B_t$ be variance 2 Brownian motion run until the first time $\tau_a$ it hits $a$, and independently let $B'_s$ be variance 2 Brownian motion conditioned on the event $\{ B'_s \leq 0 \text{ for all } s \geq 0\}$. Namely, $-\frac{1}{\sqrt2}B_s'$ is a 3D Bessel process starting from 0. Define $X_t = B_t$ for $t \leq \tau_a$ and $X_t = a + B'_{t- \tau_a}$ for $t > \tau_a$. Let $P_a$ be the law of $X_t$. Slightly abusing notation, we sample $(X_t, \mathbf a) $ from $ P_a 1_{a>0} da$ and let $M^{Q^-}$ be the marginal law of $X_t$ under this infinite measure.

We will define the Liouville field with one or more insertions of size $Q^-$ via $M^{Q-}$. 
We will put insertions at the boundary points $(+\infty, -\infty,1)$ of $\mathcal S$; we choose the third boundary point 1 rather than 0 to avoid interfering with the GFF normalization (mean zero on $\{0\} \times [0,\pi]$). 
Recall that $H(\mathcal S)$ is the closure of the space of smooth  functions on $\mathcal S$ of finite Dirichlet energy with respect to the Dirichlet inner product. 
Let $H_1 \subset H(\mathcal S)$ be the subspace of functions which are zero on $(-\infty, 10] \times [0,\pi]$ and constant on each segment $\{t\} \times [0,\pi]$ for $t \geq 10$. Let $H_2$ be the subspace of functions which are zero on $[-10, \infty) \times [0,\pi]$ and constant on each segment $\{ t \} \times [0,\pi]$ for $t \leq -10$. Let $H_3$ be the subspace of functions which are zero on $\{ z \in \mathcal S \: : \: |z - 1| \geq 1\}$ and constant on each semicircle $\{ z \in \mathcal S \: : \: |z-1| = e^{-t}\}$ for $t \geq 0$.  Let $H_0$ be the orthogonal complement of $H_1 \oplus H_2 \oplus H_3$.  {Functions in $H_0$ have the same average value on each  segment $\{t\} \times [0,\pi]$ for $ t\geq 10$,} %Functions in $H_0$ are harmonic on $(10, \infty) \times [0,\pi]$ (with Neumann boundary conditions on $\{10, \infty\} \times [0,\pi]$) {[*I don't know why functions in $H_0$ are harmonic as well*]} and have the same average value on $\{t \} \times [0,\pi]$ for all $t \geq 10$.  Functions in $H_0$ likewise 
and have similar behavior in $(-\infty, -10) \times [0,\pi]$ and $\{ z \in \mathcal S : |z-1| < 1\}$. 

Let $\mathcal P$ be the set of probability  measures $\rho$ compactly supported in $\{ z \in (-10,10) \times (0,\pi) \:: \: |z-1| > 1\}$ such that $\int G_{\mathcal S}(z,w) \rho(dw)\rho(dz)<\infty$; such measures can be integrated against a GFF. In particular $\mathcal P$ contains the uniform probability measure on $\{0\}\times[0,\pi]$.  Let $P_\rho$ be the law of the GFF $h$ on $\mathcal S$ normalized so $(h, \rho) = 0$. 
Using the decomposition $H(\mathcal S) = H_0 \oplus H_1 \oplus H_2 \oplus H_3$ we can decompose a GFF $h \sim P_\rho$ as 
\begin{equation}\label{eq-gff-decomp-4}
	h = g_0 + g_1 + g_2 + g_3    
\end{equation}
where the $g_i$ are independent and correspond to projections to $H_i$. 

Let $\rho_1$, $\rho_2$ and $\rho_3$ be the uniform probability measures on $\{10\} \times [0,\pi]$, $\{-10\} \times [0,\pi]$ and $\{z \in \mathcal{S}\:: \: |z| = 1\}$ respectively. For real $\beta_1, \beta_2, \beta_3$ define the {\emph{non-probability} measure 
	\[
	P^{(\beta_i)_i}_\rho(dh) = \eps_0^{\frac14((\beta_1-Q)^2 + (\beta_2-Q)^2)} e^{\frac{\beta_1-Q}2(h, \rho_1) + \frac{\beta_2-Q}2 (h, \rho_2) + \frac{\beta_3}2 (h, \rho_3)} P_\rho(dh)
	, \qquad \eps_0 := e^{-10}.\] }
For $\beta \in \bbR$ let $M^\beta$ be the law of Brownian motion with variance 2 and drift $-(Q - \beta)$; in particular $|M^\beta| = 1$. We now extend the definition of the Liouville field to allow insertions of size $Q^-$.  {This is the definition one lands upon when taking $\beta \uparrow Q$ and renormalizing appropriately, as we will see later in Proposition~\ref{prop-limit-crit}.}

\begin{definition}\label{def-critical}
	Suppose $\beta_1, \beta_2, \beta_3 \in \bbR\cup \{ Q^-\}$ and $\rho \in \mathcal P$. Let $\hat \beta_i = Q$ if $\beta_i = Q^-$, and let $\hat \beta_i = \beta_i$ otherwise.
	Let $s = \frac12(\sum \hat\beta_i) - Q$. Sample $(h, \mathbf c, X_t^1, X_t^2, X_t^3) \sim P_\rho^{(\hat \beta_i)_i} \times [e^{sc}\,dc] \times M^{\beta_1} \times M^{\beta_2} \times M^{\beta_3}$. Decompose $h = g_0 + g_1 + g_2 + g_3$ as in~\eqref{eq-gff-decomp-4}. Let $\hat g_1$ be the function which is zero on $(-\infty, 10) \times [0,\pi]$ and equals $X_{t}^1$ on each segment $\{t+10\} \times [0,\pi]$ for $t \geq 0$. Let $\hat g_2$ be the function which is zero on $(-10, \infty)\times[0,\pi]$ and equals $X_{t}^2$ on each segment $\{ -t -10 \} \times [0,\pi]$ for $t \geq 0$. Let $\hat g_3$ be the function which is zero on $\{ z \in \mathcal S : |z| \geq 1\}$ and equals $X_t^3 + Qt$ on each semicircle $\{ z \in \mathcal S \: : \: |z| = e^{-t}\}$ for $t \geq 0$. Let $\phi = g_0 + \hat g_1 + \hat g_2 + \hat g_3 + \mathbf c$. We denote the law of $\phi$ by $\LF_\mathcal{S}^{(\beta_1, +\infty), (\beta_2, -\infty), (\beta_3, 1)}$. 
\end{definition}

%{Note that for $\beta_1, \beta_2, \beta_3 \neq Q^-$, the measure $\LF_\mathcal{S}^{(\beta_1, +\infty), (\beta_2, -\infty), (\beta_3, 1)}$ is not a normalized measure, while if any of $\beta_1, \beta_2, \beta_3$ equals $Q^-$, then the measure $\LF_\mathcal{S}^{(\beta_1, +\infty), (\beta_2, -\infty), (\beta_3, 1)}$ is renormalized.}
\begin{lemma}
	Definition~\ref{def-critical} does not depend on the choice of $\rho$. Moreover, if $\beta_1, \beta_2, \beta_3 \in \bbR$, then the definition agrees with Definition~\ref{def-lf-strip}.
\end{lemma}
\begin{proof}
	{We first check that if $\rho$ is the uniform probability measure on $\{0\} \times [0,\pi]$, then Definition~\ref{def-critical} agrees with Definition~\ref{def-lf-strip}. For the special case  $(\beta_1, \beta_2, \beta_3) = (Q,Q,0)$, we have $P_\rho^{(\beta_i)_i} = P_{\mathcal S}$, and for $h \sim P_{\mathcal S}$  the field average processes described by $g_1, g_2, g_3$ from~\eqref{eq-gff-decomp-4} each have the law of variance 2 Brownian motion (see e.g.\ \cite[Section 4.1.6]{DMS14}), so the claim is immediate. This gives the decomposition identifying $\LF_{\mathcal S}^{(Q, +\infty), (Q, -\infty)}$ with $P_{\mathcal S}    \times dc \times M^{Q} \times M^Q \times M^0$. Now we explain how to extend to the case $\beta_1 \in \bbR$, $\beta_2 = Q$ and $\beta_3 = 0$. Parametrizing in $\mathcal S$ rather than $\bbH$, an immediate consequence of Lemma~\ref{lmm-lf-insertion-bdry} is
		\begin{align*}
			\LF_{\mathcal S}^{(\beta_1, +\infty), (Q, -\infty)}(d\phi) &= \lim_{\eps \to 0} \eps^{\frac14(\beta_1 - Q)^2} e^{\frac12(\beta_1-Q) (\phi, \theta_\eps)} \LF_{\mathcal S}^{(Q, +\infty), (Q, -\infty)}(d\phi),
			%\\&= \lim_{\eps \to 0} \eps^{\frac14(\beta_1^2 - Q^2)} e^{\frac12(\beta_1-Q) (\phi, \rho_\eps)} \LF_{\mathcal S}^{(Q, +\infty), (Q, -\infty)}(d\phi) \\
		\end{align*}
		where $\theta_\eps$ is the uniform probability measure on $\{-\log \eps\} \times [0,\pi]$. Identifying $\LF_{\mathcal S}^{(Q, +\infty), (Q,-\infty)}$ with $P_{\mathcal S} \times dc \times M^Q \times M^Q \times M^0$, this limit can be written as 
		\begin{align*}
			&\lim_{\eps \to 0} \eps^{\frac14 (\beta_1 - Q)^2}  e^{\frac12(\beta_1 - Q) ((h, \rho_1)+ X^1_{-\log (\eps/\eps_0)} + c)} P_{\mathcal S} (dh) dc M^Q(dX^1)M^Q(dX^2) M^0(dX^3) \\
			&=\lim_{\eps \to 0} (\eps/\eps_0)^{\frac14 (\beta_1 - Q)^2}  e^{\frac12(\beta_1 - Q) X^1_{-\log (\eps/\eps_0)}} P^{(\beta_i)_i}_{\rho} (dh) e^{sc} dc M^Q(dX^1)M^Q(dX^2) M^0(dX^3)\\
			&=P^{(\beta_i)_i}_{\rho} (dh) e^{sc} dc M^{\beta_1}(dX^1)M^Q(dX^2) M^0(dX^3).
		\end{align*}
		To obtain the last equality above, by Girsanov's theorem the law of $X^1_t$ under the probability measure $(\eps /\eps_0)^{\frac14(\beta_1 - Q)^2} \\e^{\frac12(\beta_1 - Q)X^1_{-\log(\eps/\eps_0)}} M^Q (dX^1)$ is Brownian motion with variance $2$, with drift $-(Q -\beta_1)$ until time $-\log(\eps/\eps_0)$ and zero drift afterwards; this converges as $\eps \to 0$ to $M^{\beta_1}$ in the topology of uniform convergence on finite intervals. Thus $\LF_{\mathcal S}^{(\beta_1, +\infty), (Q, -\infty)}$ can be identified with $P_\rho^{(\beta_i)_i} \times e^{sc}dc \times  M^{\beta_1} \times M^{Q} \times  M^0$ as desired. Here we discussed $\beta_1 \in \bbR, \beta_2 = Q, \beta_3 = 0$ to lighten notation, but the same argument applies for $\beta_1,\beta_2, \beta_3 \in \bbR$. We conclude that if $\rho$ is the uniform probability measure on $\{0\} \times [0,\pi]$ then Definition~\ref{def-critical} agrees with Definition~\ref{def-lf-strip}.
	}

	{
		Now let $\rho$ be arbitrary. It remains to verify that Definition~\ref{def-critical} does not depend on the choice of $\rho$. 
		If $h \sim P_\rho$ then the law of $h$ viewed as a distribution modulo additive constant does not depend on $\rho$, so by the translation invariance of $dc$, 
		for  $(h, \mathbf c)$ sampled from $P_\rho \times dc$ the law of $h+\mathbf c$ does not depend on $\rho$. Consequently, if we sample $(h, \mathbf c)$ from 
		\[P_\rho^{(\beta_i)_i} \times [e^{sc}dc] =  e^{\frac{\beta_1-Q}2(h+c, \rho_1) + \frac{\beta_2-Q}2 (h+c, \rho_2) + \frac{\beta_3}2 (h+c, \rho_3)} P_\rho(dh) dc,\]
		the law of $h+\mathbf c$ does not depend on $\rho$; since $\phi$ is a function of  $h+\mathbf c$ and randomness independent of $(h, \mathbf c)$, the claim follows. }
\end{proof}

We will prove that the Liouville field with one or more $Q^-$ insertions arises as a $\beta \uparrow Q$ limit.  The key is the  Brownian motion description of $ \frac1{Q-\beta}M^\beta$  and its convergence to $ M^{Q^-}$ under a suitable topology.

\begin{lemma}\label{lem-Mbeta}
	Let $\beta < Q$. For $X_t \sim \frac{1}{Q-\beta} M^\beta$, the law of $\mathbf a = \sup_t X_t$ is $1_{a>0} e^{-(Q-\beta)a}\,da$. Moreover, conditioned on $\mathbf a$, the conditional law of $X_t$ is that of variance 2 Brownian motion with upward drift $(Q - \beta)t$ run until it hits $\mathbf a$, then variance 2 Brownian motion with downward drift $-(Q-\beta)t$ started at $\mathbf a$ and conditioned to stay below $\mathbf a$.
\end{lemma}
\begin{proof}
	The law of $\mathbf a$ follows from a standard Brownian motion computation, and the conditional law of $X_t$ given $\mathbf a$ follows from the Williams decomposition~\cite{williams1974path}. 
\end{proof}

\begin{lemma}\label{lem-BM-lim}
	For $A>0$ let $E'_A$ be the event that a process $X_t$ satisfies $\sup_t X_t < A$. Then we have the weak limit $\lim_{\beta \uparrow Q} \frac1{Q-\beta}M^\beta|_{E'_A} = M^{Q^-}|_{E'_A}$, where the topology on function space is uniform convergence on compact sets. 
\end{lemma}
\begin{proof}
	Comparing the description of $\frac1{Q-\beta}M^\beta$ in Lemma~\ref{lem-Mbeta} to the definition of $M^{Q^-}$, the result follows.
\end{proof}

\begin{proposition}\label{prop-limit-crit}
	Let $\beta_1, \beta_2, \beta_3 \in \bbR \cup \{Q^-\}$ and $\rho \in \mathcal P$. 
	For $\beta_i \in \bbR\backslash \{Q\}$ let $(\beta_i^n)_{n \geq 1}$ be a sequence with limit $\beta_i$. For $\beta_i = Q$ let $(\beta_i^n)_{n \geq 1}$ be a nonincreasing sequence with limit $Q$. For $\beta_i = Q^-$ let $(\beta_i^n)_{n \geq 1}$ be an increasing sequence with $\lim_{n \to \infty} \beta_i^n = Q$. 
	
	{Let $x_1 = +\infty, x_2 = -\infty, x_3 = 1$.} Let $\mathcal I \subset \{(-\infty, 1), (1, +\infty), \bbR \times \{\pi\} \}$, 
	let $K > 0$ and let $E_K = \{ \phi \::\: |(\phi, \rho)|<K \text{ and } \nu_\phi(I) \leq K \text{ for }I \in \mathcal I\}$.
	Suppose that for any $i$ such that $\beta_i = Q^-$ (resp.\ $\beta_i \geq Q$) the point $x_i$ is an endpoint of some (resp.\ no) interval in $\mathcal I$.
	Then, as a limit in the space of finite measures, 
	\[\lim_{n \to \infty} \left(\prod_{i\: :\: \beta_i = Q^-} \frac1{Q - \beta_i^n}\right) \LF_\mathcal{S}^{(\beta_i^n,x_i)_i} |_{E_K} = \LF_\mathcal{S}^{(\beta_i, x_i)_i}|_{E_K}.\]
	Here, we equip the space of distributions {(on which $\LF_\mathcal{S}^{(\beta_i^n,x_i)_i}$ is a measure)} %which supports \xin{$\LF_\mathcal{S}^{(\beta_i^n,x_i)_i}$}  {[maybe ``space of distributions (on which $\LF$ is a measure)''?]} 
	with the weak-$*$ topology from testing against smooth compactly supported functions. 
\end{proposition}
\begin{proof}
	For a field $\phi$ defined in Definition~\ref{def-critical} and for $A>0$, let $E'_{K, A}$ be the event that $|(\phi, \rho)| < K$ and $\sup_t X_t^i < A$ for all $i$. %We claim that Proposition~\ref{prop-limit-crit} holds when $E_K$ is replaced by $E_{K,A}'$. %Indeed, for a process $X_t$ sampled from $M^\beta$, writing $E'$ for the event that $\sup_t X_t < A$, we have $\lim_{\beta \uparrow Q} \frac1{Q- \beta} M^\beta = M^{Q-}$ in the topology of weak convergence on compact sets. 
	%Indeed, $\lim_{n \to \infty} C_{\mathcal{S}}^{(\beta_1^n, +\infty), (\beta_2^n, -\infty), (\beta_3^n, 1)} = C$ and $\lim_{n \to \infty} \frac12\sum \beta_i^n - Q = s$, and combining this with Lemma~\ref{lem-BM-lim} and the definition of Liouville field from Definition~\ref{def-critical} then gives the claim. 
	By Lemma~\ref{lem-BM-lim} and the definition of Liouville field from Definition~\ref{def-critical}, Proposition~\ref{prop-limit-crit} holds when $E_K$ is replaced by $E_{K,A}'$.
	
	By the above claim, since the conditional probabilities $\LF_\mathcal{S}^{(\beta_i^n, x_i)_i}[E_K \mid E'_{K, A}]$ and $\LF_\mathcal{S}^{(\beta_i, x_i)_i}[E_K \mid E'_{K, A}]$ are uniformly bounded from below uniformly for all $n$, Proposition~\ref{prop-limit-crit} holds when $E_K$ is replaced by $E_K \cap E'_{K,A}$. 
	To bootstrap this to the desired statement, it suffices to show 
	\eqb\label{eq-unif-crit-A}
	\lim_{A \to \infty}\left(\prod_{i : \beta_i = Q^-} (Q - \beta_i^n)^{-1}\right)\LF_\mathcal{S}^{(\beta_i^n, x_i)_i}[E_K \cap (E'_{K, A})^c]= 0 \text{ uniformly in }n.
	\eqe
	To simplify notation we explain this for the case that $\beta_1, \beta_2, \beta_3 = Q^-$; the other cases are similarly shown.  Let $\rho$ be the uniform probability measure on $\{0\} \times [0,\pi]$. Let  $s^n = \frac12\sum \beta_i^n - Q$.  
	A sample $\phi$ from $\left(\prod_{i : \beta_i = Q^-} (Q - \beta_i^n)^{-1}\right)\LF_\mathcal{S}^{(\beta_i^n, x_i)_i}$ can be obtained by sampling 
	\[(h, \mathbf c, X_t^1, X_t^2, X_t^3) \sim P_\rho^{( \beta_i^n)_i} \times [e^{s^n c }\,dc] \times (\frac1{Q - \beta_1^n} M^{\beta_1^n}) \times (\frac1{Q - \beta_2^n} M^{\beta_2^n}) \times (\frac1{Q - \beta_3^n} M^{\beta_3^n})\]
	and combining them to give $\phi$ as in Definition~\ref{def-critical}. 
	Let $\mathbf a_i  = \sup_t X_t^i$ and let $F_{K,A}^1$ be the event that $\mathbf a_1  \geq \max(A, \mathbf a_2, \mathbf a_3)$ and $|(\phi,\rho)|<K$. By Lemma~\ref{lem-Mbeta}, the law of
	$(h, \mathbf c, \mathbf a_1, \mathbf a_2, \mathbf a_3)$ restricted to $F_{K,A}^1$ is 
	\begin{equation}\label{eq-maxima-law}
		P_\rho^{( \beta_i^n)_i} \times [1_{|c| < K} e^{s^n c} dc] \times [1_{a_1>A} e^{-(Q -\beta_1^n)a_1}\, da_1] \prod_{i=2}^3 [1_{0 < a_i < a_1} e^{-(Q -\beta_i^n)a_i}\, da_i].
	\end{equation}
	Let $G$ be the average of $h$ on $\{10\} \times [0,\pi]$, so the maximum of the field average of $\phi$ on $\{t\} \times [0,\pi]$ for $t \geq 10$ is $\mathbf a_1 + G + \mathbf c$. When $\mathbf a_1 + G + \mathbf c$ is large, the LQG-length of any $I \in\mathcal I$ adjacent to $+\infty$ is likely large, so $E_K$ likely does not occur. We quantify this via the existence of the $-p$th GMC moment \cite[Proposition 3.6]{RV10} for any $p>0$, and Markov's inequality: 
	\begin{align*}
		\mathbb P[\nu_{\phi}([1,\infty)) \leq K \mid \mathbf a_1, G, \mathbf c] &= \mathbb P[\nu_{\phi - (\mathbf a_1 + G + \mathbf c)}([10,\infty)) \leq Ke^{-\frac\gamma2(\mathbf a_1 + G + \mathbf c)} \mid \mathbf a_1, G, \mathbf c] 
		\\ &\leq (Ke^{-\frac\gamma2 (\mathbf a_1 + G + \mathbf c)})^p \bbE[\nu_{\phi - (\mathbf a_1 + G + \mathbf c)}([10,\infty))^{-p} \mid \mathbf a_1, G, \mathbf c]  \lesssim (Ke^{-\frac\gamma2 (\mathbf a_1 + G + \mathbf c)})^p.
	\end{align*}
	The last inequality above holds because the field average of $\phi - (\mathbf a_1 + G + \mathbf c)$ on $\{t \} \times [0,\pi]$ for $t \geq 10$ is negative,  {and the projection of $h$ to $H_2(\mathcal S)$ as in~\eqref{eqn-gff-decom} is translation invariant in law and has negative GMC moments on boundary intervals. (See the moment bound in the proof of \cite[Lemma A.1.4]{DMS14} for a similar argument.)} Thus, taking the expectation with respect to $\mathbf a_1, G, \mathbf c$ gives
	\[\prod(Q-\beta_i^n)^{-1} \LF_\mathcal{S}^{(\beta_i^n, x_i)_i} [E_K \cap F_{K,A}^1] \lesssim \int_{-K}^K\int_A^\infty \bbE[ (K e^{-\frac\gamma2(a_1 + G_0 +c)})^p ]e^{-(Q-\beta_i^n)a_1}a_1^2\,da_1\, e^{s^n c} \,dc\]
	%\[\prod(Q-\beta_i^n)^{-1} \LF_\mathcal{S}^{(\beta_i^n, x_i)_i} [E_K \cap F_{K,A}^1] \lesssim \int_A^\infty \bbE[a_1^2 (K e^{-\frac\gamma2(a_1 + G_0 - K)})^p ]da_1 \lesssim A^2 e^{-\frac{\gamma p}2 A} K^p e^{\frac{\gamma p }2 K}\]
	where the implicit constant depends on $p$ but not on $A$ or $n$, and the expectation is taken with respect to the {Gaussian  $G_0$ defined as the average of a GFF sampled from $P_\rho^{(\beta_i^n)_i}/|P_\rho^{(\beta_i^n)_i}|$ on {$\{10\} \times[0,\pi]$}}. In the inequality, the term $a_1^2$ comes from the integrals over $da_2$ and $da_3$ in~\eqref{eq-maxima-law}. Since $\int_A^\infty e^{-\frac\gamma2 p a_1} a_1^2\, da_1 \lesssim A^2 e^{-\frac\gamma2pA}$ {and $G_0$ has mean and variance uniformly bounded in $n$}, the upper bound can be bounded above by a constant times $K^{p+1} e^{-\frac\gamma2 pK} e^{|s|K} A^2 e^{-\frac\gamma2pA}$.
	Defining $F_{K, A}^i$ for $i = 2,3$, the above estimate also holds for these events, so since $(E_{K, A}')^c \subset \bigcup_i F_{K,A}^i$, we obtain 
	\[\prod(Q-\beta_i^n)^{-1} \LF_\mathcal{S}^{(\beta_i^n, x_i)_i} [E_K \cap (E_{K,A}')^c] \lesssim K^{p+1} e^{-\frac\gamma2 pK} e^{|s|K} A^2 e^{-\frac\gamma2pA}.\]
	This gives the desired uniform estimate~\eqref{eq-unif-crit-A}. 
\end{proof}

\begin{definition}\label{def-qt-thick-crit}
	Fix {$W_1, W_2, W_3\geq\frac{\gamma^2}{2}$}. For $W_i > \frac{\gamma^2}2$ let $\beta_i =  \gamma+\frac{2-W_i}{\gamma}<Q$, and for $W_i = \frac{\gamma^2}2$ let $\beta_i = Q^-$. Sample $\phi$ from $\prod_{i: \beta_i \neq Q^-} (Q - \beta_i)^{-1}\LF_{\mathcal{S}}^{(\beta_1, +\infty), (\beta_2, -\infty), (\beta_3, 1)}$. Let $\textup{QT}(W_1, W_2, W_3)$ be the law of $(\mathcal{S}, \phi, +\infty, -\infty, 1)/{\sim_\gamma}$. 
	\\For general $W_1, W_2, W_3 > 0$ with one or more weights equal to $\frac{\gamma^2}2$, define $\QT(W_1, W_2, W_3)$ as in Definition~\ref{def-qt-thin}.
\end{definition}

\begin{remark}\label{rmk:RV-problem}
	A variant of the limiting statement in Proposition~\ref{prop-limit-crit} for Liouville CFT on the sphere was stated {as a conjecture} in \cite[Remark 2.5]{DKRV17a}. We expect that adapting our argument will lead to a proof.
\end{remark}

\section{Imaginary geometry and $\SLE_\kappa(\rho_-;\rho_+,\rho_1)$}\label{sec-pre-ig}
We briefly go over the GFF/SLE coupling in Imaginary Geometry \cite{MS16a} in Section~\ref{subsec:ig}. Then  in Section~\ref{subsec:sle-resampling} we state the \emph{SLE resampling properties} \cite[Section 4]{MS16b}, which will frequently appear in our later proofs. 
{Finally we prove Theorem~\ref{thm-sle-reverse} in Section~\ref{sec-sle-reverse}.}

\subsection{Background on $\SLE_\kappa(\underline{\rho})$ and imaginary geometry}\label{subsec:ig}

The SLE$_\kappa$ curves, as introduced in \cite{Sch00}, is a conformally invariant measure on continuously growing compact hulls $K_t$ with the Loewner driving function $W_t=\sqrt{\kappa}B_t$ (where $B_t$ is the standard Brownian motion). When the background domain is the upper half plane, this can be described by 
\begin{equation}\label{eqn-def-sle}
	g_t(z) = z+\int_0^t \frac{2}{g_s(z)-W_s}ds, \ z\in\mathbb{H},
\end{equation} 
and $g_t$ is the unique conformal transformation from $\mathbb{H}\backslash K_t$ to $\mathbb{H}$ such that $\lim_{|z|\to\infty}|g_t(z)-z|=0$. SLE$_\kappa$ curves also has a natural variant called SLE$_\kappa(\underline{\rho})$, which first appeared in \cite{LSW03} and studied in \cite{Dub05,MS16a}. Fix  $x^{k,L}<...<x^{1,L}\le 0^-\le 0^+\le x^{1,R}<...<x^{\ell, R}$, which are called \textit{force points}, and set $\underline{x} = (\underline{x}_L, \underline{x}_R) = (x^{1,L}, ..., x^{k,L};x^{1,R}, ..., x^{\ell, R})$. For each for each force point $x^{i,q}$, $q\in\{L,R\}$ we assign a \textit{weight} $\rho^{i,q}\in\mathbb{R}$. Let $\underline{\rho}$ be the vector of weights. The  SLE$_\kappa(\underline{\rho})$ process with force points $\underline{x}$ is the measure on compact hulls $(K_t)_{t\ge 0}$ growing the same as ordinary SLE$_\kappa$ (i.e, satisfies \eqref{eqn-def-sle}) except that the Loewner driving function $(W_t)_{t\ge 0}$ are now characterized by 
\begin{equation}\label{eqn-def-sle-rho}
	\begin{split}
		&W_t = \sqrt{\kappa}B_t+\sum_{q\in\{L,R\}}\sum_i \int_0^t \frac{\rho^{i,q}}{W_s-V_s^{i,q}}ds; \\
		& V_t^{i,q} = x^{i,q}+\int_0^t \frac{2}{V_s^{i,q}-W_s}ds, \ q\in\{L,R\}.
	\end{split}
\end{equation}
It has been shown in \cite{MS16a} that  SLE$_\kappa(\underline{\rho})$ processes a.s. exists, is unique and generates a continuous curve until the \textit{continuation threshold}, the first time $t$ such that $W_t = V_t^{j,q}$ with $\sum_{i=1}^j\rho^{i,q}\le -2$ for some $j$ and $q\in\{L,R\}$. Let $f_t := g_t-W_t$ be the \textit{centered Loewner flow.}

Now we recall the notion of \emph{the  GFF flow lines}. Heuristically, given a GFF $h$, $\eta(t)$ is a flow line of angle $\theta$ if
\begin{equation}
	\eta'(t) = e^{i(\frac{h(\eta(t))}{\chi}+\theta)}\ \text{for}\ t>0, \ \text{where}\ \chi = \frac{2}{\sqrt{\kappa}}-\frac{\sqrt{\kappa}}{2}.
\end{equation} 
{To be more precise,  \cite[Theorem 1.1]{MS16a} introduces an exact coupling of a Dirichlet GFF with an $\SLE_\kappa(\underline{\rho})$, which we briefly recap as follows. Let $(K_t)_{t\ge 0}$ be the hull at time $t$ of the SLE$_\kappa(\underline{\rho})$ process described by the Loewner flow \eqref{eqn-def-sle} with $(W_t, V_t^{i,q})$ solving \eqref{eqn-def-sle-rho} with filtration $\mathcal{F}_t$.  Let $\mathfrak{h}_t^0$ be the harmonic function on $\mathbb{H}$ with boundary values 
	$$
	-\lambda(1+\sum_{i=0}^j \rho^{i,L})\ \ \text{on} \ [V_t^{j+1, L}, V_t^{j,L})\ \  \text{and}\ \ \lambda(1+\sum_{i=0}^j \rho^{i,R})\ \text{on}\ \ [V_t^{j, R}, V_t^{j+1,R})
	$$
	where $\lambda = \frac{\pi}{\sqrt{\kappa}}$, $\rho^{0,R} = \rho^{0,L}=0$, $x^{0,L} = 0^-, x^{0,R} = 0^+, x^{k+1, L} = -\infty, x^{\ell+1, R} = +\infty$. Set {$\mathfrak{h}_t(z) = \mathfrak{h}_t^0(g_t(z))-\chi\arg g_t'(z)$}. Let  $\tilde{h}$ be  a zero boundary GFF on $\mathbb{H}$ and $h = \tilde{h}+ \mathfrak{h}_0$. Then for any $\mathcal{F}_t$-stopping time $\tau$ before the continuation threshold, $K_\tau$ is a local set for $h$ and the conditional law of $h|_{\mathbb{H}\backslash K_\tau}$ given $\mathcal{F}_\tau$ is the same as the law of $\mathfrak{h}_\tau+\tilde{h}\circ f_\tau$.}

% state the SLE and GFF coupling result \cite[Theorem 1.1]{MS16a}. Set .

% \begin{theorem}\label{thm-ig1}
	% 	Fix $\kappa>0$, a vector $\underline{\rho}$ of weights and a vector $\underline{x}$ of force points. Let $(K_t)_{t\ge 0}$ be the hull at time $t$ of the SLE$_\kappa(\underline{\rho})$ process described by the Loewner flow \eqref{eqn-def-sle} with $(W_t, V_t^{i,q})$ solving \eqref{eqn-def-sle-rho}. Let $\mathfrak{h}_t^0$ be the harmonic function on $\mathbb{H}$ with boundary values 
	% 	$$
	% 	-\lambda(1+\sum_{i=0}^j \rho^{i,L})\ \ \text{on} \ [V_t^{j+1, L}, V_t^{j,L})\ \  \text{and}\ \ \lambda(1+\sum_{i=0}^j \rho^{i,R})\ \text{on}\ \ [V_t^{j, R}, V_t^{j+1,R})
	% 	$$
	% 	where $\rho^{0,R} = \rho^{0,L}=0$, $x^{0,L} = 0^-, x^{0,R} = 0^+, x^{k+1, L} = -\infty, x^{\ell+1, R} = +\infty$. Set $\mathfrak{h}_t(z) = \mathfrak{h}_t^0(f_t(z))-\chi\arg f_t'(z)$. Let $\mathcal{F}_t$ be the filtration generated by $(W, V^{i,q})$. Then there exists a coupling $(K,h)$ where $h = \tilde{h}+ \mathfrak{h}_0$ with $\tilde{h}$ being a zero boundary GFF on $\mathbb{H}$ such that the following is true. For any $\mathcal{F}_t$-stopping time $\tau$ before the continuation threshold , $K_\tau$ is a local set for $h$ and the conditional law of $h|_{\mathbb{H}\backslash K_\tau}$ given $\mathcal{F}_\tau$ is the same as the law of $\mathfrak{h}_\tau+\tilde{h}\circ f_\tau$.
	% \end{theorem}
For $\kappa<4$, the SLE$_\kappa(\underline{\rho})$ coupled with the GFF $h$ as above is referred as the \text{flow lines} of $h$, and we say an SLE$_\kappa(\underline{\rho})$ curve is a flow line of angle $\theta$ if it can be coupled with $h+\theta\chi$. Moreover, \cite[Theorem 1.2]{MS16a} shows that these flow lines are a.s. determined by the GFF $h$, and we can simultaneously consider flow lines starting from different boundary points. Furthermore, by \cite[Theorem 1.5]{MS16a}, the interaction of these flow lines (i.e., crossing, merging, etc.) are completely determined by their angles.   One can also make sense of flow lines of GFF starting from interior points, and see \cite{MS17} for more details.

\subsection{Coupling of two flow lines}\label{subsec:sle-resampling}
One important consequence is that, as argued in \cite[Section 6]{MS16a}, suppose $\eta_1$ and $\eta_2$ are flow lines of $h$, then given $\eta_1$, the conditional law of $\eta_2$ is the same as the law of the flow line (with some angle) of the GFF in $\mathbb{H}\backslash\eta_1$ with the \textit{flow line boundary conditions} (one can go to \cite[Figure 1.10]{MS16a} for more explanation) induced by $\eta_1$, and vice versa for the law of $\eta_1$ given $\eta_2$. The \emph{SLE resampling property} states that these two conditional laws actually uniquely characterize the joint law of $(\eta_1,\eta_2)${, at least in some parameter ranges. 
	
	We summarize the imaginary geometry input we need for domains with three marked points in Proposition~\ref{prop:ig-flow-descriptions} below. Suppose $\eta_1$ is a simple curve in $\bbH$ from $0$ to $\infty$ which does not hit 1. We want to make sense of the curve $\eta_2 \sim \SLE_\kappa(\rho_-,\rho_0;\rho_+)$ in $\bbH \backslash \eta_1$ in the region to the right of $\eta_1$ from 1 to $\infty$. The definition is clear if $\eta_1 \cap (0, \infty) = \emptyset$, {where the force points are located at $1^-,0;1^+$}. Otherwise, let $p$ be the rightmost point of $\eta_1 \cap [0,1)$ and $q$ the leftmost point of $\eta_1 \cap (1, +\infty)$ (with $q = \infty$ if $\eta_1$ is disjoint from $(1, \infty)$). Let $D$ be the connected component of $\bbH \backslash \eta_1$ with $1$ on its boundary, and sample an $\SLE_\kappa(\rho_-,\rho_0;\rho_+)$ curve in $D$ from $1$ to $q$ with {force points at $1^-,p;1^+$}.  If $q = \infty$ then $\eta_2$ is this curve. Otherwise, in each connected component of $\bbH \backslash \eta_1$ to the right of $D$ we sample an independent $\SLE(\rho_0+\rho_-;\rho_+)$ and let $\eta_2$ be the concatenation of all the sampled curves. Similarly, if $\eta_2$ is a curve in $\bbH$ from $1$ to $\infty$ which does not hit $0$, we can define $\SLE_\kappa(\rho_-;\rho_+,\rho_1)$ in $\bbH \backslash \eta_2$ in the region to the left of $\eta_2$ from 0 to $\infty$. 
	
	% {[change notation to avoid conflict with Section 6 where $a_1,a_2, a_3 \in \partial D$...]}
	\begin{proposition}\label{prop:ig-flow-descriptions}
		Let $\theta_1,\theta_2,x_1, x_2, x_3\in\bbR$ such that $\theta_1>\theta_2$ and
		\begin{equation}\label{eqn-ig-range}
			x_1<\lambda-\theta_1\chi; \ x_3>-\lambda-\theta_2\chi;\ -\lambda-\theta_1\chi<x_2<\lambda-\theta_2\chi. 
		\end{equation}
		The following two laws on pairs of curves $(\eta_1, \eta_2)$ agree:
		\begin{itemize}
			\item Sample $\eta_1$ in $\bbH$ from $0$ to $\infty$ as $\SLE_\kappa(-\frac{x_1+\theta_1\chi}{\lambda}-1; \frac{x_2+\theta_1\chi}{\lambda}-1, \frac{x_3-x_2}{\lambda})$. In $\bbH \backslash \eta_1$ in the region to the right of $\eta_1$, sample $\eta_2$ from $1$ to $\infty$ as  $\SLE_\kappa(- \frac{x_2+\theta_2\chi}{\lambda}-1,\frac{x_2+\theta_1\chi}{\lambda}-1;\frac{x_3+\theta_2\chi}{\lambda}-1 )$.
			\item Sample $\eta_2$ in $\bbH$ from $1$ to $\infty$ as $\SLE_\kappa(- \frac{x_2+\theta_2\chi}{\lambda}-1,\frac{x_2-x_1}{\lambda};\frac{x_3+\theta_2\chi}{\lambda}-1 )$. In  $\bbH \backslash \eta_2$ in the region to the left of $\eta_2$, sample $\eta_1$ from $0$ to $\infty$ as  $\SLE_\kappa(-\frac{x_1+\theta_1\chi}{\lambda}-1; \frac{x_2+\theta_1\chi}{\lambda}-1, \frac{-\theta_2\chi-x_2}{\lambda}-1)$.
		\end{itemize}
		Furthermore, for $(\theta_1-\theta_2)\chi\ge\frac{\sqrt{\kappa}\pi}{2}$, this law on $(\eta_1, \eta_2)$ is characterized by the following:
		\begin{itemize}
			\item Almost surely $\eta_1 \cap [1, \infty) = \emptyset$ and $\eta_2 \cap (-\infty, 0] = \emptyset$. Moreover, the conditional law of $\eta_1$ given $\eta_2$ is $\SLE_\kappa(-\frac{x_1+\theta_1\chi}{\lambda}-1; \frac{x_2+\theta_1\chi}{\lambda}-1, \frac{-\theta_2\chi-x_2}{\lambda}-1)$, and the conditional law of $\eta_2$ given $\eta_1$ is $\SLE_\kappa(- \frac{x_2+\theta_2\chi}{\lambda}-1,\frac{x_2+\theta_1\chi}{\lambda}-1;\frac{x_3+\theta_2\chi}{\lambda}-1 )$.
		\end{itemize}
\end{proposition}}
See Figure \ref{fig-ig-resample} for an illustration of the setting. The first statement is clear from the flow line conditioning in \cite[Section 6]{MS16a}. The second statement is the resampling property of flow lines in \cite{MS16b}. We remark that the original statement  \cite[Theorem 4.1]{MS16b} is for curves with same starting and ending points; the proof is based on a Markov chain mixing argument and the first step is to apply the SLE duality argument to separate the initial and terminal points of $\eta_1$ and $\eta_2$. Therefore the same argument readily applies (and is simpler) in the case described in Proposition \ref{prop:ig-flow-descriptions}. 

\begin{figure}[ht]
	\centering
	\begin{tabular}{ccc} 
		\includegraphics[scale=0.43]{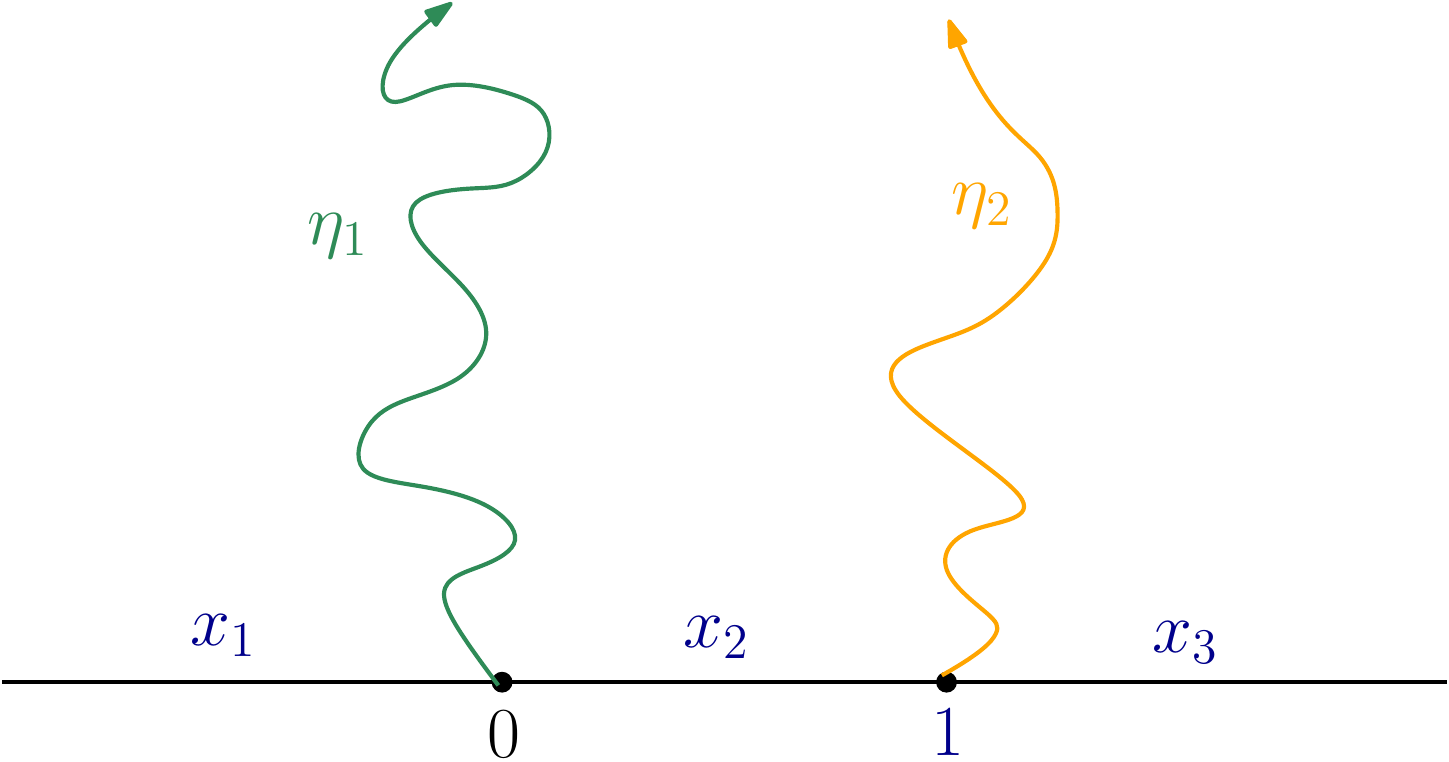}
		& &
		\includegraphics[scale=0.48]{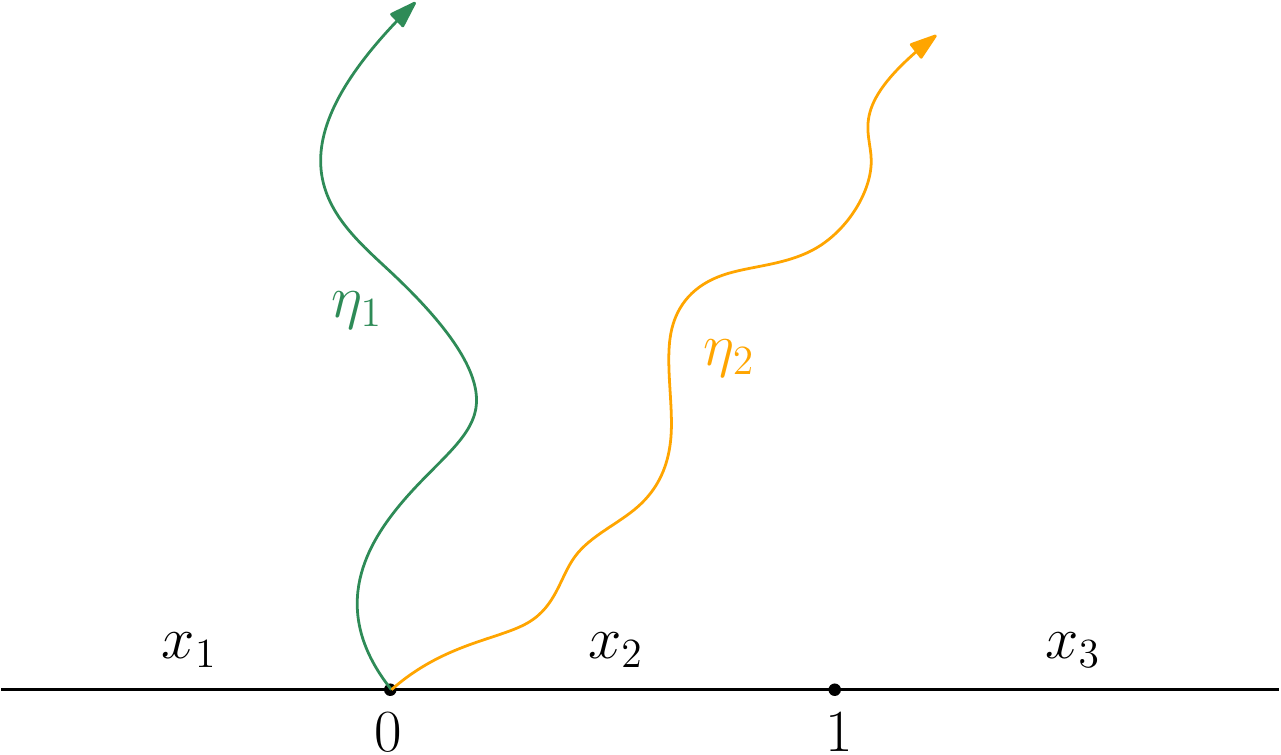}
	\end{tabular}
	\caption{\textbf{Left}: Suppose $h$ is a Dirichlet GFF with piecewise constant boundary condition $x_1\textbf{1}_{(-\infty,0)}(x)+x_2\textbf{1}_{[0,1)}(x)+x_3\textbf{1}_{[1,+\infty)}(x)$. Let $\eta_1$ (resp. $\eta_2$) be the flow line of $h$ starting from 0 (resp. 1) of angle $\theta_1$ (resp. $\theta_2$). Then  the marginal law of $\eta_1$ is SLE$_\kappa(-\frac{x_1+\theta_1\chi}{\lambda}-1; \frac{x_2+\theta_1\chi}{\lambda}-1, \frac{x_3-x_2}{\lambda})$ with force points at $(0^-;0^+,1)$, and the marginal law of $\eta_2$ is SLE$_\kappa(- \frac{x_2+\theta_2\chi}{\lambda}-1,\frac{x_2-x_1}{\lambda};\frac{x_3+\theta_2\chi}{\lambda}-1 )$ force points at $(1^-,0;1^+)$. Suppose $\theta_1\ge\theta_2$. By \cite[Section 6]{MS16a}, one can also read the conditional law of $\eta_1$ given $\eta_2$, which is SLE$_\kappa(-\frac{x_1+\theta_1\chi}{\lambda}-1; \frac{x_2+\theta_1\chi}{\lambda}-1, \frac{-\theta_2\chi-x_2}{\lambda}-1)$ in the left component of $\mathbb{H}\backslash \eta_2$, and similarly conditional law of $\eta_2$ given $\eta_1$ is  SLE$_\kappa(- \frac{x_2+\theta_2\chi}{\lambda}-1,\frac{x_2+\theta_1\chi}{\lambda}-1;\frac{x_3+\theta_2\chi}{\lambda}-1 )$  in the right component of $\mathbb{H}\backslash \eta_1$.
		% {[The $\eta_2 \mid \eta_1$ law is wrong. Near the start of the curve there should be no $\theta_1$ dependence.]}{I think there is no issue here since the $\SLE$ weights corresponds to $1^-$, 0, $1^+$.}
		If $\theta_1<\theta_2$, then given $\eta_2$, the segment of $\eta_1$ before crossing $\eta_2$ is the same as non-crossing case and $\eta_1$ can be continued after crossing. Finally, by the resampling property, the conditional laws of $\eta_1|\eta_2$ and $\eta_2|\eta_1$ uniquely characterize the joint distribution of $(\eta_1,\eta_2)$, as constructed using Imaginary Geometry. \textbf{Right}: One can similarly consider the flow lines from 0 and read off the marginal and conditional laws as in Lemma \ref{lm:ig-commute}.   }\label{fig-ig-resample} 
\end{figure}

\subsection{Reversibility of SLE$_\kappa(\rho_-;\rho_+,\rho_1)$}\label{sec-sle-reverse}

In this section, as an application of the Imaginary Geometry flow lines and the curve resampling properties, we extend the result on reversibility of SLE$_\kappa(0;\rho_+,\rho_1)$ in \cite{Zhan19} to SLE$_\kappa(\rho_-;\rho_+,\rho_1)$ curves. 

To begin with, let us recall the notion of \emph{SLE weighted by conformal derivative}. Given $\rho_-,\rho_+>-2$, $\rho_1>-2-\rho_+$ (which implies that the continuation threshold is never hit) and $\alpha\in\bbR$, we define the measure $\widetilde{\SLE}_\kappa(\rho_-;\rho_+,\rho_1;\alpha)$ on curves $\eta$ from 0 to $\infty$ on $\mathbb{H}$ as follows. Let $D_\eta$ be the component of $\mathbb{H}\backslash \eta$ containing $1$, and $\psi_\eta$ the unique conformal map from $D_\eta$ to $\mathbb{H}$ fixing 1 and sending the first (resp. last) point on $\partial D_\eta$ hit by $\eta$ to 0 (resp. $\infty$). Then our $\widetilde{\SLE}_\kappa(\rho_-;\rho_+,\rho_1;\alpha)$ on $\bbH$ is defined by
\begin{equation}\label{eqn-sle-CR-1}
	\frac{d\widetilde{\SLE}_\kappa(\rho_-;\rho_+,\rho_1;\alpha)}{d{\SLE}_\kappa(\rho_-;\rho_+,\rho_1)}(\eta) = |\psi_{\eta}(1)'|^\alpha
\end{equation}
where the force points of ${\SLE}_\kappa(\rho_-;\rho_+,\rho_1)$ is $0^-,0^+,1$.  This definition can be  extended to other domains via conformal transforms, while by symmetry, we can also define the version $d\widetilde{\SLE}_\kappa(\rho_-,\rho_{-1};\rho_+;\alpha)$ with $1$ replaced by $-1$ and force points {$0^-,-1;0^+$} similarly. Also let $\mathcal{R}(\eta)$ be the time reversal of $\eta$. {With these notations, we state the result in \cite{Zhan19} as follows.}
{\begin{theorem}\label{thm-dapeng}
		Suppose $\eta$ an $\SLE_\kappa(0;\rho_+,\rho_1)$ curve in $\bbH$ from $0$ to $\infty$ with force located at $0^+$ and 1, and $\rho_+>-2, \rho_+ + \rho_1>-2$. Let $\mathcal{L}$ be the law of the time reversal $\mathcal{R}(\eta)$ under the conformal mapping $z\mapsto-\frac{1}{z}$. Then $\mathcal{L}$ is a constant multiple of the measure $\widetilde{\SLE}_\kappa(\rho_++\rho_1,-\rho_1;0;\frac{\rho_1(4-\kappa)}{2\kappa})$. 
\end{theorem}}
We note that the theorem above  is implicitly shown in Theorem 1.1 and Section 3.2 of \cite{Zhan19} via the construction of the reversed curve.  The statement is  for general $\SLE_\kappa(\underline{\rho})$ curves with all force points lying on the same side of 0, while in this paper we only work on the $0^+,1$ force point case for simplicity. To prove   Theorem~\ref{thm-sle-reverse}, we begin with the following variant of Proposition~\ref{prop:ig-flow-descriptions}. Again suppose we want to sample $\eta_2\sim\widetilde{\SLE}_\kappa(\rho_-;\rho_+,\rho_1;\alpha)$ {going from 0 to $\infty$} to the right of $\eta_1$ in $\bbH\backslash\eta_1$ when  $\eta_1$ is hitting $(0,\infty)$, let $p,q$ be the left and right most point on $\bbR$ on the boundary of the connected component of $D_{\eta_1}$. In each component of $\bbH\backslash\eta_1$ whose boundary contains a segment of $(0,p)$ (resp.\ $(q,\infty)$), we sample an independent $\SLE_\kappa(\rho_-;\rho_+)$ (resp.\ $\SLE_\kappa(\rho_-;\rho_++\rho_1)$), and in $D_{\eta_1}$ we sample an $\widetilde{\SLE}_\kappa(\rho_-;\rho_+,\rho_1;\alpha)$ curve from $p$ to $q$. Then $\eta_2$ is the concatenation of these curves.

% {[change notation to avoid conflict with Section 6 where $a_1,a_2, a_3 \in \partial D$...]}
\begin{lemma}\label{lm:ig-commute}
	Let $x_1,x_2,x_3,\alpha,\theta_1,\theta_2\in\bbR$ such that  
	$$\theta_1>\theta_2;\ x_1<\lambda-\theta_1\chi;\ x_2,x_3>-\lambda-\theta_2\chi.  $$ The following two laws on pairs of curves $(\eta_1, \eta_2)$ agree:
	\begin{itemize}
		\item Sample $\eta_1$ in $\bbH$ from $0$ to $\infty$ as $\widetilde{\SLE}_\kappa(-\frac{x_1+\theta_1\chi}{\lambda}-1; \frac{x_2+\theta_1\chi}{\lambda}-1, \frac{x_3-x_2}{\lambda};\alpha)$. In $\bbH \backslash \eta_1$ in the region to the right of $\eta_1$, sample $\eta_2$ from $0$ to $\infty$ as  $\widetilde{\SLE}_\kappa(- \frac{(\theta_1-\theta_2)\chi}{\lambda}-2;\frac{x_2+\theta_2\chi}{\lambda}-1, \frac{x_3-x_2}{\lambda};\alpha)$.
		\item Sample $\eta_2$ in $\bbH$ from $0$ to $\infty$ as $\widetilde{\SLE}_\kappa(-\frac{x_1+\theta_2\chi}{\lambda}-1; \frac{x_2+\theta_2\chi}{\lambda}-1, \frac{x_3-x_2}{\lambda};\alpha)$. In  $\bbH \backslash \eta_2$ in the region to the left of $\eta_2$, sample $\eta_1$ from $0$ to $\infty$ as  $\SLE_\kappa(-\frac{x_1+\theta_1\chi}{\lambda}-1; \frac{(\theta_1-\theta_2)\chi}{\lambda}-2)$.
	\end{itemize}
\end{lemma}

\begin{proof}
	For $\alpha=0$ case, the result is straightforward from the flow line conditioning argument in \cite[Section 6]{MS16a} as drawn in Figure \ref{fig-ig-resample}. If $\alpha\neq 0$, let $\mathcal{P}$ be the measure on curves $(\hat{\eta}_1,\hat{\eta}_2)$ as in the statement with $\alpha=0$. If we let $\mathcal{L}$ be the joint law of $(\eta_1,\eta_2)$ constructed from the second way (i.e.\ start with $\eta_2$ then sample $\eta_1$) 
	\begin{equation}\label{eqn-sle-reverse-1}
		\frac{d\mathcal{L}}{d\mathcal{P}}(\eta_1,\eta_2) = |\psi_{\eta_2}'(1)|^{\alpha}.
	\end{equation}
	Meanwhile, if we first sample $\eta_1$ and then $\eta_2$ conditioned on $\eta_1$ as in the statement and let $\tilde{\mathcal{L}}$ be the joint law of $(\eta_1,\eta_2)$, then 
	\begin{equation}\label{eqn-sle-reverse-2}
		\frac{d\tilde{\mathcal{L}}}{d\mathcal{P}}({\eta}_1,{\eta}_2) = |\psi_{{\eta}_1}'(1)|^{\alpha}|\psi_{{\eta}_2|{\eta}_1}'(1)|^{\alpha}
	\end{equation}
	where $\psi_{{\eta}_2|{\eta}_1}$ is the conformal map from the component of $\bbH\backslash\psi_{{\eta}_1}({\eta}_2)$ containing 1 to $\bbH$ fixing 1 and sending the first (resp.\ last) point hit by $\psi_{{\eta}_1}({\eta}_2)$ to 0 (resp.\ $\infty$). Then we observe that $\psi_{{\eta}_2} = \psi_{{\eta}_2|{\eta}_1}\circ\psi_{{\eta}_1}$ and therefore the two Radon-Nikodym derivatives \eqref{eqn-sle-reverse-1} and \eqref{eqn-sle-reverse-2} are the same.
\end{proof}
% Using the trick of conditioning, we can then prove Theorem~\ref{thm-sle-reverse}.  {[either be more specific about what this trick is, or remove this sentence.]}
%\begin{theorem}\label{thm-sle-reverse}
%	Fix $\rho_+>-2$, $\rho_->-2$, and $\rho_1>-2-\rho_+$. Let $\eta$ be an $\SLE_\kappa(\rho_-;\rho_+,\rho_1)$ curve in $\bbH$ from $0$ to $\infty$ with force located at $0^-,0^+$ and 1. Then the law of $\mathcal{R}(\eta)$ is the probability measure  proportional to $\widetilde{\SLE}_\kappa(\rho_++\rho_1,-\rho_1;\rho_-;\frac{\rho_1(4-\kappa)}{2\kappa})$. 
%\end{theorem}

\begin{proof}[Proof of Theorem~\ref{thm-sle-reverse}]
	We start with the case {$\rho_-\leq 0$. The $\rho_-=0$ case is already covered in Theorem~\ref{thm-dapeng}.}. %First assume that $\rho_+,\rho_++\rho_1\ge \frac{\kappa}{2}-2$ so that the $\SLE_\kappa(\rho_-;\rho_+,\rho_1)$ curve $\eta$ in $\bbH$ is a.s.\ not hitting $(0,\infty)$  {[what is $\eta$? Also I think this simple topology $\rho_- \leq 0$ proof can be clearer and more succinct. Essentially we would just write down the text from Figure 4, and add like two more sentences to explain why the law of $\eta_1$ given $\eta_2$ is as described. ]}.
	We sample a curve $\eta_1$ from $\widetilde{\SLE}_\kappa(\rho_++\rho_1,-\rho_1;0;\frac{\rho_1(4-\kappa)}{2\kappa})$ and a curve $\eta_2\sim \SLE_\kappa(\rho_-;-\rho_--2)$ with force points at $0^-;0^+$ on the right component of $\bbH\backslash\eta_1$. Then by Theorem \ref{thm-dapeng} we know that the marginal law of $\mathcal{R}(\eta_1)$ is now $\SLE_\kappa(0;\rho^+,\rho_1)$; furthermore, by \cite[Theorem 1.1]{MS16b}, the conditional law of  $\mathcal{R}(\eta_2)$ given $\mathcal{R}(\eta_1)$ is $\SLE_\kappa(-\rho_--2;\rho_-)$. Therefore by Lemma~\ref{lm:ig-commute}, the conditional law of $\mathcal{R}(\eta_1)$ given $\mathcal{R}(\eta_2)$ is precisely $\SLE_\kappa(\rho_-;\rho_+,\rho_1)$. See Figure \ref{fig-sle-reverse-1}.
	
	\begin{figure}[ht]
		\begin{tabular}{ccc} 
			\includegraphics[width=0.49\textwidth]{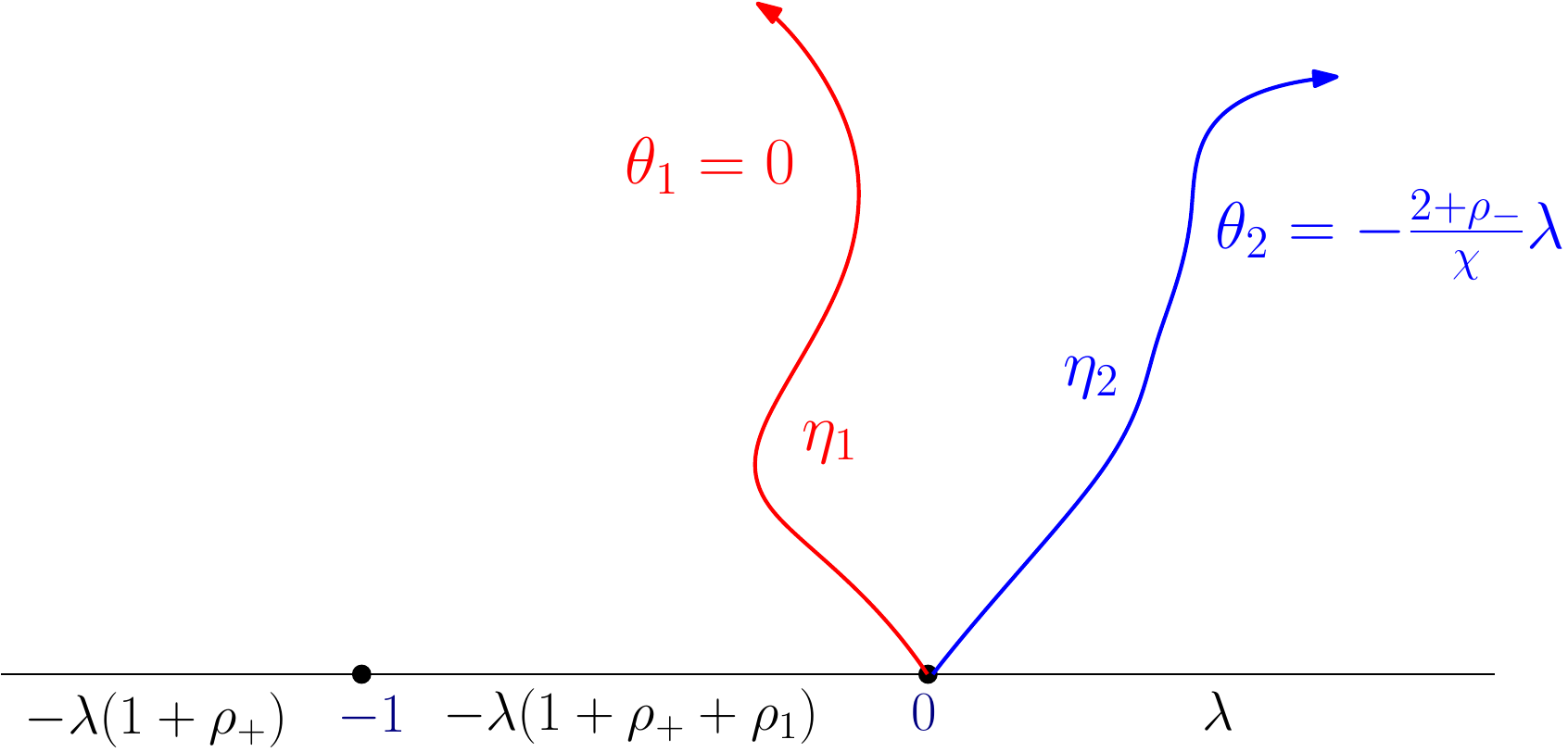}
			& &
			\includegraphics[width=0.49\textwidth]{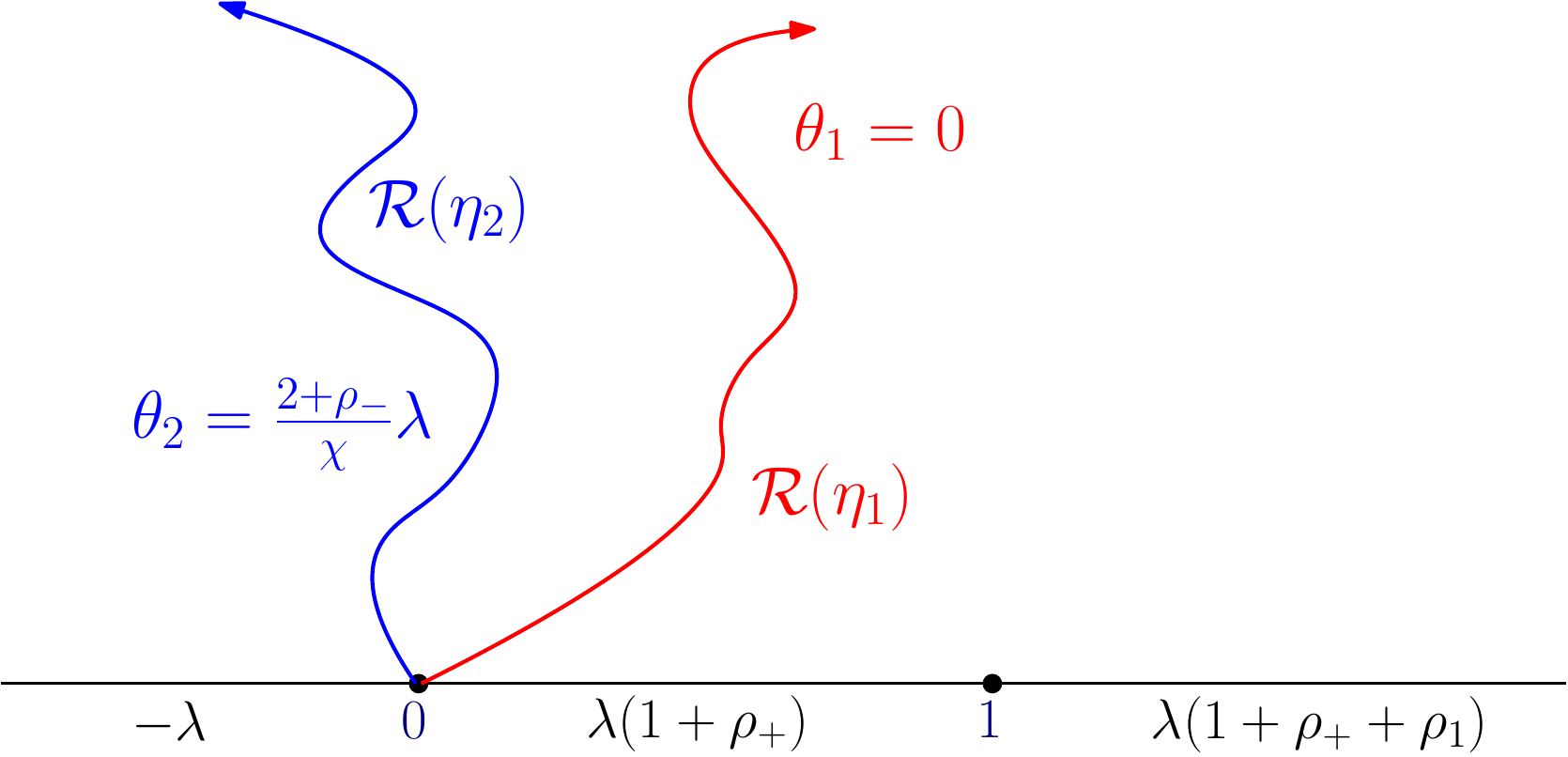}
		\end{tabular}
		\caption{\textbf{Left:} The curves $(\eta_1,\eta_2)$, whose the law has Radon-Nikodym derivative $|\psi_{\eta_1}'(-1)|^{\frac{\rho_1(4-\kappa)}{2\kappa}}$ with respect to the corresponding flow lines of the GFF with the depicted boundary values. By Lemma~\ref{lm:ig-commute} the conditional law of $
			\eta_1$ given $\eta_2$ is $\widetilde{\SLE}_\kappa(\rho_++\rho_1,-\rho_1;\rho_-;\frac{\rho_1(4-\kappa)}{2\kappa})$.   \textbf{Right:} An Imaginary Geometry coupling of $\mathcal{R}(\eta_1)$ and $\mathcal{R}(\eta_2)$ where the marginal law of $\mathcal{R}(\eta_1)$ is $\SLE_\kappa(0;\rho^+,\rho_1)$ by Theorem~\ref{thm-dapeng} and the law of $\mathcal{R}(\eta_2)$ given $\mathcal{R}(\eta_1)$ is $\SLE_\kappa(-\rho_--2;\rho_-)$. Another application of Lemma~\ref{lm:ig-commute} gives that the law of $\mathcal{R}(\eta_1)$ given $\eta_2$ is $\SLE_\kappa(\rho_-;\rho_+,\rho_1)$.  }\label{fig-sle-reverse-1}
	\end{figure}

	{Now suppose $\rho^-\in (0,2]$}. We first sample a curve $\eta_2$ on $\bbH$ from 0 to $\infty$ from $\widetilde{\SLE}_\kappa(2+\rho_++\rho_1,-\rho_1;\rho_--2;\frac{\rho_1(4-\kappa)}{2\kappa})$ and then $\eta_1$ from $\widetilde{\SLE}_\kappa(\rho_++\rho_1,-\rho_1;0;\frac{\rho_1(4-\kappa)}{2\kappa})$ on the left component of $\mathbb{H}\backslash\eta_2$. Then using the same conformal map composing argument, we observe that the Radon-Nikodym derivative of the law of $(\eta_1,\eta_2)$ with respect to the flow lines of the GFF with the corresponding boundary values in Figure \ref{fig-sle-reverse-2} is $|\psi_{\eta_1}'(-1)|^{\frac{\rho_1(4-\kappa)}{2\kappa}}$, and the marginal law of $\eta_1$ is $\widetilde{\SLE}_\kappa(\rho_++\rho_1,-\rho_1;\rho_-;\frac{\rho_1(4-\kappa)}{2\kappa})$. By Theorem \ref{thm-dapeng}, we know that the conditional law of $\mathcal{R}(\eta_1)$ given $\mathcal{R}(\eta_2)$ is $\SLE_\kappa(0;\rho_+,\rho_1)$, while by what we have just proved, since $\rho_--2\le 0$,  the marginal law of $\mathcal{R}(\eta_2)$ is $\SLE_\kappa(\rho_--2;2+\rho^+,\rho_1)$. Therefore using the Imaginary Geometry coupling we observe that the marginal law of $\mathcal{R}(\eta_1)$ is $\SLE_\kappa(\rho_-;\rho_+,\rho_1)$, which concludes the proof for $\rho_-\in[0,2]$ case. Also see Figure \ref{fig-sle-reverse-2} for an illustration.
	
	Finally we notice that the above argument (i.e., the coupling in Figure \ref{fig-sle-reverse-2}) can be iterated, giving the reversibility for  $\rho_-\in (2,4], (4,6]$, ..., etc.. This finishes the proof of Theorem \ref{thm-sle-reverse}.  
	\begin{figure}[ht]
		\begin{tabular}{ccc} 
			\includegraphics[width=0.49\textwidth]{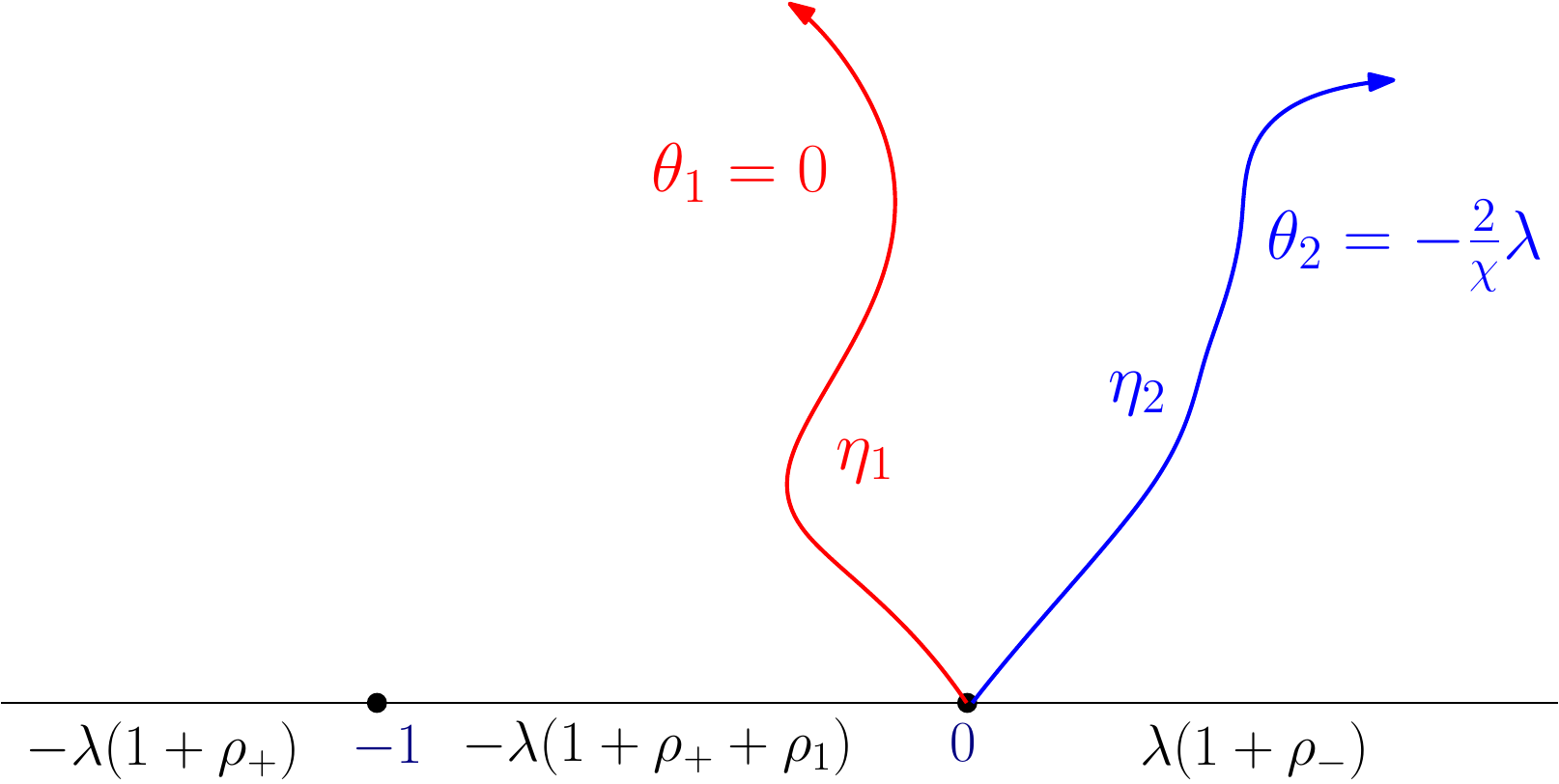}
			& &
			\includegraphics[width=0.49\textwidth]{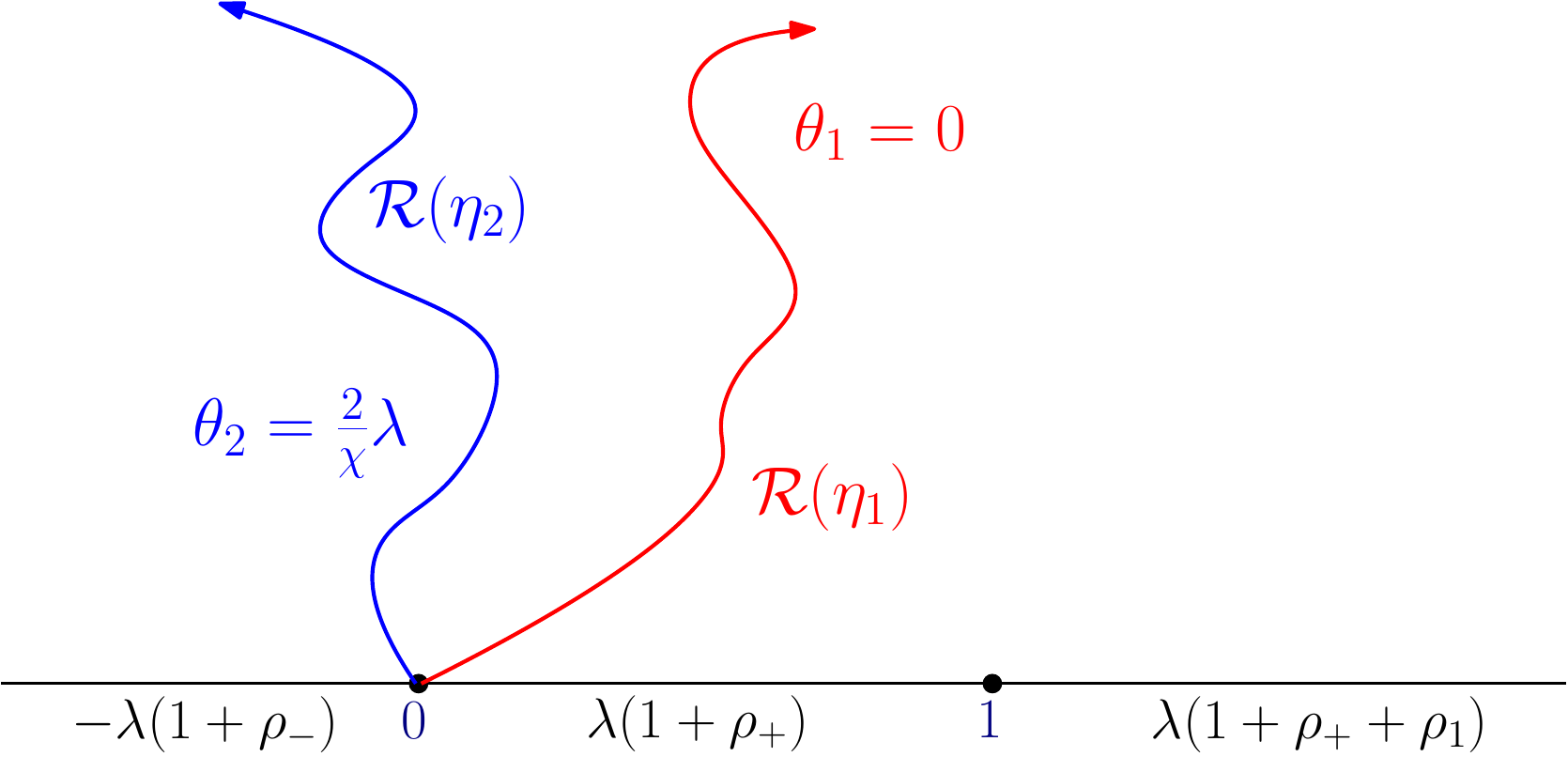}
		\end{tabular}
		\caption{\textbf{Left:} The curves $(\eta_1,\eta_2)$, whose the law has Radon-Nikodym derivative $|\psi_{\eta_1}'(-1)|^{\frac{\rho_1(4-\kappa)}{2\kappa}}$ with respect to the corresponding flow lines of the GFF with the depicted boundary values. One can show that the marginal law of $
			\eta_1$ is $\widetilde{\SLE}_\kappa(\rho_++\rho_1,-\rho_1;\rho_-;\frac{\rho_1(4-\kappa)}{2\kappa})$.   \textbf{Right:} An Imaginary Geometry coupling of $\mathcal{R}(\eta_1)$ and $\mathcal{R}(\eta_2)$ where the marginal law of $\mathcal{R}(\eta_2)$ is $\SLE_\kappa(\rho_--2;2+\rho^+,\rho_1)$ and the conditional law of $\mathcal{R}(\eta_1)$ given $\mathcal{R}(\eta_2)$ is $\SLE_\kappa(0;\rho_+,\rho_1)$. Then we see that the marginal law of $\mathcal{R}(\eta_1)$ is $\SLE_\kappa(\rho_-;\rho_+,\rho_1)$.  }\label{fig-sle-reverse-2}
	\end{figure}
\end{proof}

\section{Conformal welding of $\QT(W,W,2)$ and two thin quantum disks}\label{sec:WW2}

{In this section we prove Proposition~\ref{thm-w-2u}, a result on the conformal welding of $\QT(W,W,2)$ and two thin quantum disks. See Figure \ref{fig-w-2u} for an illustration.  This is the first step towards Theorem~\ref{thm:M+QTp-rw} where $W_1\neq W_2$, since Proposition~\ref{thm-w-2u} involves the conformal welding of a quantum disk and a quantum triangle such that the weights of the quantum triangle vertices along the interface are not equal.}

%{In Sections~\ref{sec:WW2} to~\ref{sec-6-thin-qt} we will prove the main theorems. The $W_1=W_2$ case is essentially covered in~\cite{AHS21} (see also Proposition~\ref{prop-3-pt-disk} below), and we shall extend to the $W_1\neq W_2$ case. To begin with, in this section we prove Proposition~\ref{thm-w-2u}, a result about conformal welding of $\QT(W,W,2)$ and two thin quantum disks. See Figure \ref{fig-w-2u} for an illustration.  This is the first step towards Theorem~\ref{thm:M+QTp-rw} where $W_1\neq W_2$, since Proposition~\ref{thm-w-2u} involves the conformal welding of a quantum disk and a quantum triangle such that the weights of the quantum triangle vertices along the interface are not equal. Then later in Section~\ref{sec:resampling}, Proposition~\ref{thm-w-2u} will serve as a critical input for the Markov chain resampling argument to prove Theorem~\ref{thm:M+QTp-rw} for a restricted range. }

%The aim of this section is to prove the following result. 

\begin{proposition}\label{thm-w-2u}
	Fix $W>\frac{\gamma^2}{2}$ and $U\in(0,\frac{\gamma^2}{2})$. Take a triangle from $\QT(W+U,W+U,2+2U)$ embedded as $(\bbH, \phi, 0, \infty, 1)$ with 1 being the weight $2+2U$ point. Then there exists some constant $c$ depending only on $W$ and $U$, and some probability measure $\mathsf{m}(W;U)$ of pairs of curves $(\eta_1,\eta_2)$ where $\eta_1$ runs from 1 to 0 and $\eta_2$ runs from 1 to $\infty$  such that the following welding equation holds:
	\begin{equation}\label{eqn-w-2u}
		\QT(W+U,W+U,2+2U)\otimes\mathsf{m}(W;U)  = c\int_0^\infty \int_0^\infty \Wd(\QT(W,W,2;\ell,\ell'), \Md_2(U;\ell), \Md_2(U;\ell')d\ell d\ell'
	\end{equation}
	where $\QT(W,W,2;\ell,\ell')$ is disintegration over the length of the two boundary arcs containing the weight 2 vertex and $\Wd$ stands for identifying the edges of lengths $\ell$, $\ell'$. 
\end{proposition}

\begin{figure}[ht]
	\centering
	\includegraphics[width=0.7\textwidth]{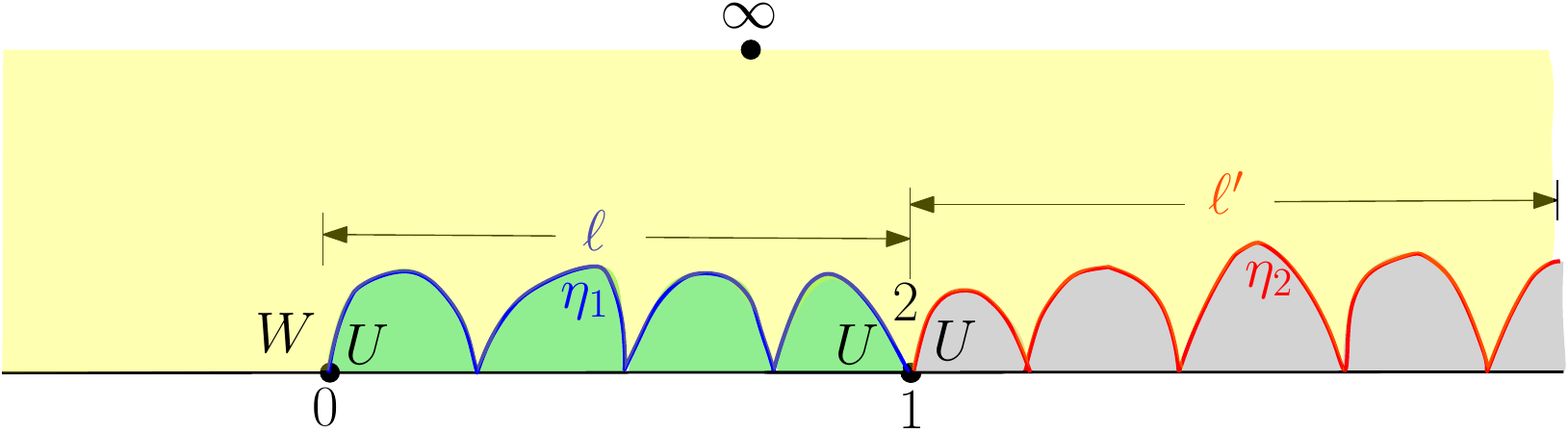}
	\caption{Setup of Proposition~\ref{thm-w-2u}. We claim that welding a thick disk from  $\Md_{2,\bullet}(W)$ (which is the same as $\QT(W,W,2)$) with two weight $U$ thin disks along the boundary arc containing the third marked point produces a three-pointed disk with law $\QT(W+U,W+U,2+2U)$ embedded as $(\bbH, \phi, 0,\infty,1)$.  }\label{fig-w-2u}
\end{figure}

This section is organized as follows. In Section \ref{sec-w2u-weld} we recall the notion of conformal welding and the result from \cite[Proposition 4.5]{AHS21}, which states the welding of a two-pointed disk with a three-pointed disk. Then using a limiting procedure over this result, in Section \ref{sec-w2u-proof} we give the proof of Proposition~\ref{thm-w-2u}.

\subsection{Conformal welding of two-pointed and three-pointed disks}\label{sec-w2u-weld}

We first recall the  the \emph{conformal welding} of quantum surfaces. Let $n\ge 1$ and $\mathcal{M}^1, ..., \mathcal{M}^n$ be measures on quantum surfaces. Fix some boundary arcs $\tilde{e}_1, e_1, ...,  \tilde{e}_n, e_n$ such that $e_i$ and $\tilde{e}_i$ are different boundary arcs on samples from $\mathcal{M}_i$. Suppose we have the disintegration 
$$\mathcal{M}^i = \int_0^\infty\int_0^\infty  \mathcal{M}^i(\ell_{i-1}, \ell_i)d\ell_id\ell_{i-1}\ \ \text{for}\ \ i=1, ..., n  $$
over the quantum lengths of $e_i$ and $\tilde{e}_i$.  Given a tuple of \emph{independent} surfaces from $\mathcal{M}^1(\ell_0,\ell_1)\times \mathcal{M}^2(\ell_1;\ell_2)\times\cdots\mathcal{M}^n(\ell_{n-1},\ell_n)$, suppose that  they can a.s.\ be \emph{conformally welded} along the pairs of arcs $(e_1,\tilde{e}_{i+1})$ for $i=1,..., n-1$,  yielding a large surface decorated with interfaces from the gluing. We write $$\text{Weld}( \mathcal{M}^1(\ell_0,\ell_1), \mathcal{M}^2(\ell_1,\ell_2),...,  \mathcal{M}^n(\ell_{n-1},\ell_n))$$ for the law of the resulting curve-decorated surface.  On the other hand, suppose we have a quantum surface sampled from some measure $\mathcal{M}$ and embedded on domain $D$ and we also sample an \emph{independent} family of curves on $D$ from some measure $\mathcal{P}$ with conformal invariance property. Then we write $\mathcal{M}\otimes \mathcal{P}$ for the law of this curve-decorated surface.

We emphasize that for all the quantum surfaces discussed in this paper, including the (two and three pointed) quantum disks and quantum triangles, the conformal welding as above is well-defined. This is because near a point $x$ with weight $W\ge \frac{\gamma^2}{2}$, the field is locally absolutely continuous to that of a weight $W$ quantum wedge near its finite-volume endpoint, while near a point $x$ with weight $W< \frac{\gamma^2}{2}$ the surface is a Possionian chain of weight $\gamma^2-W$ disks so local absolute continuity with respect to the weight $W$ quantum wedge still holds. Therefore from the conformal welding of quantum wedges \cite[Theorem 1.2]{DMS14}, our conformal weldings for quantum disks and triangles are well-defined. See e.g. \cite{She16a}, \cite[Section 3.5]{DMS14}  or \cite[Section 4.1]{GHS19} for more background on conformal welding.

We state the conformal welding of two-pointed quantum disks as below. Recall the notion of the measure $\mathcal{P}^{\text{disk}}(W_1, ..., W_n)$ in \cite[Definition 2.25]{AHS20} on tuple of curves $(\eta_1,..., \eta_{n-1})$ in a domain $(D,x,y)$, which is the same as $\textup{SLE}_\kappa (W_1-2;W_2-2)  $ from $x$ to $y$ for $n=2$ and  defined recursively for $n\ge 3$ by first sampling $\eta_{n-1}$ from $\textup{SLE}_\kappa (W_1+...+W_{n-1}-2;W_n-2)  $ then $(\eta_1,...,\eta_{n-2})$ from $\mathcal{P}^{\text{disk}}(W_1, ..., W_{n-1})$ on each connected component $(D',x',y')$ on the left of $D\backslash\eta_{n-1}$ where $x'$ and $y'$ are the first and the last point hit by $\eta_{n-1}$.
\begin{theorem}[Theorem 2.2 of \cite{AHS20}]\label{thm-disk-2}
	Fix $W_1, ..., W_n>0$ and $W = W_1+...+W_n$. Then there exists a constant $c=c_{W_1, ..., W_n}\in (0,\infty)$ such that for all $\ell, r>0$, the identity
	\begin{equation}
		\begin{split}
			&\mathcal{M}_2^{\textup{disk}}(W;\ell,r)\otimes \mathcal{P}^{\textup{disk}}(W_1, ..., W_n) \\&= c\iiint_0^\infty \textup{Weld}( \mathcal{M}_2^{\textup{disk}}(W_1;\ell,\ell_1),\mathcal{M}_2^{\textup{disk}}(W_2;\ell_1,\ell_2),..., \mathcal{M}_2^{\textup{disk}}(W_n;\ell_{n-1},r))d\ell_1...d\ell_{n-1}
		\end{split}
	\end{equation}
	holds as measures on the space of curve-decorated quantum surfaces.
\end{theorem}

Next we present the welding of two-pointed quantum disk with three-pointed quantum disks as in \cite[Proposition 4.5]{AHS21}, which adds a marked point to the boundary arc in Theorem \ref{thm-disk-2} above.
Recall the notion of SLE weighted by conformal radius in Section \ref{sec-sle-reverse}.

\begin{proposition}\label{prop-3-pt-disk}
	Suppose $W_1,W_2>0$, {$W_1+W_2\neq \frac{\gamma^2}{2}$} and $W_2\neq\frac{\gamma^2}{2}$. Then there exists a constant $c_{W_1, W_2}\in (0,\infty)$ such that for all $\beta\in\mathbb{R}$ and $\ell>0$, 
	\begin{equation}\label{eqn-3-pt-disk}
		\begin{split}
			\mathcal{M}_{2,\bullet}^{\textup{disk}}(W_1+W_2;\beta;\ell)&\otimes \widetilde{\SLE}_\kappa(W_1-2;W_2-2,0;1-\Delta_{\beta})\\ &= c_{W_1, W_2}\int_0^\infty \textup{Weld}( \mathcal{M}_2^{\textup{disk}}(W_1;\ell,x), \mathcal{M}_{2,\bullet}^{\textup{disk}}(W_2;\beta;x))dx.
		\end{split}
	\end{equation} 
	where again $\Delta_\beta$ is determined by \eqref{eqn-delta-alpha}.
\end{proposition}

Note that if $W_1+W_2<\frac{\gamma^2}{2}$, the interface above is understood as a chain of $\SLE_\kappa(W_1-2;W_2-2)$ curves except that the segment of curve on the disk containing the marked point is replaced by $\widetilde{\SLE}_\kappa(W_1-2;W_2-2,0;1-\Delta_{\beta})$. If {$\beta=\gamma$} then since $\Delta_\gamma=1$, the interface is simply $\SLE_\kappa(W_1-2;W_2-2)$ without any reweighting.

\begin{proof}
	When $W_1\ge \frac{\gamma^2}{2}$, the statement is precisely the same as \cite[Proposition 4.5]{AHS21}. Now suppose $W_1< \frac{\gamma^2}{2}$. We start with a sample from $\mathcal{M}_{2,\bullet}^{\textup{disk}}(2+W_1+W_2;\beta)$ embedded as $(\mathbb{H}, \phi, 0, \infty, 1)$ where the $\beta$-insertion is located at 1. Sample an independent curve $\eta_2$ from $\widetilde{\SLE}_{\kappa}(W_1,W_2-2;1-\Delta_\beta)$, and given $\eta_2$, independently sample a curve $\eta_1$ from SLE$_\kappa(0;W_1-2)$ on the left component of $\mathbb{H}\backslash\eta_2$. Let $\tilde{\mathcal{P}}(2,W_1, W_2)$ be the joint law of $(\eta_1, \eta_2)$. 
	Then by Theorem \ref{thm-disk-2}, we obtain that for some $c=c_{W_1,W_2}\in (0,\infty)$,
	\begin{equation}\label{eqn-3-pt-disk-a}
		\mathcal{M}_{2,\bullet}^{\textup{disk}}(2+W_1+W_2;\beta)\otimes \tilde{\mathcal{P}}(2,W_1, W_2) = c\iint_{[0,\infty)^2}\text{Weld}(\mathcal{M}_2^{\textup{disk}}(2;\ell),\mathcal{M}_2^{\textup{disk}}(W_1;\ell,x),\mathcal{M}_{2,\bullet}^{\textup{disk}}(W_2;\beta;x))dxd\ell.
	\end{equation}
	
	On the other hand, using the same trick as in the proof of Theorem \ref{thm-sle-reverse}, the marginal law of $\eta_1$ under $\tilde{\mathcal{P}}$ is $\widetilde{\SLE}_\kappa(0;W_1+W_2-2,0;1-\Delta_\beta)$, and by the existing argument for $W_1\ge\frac{\gamma^2}{2}$, given the interface and its quantum length $\ell$, the quantum surface to the right of $\eta_1$ has law  $\mathcal{M}_{2,\bullet}^{\textup{disk}}(W_1+W_2;\beta;\ell)$. The law of $\eta_2$ given $\eta_1$ is $\widetilde{\SLE}_\kappa(W_1-2;W_2-2,0;1-\Delta_\beta)$ on the right component of $\mathbb{H}\backslash\eta_1$, and therefore disintegrating \eqref{eqn-3-pt-disk-a} over $\ell$ and $\eta_1$ yields the proposition.
\end{proof}

Recall from Remark \ref{remark-qt-3-disk}, if $W_3>\frac{\gamma^2}{2}$ and $\beta_3 =\gamma+ \frac{2-W_3}{\gamma}$, then the measure $\mathcal{M}_{2,\bullet}^{\textup{disk}}(W;\beta_3;\ell) $ is some multiple constant of our quantum triangle QT$(W, W, W_3;\ell)$. Therefore we can rewrite \eqref{eqn-3-pt-disk} as  
\begin{equation}\label{eqn-qt-disk}
	\begin{split}
		\text{QT}(W_1+W_2, W_1+W_2, W_3;\ell)&\otimes \widetilde{\SLE}_\kappa(W_1-2;W_2-2,0;1-\Delta_{\beta_3}) = \\&c_{W_1, W_2}\int_0^\infty \textup{Weld}( \mathcal{M}_2^{\textup{disk}}(W_1;\ell,x), \text{QT}(W_2, W_2, W_3;x))dx.
	\end{split}
\end{equation} 
We emphasize that \eqref{eqn-qt-disk} continues to hold for $W_3<\frac{\gamma^2}{2}$ by the thick-thin duality. This is because  concatenating weight $W_3$ quantum disks to both sides of \eqref{eqn-qt-disk} (with $W_3$ replaced by $\gamma^2-W_3$) does not affect the equation, while from \eqref{eqn-delta-alpha}, the corresponding $\Delta_\beta$'s are the same for $W_3$ and $\gamma^2-W_3$ and therefore the interfaces are the same. 

\subsection{Proof of Proposition~\ref{thm-w-2u}}\label{sec-w2u-proof}

The idea of proving Proposition~\ref{thm-w-2u} is as follows. First assume $W\in (\frac{\gamma^2}{2},2]$ and $U\in(0,\frac{\gamma^2}{2})$. We take $(W_1,W_2)$ to be $(W,U)$ in Proposition \ref{prop-3-pt-disk}
and let $\beta\downarrow\beta_0:=\gamma-\frac{2U}{\gamma}$. In this limiting procedure, we will show that the SLE excursion containing the point 1 shrinks into a single point, yielding the desired welding picture. Finally if $W>2$, we can split the weight $W$ disk into a weight $W-2$ quantum disk and a weight $2$ quantum disk and apply Proposition \ref{prop-3-pt-disk}. 

\begin{figure}[ht]
	\centering
	\begin{tabular}{c} 
		\includegraphics[width=0.8\textwidth]{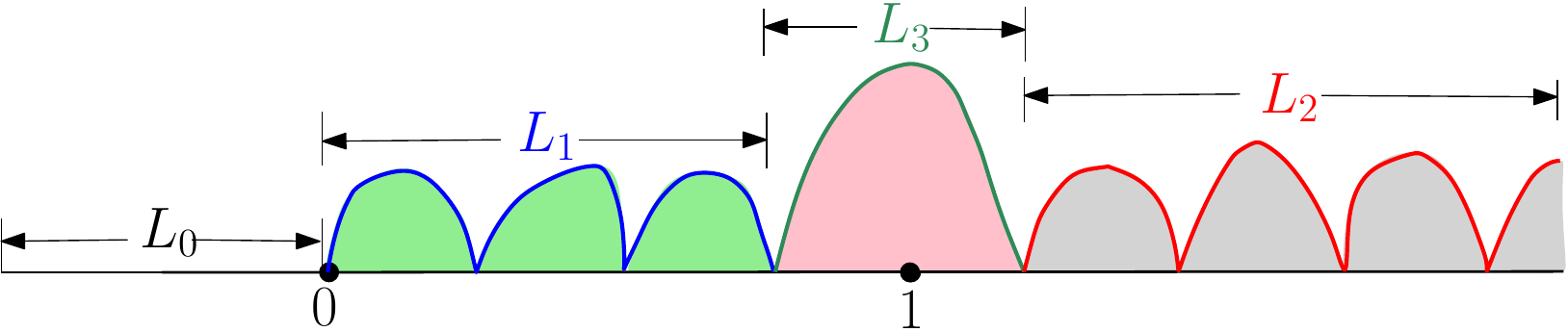}\\
		\includegraphics[width=0.8\textwidth]{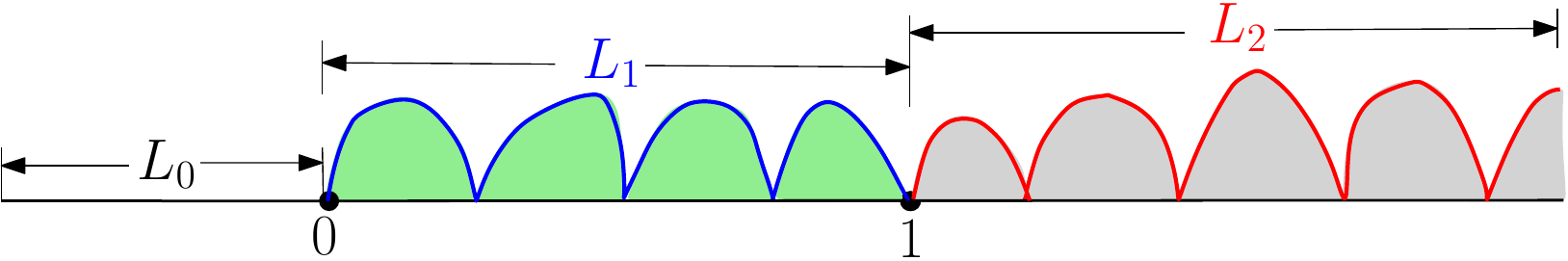}
	\end{tabular}
	\caption{Illustration of the proof of Proposition~\ref{thm-w-2u}. \textbf{Top:} A three-pointed disk from $\Md_{2,\bullet}(W+U;\beta;\ell_0)$ embedded as $(\bbH,\phi,0,\infty,1)$ decorated by an independent $\widetilde{\SLE}_\kappa(W-2;U-2,0;1-\Delta_{\beta})$ curve on the top. This splits the surface into a weight $W$ thick disk and a three-pointed disk of weight $U$, which can further be decomposed into two weight $U$ disks (on the left and right of 1) and a disk from $\Md_{2,\bullet}(\gamma^2-U;\beta)$. \textbf{Bottom:} As $\beta\downarrow\beta_0$, the disk containing the point 1 shrinks to a single point, yielding the picture in Proposition~\ref{thm-w-2u}. }\label{fig-w-2u-a}
\end{figure}

Consider a three-pointed quantum disk from $\Md_{2,\bullet}(W+U;\beta;\ell_0)^\#$ embedded as $(\bbH,\phi_\beta,0,\infty,1)$ (with 1 being the $\beta$-insertion and $\ell_0$ being the quantum length of $(-\infty,0)$), and draw an independent curve $\eta$ from $\widetilde{\SLE}_\kappa(W-2;U-2,0;1-\Delta_{\beta})^\#$. Note that by \cite[Theorem 1.1]{AHS21} $|\widetilde{\SLE}_\kappa(W-2;U-2,0;1-\Delta_{\beta})|<\infty$ for $\beta\in (\beta_0,\gamma)$ and $|\widetilde{\SLE}_\kappa(W-2;U-2,0;1-\Delta_{\beta_0})|=\infty$. This curve is boundary-hitting, and let $\tau$ (resp.\ $\sigma$) be the start (resp.\ end) time of the excursion containing the point 1. Let $L_1$, $L_2$, $L_3$ be the quantum lengths of $\eta|_{[0,\tau]}$,$\eta|_{[\sigma,\infty)}$  and $\eta|_{[\tau,\sigma]}$. (See also Figure \ref{fig-w-2u-a}.) By Proposition \ref{prop-3-pt-disk}, given the interface $\eta$ and its quantum length $\ell$, the surface above $\eta$ is a weight $W$ quantum disk from $\Md_2(W;\ell)$, while the beaded surface below $\eta$ is a three-pointed quantum disk from $\Md_{2,\bullet}(U;\beta;\ell)$, which by Definition \ref{def-m2dot-thin} can further be realized as $\Md_2(U)\times\Md_2(\gamma^2-U;\beta)\times\Md_2(U)$.

\begin{lemma}\label{lm-w-2u}
	In the above setting, assume $W\in (\frac{\gamma^2}{2},2]$ and $\beta<\gamma$. Then as $\beta\downarrow\beta_0$, under the normalized measure $\Md_{2,\bullet}(W+U;\beta;\ell_0)^\#\otimes \widetilde{\SLE}_\kappa(W-2;U-2,0;1-\Delta_{\beta})^\#$, {$L_3$} converges to 0 in probability.
\end{lemma}

\begin{proof}
	From Proposition \ref{prop-qt-bdry-thin} we know that $|\Md_{2,\bullet}(W+U;\beta;\ell_0)|$ is finite for $\beta\in[\beta_0,\gamma]$ while $|\widetilde{\SLE}_\kappa(W-2;U-2,0;1-\Delta_{\beta_0})|=\infty$, it suffices to prove that for any $\e>0$, there is some constant $C>0$ not depending on $\beta\in(\beta_0,\gamma)$ such that under  $\Md_{2,\bullet}(W+U;\beta;\ell_0)\otimes \widetilde{\SLE}_\kappa(W-2;U-2,0;1-\Delta_{\beta})$, the event $\{\ell_3>\e\}$ has measure no larger than $C$.
	
	By Proposition \ref{prop-3-pt-disk} and Definition \ref{def-m2dot-thin}, there exists some constants $c$ depending only on $\gamma,U,W$ (which might vary in the lines of the equation) but not on $\beta$ such that 
	\begin{equation}\label{eqn-w-2u-lm-1}
		\begin{split}
			&\ \ \ \ \big(\Md_{2,\bullet}(W+U;\beta;\ell_0)\otimes \widetilde{\SLE}_\kappa(W-2;U-2,0;1-\Delta_{\beta})\big)\big[{L_3>\e}\big]\\&= c\int_{\e}^\infty\int_0^\infty\int_0^\infty |\Md_2(W;\ell_0,\ell_1+\ell_2+\ell_3)||\Md_2(U;\ell_1)||\Md_2(U;\ell_2)||\Md_2(\gamma^2-U;\beta;\ell_3)|d\ell_1d\ell_2d\ell_3\\
			&= c\int_\e^\infty\int_{\e}^\ell\int_0^{\ell-\ell_3}|\Md_2(W;\ell_0,\ell)|\ell_1^{-\frac{2U}{\gamma^2}}(\ell-\ell_3-\ell_1)^{-\frac{2U}{\gamma^2}}\ell_3^{\frac{1}{\gamma}(\beta-\gamma+\frac{2U}{\gamma})-1}d\ell_1d\ell_3d\ell\\
			&=c\int_{\e}^\infty\int_{\e}^\ell |\Md_2(W;\ell_0,\ell)|(\ell-\ell_3)^{-\frac{4U}{\gamma^2}+1}\ell_3^{\frac{1}{\gamma}(\beta-\gamma+\frac{2U}{\gamma})-1}d\ell_3d\ell\\
			&=c\int_{\e}^\infty |\Md_2(W;\ell_0,\ell)|\ell^{\frac{\beta}{\gamma}-\frac{2U}{\gamma^2}}\int_{\frac{\e}{\ell}}^1(1-x)^{-\frac{4U}{\gamma^2}+1}x^{\frac{1}{\gamma}(\beta-\beta_0)-1}dxd\ell
		\end{split}
	\end{equation}
	where in the third line we used Proposition \ref{prop-disk-bdry-law} and Proposition \ref{prop-qt-bdry-thick}. Now we fix $\delta>0$ small and observe that 
	\begin{equation}\label{eqn-w-2u-lm-2}
		\begin{split}
			&\int_s^1(1-x)^{-\frac{4U}{\gamma^2}+1}x^{\frac{1}{\gamma}(\beta-\beta_0)-1}dx\le s^{-\delta}\int_0^1(1-x)^{-\frac{4U}{\gamma^2}+1}x^{\frac{1}{\gamma}(\beta-\beta_0)-1+\delta}dx 
			\le C(\delta)s^{-\delta}.
		\end{split}
	\end{equation}
	Plugging \eqref{eqn-w-2u-lm-2} in, we observe that the quantity in \eqref{eqn-w-2u-lm-2} is controlled by 
	\begin{equation}
		C\int_{\e_0}^\infty |\Md_2(W;\ell_0,\ell)|\ell^{\frac{\beta}{\gamma}-\frac{2U}{\gamma^2}+\delta}d\ell
	\end{equation}
	where $C=C(\delta,\epsilon,\gamma,W,U)$ is some constant. Now we take $\delta = \frac{U}{\gamma^2}$ so $\frac{\beta}{\gamma}-\frac{2U}{\gamma^2}+\delta$ varies between $(0,1-\frac{U}{\gamma^2})$. To conclude the proof, it suffices to verify that
	\begin{equation}\label{eqn-w-2u-lm-3}
		\int_0^\infty |\Md_2(W;\ell_0,\ell)|\ell^{1-\frac{U}{\gamma^2}} d\ell<\infty.
	\end{equation}
	We observe that by Proposition \ref{prop-qt-bdry-thin}, a three-pointed disk from $\Md_{2,\bullet}(\frac{U}{2};\gamma)$ (or equivalently $\QT(\frac{U}{2}, \frac{U}{2}, 2)$) has unmarked boundary length law $c\ell^{1-\frac{U}{\gamma^2}}d\ell$, and by Proposition \ref{prop-3-pt-disk}, \eqref{eqn-w-2u-lm-3} is a constant times
	\begin{equation}
		\int_0^\infty |\Md_2(W;\ell_0,\ell)||\Md_{2,\bullet}(\frac{U}{2};\gamma;\ell)|d\ell = c|\Md_{2,\bullet}(W+\frac{U}{2};\gamma;\ell_0)|.
	\end{equation}
	However, we know from Proposition \ref{prop-qt-bdry-thick} that $|\Md_{2,\bullet}(W+\frac{U}{2};\gamma;\ell_0)|<\infty$, which concludes the proof.
\end{proof}

The next lemma gives the interpretation of the right hand side of \eqref{eqn-w-2u}. We write $\bar{\mathcal{M}}_2(U)$ for the law of the surface constructed by concatenating a pair of samples from $\Md_2(U)\times\Md_2(U)$, giving the disintegration
\begin{equation}\label{eqn-w-2u-barm}
	\bar{\mathcal{M}}_2(U;\ell) = \int_0^\ell \Md_2(U;r)\times\Md_2(U;\ell-r)dr.
\end{equation}

\begin{lemma}\label{lm-w-2u-a}
	The triply marked surface on the right hand side of \eqref{eqn-w-2u} is the same as
	\begin{equation}
		\int_0^\infty \Wd(\Md_2(W;\ell),\bar{\mathcal{M}}_2(U;\ell))d\ell.
	\end{equation}
\end{lemma}
\begin{proof}
	We start with a sample from $\Md_{2}(W;\ell)$ where $\ell$ is its right boundary length. Then sample $r\sim\text{Leb}_{[0,\ell]}$ and mark the point on the right boundary arc with distance $r$ to the top endpoint. Recall that from Definition \ref{three-pointed-disk} and Proposition \ref{prop-m2dot}, once given $r$, after adding a third point onto the right boundary of weight $W$ disk,  the law of the surface we get is precisely $\Md_{2,\bullet}(W;r,\ell-r)$. Therefore the lemma follows by simultaneously welding a pair of samples from $\Md_2(U;r)\times\Md_2(U;\ell -r)$ to the right boundary arc according to quantum length and recalling the definition \eqref{eqn-w-2u-barm}.
\end{proof}
We notice that as in the proof of Lemma \ref{lm-w-2u}, 
\begin{equation}
	|\int_0^\infty \Wd(\Md_2(W;\ell_0,\ell),\bar{\mathcal{M}}_2(U;\ell))d\ell| = c\int_0^\infty |\Md_2(W;\ell_0,\ell)|\ell^{1-\frac{4U}{\gamma^2}}d\ell = c|\Md_{2,\bullet}(W+2U;\ell_0)|<\infty,
\end{equation}
which means that we may sample a quantum surface from the normalized version of the measure on the right hand side of \eqref{eqn-w-2u} embedded as $(\bbH, \tilde{\phi}, 0, \infty, 1)$ with $(\tilde{\eta}_1, \tilde{\eta}_2)$ being curves joining 1 with 0 and $\infty$.
To prove the theorem, we need to show that the law of $(\bbH, \tilde{\phi}, 0, \infty, 1)$  is $\Md_{2,\bullet}(W+U;\beta_0;\ell_0)^\#$, and  $(\tilde{\eta}_1, \tilde{\eta}_2)$ are independent of the surface. 

We go back to the setting as in Lemma \ref{lm-w-2u} and Figure \ref{fig-w-2u-a}.  Let $S_\beta$ be the connected component of $\bbH\backslash\eta$ containing 1, and $x_\beta$ be the quantum midpoint of the left boundary of $(\bbH,\phi_\beta,0,\infty,1)$ (i.e.\ $\nu_{\phi_\beta}((-\infty, x_\beta)) = \nu_{\phi_\beta}((x_\beta, 0)) = \frac{\ell_0}{2}$). Consider  the conformal map $g_\beta$ from $\bbH\backslash S_\beta$ to $\bbH$ that fixes 0, $\infty$ and $x_\beta$. For any $\e>0$ let $\bbH_\e = \{z\in\bbH:|z-1|>\e\}$. Since it is clear that the law of $(\bbH_\e, {\phi}_\beta, 0, \infty)$ converges in total variation to $(\bbH_\e, {\phi}_{\beta_0}, 0, \infty)$ (which could be seen from the LCFT definition and the disintegration description in \eqref{eqn-qt-disintegration-thick}), we may couple $(\bbH_\e, {\phi}_\beta, 0, \infty)$ with $(\bbH_\e, {\phi}_{\beta_0}, 0, \infty)$ such that the corresponding $x_\beta$ agrees with $x_{\beta_0}$ with probability $1-o_\beta(1)$. We shall work on the surface $(\bbH\backslash S_\beta, 0, \infty)$, which is equivalent to $(\bbH, \hat{\phi}_\beta, 0, \infty)$ where $\hat{\phi}_\beta = \phi_\beta\circ g_\beta^{-1}+Q\log |(g_\beta^{-1})'|$.

\begin{lemma}\label{lm-w-2u-b}
	Fix $\e>0$. Under the measure $\Md_{2,\bullet}(W+U;\beta;\ell_0)^\#\otimes\widetilde{\SLE}_\kappa(W-2;U-2,0;1-\Delta_{\beta})^\#$, as $\beta\downarrow\beta_0$, the law of the surface $(\bbH_\e, \hat{\phi}_\beta, 0, \infty)$ converges {weakly} to that of $(\bbH_\e, \tilde{\phi}, 0, \infty)$.
\end{lemma}

\begin{proof}
	From Lemma \ref{lm-w-2u}, the quantum length  {$L_3$} is converging in probability to zero. {In particular, using conformal covariance property of quantum length, this also implies that the harmonic measure of $\partial S_\beta$ in $\bbH\backslash\partial S_\beta$ viewed from $x_\beta$ (after a reflection over $\bbR^-$) converges in probability to zero. Then adapting the the same proof in~\cite[Lemma 5.16]{AGS21},  the claim follows   from} the continuity of the disintegration of quantum disks over quantum length (see e.g. \cite[Proposition 2.23]{AHS20} and \cite[Lemma 5.17]{AGS21}) and the description provided by Lemma \ref{lm-w-2u-a}.
\end{proof}

\begin{proof}[Proof of Proposition~\ref{thm-w-2u}]
	\emph{Step 1. Identifying the field.}
	Assume that we are in the setting of Lemma \ref{lm-w-2u} and \ref{lm-w-2u-b}, and $W\in (\frac{\gamma^2}{2},2]$. We prove that, for any $\e>0$,  the distributions $\hat{\phi}_\beta$ converges weakly to $\phi_{\beta_0}$ in the domain $\bbH_\e$, where again $\phi_{\beta_0}$ is sampled from $\Md_{2,\bullet}(W+U;\beta_0;\ell_0)^\#$. Then Lemma \ref{lm-w-2u-b} implies that the law of $(\bbH, \tilde{\phi}, 0, \infty, 1)$ is $\QT(W+U,W+U,2+2U;\ell_0)^\#$.
	
	We start by extending  $g_\beta$ to the conformal map from $\mathbb{C}\backslash (S_\beta^*\cup\bbR_+)$ to $\mathbb{C}\backslash \bbR_+$ via  Schwartz reflection, where $S_\beta^* = S_\beta\cup \{z:\bar{z}\in S_\beta\}$. Fix $\delta>0$ and work on the event that $x_\beta<-\delta$, which has probability $1-o_\delta(1)$. Then since the quantum length $\ell_3$ goes to 0 in probability, if we let $\beta\downarrow\beta_0$, the probability that an independent Brownian motion starting from {$x_\beta$} exits $\mathbb{C}\backslash(S_\beta^*\cup\bbR_+)$ through $\partial S_\beta$ goes to 0.
	
	Consider the conformal map $\psi_\beta$ from $\mathbb{C}\backslash\bbR_+$ to the unit disk sending $x_\beta$ to 0 and $\infty$ to 1. Then the Beurling estimate (see e.g.\ \cite[Section 3.8]{Law08}) implies that for any fixed $\e>0$, with probability 1$-o_\beta(1)$, the set $\psi_\beta(S_\beta^*)$ is contained in $\{z:1-\e<|z|<1\}$. This implies that the kernel of the set  $\mathbb{C}\backslash (S_\beta^*\cup\bbR_+)$ is $\mathbb{C}\backslash\bbR_+$ with probability $1-o_\delta(1)$, and the Carath\'{e}odory kernel theorem (see e.g.\ \cite[Section 3.6]{Law08}) implies that the conformal maps $g_\beta^{-1}$ converges uniformly on compact sets of $\mathbb{C}\backslash \bbR_+$ to the identity function. Then in $\bbH_\e$, since $\hat{\phi}_\beta = \phi_\beta\circ g_\beta^{-1}+Q\log |(g_\beta^{-1})'|$ and $\phi_\beta$ converges in total variation distance to $\phi_{\beta_0}$, it is clear that as we first send $\beta\downarrow\beta_0$ and then $\delta\to0$, $\hat{\phi}_\beta$ converges weakly to  $\phi_{\beta_0}$, which concludes the first step of the proof.
	
	\emph{Step 2. Identifying the interface.} In Step 1 we have shown that in the $(\bbH, \tilde{\phi}, 0, \infty, 1)$ obtained by welding two weight $U$ disks with a weight $W$ disk, the field $\tilde{\phi}$ is precisely the Liouville field.  
	Now we show that the law of the interfaces $(\tilde{\eta}_1,\tilde{\eta}_2)$ on the right hand side of \eqref{eqn-w-2u} can be characterized by some SLE resampling property and independent of the field.	
	
	Recall that in curve-decorated surface $\Md_{2,\bullet}(W+U;\beta;\ell_0)\otimes \widetilde{\SLE}_\kappa(W-2;U-2,0;1-\Delta_{\beta})$, as we remove the bubble $S_\beta$, the interfaces $(\eta_1, \eta_2)$ are given by $(\eta|_{[0,\tau]},\eta|_{[\sigma,\infty)})$ where $\tau$ (resp.\ $\sigma$) be the start (resp.\ end) time of the excursion containing the point 1. Then by SLE Markov property, given $S_\beta$ and ${\eta_1}$, the law of $\eta_2$ is $\SLE_\kappa(W-2;U-2)$ with force points at $0$ and $\eta(\sigma)_+$ in the right connected component of $\bbH\backslash\eta|_{[0,\sigma]}$. Similarly, using the SLE reversibility statement \cite[Theorem 1.1]{MS16b}, the law of $\eta_1$ given $S_\beta$ and $\eta_2$ is the $\SLE_\kappa(U-2;W-2)$ process from 1 to 0 in the left connected component of $\bbH\backslash\eta|_{[\tau,\infty)}$ with force points at $\eta(\tau)_-$ and $\infty$. Therefore it follows from Lemma \ref{lm-w-2u-b} that the law of $\tilde{\eta}_1$ given $\tilde{\eta}_2$ is the $\SLE_\kappa(U-2;W-2)$ process from 1 to 0 in the left connected component of $\bbH\backslash\tilde{\eta}_2$ with force points at $1^-$ and $\infty$, while the law of $\tilde{\eta}_2$ given $\tilde{\eta}_1$ is the $\SLE_\kappa(W-2;U-2)$ process from 1 to $\infty$ in the left connected component of $\bbH\backslash\tilde{\eta}_2$ with force points at $0$ and $1^+$. Therefore it follows from the SLE resampling property (Proposition~\ref{prop:ig-flow-descriptions}) that the joint law of $(\tilde{\eta}_1, \tilde{\eta}_2)$ is unique and independent of the field, and thus concluding the proof for $W\in (\frac{\gamma^2}{2}, 2]$.
	
	\emph{Step 3. Extension to $W>2$.} In Figure \ref{fig-w-2u}, by Theorem \ref{thm-disk-2}, we can weld the weight $W$ disk into a weight $W-2$ disk on the left and a weight 2 disk on the right with interface $\eta_0$. Then by Steps 1 and 2, the law of the quantum surface on the right of $\eta_0$ is a three-pointed disk $\Md_{2,\bullet}(2+U;\beta_0)$, and therefore by Proposition \ref{prop-3-pt-disk} the whole surface has law   $\Md_{2,\bullet}(W+U;\beta_0)$. Moreover the marginal law of $\eta_0$ is $\widetilde{\SLE}_\kappa(W-4;U,0;1-\Delta_{\beta_0})$, while the law of the interfaces $(\eta_1, \eta_2)$ given $\eta_0$ are characterized by the SLE resampling properties. Therefore the law of $(\eta_1,\eta_2)$ is independent of the field, which concludes the proof of the Theorem.
\end{proof}

\section{Proof of Theorem \ref{thm:M+QTp-rw} for a restricted range}\label{sec:resampling}

{In this section we prove Theorem~\ref{thm:u2v}, which is Theorem~\ref{thm:M+QTp-rw} for a restricted parameter range.}

\begin{theorem}\label{thm:u2v}
	Suppose $0 < U < \frac{\gamma^2}2 < W$. 
	Sample a curve-decorated quantum surface from 
	\[\int_0^\infty \mathrm{Weld}(\Md_2(U; \ell), \QT(W,2,W; \ell)) \, d\ell\]
	where the welding identifies a boundary edge of the quantum disk with a boundary edge of the quantum triangle with endpoints of weights $2,W$. 
	Embed it as {$(\bbH, \phi, \eta,  \infty,0,1)$}, where the boundary points with weights {$(U+W, U+2, W)$} are mapped to $(\infty,0,1)$. Then there is a finite constant $C = C(U,W)$ such that the law of $(\phi, \eta)$ is $C \LF_{\bbH}^{(\beta_1,\infty), (\beta_2, 0),(\beta_3, 1)} \times \SLE_\kappa(U-2; 0, W-2)$, where {$\beta_1 = Q+\frac{\gamma}{2} - \frac{W+U}{\gamma}$, $\beta_2 = Q+\frac\gamma2 - \frac {2+U}\gamma$ and $\beta_3 = Q + \frac\gamma2 - \frac{W}\gamma$. In other words, Theorem~\ref{thm:M+QT} holds for $(W,W_1,W_2,W_3) = (U,W,2,W)$.} 
\end{theorem}

We point out that in the special case $W = 2$ this is already known.
\begin{proposition}\label{prop-W=2}
	Theorem~\ref{thm:u2v} holds when $W = 2$.
\end{proposition}
\begin{proof}
	This is \cite[Lemma 4.4]{AHS21} with the parameters $(W_-, W_+) = (U, 2)$.
\end{proof}

We prove the $W > 2$ case in Section~\ref{sec-W>2} and the $W \in (\frac{\gamma^2}2, 2)$ case in Section~\ref{sec-W<2}, and thus complete the proof of Theorem~\ref{thm:u2v}. 
The key is a Markovian characterization of Liouville fields with three insertions.

\begin{proof}[Proof of Theorem~\ref{thm:u2v}]
	The various cases are proved in Propositions~\ref{prop-W=2},~\ref{prop-W>2} and~\ref{prop-W<2}.
\end{proof}

\subsection{The case $W >2$}\label{sec-W>2}

{
	In this section we prove the $W>2$ case (which we state as Proposition~\ref{prop-W>2}).
	
	Recall $(\phi, \eta)$ in Theorem~\ref{thm:u2v}. Roughly speaking, Proposition~\ref{prop-indep-curves-thick} shows that $\phi$ and $\eta$ are independent and identifies the law of $\eta$. Lemma~\ref{def-markov-kernel} gives a Markov property for the Liouville field which is directly inherited from that of the Gaussian free field. Using this and Proposition~\ref{thm-w-2u}, we obtain Markov properties for $\phi$, where one can resample the field in three subsets od $\bbH$ which together cover $\bbH$ (Lemmas~\ref{lem-markov1},~\ref{lem-markov0} and~\ref{lem-markovinfty}). Finally, these three resampling properties are enough to characterize the law of $\phi$ and hence complete the proof of Proposition~\ref{prop-W>2}.  } 

\begin{proposition}\label{prop-indep-curves-thick}
	In the setting of Theorem~\ref{thm:u2v} with $W > 2$, let $M$ be the law of the field $\phi$. Then the joint law of $(\phi, \eta)$ is $M \times \SLE_\kappa(U-2; 0, W-2)$. 
\end{proposition}
\begin{proof}
	%To simplify the exposition, we focus on the setting where $U, V \in [\frac{\gamma^2}2,2)$. The case where one or both of $U,V$ lie in $(0, \frac{\gamma^2}2)$ has the exact same argument but is messier to describe for topological reasons. 
	Let $P_{U,W}$ be the law of $(\eta_1,\eta_2)$ from Figure~\ref{fig-ig-resample} where $(a_1, a_2, a_3) = (\lambda(1-U), \lambda, \lambda(W-1))$ and $(\theta_1, \theta_2) = (0, -\frac{2\lambda}\chi)$, so the curve $\eta_1$ is $\SLE_\kappa(U-2;0, W-2)$ from $0$ to $\infty$, and $\eta_2$ is $\SLE_\kappa(U, 0; W-4)$ from $1$ to $\infty$. 
	Let $D_2$ be the connected component of $\bbH \backslash \eta$ having $1$ on its boundary, and let $\eta_2$ be $\SLE_\kappa(0; W-4)$ in $D_2$ independent of $\phi$. By {Proposition~\ref{prop-W=2}} the law of $(\bbH, \phi, \eta, \eta_2, 0,1, \infty)/{\sim_\gamma}$ is 
	\begin{equation}\label{eq-three-disks}
		C \iint_0^\infty \mathrm{Weld}(\Md_2(U; \ell), \QT(2,2,2;\ell,\ell'), \Md_2(W-2;\ell'))\,d\ell\,d\ell'.
	\end{equation}
	Theorem~\ref{thm-disk-2} implies that the conditional law of $\eta$ given $(\phi, \eta_2)$ is $\SLE(U-2;0)$ in $(D_1, 0, \infty)$ where $D_1$ is the connected component of $\bbH \backslash \eta_2$ having $0$ on its boundary. %Likewise, the conditional law of $\eta_2$ given $(\phi, \eta)$ is $\SLE(0;V-2)$ in $(D_2, 1,\infty)$ where $D_2$ is the connected component of $\bbH \backslash \eta$ having $1$ on its boundary. 
	By Proposition~\ref{prop:ig-flow-descriptions}, conditioned on $\phi$, the conditional law of $(\eta, \eta_2)$ is $P_{U,V}$, and so the conditional law of $\eta$ is $\SLE_\kappa(U-2;0, W-2)$ as desired. 
\end{proof}

{Recall from Proposition~\ref{prop:Markov} that Gaussian free fields satisfy the domain Markov property. We now show that Liouville fields with three insertions satisfy a variant of the domain Markov property.  
	In Proposition~\ref{prop-W>2} we will show that this Markov property characterizes such  Liouville fields. This will allow us to identify $M$ from Proposition~\ref{prop-indep-curves-thick}  hence prove Theorem~\ref{thm:u2v} in the $W>2$ case.} {Since $\LF_\bbH^{(\beta_1,0),(\beta_2,1),(\beta_3,\infty)}$ is an infinite measure, we first need to specify the definition of conditioning in terms of Markov kernels as below.
	\begin{definition}\label{def-markov-kernel}
		Suppose $(\Omega, \cF)$ and $(\Omega', \cF')$ are measurable spaces. We say $\Lambda: \Omega \times \cF' \to [0,1]$ is a Markov kernel if $\Lambda(\omega, \cdot)$ is a probability measure on $(\Omega', \cF')$ for each $\omega \in \Omega$, and $\Lambda(\cdot, A)$ is $\cF$-measurable for each $A \in \cF'$.  If $(X,Y)$ is a sample from $\Lambda(x, dy) \mu(dx)$ for a measure $\mu$ on $(\Omega, \cF)$, we say the conditional law of $Y$ given $X$ is $\Lambda(X, \cdot)$.
\end{definition}}

\begin{lemma}\label{lem-markov-LF}
	Suppose $\psi\sim \LF_\bbH^{(\beta_1, 0), (\beta_2, 1), (\beta_3, \infty)}$, and the random set $  {S = S(\psi)} \subset \bbH$  is either the empty set or a bounded neighborhood of $0$ with $\ol S \cap [1,+\infty) = \emptyset$. Suppose that for any open $U \subset \bbH$, the event $\{ (\bbH \backslash S) \subset U\}$ is measurable with respect to  {$\psi|_U$}. Then conditioned on $(S, \psi|_{\bbH \backslash S})$ and on $\{S \neq \emptyset\}$ {in the sense of Defintion~\ref{def-markov-kernel}}, we have $\psi|_S \stackrel d= h + \mathfrak h + \frac{\beta_1}2 G_S(\cdot, 0)$ where $h$ is a GFF on $S$ with zero (resp.\ free) boundary conditions on $\partial S \cap \bbH$ (resp.\ $\partial S \cap \bbR$), $\mathfrak h$ is the harmonic extension of $\psi|_{\bbH \backslash S}$ to $S$ with normal derivative zero on $\partial S \cap \bbR$, and $G_S$ is the Green function of $h$. 
	
	The same holds if $S$ is either the empty set or a bounded neighborhood of $1$ with $\ol S \cap (-\infty, 0] = \emptyset$, and we replace $\frac{\beta_1}2 G_S(\cdot, 0)$ with $\frac{\beta_2}2 G_S(\cdot, 1)$. 
	
	The same holds if $S$ is either the empty set or a neighborhood of $\infty$ bounded away from $\{0,1\}$, and we replace $\frac{\beta_1}2 G_S(\cdot, 0)$ with $(\frac{\beta_3}2-Q) G_S(\cdot, \infty)$.
\end{lemma}
\begin{proof}
	When $\beta_1 = 0$, the set $\bbH \backslash S$ is a \emph{local set} as defined in \cite{SS13}, and the statement follows from \cite[Lemma 3.9]{SS13}. 
	When $\beta_1 \neq 0$, the result is obtained by weighting the $\beta_1 = 0$ case by $\eps^{\frac{\beta_1^2}4}e^{\frac{\beta_1}2 \phi_\eps(0)}$ and sending $\eps \to 0$. The other two cases are similar.
\end{proof}

Next, we will use Lemma~\ref{lem-markov-LF} to derive corresponding Markov properties for $M$ in Lemmas~\ref{lem-markov1},~\ref{lem-markov0} and~\ref{lem-markovinfty}. 

Recall that $\beta_1 = \gamma - \frac{U}\gamma$, $\beta_2 = \gamma - \frac{W-2}\gamma$ and $\beta_3 = \gamma - \frac{U+W-2}\gamma$.
\begin{lemma}\label{lem-markov1}
	Let $A \subset \bbH$ be a bounded neighborhood of $1$ such that $A$ and $\bbH \backslash A$ are simply connected and $\ol A \cap (-\infty, 0] = \emptyset$. For $\phi \sim M$, conditioned on $\phi|_{\bbH \backslash A}$ we have $\phi|_A \stackrel d= h + \mathfrak h + \frac{\beta_2}2 G_A(\cdot, 1)$ where $h$ is a mixed boundary GFF in $A$ with zero (resp. free) boundary conditions on $\partial A \cap \bbH$ (resp. $\partial A \cap \bbR$), $\mathfrak h$ is the harmonic extension of $\phi|_{\bbH \backslash A}$ to $A$ with normal derivative zero on $\partial A \cap \bbR$, and $G_A$ is the Green function describing the covariance of $h$.
\end{lemma}
\begin{proof}
	Sample $(\phi, \eta)\sim M \times \SLE_\kappa(U-2; 0, W-2)$. Let $D$ be the connected component of $\bbH \backslash \eta$ with $1$ on its boundary. Let $f: D \to \bbH$ be the conformal map fixing the three boundary points $\{0, 1, \infty\}$. Let $\cD= (\bbH \backslash D, \phi, 0, \infty)/{\sim_\gamma}$ and let $\psi = f \bullet_\gamma \phi$. See Figure~\ref{fig-markov1} (left).
	
	\begin{figure}[ht]
		\centering
		\includegraphics[scale = 0.54]{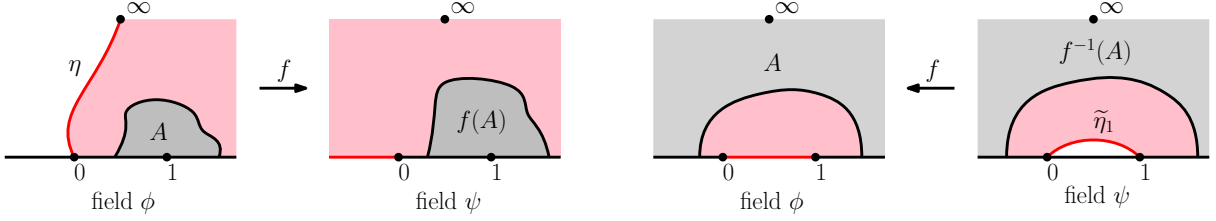}
		\caption{\textbf{Left:} Illustration for the proof of  Lemma~\ref{lem-markov1}. \textbf{Right:} Illustration for the proof of Lemma~\ref{lem-markovinfty}. }\label{fig-markov1}
	\end{figure}
	
	By Proposition~\ref{prop-indep-curves-thick} and the definition of $M$, there is a constant $c$ such that the law of $(\psi, \cD )$ is $c \int \LF_\bbH^{(\gamma, 0), (\beta_2,1),(\beta_2, \infty)}(\ell) \times \Md_2(U; \ell) \, d\ell$, where $\LF_\bbH^{(\gamma, 0), (\beta_2,1),(\beta_2, \infty)}(\ell)(d\psi)$ is defined as the disintegration of the measure $\LF_\bbH^{(\gamma, 0), (\beta_2,1),(\beta_2, \infty)}(d\psi)$ on the event $\{\nu_\psi(-\infty,0) = \ell \}$.

	Since $\eta$ is the interface when $(\bbH, \psi, 0, 1, \infty)/{\sim_\gamma}$ is conformally welded to $\cD$, the curve $\eta$ is measurable with respect to $\sigma(\cD, \nu_\psi|_{(-\infty,0)})$, thus $E := \{\eta \subset \bbH \backslash A\} \in \sigma
	(\cD, \nu_\psi|_{(-\infty,0)})$. On $E$, define $S = f(A)$, and on $E^c$, define $S = \emptyset$.  Lemma~\ref{lem-markov-LF} is applicable with this choice of $S$.  Consequently, conditioned on $E$ and on $(\cD, \psi|_{\bbH \backslash f(A)})$, we have $\psi|_{f(A)} \stackrel d= \wt h + \wt{\mathfrak h} + \frac{\beta_2}2 G_{f(A)}(\cdot, 1)$, where $\wt h$ is a GFF on $f(A)$ with zero (resp.\ free) boundary conditions on $\partial f(A) \cap \bbH$ (resp.\ $\partial f(A) \cap \bbR$) and $\wt{\mathfrak h}$ is the harmonic extension of $\psi|_{\bbH \backslash f(A)}$ to $f(A)$ having normal derivative zero on $\partial(f(A)) \cap \bbR$. By conformal invariance, we conclude that conditioned on $E$ and on $(\phi|_{\bbH \backslash A}, \eta)$, we have $\phi|_A \stackrel d= h + \mathfrak h + \frac{\beta_2}2 G_A(\cdot, 1)$.
	
	Finally, since $(\phi, \eta)\sim M \times \SLE_\kappa(U-2; 0, W-2)$, and the event $E$ only depends on $\eta$, we deduce the Markov property for $\phi$. 
\end{proof}

\begin{lemma}\label{lem-markov0}
	Let $A \subset \bbH$ be a bounded neighborhood of $0$ such that $A$ and $\bbH \backslash A$ are simply connected and $\ol A \cap [1,\infty) = \emptyset$. For $\phi \sim M$, conditioned on $\phi|_{\bbH \backslash A}$ we have $\phi|_A \stackrel d= h + \mathfrak h + \frac{\beta_1}2 G_A(\cdot, 0)$ where $h$ is a mixed boundary GFF in $A$ with zero (resp. free) boundary conditions on $\partial A \cap \bbH$ (resp. $\partial A \cap \bbR$), $\mathfrak h$ is the harmonic extension of $\phi|_{\bbH \backslash A}$ to $A$ with normal derivative zero on $\partial A \cap \bbR$, and $G_A$ is the Green function describing the covariance of $h$.
\end{lemma}

\begin{proof}
	Define $\eta_2$ as in the argument of Proposition~\ref{prop-indep-curves-thick}. The same argument as in Lemma~\ref{lem-markov1} applied to $(\phi, \eta_2)$ yields the result. Indeed, the picture is symmetric if we interchange $U$ and $W-2$. 
\end{proof}

Before proving the last Markov property Lemma~\ref{lem-markovinfty}, we first introduce a weighted quantum disk measure $\wt{\mathcal M}_2^\mathrm{disk} (U)$; this is not strictly necessary but simplifies the later exposition. 

\begin{lemma}\label{lem-weight-other-side}
	For $U \in (0,2]$ and $p \in (-1, \frac4{\gamma^2})$, if we sample a quantum disk from $R^p \Md_2(U)$ then the law of $L$ is $1_{\ell > 0} c \ell^{-\frac{2U}{\gamma^2} + p}\, d\ell$ where $c \in (0,\infty)$; here $L$ and $R$ are the left and right boundary arc lengths of the quantum disk. In particular, for $U<2$ the law of the left boundary arc length of $\tildeMd_2(U):=R^{\frac{2U}{\gamma^2}} \Md_2(U)$ is $c1_{\ell > 0} \, d\ell$ for some $c \in (0,\infty)$.
\end{lemma}
%{[*Using change of variables to $\int_0^\infty |\Md_2(U;\ell,r)|r^qdr$ should directly give the answer as well, but we possibly need something like $\int_0^\infty |\Md_2(U;1,r)|r^qdr<\infty$?*]}
\begin{proof}
	{We first prove the lemma except for the finiteness claim $c < \infty$.  Let $P$ denote the law of $\hat \psi$ in Definition~\ref{def-thick-disk} (with $\beta = \gamma + \frac{2-U}\gamma$), so for $(\hat \psi, \mathbf c)\sim P \times [\frac\gamma2 e^{(\beta-Q)c}\,dc]$, the law of $(\mathcal S, \hat \psi +\mathbf c, -\infty, +\infty)/{\sim_\gamma}$ is $\Md_2(W)$. Let $\partial_\ell \mathcal S$ and $\partial_r \mathcal S$ be the boundary arcs of $(\mathcal S, -\infty, +\infty)$. 
		By Definition~\ref{def-thick-disk}, for an interval $I$ the size of the event $\{L \in I\}$ is 
		\[\bbE \left[\int_{-\infty}^\infty 1_{e^{\frac\gamma2c} \nu_{\hat \psi}(\partial_\ell \mathcal S) \in I} (e^{\frac\gamma2 c} \nu_{\hat \psi}(\partial_r \mathcal S)   )^p \cdot \frac\gamma2 e^{(\beta-Q)c} \, dc \right]\]
		where  $\bbE$ denotes expectation with respect to $P$. Using the change of variables $y = e^{\frac\gamma2 c} \nu_{\hat \psi}(\partial_\ell \mathcal S)$, this equals
		\[\bbE\left[ \int_0^\infty 1_{y \in I} \nu_{\hat \psi}(\partial_r \mathcal S)^p (\frac y{\nu_{\hat \psi}(\partial_\ell \mathcal S)})^{\frac2\gamma(\beta-Q)+p} y^{-1}dy \right] = \bbE[\nu_{\hat \psi}(\partial_r \mathcal S)^p \nu_{\hat \psi}(\partial_\ell \mathcal S)^{-\frac2\gamma(\beta-Q)-p}] \int_I y^{\frac2\gamma (\beta-Q)+p-1}\,dy. \]
		Since $\frac2\gamma (\beta-Q)+p-1 = -\frac{2U}{\gamma^2}+p$, this yields the claim apart from the finiteness of the constant. 
	}
	
	If $U = 2$, the finiteness of $c$  is immediate from the joint law $c(\ell + r)^{-\frac4{\gamma^2}-1} \, d\ell\, dr$ for $(L,R)$, where $c<\infty$ is a constant, see e.g.\ \cite[Proposition 7.8]{AHS20}. For $U <2$, this follows from the $U=2$ result and the fact that conformally welding a weight $U$ disk to a weight $(2-U)$ disk gives a weight 2 disk (Theorem~\ref{thm-disk-2}).
\end{proof}

\begin{lemma}\label{lem-markovinfty}
	Let $A \subset \bbH$ be a neighborhood of $\infty$ such that $A$ and $\bbH \backslash A$ are simply connected and $\ol A \cap [0,1] = \emptyset$. For $\phi \sim M$, conditioned on $\phi|_{\bbH \backslash A}$ we have $\phi|_A \stackrel d= h + \mathfrak h + (\frac{\beta_3}2-Q) G_A(\cdot, \infty)$ where $h$ is a mixed boundary GFF in $A$ with zero (resp. free) boundary conditions on $\partial A \cap \bbH$ (resp. $\partial A \cap \bbR$), $\mathfrak h$ is the harmonic extension of $\phi|_{\bbH \backslash A}$ to $A$ with normal derivative zero on $\partial A \cap \bbR$, and $G_A$ is the Green function describing the covariance of $h$.
\end{lemma}
\begin{proof}
	Let $\mathsf{m}(W;U)$ be the probability measure on pairs of curves from Proposition~\ref{thm-w-2u}. Reflect this pair of curves across the line $\mathrm{Re}\, z = \frac12$ to get a pair $(\wt \eta_1, \wt \eta_2)$ where $\wt \eta_1$ joins $0$ and $1$ and $\wt \eta_2$ joins $0$ and $\infty$. Let $\wt {\mathsf{m}} (W;U)$ be the law of $(\wt \eta_1, \wt \eta_2)$.
	
	Let $\alpha = \gamma - \frac{2U}\gamma$.
	Sample 
	\[(\psi, \wt \eta_1, \wt \eta_2) \sim \nu_\psi(0, 1)^{\frac{2U}{\gamma^2}} \LF_\bbH^{(\alpha, 0), (\beta_3, 1) (\beta_3, \infty)}(d\psi) \times \wt {\mathsf{m}}(W;U) .\]
	See Figure~\ref{fig-markov1}. By Proposition~\ref{thm-w-2u}, the decorated quantum surface $(\bbH, \psi, \wt \eta_1, \wt \eta_2, 0, 1, \infty)/{\sim_\gamma}$ has law 
	\[\iint_0^\infty \mathrm{Weld}(\Md_2(U;\ell), \QT(W, W, 2; \ell;\ell'),  \tildeMd_2(U; \ell'))\, d\ell \, d\ell', \]
	where $\tildeMd_2(U)$ is the weighted quantum disk defined in Lemma~\ref{lem-weight-other-side} and $\tildeMd_2(U; \ell')$ its disintegration by the unweighted boundary arc length. 
	
	By Lemma~\ref{lem-weight-other-side} we have $|\tildeMd_2(U; \ell')| = c$ for all $\ell'$ for some finite constant $c$, so the marginal law of the decorated quantum surface above $\wt \eta_1$ is $c\int_0^\infty \mathrm{Weld}(\Md_2(U;\ell), \QT(W,W,2;\ell))\, d\ell$. Let $f$ be the conformal map sending the connected component of $\bbH \backslash \wt \eta_1$ above $\wt \eta_1$ to $\bbH$ such that $f$ fixes $(0,1,\infty)$, and let $\phi = f \bullet_\gamma \psi$. Then the marginal law of $\phi$ is $cM$. 
	
	Let $S = f^{-1}(A)$.
	Since $S$ is measurable with respect to {$\wt \eta_1$} and $\psi$ is independent of $\wt \eta_2$, Lemma~\ref{lem-markov-LF} tells us that conditioned on $S$ and $\psi|_{\bbH \backslash S}$, we have $\psi|_S \stackrel d= h + \mathfrak h + (\frac{\beta_3}2 - Q) G_S(\cdot, \infty)$ where $h$ is a GFF on $S$ with zero (resp.\ free) boundary conditions on $\partial S \cap \bbH$ (resp.\ $\partial S \cap \bbR$) and $\mathfrak h$ is the harmonic extension of $\psi|_{\bbH \backslash S}$ to $S$ with normal derivative zero on $\partial S \cap \bbR$. By the conformal invariance of the GFF and $\phi = f \bullet_\gamma \psi$, we obtain the desired Markov property for $\phi$. 
\end{proof}

\begin{figure}[ht]
	\centering
	\includegraphics[scale = 0.56]{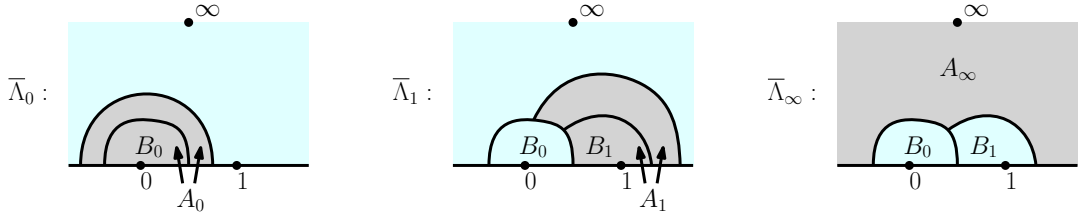}
	\caption{Illustration for the proof of Proposition~\ref{prop-W>2}. Each figure describes a Markov kernel where we resample the field in the grey region conditioned on the field in the blue region.  }\label{fig-irreducible}
\end{figure}

\begin{lemma}\label{lem-sigma-finite-M}
	The measure $M$ is $\sigma$-finite.
\end{lemma}
\begin{proof}
	Let $V_1, \dots, V_n \in (0, 2]$ satisfy $\sum_{i=1}^n V_i = W-2$. Sample $(\eta, \eta_2) \sim P_{U,W}$ (defined in the proof of Proposition~\ref{prop-indep-curves-thick}), then in the region to the right of $\eta_2$ sample curves $(\hat \eta_1, \dots, \hat \eta_{n-1}) \sim \mathcal P^\mathrm{disk}(V_1, \dots, V_n)$ where $P^\mathrm{disk}(V_1, \dots, V_n)$ is the measure defined before Theorem~\ref{thm-disk-2}. Let $\cL$ denote the law of $(\eta, \eta_2, \hat \eta_1, \dots, \hat \eta_{n-1})$.
	
	Sample $(\phi, \eta, \eta_2, \hat \eta_1, \dots, \hat \eta_{n-1}) \sim M \times \cL$, then the argument of Proposition~\ref{prop-indep-curves-thick} gives that the quantum surface $(\bbH, \phi, \eta, \eta_2, 0, 1,\infty)/{\sim_\gamma}$ has law~\eqref{eq-three-disks}. Applying Theorem~\ref{thm-disk-2}, we see that the law of the quantum surface $(\bbH, \phi, \eta, \eta_2, \hat \eta_1, \dots, \hat \eta_{n-1}, 0, 1, \infty)/{\sim_\gamma}$ is
	\[C' \iiint_0^\infty \mathrm{Weld}(\Md_2(U; \ell), \QT(2,2,2; \ell, \ell'), \Md_2(V_1; \ell', \ell_1), \dots, \Md_2(V_n; \ell_{n-2}, \ell_{n-1}))\, d\ell \, d\ell' \, d\ell_1 \dots d\ell_{n-1}. \]
	Thus, for any $N>0$ the event $E_N$ that the quantum lengths of $\eta, \eta_2, \hat \eta_1, \dots, \hat \eta_{n-1}$ all lie in $(\frac1N, N)$ has finite measure with respect to $M \times \cL$, and the events $\{E_N\}_{N \geq 0}$ exhaust the sample space. Thus $M \times \cL$ is $\sigma$-finite.
	
	We now show that $M$ is $\sigma$-finite. Let $F_N$ be the set of $\phi$ such that conditioned on $\phi$, the conditional probability of $E_N$ is at least $\frac1N$. Then
	\[ \infty > (M \times \cL)[E_N] \geq \frac1N M[F_N],\]
	so $M[F_N] < \infty$. Since $\{E_N\}_{N \geq 0}$ exhaust the sample space, the events $\{F_N\}$ also exhaust the sample space. 
\end{proof}

We say a Markov kernel $K: \Omega \to \cF$ on a measurable space $(\Omega, \cF)$ is \emph{irreducible} if there exists a measure $\rho$ such that for any $\omega \in \Omega$ and $A \in \cF$ with $\rho(A) >0$ we have $K^n(\omega, A) >0$ for some $n>0$. 
\cite[Propositions 4.2.1 and 10.1.1, Theorem 10.0.1]{Meyn-Tweedie} states that irreducible Markov chains with invariant probability measures have unique invariant probability measures. We give a $\sigma$-finite variant of this result  {if we assume irreducibility, but more strongly, in the criterion of irreducibility we have $n \equiv 1$ and $\rho$ is an invariant measure of $K$}. 
\begin{lemma}\label{lem-irreducible}
	Suppose a Markov kernel $K: \Omega \to \cF$ on a measurable space $(\Omega, \cF)$ has two $\sigma$-finite invariant measures $\mu_1, \mu_2$ such that for every $\omega \in \Omega$  {the measure $\mu_1$ is} absolutely continuous with respect to $K(\omega, -)$. Further assume that for $i=1,2$ we have $K(x, dy)\mu_i(dx) = K(y, dx) \mu_i(dy)$.
	Then $\mu_1 = c \mu_2$ for some $c \in (0,\infty)$.
\end{lemma}
\begin{proof}
	Let $E \in \cF$ satisfy $\mu_1[E], \mu_2[E]<\infty$. Define the reflected Markov kernel $K_E(x, A) := K(x, A \cap E) + 1_{x \in A} K(x,  {\Omega} \backslash E)$, i.e.\ if a step of a random walk would leave $E$ it instead stays in place. By reversibility, the measures $\mu_1|_E$ and $\mu_2|_E$ are invariant under $K_E$. Moreover $\mu_1|_E$ is absolutely continuous with respect to $K_E(\omega, -)$ for all $\omega \in E$,  {so we can set $\rho = \mu_1|_E$ and $n = 1$ in the definition of irreducibility to conclude that} $K_E$ is irreducible. By \cite[Propositions 4.2.1 and 10.1.1, Theorem 10.0.1]{Meyn-Tweedie} we have $\mu_1|_E = c \mu_2|_E$ for some constant $c$, and sending $E \uparrow \Omega$ gives the full result. 
\end{proof}

\begin{proposition}\label{prop-W>2}
	Theorem~\ref{thm:u2v} holds for $W > 2$. 
\end{proposition}
\begin{proof}
	Let $M$ be the law of the field $\phi$, then Proposition~\ref{prop-indep-curves-thick} identifies the law of $(\phi, \eta)$ as $M \times \SLE_\kappa(U-2; 0, W-2)$. Thus it suffices to show that $M$ agrees with a multiple of $M' := \LF_\bbH^{(\beta_1, 0), (\beta_2, 1), (\beta_3, \infty)}$.
	
	We define three Markov transition kernels $\ol \Lambda_0$, $\ol \Lambda_1$ and $\ol \Lambda_2$ such that $M$ and $M'$ are invariant measures under each Markov kernel, see Figure~\ref{fig-irreducible}. Let $B_0 \subset A_0$ be bounded neighborhoods of $0$ in $\bbH$ such that $\ol A_0 \cap [1,\infty) = \emptyset$. Let $B_1 \subset A_1$ be bounded neighborhoods of $1$ in $\bbH \backslash B_0$ such that $\ol A_1 \cap (-\infty, 0] = \emptyset$ and $\ol B_1 \cap \partial B_0 \neq \emptyset$. Finally let $A_\infty = \bbH \backslash \ol{B_0 \cap B_1}$. 
	
	For $z \in \{0,1,\infty\}$,
	let $\ol \Lambda_z(\phi, d\psi)$ be the law of $\psi$ defined via $\psi|_{\bbH \backslash A_z} = \phi|_{\bbH \backslash A_z}$ and $\psi|_{A_z} = h + \mathfrak h + \frac{\alpha_z}2 G_{A_z}(\cdot, z)$ where $h$ is a GFF in $A_z$,  $\mathfrak h$ is the harmonic extension of $\phi|_{\bbH \backslash A_z}$ to $A_z$ having zero normal derivative on $\partial A_z \cap \bbR$, and $(\alpha_0, \alpha_1, \alpha_\infty) = (\beta_1, \beta_2, \beta_3 - 2Q)$. By Lemmas~\ref{lem-markov-LF},~\ref{lem-markov1},~\ref{lem-markov0} and~\ref{lem-markovinfty}, the measures $M$ and $M'$ are invariant under $\ol \Lambda_0, \ol \Lambda_1$ and $\ol \Lambda_\infty$, and more strongly we get reversibility: we have $\ol \Lambda_j(x, dy)M(dx) = \ol \Lambda_j(y, dx)M(dy)$ for $j \in \{0,1,\infty\}$, and the same holds for $M'$.
	
	Let $\wt K(\phi, dz) = \iint \ol \Lambda_\infty(y, dz) \ol \Lambda_1(x, dy) \ol \Lambda_0(\phi, dx)$, then $M$ and $M'$ are invariant measures of $\wt K$. We now check that $M'$ is absolutely continuous with respect to $\wt K(\phi, -)$ for all $\phi$. It is well known that if $h$ is a GFF in $A_0$ with zero (resp.\ free) boundary conditions on $\partial A_0 \cap \bbH$ (resp.\ $\partial A_0 \cap \bbR$) and $g$ is a smooth function on $A_0$, then the laws of $h|_{B_0}$ and $(h+g)|_{B_0}$ are mutually absolutely continuous, see e.g.\ the argument of \cite[Proposition 2.9]{MS17}. Thus, the $M'(d\psi)$-law of  $\psi|_{B_0}$ is absolutely continuous with respect to the $\ol \Lambda_0(\phi, dx)$-law of $x|_{B_0}$. Similarly, the $M'(d\psi)$-law of $\psi|_{\ol{B_0 \cup B_1}}$  is  absolutely continuous with respect to the $\int \ol\Lambda_1(x, dy) \ol \Lambda_0(\phi, dx)$-law of $y|_{\ol{B_0 \cup B_1}}$, and finally, $M'$ is absolutely continuous with respect to $\iint \ol \Lambda_\infty(y, -) \ol \Lambda_1(x, dy) \ol \Lambda_0(\phi, dx)$.
	
	Now, let $K(\phi, dz) = \iint  \ol \Lambda_0(y, dz) \ol \Lambda_1 (x, dy) \wt K(\phi, dx)$. The measure $M'$ is absolutely continuous with respect to $\wt K(\phi, -)$, and hence $K(\phi, -)$, for all $\phi$. Moreover, since $K = \ol \Lambda_0 \ol \Lambda_1 \ol \Lambda_\infty  \ol \Lambda_1  \ol \Lambda_0$, we get reversibility. 
	Finally, $M$ is $\sigma$-finite (Lemma~\ref{lem-sigma-finite-M}), and so is $M'$ since the event $F_N$ that the average of $\psi$ on $(\partial B_1(0)) \cap \bbH$ lies in $[-N,N]$ is finite satisfies $M'[F_N] <\infty$, and $\{F_N\}_{N \geq 0}$ exhaust the sample space. By Lemma~\ref{lem-irreducible} $M = cM'$ for some constant $c$, as desired.

	%   {[*The reversibility with respect to $\wt K$ is not so clear: by repeatedly applying reversibility of each of the $\ol \Lambda$'s, one only obtains that $\wt K(x, dy) \mu_i(dx) = \widehat K(y, dx) \mu_i (dy)$ where $\widehat K (\phi, dz) := \iint \ol \Lambda_0(y, dz) \ol \Lambda_1(x, dy) \ol \Lambda_\infty (\phi, dx)$.*]}
\end{proof}

\subsection{The case $W \in (\frac{\gamma^2}2, 2)$}\label{sec-W<2}
The case $W \in (\frac{\gamma^2}2,2)$ will be handled with the same proof structure as the $W>2$ case discussed in Section~\ref{sec-W>2}. 
The first step is to prove that the field and curve are independent, and identify the curve (Proposition~\ref{prop-indep-curves-thinner}). To that end, we need the following conformal welding result. 

%\begin{proposition}[{\cite[Lemma 4.4]{AHS21}}]\label{prop-weld-M2bullet}
%For $0<W_- < \frac{\gamma^2}2 < W_+$, we have \[\QT(W_- + W_+, W_- + W_+, 2) \otimes \SLE_\kappa(W_- - 2, W_+ - 2) = C\int_0^\infty \mathrm{Weld}(\Md_2(W_-;\ell), \QT(W_+, W_+,2;\ell))\, d\ell. \]\end{proposition}

\begin{lemma}\label{lem-cut-V}
Let $U \in (0,2)$, $V \in (0, 2-\frac{\gamma^2}2)$ and $W = 2-V$. {Let $\beta_1 = Q+\frac{\gamma}{2}-\frac{2+U}{\gamma}$.}
Let $\wt P_{U,V}$ be the law of the curves $(\eta_1, \eta_2)$ in Figure~\ref{fig-ig-resample} with parameters {$(x_1,x_2,x_3) = (\lambda(1-U), \lambda, \lambda)$} and $(\theta_1, \theta_2) = (0, \frac{\lambda(V-2)}{\chi})$. Sample
\[ (\psi,\eta_1,\eta_2) \sim \LF_\bbH^{(\beta_1, 0), (\gamma, 1), (\beta_1,\infty)} \times \wt P_{U,V}.\]
Then the decorated quantum surface $(\bbH, \psi, \eta_1, \eta_2, 0, 1, \infty)/{\sim_\gamma}$ has law \[C\iint_0^\infty \mathrm{Weld}(\Md_2(U; \ell), \QT(W, 2, W; \ell, \ell'), \Md_2(V; \ell')) \, d\ell \, d\ell', \qquad C \in (0,\infty).\]
\end{lemma}
\begin{proof}
%For $(\eta_1, \eta_2) \sim \wt P_{U,V}$, the marginal law of $\eta_1$ is $\SLE_\kappa(U-2; 0)$ in $(\bbH, 0, \infty)$, and the conditional law of $\eta_2$ given $\eta_1$ is $\SLE_\kappa(-V;V-2)$ in $(D_2, 1, \infty)$ where $D_2$ is the connected component of $\bbH \backslash \eta_1$ with $1$ on its boundary.

By \cite[Lemma 4.4]{AHS21}, we have \[\int_0^\infty \mathrm{Weld}(\QT(W,2,W;\ell'), \Md_2(V;\ell'))\, d\ell' = C_1 \QT(2,2,2) \otimes \SLE_\kappa(-V;V-2), \qquad C_1 \in (0,\infty).\] By Proposition~\ref{prop-W=2}, we have \[\int_0^\infty \mathrm{Weld}(\Md_2(U;\ell), \QT(2,2,2;\ell))\, d\ell = C_2 \QT(U+2,U+2, 2) \otimes \SLE_\kappa(U-2; 0), \qquad  C_2 \in (0,\infty).\]
Combining these yields the result. 
\end{proof}

\begin{figure}[ht]
\centering
\includegraphics[scale=0.56]{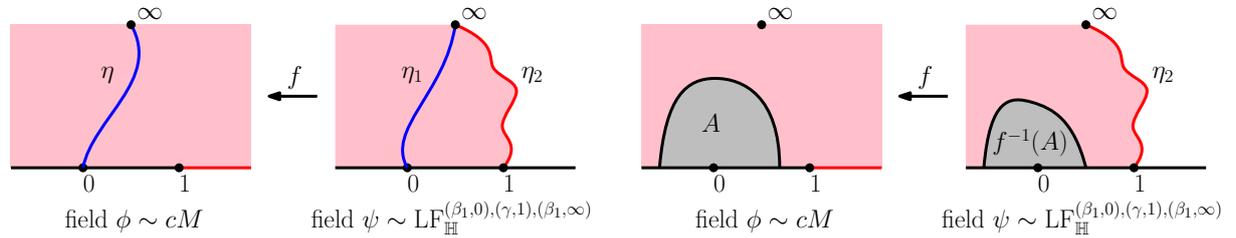}
\caption{\textbf{Left:} Illustration for the proof of  Proposition~\ref{prop-indep-curves-thinner}. \textbf{Right}: Illustration for the proof of Lemma~\ref{lem-markov-welded-02}. }\label{fig-markov-welded02}
\end{figure}

\begin{proposition}\label{prop-indep-curves-thinner}
In the setting of Theorem~\ref{thm:u2v} with $W \in(\frac{\gamma^2}2, 2)$, let $M$ be the law of the field $\phi$. Then the joint law of $(\phi, \eta)$ is $M \times \SLE_\kappa(U-2; 0, W-2)$. 
\end{proposition}
\begin{proof}
Let $V = 2-W$. Recall the law $\wt P_{U,V}$ of Lemma~\ref{lem-cut-V}. 
The marginal law of $\eta_1$ is $\SLE_\kappa(U-2; 0)$ in $(\bbH, 0, \infty)$, and the conditional law of $\eta_2$ given $\eta_1$ is $\SLE_\kappa(-V;V-2)$ in $(D_2, 1, \infty)$ where $D_2$ is the connected component of $\bbH \backslash \eta_1$ with $1$ on its boundary. 

Let $(\beta_1,\beta_2,\beta_3)$ be insertions corresponding to weights $(2+U, 2-V, 2+U-V)$.  
Sample
\[ (\psi,\eta_1,\eta_2) \sim \nu_\psi(1,\infty)^{\frac{2V}{\gamma^2}}\LF_\bbH^{(\beta_1, 0), (\gamma, 1), (\beta_1,\infty)}(d\psi) \times \wt P_{U,V}.\]
See Figure~\ref{fig-markov-welded02}. 
By Lemma~\ref{lem-cut-V}, the decorated quantum surface $(\bbH, \psi, \eta_1, \eta_2, 0, 1, \infty)/{\sim_\gamma}$ has law \[\iint_0^\infty \mathrm{Weld}(\Md_2(U; \ell), \QT(W, 2, W; \ell, \ell'), \tildeMd_2(V; \ell')) \, d\ell \, d\ell',\]
where $\tildeMd(V)$ is the weighted quantum disk defined in Lemma~\ref{lem-weight-other-side}, and $\tildeMd(V;\ell')$ is the disintegration of $\tildeMd(V)$ by the unweighted boundary arc length. 

By Lemma~\ref{lem-weight-other-side} we have $|\tildeMd_2(V;\ell)| = c$ for some finite constant $c$, so the marginal law of the decorated quantum surface to the left of $\eta_2$ is $c\int_0^\infty \mathrm{Weld}(\Md_2(U; \ell), \QT(W,2,W;\ell))\, d\ell$. 
Let $f$ be the conformal map sending the connected component of $\bbH \backslash \eta_2$ to the left of $\eta_2$ to $\bbH$ such that $f$ fixes $(0,1,\infty)$, and let $\phi = f \bullet_\gamma \psi$ and $\eta = f(\eta_1)$. Then the marginal law of $\phi$ is $cM$. 

Finally, by Proposition~\ref{prop:ig-flow-descriptions}, when $(\eta_1, \eta_2) \sim \wt P_{U,V}$, the conditional law of $\eta_1$ given $\eta_2$ is $\SLE_\kappa(U-2; 0, -V)$ the region to the left of $\eta_2$. Since $f$ is measurable with respect to $\sigma_2$, we deduce that the law of $(\phi, \eta)$ decomposes as a product measure $cM \times \SLE_\kappa(U-2; 0, -V)$, as desired. 
\end{proof}

\begin{lemma}\label{lem-markov-welded-02}
For the measure $M$ defined in Proposition~\ref{prop-indep-curves-thinner}, the statements of Lemmas~\ref{lem-markov1},~\ref{lem-markov0} and~\ref{lem-markovinfty} hold with $\beta_1 = Q + \frac\gamma2 - \frac{2+U}\gamma$, $\beta_2 = Q + \frac\gamma2 - \frac{W}\gamma$ and $\beta_3 = Q + \frac\gamma2 - \frac{U+W}\gamma$.
\end{lemma}
\begin{proof}
The analogues of Lemmas~\ref{lem-markov1}  and~\ref{lem-markovinfty} have exactly the same proofs as the original lemmas. We now discuss the analogue of Lemma~\ref{lem-markov0}.

The argument is very similar to that of Lemma~\ref{lem-markovinfty} so we will be brief.
We work in the setting of the argument of Proposition~\ref{prop-indep-curves-thinner}, see Figure~\ref{fig-markov-welded02}.  Let $A' = f^{-1}(A)$, then since $f$ is measurable with respect to $\eta_2$ and $\psi$ is independent of $\eta_2$, the Markov property of the Liouville field gives a Markov property for $\psi|_{A'}$ given $\psi|_{\bbH \backslash A'}$. Using the map $f$, this gives the Markov property for $\phi_A$ given $\phi|_{\bbH \backslash A}$. 
\end{proof}

\begin{lemma}\label{lem-sigma-finite-M-2}
$M$ is $\sigma$-finite. 
\end{lemma}
\begin{proof}
For $(\phi, \eta)\sim M \times \SLE_{\kappa}(U-2; 0, W-2)$, the law of $(\bbH, \phi, \eta, 0, 1, \infty)/{\sim_\gamma}$ is 
\[c \int_0^\infty \mathrm{Weld}(\Md_2(U;\ell), \QT(W,2,W; \ell))\, d\ell\quad \text{ where }c \text{ is a constant},\] so the event $E_N$ that the quantum length of $\eta$ lies in $[\frac1N, N]$ has finite measure, and the events $\{E_N\}_{N \geq 0}$ exhaust the sample space. The rest of the argument is the same as that of Lemma~\ref{lem-sigma-finite-M}.
\end{proof}

\begin{proposition}\label{prop-W<2}
Theorem~\ref{thm:u2v} holds for $W \in (\frac{\gamma^2}2,2)$. 
\end{proposition}
\begin{proof}
The argument is identical to that of Proposition~\ref{prop-W>2}. We can define Markov kernels $\ol \Lambda_z(\phi, d\psi)$ for $z \in \{0,1,\infty\}$, under which $M$ and $M' := \LF_\bbH^{(\beta_1, 0), (\beta_2, 1), (\beta_3, \infty)}$ are invariant by Lemmas~\ref{lem-markov-LF} and~\ref{lem-markov-welded-02}. We then define the Markov kernel $K = \ol \Lambda_0 \ol \Lambda_1 \ol \Lambda_\infty  \ol \Lambda_1  \ol \Lambda_0$; it has $M$ and $M'$ as invariant measures, and by the argument of Proposition~\ref{prop-W>2} $K$ is irreducible. Finally, $M$ is $\sigma$-finite by Lemma~\ref{lem-sigma-finite-M-2}, and $M'$ is $\sigma$-finite by the same argument of Proposition~\ref{prop-W>2}, so Lemma~\ref{lem-irreducible} yields $M = cM'$ for some constant $c$. 
\end{proof}

\section{Proof of Theorem \ref{thm:M+QTp-rw} in the full range}\label{sec-6-thin-qt}

In this section we prove Theorem~\ref{thm:M+QTp-rw}. The first three subsections aim to establish the following weaker version of Theorem \ref{thm:M+QTp-rw}. 

Let $\mathsf m$ be a measure on the space of curves in $(\bbH, 0, 1, \infty)$ from $0$ to $\infty$ which do not hit $1$. Let $W,W_1,W_2,W_3 > 0$ such that none of $W+W_1, W+W_2, W_1, W_2, W_3$ equal $\frac{\gamma^2}2$. 
Sample a pair $S, \eta$ from $\QT(W+W_1,W+W_2,W+W_3) \times \mathsf m$, let $(D, a_1, a_2, a_3)$ be an embedding of $S$ and let $(\tilde D, \tilde a_1, \tilde a_2, \tilde a_3)$ be the corresponding embedding of the core of $S$. If $a_1 \neq \tilde a_1$ let $\eta_1$ be independent $\SLE_{\kappa}(W-2; W_1-2)$ in each component in the interior of $D$ from $\tilde a_1$ to $a_1$. Likewise if $a_2 \neq \tilde a_2$ let $\eta_2$ be independent $\SLE_\kappa(W-2; W_2-2)$ in each component in the interior of $D$ from $a_2$ to $\tilde a_2$. Let $ \tilde \eta \subset D$ be the image of $\eta$ under the conformal map sending $\bbH$ to $(\tilde D, \tilde a_2, \tilde a_3, \tilde a_1)$, and let $\eta'$ be the concatenation of $\tilde \eta$ with whichever of  $\eta_2, \eta_1$ we have defined. Let $\QT(W+W_1, W+W_2, {W}_3)\otimes \mathsf{m}(W;W_1,W_2,W_3)$ be the law of the decorated quantum surface $(D, \phi, a_1, a_2, a_3, \eta')$.

\begin{proposition}\label{prop-associate-c}
There exists a measure  $\mathsf{m}(W;W_1,W_2,W_3)$  and a constant $c = c_{W,W_1,W_2}\in (0,\infty)$ such that
\begin{equation}\label{eqn-qt-change-weight-2}
	\QT(W+W_1, W+W_2, {W}_3)\otimes \mathsf{m}(W;W_1,W_2,W_3) = c\int_0^\infty \Wd(\Md_2(W;\ell),\QT(W_1,W_2,{W}_3;\ell))d\ell
\end{equation} 
where we are welding the right boundary arc of the weight $W$ quantum disk with the left boundary arc of the quantum triangle linking the weight $W_1$ and $W_2$ vertices.
\end{proposition}

In particular, Theorem \ref{thm:u2v} is the special case of Proposition \ref{prop-associate-c} with $W\in(0,\frac{\gamma^2}{2})$, $W_1=W_3>\frac{\gamma^2}{2}$ and $W_2=2$. In Section \ref{sec-change-weight}, by a reweighting argument, we are going to repeatedly apply Theorem \ref{thm:u2v} and hence prove Proposition \ref{prop-associate-c} in the case where $W_1,W_2>\frac{\gamma^2}{2}$ and $W,W_3>0$. In Section \ref{sec-w-gamma2-w}, we build on this result and work on the special case where $W_1+W_2={\gamma^2}$ using the thick-thin duality. Based on this, in Section \ref{sec-general-thin} we conclude the proof of Proposition \ref{prop-associate-c}, while in Section \ref{sec-interface-law} we identify the law of the curves via the SLE resampling property and thus complete the proof of Theorem \ref{thm:M+QTp-rw} {when none of the weights $W_1,W_2,W+W_1,W+W_2,W_3$ equal $\frac{\gamma^2}{2}$}. {Finally in Section \ref{sec-gamma/2} we prove the full version of Theorem~\ref{thm:M+QTp-rw} which addresses the case when some weight equals $\frac{\gamma^2}2$}, while in Section \ref{sec-pfthm1.1} we prove Theorem \ref{thm:disk} by a quick application of Theorem \ref{thm:M+QTp-rw}.

%  {In this section we prove Theorem~\ref{thm:M+QTp-rw}. 
% First, we prove that for (ranges) we have 
% \begin{equation}\label{eqn-qt-change-weight-2}
	% \QT(W+W_1, W+W_2, {W}_3)\otimes \mathsf{m}(W;W_1,W_2,W_3) = c\int_0^\infty \Wd(\Md_2(W;\ell),\QT(W_1,W_2,{W}_3;\ell))d\ell.
	% \end{equation} 
% for some measure  $\mathsf{m}(W;W_1,W_2,W_3)$ on the space of curves, and $c = c_{W,W_1,W_2,W_3}\in (0,\infty)$ is some constant. Theorem~\ref{thm:u2v} proves~\eqref{eqn-qt-change-weight-2} for the range (). In Section~\ref{}, we extend to the range () by repeated applications of Theorem~\ref{thm:u2v} and using reweighting to change the weight $W_3$. In Section~\ref{} we build on this to prove the special case (thick-thin duality). In Section~\ref{} we extend to the range (). Finally, in Section~\ref{} we identify the law of the interface.

% This completes the proof of Theorem~\ref{thm:M+QTp-rw} when (not $\frac{\gamma^2}2$). In Section~\ref{} we use a limiting argument to resolve this final case.
% }

\subsection{The $W_1, W_2 > \frac{\gamma^2}2$ regime}\label{sec-change-weight}

This section serves to prove the following:
\begin{proposition}\label{prop-thick-welding}
Fix $W_1,W_2>\frac{\gamma^2}{2}$ and $W,W_3>0$. Then there exists some measure $\mathsf{m}(W;W_1,W_2,W_3)$ such that \eqref{eqn-qt-change-weight-2} holds.
\end{proposition}
To start with, we extend the idea of changing the weight of the third point as in \cite[Proposition 4.5]{AHS21} (See also Section \ref{sec-w2u-weld}) to quantum triangles with general weights in Proposition~\ref{prop-change-weight}. {Although this section only requires $W_1, W_2 > \frac{\gamma^2}2$, we state and prove  Proposition~\ref{prop-change-weight} for a larger range for later sections.} Suppose we have a curve-decorated surface from $\QT(W+W_1, W+W_2, \tilde{W}_3)\otimes \mathsf{m}(W;W_1,W_2,\tilde{W}_3)$ embedded as $(D,\phi,\eta, a_1,a_2,a_3)$ and $\eta$ is a curve from $a_2$ to $a_1$. We choose $D$ such that each component of the interior of $D$ has smooth boundary. Let $(\tilde D, \tilde a_1, \tilde a_2, \tilde a_3)$ be the corresponding embedding of the core of the quantum triangle, and 
let $\tilde D_\eta$ be the connected component of $\tilde D \backslash \eta$ with $\tilde a_3$ on its boundary. See Figure~\ref{fig-reweighting}.  Let $\psi_\eta:\tilde D_\eta\to \bbH$   be the conformal mapping sending the first (resp.\ last) point on $\partial D_\eta$ (resp.\ $\partial D$) hit by $\eta$ to 0 (resp.\ $\infty$) and sending {$\tilde{a}_3$ to $1$, and $\psi: \tilde D \to \bbH$ be the conformal map sending $(\tilde{a}_2,\tilde{a}_1,\tilde{a}_3)$ to $(0,\infty,1)$}. %Note also that for the derivative to exist we might impose e.g.\ that each connected component of $D$ has smooth boundary arcs.
%For $\tilde{W}_3<\frac{\gamma^2}{2}$, we can first drop the weight $\tilde{W}_3$ thin disk and let $\psi,\psi_\eta$ be the same as the ones induced by the remaining weight $(W+W_1,W+W_2,\gamma^2-\tilde{W}_3)$ quantum triangle. 

\begin{proposition}\label{prop-change-weight}
Suppose that given $W, W_1, W_2,W_3,\tilde{W}_3>0$, $ W_3,\tilde{W}_3\neq \frac{\gamma^2}{2}$ with $W_1,W_2,W+W_1,W+W_2\neq\frac{\gamma^2}{2}$, there exists some measure  $\mathsf{m}(W;W_1,W_2,\tilde{W}_3)$  on random simple curves in $D$ starting from $a_2$ to $a_1$ and not hitting {$\tilde a_3$} such that we have the conformal welding of quantum triangles as in \eqref{eqn-qt-change-weight-2} with $W_3$ replaced by $\tilde{W}_3$.
% 	\begin{equation}\label{eqn-qt-change-weight-1}
	% 	\QT(W+W_1, W+W_2, \tilde{W}_3)\otimes \mathsf{m}(W;W_1,W_2,\tilde{W}_3) = c\int_0^\infty \Wd(\Md_2(W;\ell),\QT(W_1,W_2,\tilde{W}_3;\ell))d\ell
	% 	\end{equation} 
% 	where $c = c_{W, W_1, W_2}\in (0,\infty)$ is some constant. 
Then if we define the measure $\mathsf{m}(W;W_1,W_2,W_3)$ on curves by setting
\begin{equation*}
	\frac{d\mathsf{m}(W;W_1,W_2,W_3)}{d\mathsf{m}(W;W_1,W_2,\tilde{W}_3)}(\eta) = \left|\frac{\psi'_\eta({\tilde a_3})}{\psi'({\tilde a_3})} \right|^{\Delta_{\wt \beta_3} - \Delta_{\beta_3}}
\end{equation*}
where $\beta_3 = \gamma+\frac{2-W_3}{\gamma}$ and $\tilde{\beta_3} = \gamma+\frac{2-\tilde{W}_3}{\gamma}$. Then \eqref{eqn-qt-change-weight-2} holds.

\end{proposition}	
\begin{figure}[ht]
\centering
\includegraphics[width=0.75\textwidth]{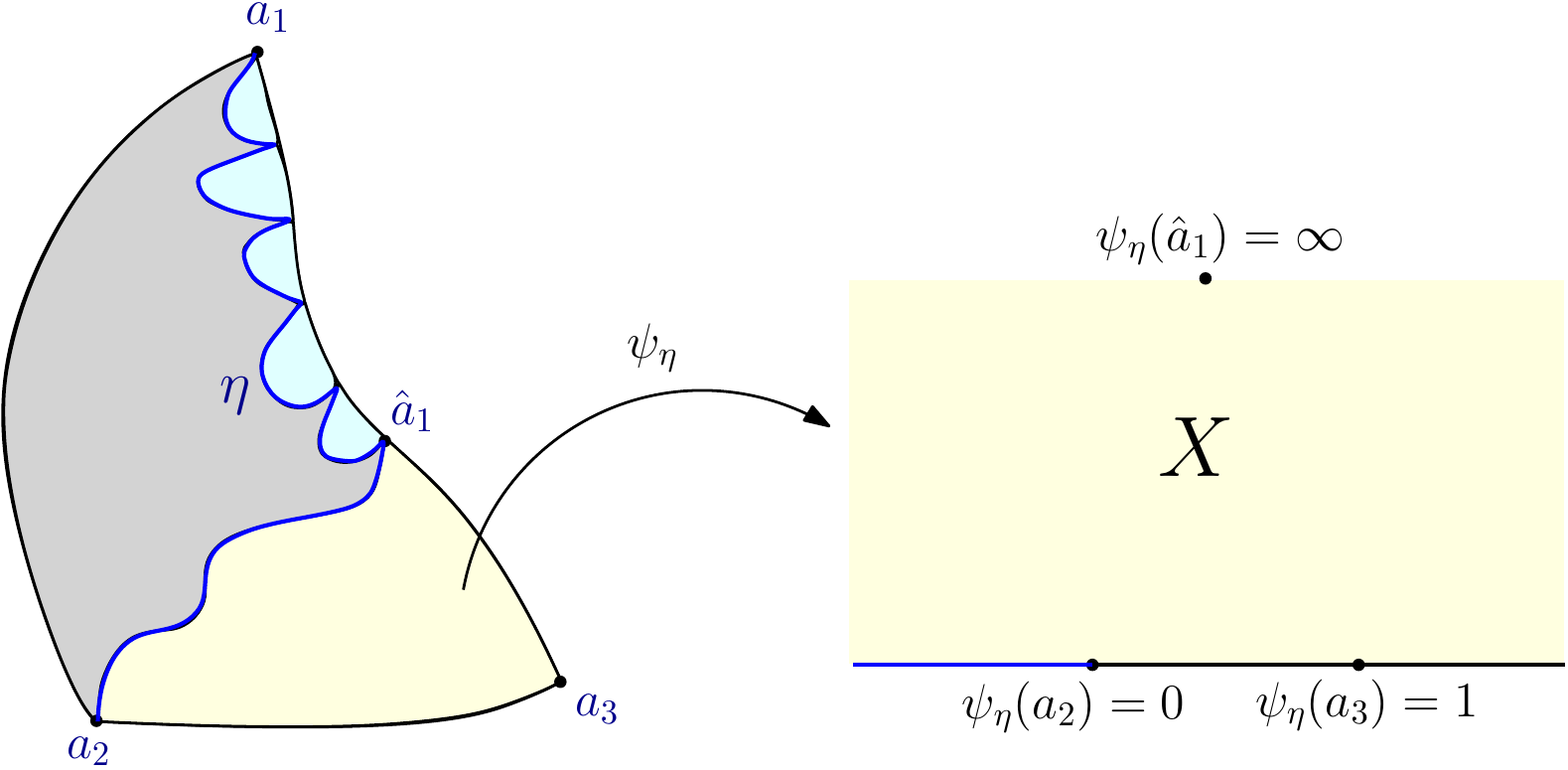}
\caption{Setup of Proposition \ref{prop-change-weight} in the case $W_1<\frac{\gamma^2}{2}$ and $W, W_2, \tilde{W}_3>\frac{\gamma^2}{2}$ and an illustration of the conformal map $\psi_\eta$. By decorating a quantum triangle from $\QT(W+W_1, W+W_2, \tilde{W}_3)$ with an independent curve $\eta$ from $\mathsf{m}(W;W_1,W_2,\tilde{W}_3)$ we get a weight $W$ disk $S_1$ (gray) and a quantum triangle with weights $W_1, W_2, \tilde{W}_3$, which has two parts $S_2$ (blue) and $S_3$ (yellow).   Consider the conformal map $\psi_\eta$ and let $X$ be the corresponding surface embedded on $\bbH$ as on the right panel. Weighting the law of curve-decorated surface by $e^{(\beta_3-\tilde{\beta}_3)X(1)}$ allows us to shift \eqref{eqn-qt-change-weight-2} to $W_3$ from the $\tilde{W}_3$ case.}\label{fig-reweighting}
\end{figure}

% and $D_3$ be the component containing $a_3$

% Recall the disintegration of the Liouville field measure $\LF_\bbH^{(\beta_1,\infty), (\beta_2, 0), (\beta_3, 1)}$ as in Definition \ref{def-qt-disintegration-thick}. Let $ \LF_{\bbH,\ell}^{(\beta_1,\infty), (\beta_2, 0), (\beta_3, 1)}$ be the law we get as we disintegrate over the quantum length $\nu_{\tilde{h}}((-\infty,0])$. % {[We actually don't define the disintegration of the Liouville field in that definition, only the quantum triangle. 
We can define a disintegration $\LF_\bbH^{(\beta_1, \infty), (\beta_2, 0), (\beta_3, 1)} = \int_0^\infty \LF_{\bbH, \ell}^{(\beta_1, \infty), (\beta_2, 0), (\beta_3, 1)} \, d\ell$ where for each $\ell>0$ the measure $\LF_{\bbH, \ell}^{(\beta_1, \infty), (\beta_2, 0), (\beta_3, 1)}$ is supported on $\{ \nu_\phi((-\infty, 0)) = \ell\}$, see e.g.\ Definition~\ref{def-qt-disintegration-thick} and Lemma~\ref{lem-qt-disintegration-thick}.
%]}
\begin{lemma}\label{lm-change-reweight}
Let $\beta_1,\beta_2<Q$, $\beta_3, \tilde{\beta}_3\in\bbR$ and $\ell>0$. In the sense of weak convergence of measures,
\begin{equation}
	\lim_{\epsilon\to 0} \e^{\frac{\beta_3^2-\tilde{\beta}_3^2}{4}}e^{\frac{\beta_3-\tilde{\beta}_3}{2}\phi_\e(1)}\LF_{\bbH,\ell}^{(\beta_1, \infty), (\beta_2, 0), (\tilde{\beta}_3,1)}(d\phi) = \LF_{\bbH,\ell}^{(\beta_1, \infty), (\beta_2, 0), ({\beta}_3,1)}(d\phi). 
\end{equation}
%	 {[We need $\beta_1, \beta_2 < Q$ for $\nu_\phi((-\infty, 0))$ to be finite...]}
\end{lemma}

The proof is identical to that of \cite[Lemma 4.6]{AHS21}, which is a direct application of the Girsanov theorem. We omit the details. 

\begin{proof}[Proof of Proposition \ref{prop-change-weight}]
%We begin by observing that it suffices to prove Proposition \ref{prop-change-weight} when both of $W_3$ and $\tilde{W_3}$ are greater than $\frac{\gamma^2}{2}$. This is because for $W_3\in (0,\frac{\gamma^2}{2})$, concatenating a weight $W_3$ disk at the weight $\gamma^2-W_3$ vertex on both sides does not change the equation, and the corresponding $\Delta_{\beta}$ values are the same for $W_3$ and $\gamma^2-W_3$.  {[This needs to be more carefully written. If the following is fine, let's use it instead.]}

{We first note that if the result holds for $W_3, \tilde W_3 > \frac{\gamma^2}2$, then it holds for the full range $W_3, \tilde W_3 \in (0, \infty) \backslash \{\frac{\gamma^2}2\}$. Indeed, if $W_3 < \frac{\gamma^2}2 < \tilde W_3$, we can use Proposition~\ref{prop-change-weight} with $W_3$ replaced by $\gamma^2 - W_3$ and concatenate a weight $W_3$ quantum disk to the weight $\gamma^2 - W_3$ vertex; since $\Delta_\beta$ takes the same value for $\beta = \gamma + \frac{2-W_3}\gamma$ and $\beta = \gamma + \frac{2-(\gamma^2-W_3)}\gamma$, the law of the curve is as desired. For $W_3$ arbitrary and $\tilde W_3 < \frac{\gamma^2}2$, recall that the quantum triangle with weights $(W_1, W_2, \tilde W_3)$ is obtained by concatenating a quantum triangle with weights $(W_1, W_2, \gamma^2 - \tilde W_3)$ and a quantum disk of weight {$\tilde W_3$}; by forgetting this quantum disk we reduce the problem to the solved case where $\tilde W_3$ is replaced by $\gamma^2 - \tilde W_3$. Henceforth we assume $W_3, \tilde W_3 > \frac{\gamma^2}2$. }

%The case $W>\frac{\gamma^2}{2}$ follows from Lemma \ref{lm-change-reweight} and the same argument of the proof of \cite[Proposition 4.5]{AHS21} since the  quantum triangle from $\QT(W+W_1, W+W_2, \tilde{W}_3)$ can be embedded as $\LF_{\bbH}^{(\beta_1, \infty), (\beta_2, 0), (\tilde{\beta}_3,1)}$ with the random curve $\eta$ going from 0 to $\infty$ not hitting 1. It remains to explain how the argument is adapted to $W\in (0,\frac{\gamma^2}{2})$ case.

First assume $W_1+W,W_2+W>\frac{\gamma^2}{2}$. The quantum triangle from $\QT(W+W_1, W+W_2, \tilde{W}_3)$ can be embedded as $c\LF_{\bbH}^{(\beta_1, \infty), (\beta_2, 0), (\tilde{\beta}_3,1)}$ with the random curve $\eta$ going from 0 to $\infty$ not hitting 1. Sample $(Y, \eta)$ from $c\LF_{\bbH}^{(\beta_1, \infty), (\beta_2, 0), (\tilde{\beta}_3,1)}\times  \mathsf{m}(W;W_1,W_2,\tilde{W}_3)$, so by definition $(\mathbb{H}, Y, \eta, \infty, 0, 1)/{\sim_\gamma}$ has the law of the left hand side of \eqref{eqn-qt-change-weight-2} with $W_3$ replaced by $\tilde{W}_3$. Let $D_\eta^1$ be the union of the components of $\bbH\backslash\eta$ whose boundaries contain a segment of $(-\infty,0)$, $D_\eta^3$ be the component of $\bbH\backslash\eta$ with 1 on its boundary as defined, and $D_\eta^2$ be the union of the remaining components (if not empty).   Set
\begin{equation}\label{eqn-qt-change-weight-3}
	X = Y\circ\psi_{\eta}^{-1}+Q\log|(\psi_{\eta}^{-1})'|
\end{equation}  
and define the quantum surfaces $S_1=(D_\eta^1, Y, \infty, 0)/{\sim_\gamma}$, $S_2 = (D_\eta^2,Y)/\sim_\gamma$, $S_3 = (\bbH, X, \infty, 0, 1)/\sim_\gamma$. See also Figure~\ref{fig-reweighting} for the setup.
By~\eqref{eqn-qt-change-weight-2} for $\tilde W_3$ instead of $W_3$ and our definition of quantum triangles, $S_2$ and $S_3$ are {conditionally independent given their left boundary lengths},  while  the conditional law of $S_1$ given $(S_2, S_3)$ is $\Md_2(W;L)$ where $L$ is $\nu_X((-\infty,0))$ plus the sum of the quantum lengths of boundary arcs of $S_2$ lying within $\bbH$. We weight the law of $(Y, \eta)$ by $\e^{\frac{\beta_3^2-\tilde{\beta}_3^2}{4}}e^{\frac{\beta_3-\tilde{\beta}_3}{2}X_\e(1)}$ and send $\e\to0$; using the argument of \cite[Proposition 4.5]{AHS21}, the conditional law of $S_1$ given the pair $(S_2, S_3)$ is unchanged. Moreover, by the {conditional independence of $S_2$ and $S_3$ given their left boundary lengths}, by Lemma \ref{lm-change-reweight}, $(S_2, S_3)$ converges in law to $c\QT(W_1,W_2,W_3)$ under the reweighting as $\e\to0$, the joint law of the quantum surfaces $(S_1, S_2,S_3)$ converges to the right hand side of \eqref{eqn-qt-change-weight-2}, while the law of $(Y, \eta)/{\sim_\gamma}$ converges to the left side of \eqref{eqn-qt-change-weight-2}. This finishes the proof for the case $W_1+W,W_2+W>\frac{\gamma^2}{2}$.

For the case where $W_i + W < \frac{\gamma^2}2$ for some $i$, we apply the above argument to the core of the quantum triangle; the proof is then identical. 
%For the rest of the cases, the weight $(W+W_1,W+W_2,\tilde{W}_3)$ triangle contains one thick quantum triangle with $a_3$ on boundary (since we assumed $\tilde{W}_3>\frac{\gamma^2}{2}$) and one or two independent thin quantum disks. On the thin quantum disk part, the measures $\mathsf{m}(W;W_1,W_2,W_3),\mathsf{m}(W;W_1,W_2,\tilde{W}_3)$ are the same. Therefore the proof immediately follows by restricting to the thick triangle part and applying the same proof for the $W_1+W,W_2+W>\frac{\gamma^2}{2}$ case. 
\end{proof}   

{Now we are ready to prove Proposition \ref{prop-thick-welding}. In this proof, we will repeatedly glue together quantum disks and quantum triangles. The key inputs are Theorem~\ref{thm:u2v} (to glue a single quantum disk to a single quantum triangle), the commutativity of multiple gluing operations, and  Proposition~\ref{prop-change-weight} (to change the weight of the vertex that is not on the welding interface).}
\begin{proof}[Proof of Proposition \ref{prop-thick-welding}]
\emph{Step 1. $W\in (0,\frac{\gamma^2}{2})$, $W_1>\frac{\gamma^2}{2}$ and $W_2\ge 2$.}
First assume $W_2\in [2,\frac{\gamma^2}{2}+2)$. We start from the weight $(W_1, 2, 2)$ triangle and weld an independent quantum disk from $\Md_2(W)$ to its left boundary and an independent disk from $\Md_2(W_2 - 2)$ to its bottom arc, see  Figure \ref{fig-associate-a}. That is, we work on the measure
\begin{equation}\label{eqn-associate-1-a}
	\int_0^\infty\int_0^\infty \Wd(\Md_2(W;\ell), \QT(W_1, 2,  2;\ell, s), \Md_2(W_2 -2;s))dsd\ell.
\end{equation}
On one hand, if we fix $s$ and integrate over $\ell$ first, i.e., we weld together the quantum triangle and the weight $W$ quantum disk, by {Theorem   \ref{thm:u2v}} and Proposition \ref{prop-change-weight}%(where we disintegrate over the bottom boundary length $s$; also note that by Proposition \ref{prop-change-weight}, the weight $W_1$ in the equation can be replaced by 2)
, \eqref{eqn-associate-1-a} is a constant multiple of
\begin{equation}\label{eqn-associate-2-a}
	\int_0^\infty \Wd\bigg(\big(\QT(W+2,2,W+W_1   ;s)\otimes \mathsf{m}(W; W_1, 2,  2)\big),\Md_2(W_2 -2;s)   \bigg)ds.
\end{equation}
Integrating over $s$, by {Theorem \ref{thm:u2v}} we see that  \eqref{eqn-associate-2-a} is a constant times 
\begin{equation}\label{eqn-associate-3-a}
	\QT(W+W_1, W+W_2, W_2)\otimes\mathsf{m}_2
\end{equation}
where $\mathsf{m}_2$ is some measure on tuple of curves $(\eta_1,\eta_2)$, such that $\eta_2$ has marginal law $\mathsf{m}(W_2-2; W+2, 2, W+W_1)$ in $\bbH$ from 1 to 0 and the conditional law of $\eta_2$ given $\eta_1$ is $\mathsf{m}(W; W_1, 2, 2)$ in $\bbH\backslash\eta_2$ from $0$ to $\infty$.% (and hence the law of the curves are independent of the field)  {[doesn't the notation with $\otimes$ already state independence of field and curves? Also, do we need the description of the curves?]}. 

On the other hand, if we fix $\ell$ and integrate over $s$ first, i.e., we weld together the quantum triangle and the weight $W_2-\gamma^2+W_1$ quantum disk, by \eqref{eqn-qt-disk}, we get a constant times
\begin{equation}\label{eqn-associate-4-a}
	\int_0^\infty \Wd\bigg(\Md_2(W;\ell), \big(\QT(W_1, W_2, W_2)\otimes \mathsf{m}(W_2-2;2, 2, W_1)\big)  \bigg)d\ell.
\end{equation}
Therefore if we forget about the curve $\eta_2$ and compare \eqref{eqn-associate-3-a} with \eqref{eqn-associate-4-a}, we obtain \eqref{eqn-qt-change-weight-2} in case $W_3 = \tilde{W}_3:=W_2$. Applying Proposition \ref{prop-change-weight} once more yields \eqref{eqn-qt-change-weight-2} for $W\in (0,\frac{\gamma^2}{2})$, $W_1>\frac{\gamma^2}{2}$ and $W_2 {= W_3}\in [2,\frac{\gamma^2}{2}+2)$. {Proposition~\ref{prop-change-weight} then allows us to choose $W_3$ arbitrary, completing the proof in this case.}

Now suppose we have proved \eqref{eqn-qt-change-weight-2} for $W\in (0,\frac{\gamma^2}{2})$, $W_1>\frac{\gamma^2}{2}$ and $W_2\in [2+\frac{k\gamma^2}{2},2+\frac{(k+1)\gamma^2}{2})$ and some $k\ge 0$. Then for $W_2\in [2+\frac{(k+1)\gamma^2}{2},2+\frac{(k+2)\gamma^2}{2})$   we pick $U<\frac{\gamma^2}{2}$ such that $W_2\in [2+\frac{k\gamma^2}{2},2+\frac{(k+1)\gamma^2}{2})$, and replace the the weight $(W_1, 2, 2)$ triangle in ~\eqref{eqn-associate-1-a} with a weight $(W_1, W_2-U, W_2-U)$ quantum triangle and the weight $W_2-2$ quantum disk with a weight $U$ quantum disk. Then~\eqref{eqn-qt-change-weight-2} follows by precisely the same argument and our assumption. This finishes the induction, so Step 1 is complete.

\begin{figure}[ht]
	\centering
	\begin{tabular}{cc} 
		\includegraphics[width=0.4\textwidth]{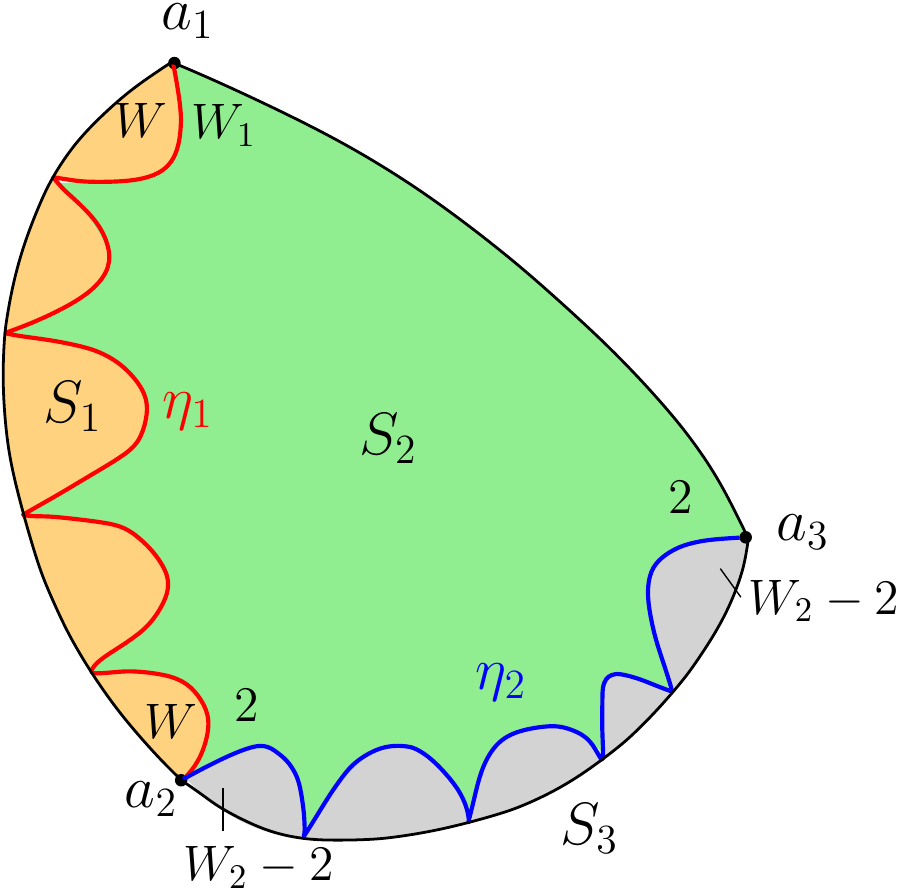}
		& 
		\includegraphics[width=0.45\textwidth]{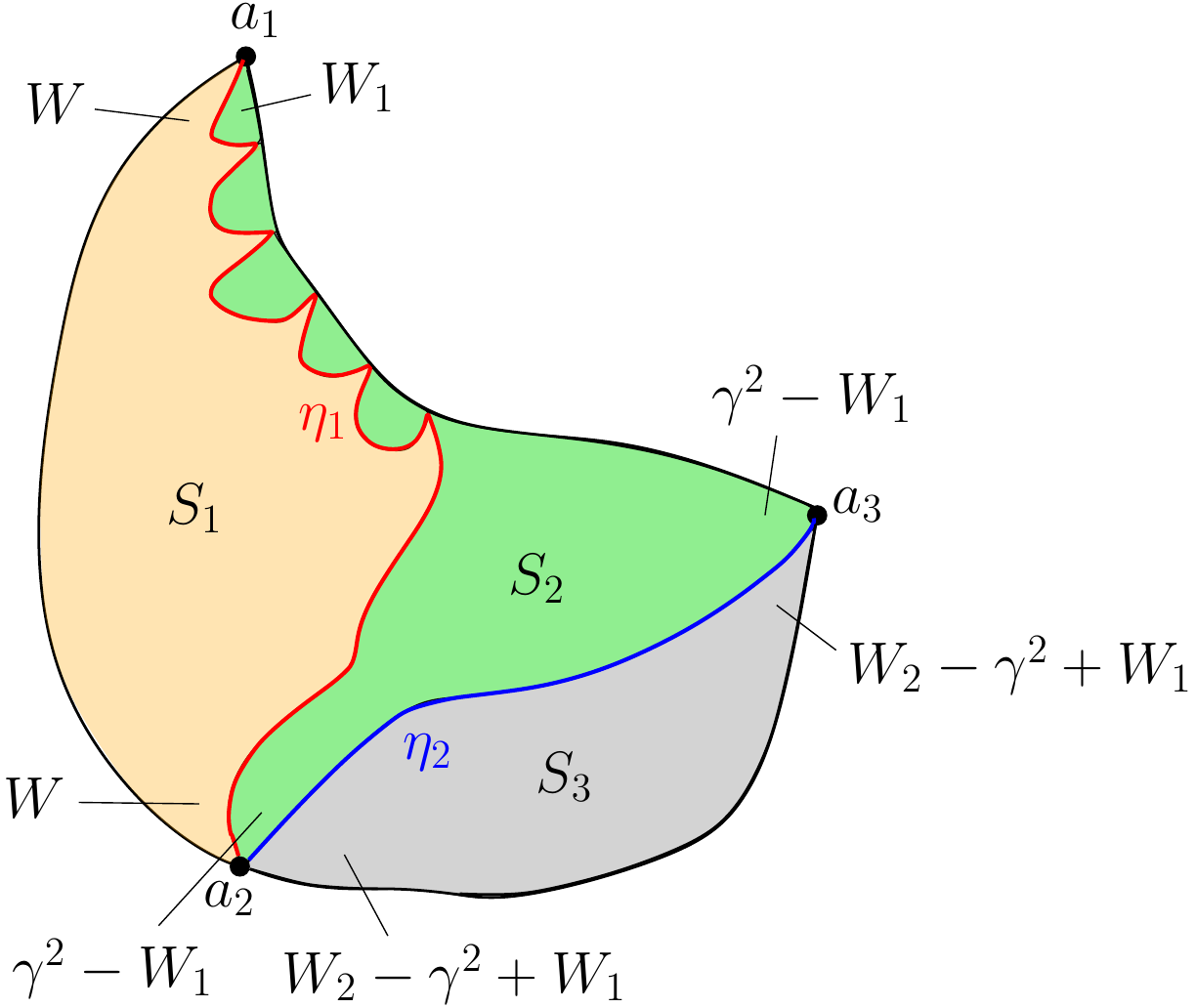}
	\end{tabular}
	\caption{\textbf{Left:} Step 1 of proof of Proposition \ref{prop-thick-welding}. We consider the simultaneous welding of the weight $W$ quantum disk (surface $S_1$), weight $(W_1, 2,2)$ quantum triangle (surface $S_2$) and weight $W_2-2$ quantum disk (surface $S_3$). If we choose to weld $S_1$ with $S_2$ first then $S_3$, we get a  quantum triangle of weight $(W+W_1,W+W_2,W_2)$ decorated with independent curves $(\eta_1,\eta_2)$. On the other hand, if we weld $S_2$ with $S_3$ first, we obtain a quantum triangle of weight $(W_1,W_2,W_2)$ decorated by curve $\eta_2$. Thus  we conclude that by further welding a weight $W$ quantum disk to the left and forgetting about $\eta_2$ we get a quantum triangle of weight $(W+W_1,W+W_2,W_2)$ decorated by an independent curve $\eta_1$.  \textbf{Right:} Step 1 for Proposition \ref{prop-associate-c}. The proof is almost identical and only the weights has been changed.    }\label{fig-associate-a}
\end{figure}

\emph{Step 2: $W>0$ and $\max\{W_1,W_2\}\ge 2$.} By symmetry we may assume $W_2\ge 2$. Suppose \eqref{eqn-qt-change-weight-2} holds for $W\in [\frac{k\gamma^2}{2},\frac{(k+1)\gamma^2}{2})$, $W_1>\frac{\gamma^2}{2}$ and $W_2\ge 2$ where again $k\ge 0$ is an integer. Note that the case $k=0$ follows directly from Step 1. Now if $W\in [\frac{(k+1)\gamma^2}{2},\frac{(k+2)\gamma^2}{2})$, again we pick some $U\in(0,\frac{\gamma^2}{2})$ such that $W-U\in [\frac{k\gamma^2}{2},\frac{(k+1)\gamma^2}{2})$. We start with a weight $(W_1, W_2, W_3)$ quantum triangle. We first glue a weight $U$ quantum disk on its left boundary, inducing an interface $\eta_2$, and then a weight {$W-U$} quantum disk to the left. That is, we are working with the measure
\begin{equation}\label{eqn-associate-5-a}
	\int_0^\infty\int_0^\infty \Wd(\Md_2(W-U;\ell),\Md_2(U;\ell,s), \QT(W_1, W_2,  W_3;s))dsd\ell.
\end{equation}
{See Figure~\ref{fig-associate-e} (left). }Again by Step 1, we can now integrate over $s$ first and weld the weight $U$ quantum disk with the quantum triangle, with the law of the curve-decorated surface being $\QT(W_1+U, W_2+U,  W_3;\ell)\otimes \mathsf{m}(U;W_1,W_2,W_3)$. Then by our induction hypothesis, integrating over $\ell$ once more and welding in the weight $(W-U)$ quantum disk, we obtain a quantum triangle of weight $(W+W_1,W+W_2,W_3)$ decorated by independent curves $(\eta_1,\eta_2)$. On the other hand, if we weld the two disks first, by Theorem \ref{thm-disk-2}, we get
\begin{equation}\label{eqn-associate-6-a}
	\int_0^\infty \Wd\bigg(\big(\Md_2(W;\ell) \otimes \SLE_\kappa(W-U-2;U-2)\big),\QT(W_1, W_2, W_3;\ell)  \bigg)d\ell.
\end{equation}
Thus if we forget about the curve $\eta_1$ and treat $\eta_2$ as the interface, we obtain \eqref{eqn-qt-change-weight-2}. This finishes the induction step and draws the conclusion.

\begin{figure}[ht]
	\centering
	\begin{tabular}{cc} 
		\includegraphics[width=0.36\textwidth]{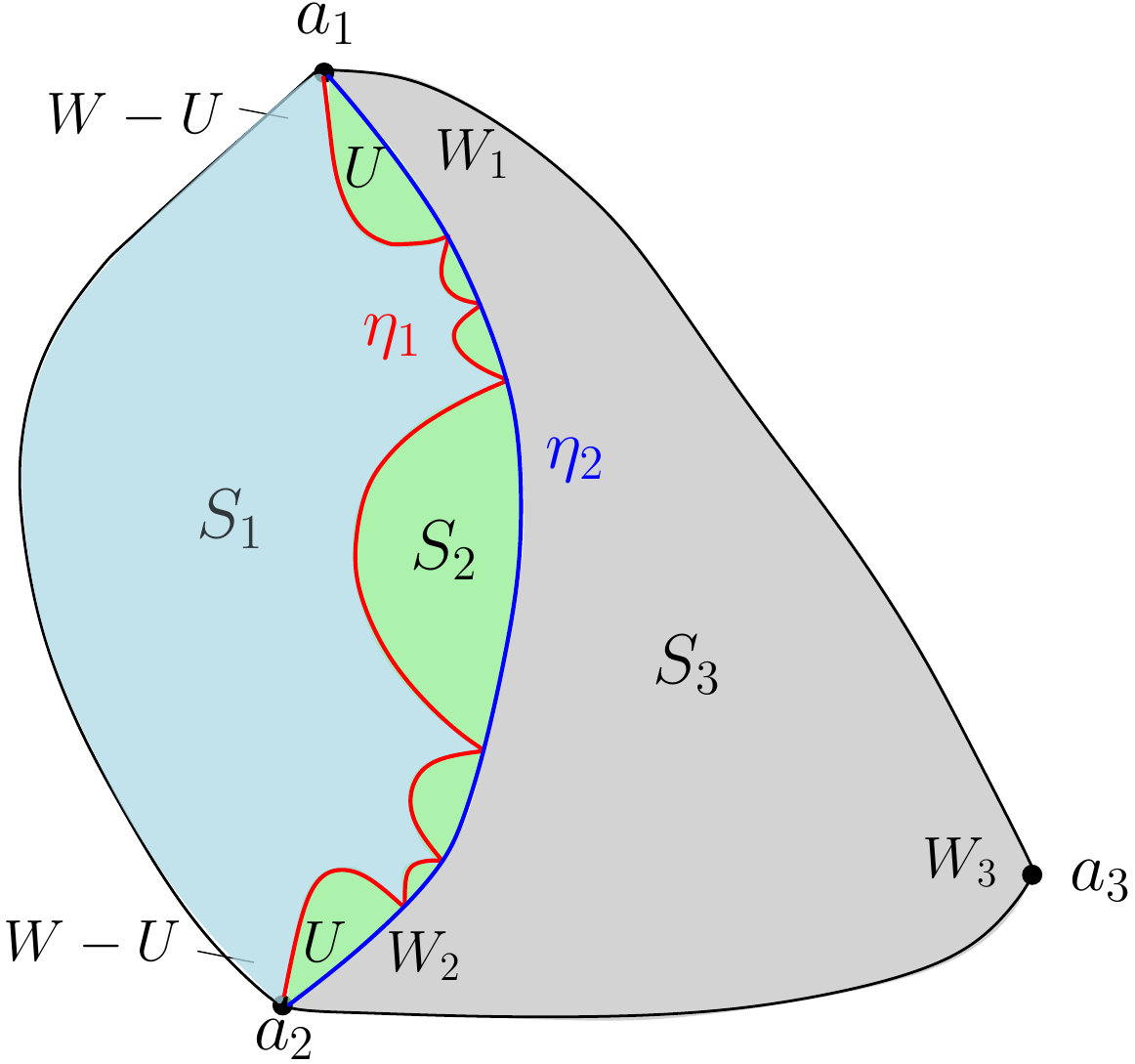}
		& 
		\includegraphics[width=0.4\textwidth]{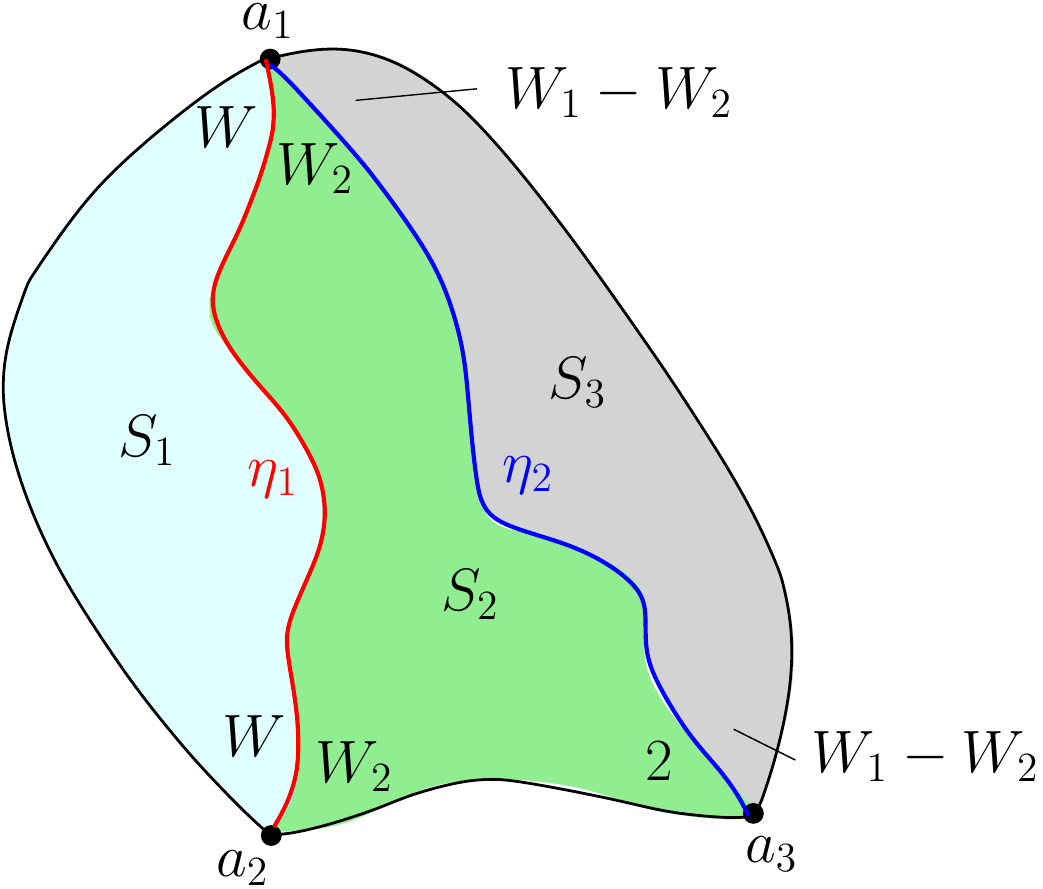}
	\end{tabular}
	\caption{\textbf{Left:} Step 2 of proof of Proposition \ref{prop-thick-welding}. Applying Step 1 and our induction hypothesis we can first weld $S_3$ to the right of $S_2$ and $S_1$ to the left to get a quantum triangle of weight $(W+W_1,W+W_2,W_3)$. If we forget about $\eta_1$ then we get the left hand side of \eqref{eqn-qt-change-weight-2}. On the other hand by Theorem \ref{thm-disk-2} we can first weld $S_1$ and $S_2$ and the picture becomes the right hand side of \eqref{eqn-qt-change-weight-2}.  \textbf{Right:} The same commutation argument in Step 3.}\label{fig-associate-e}
\end{figure}

\emph{Step 3: The general $W>0,W_1,W_2>\frac{\gamma^2}{2}$ case.} By symmetry and Step 2, we may assume $\frac{\gamma^2}{2}< W_2<W_1\le 2$. We consider the setting of the right panel of Figure \ref{fig-associate-e}. Again by comparing the procedure of first welding $S_2$ with $S_1$ (by \eqref{eqn-qt-disk}) and then $S_3$ (by Step 2; and we obtain a quantum triangle with weight $(W+W_1, W+W_2, W_1-W_2+2)$) and first welding $S_2$ with $S_3$ (by Step 3  and we get a quantum triangle of weight $(W_1, W_2, W_1-W_2+2)$) and then $S_1$, we obtain \eqref{eqn-qt-change-weight-2} with $W_3 = W_1-W_2+2$. Thus we conclude the proof by Proposition \ref{prop-change-weight}.
\end{proof}

\subsection{The $W_1 + W_2 = \gamma^2$ regime via thick-thin duality}\label{sec-w-gamma2-w}
% In this section we prove Proposition~\ref{prop-associate-c} with $W_1 + W_2 = \gamma^2$, $W_2 = 2$ and $W > \frac{\gamma^2}2$. % {[The notation in this section is unfortunately different from the main proposition -- if it's not too much work maybe we can align the notations, e.g. replace $W'$ with $W$ and $W$ with some other variable like $U$.]}

Let $W \in (0, \frac{\gamma^2}2)$.
In this section, we establish the conformal welding of a quantum triangle of weights $W, \gamma^2-W, 2$ with a thick quantum disk, via Theorem \ref{prop-w-gamma2-w}. The key observation is that, using the thick-thin duality, the concatenation point on the quantum triangle has weight $2+\gamma^2$ in the global field and hence the $\beta$-value for insertion becomes $\beta = \gamma+\frac{2-(2+\gamma^2)}{\gamma} = 0$.

\begin{theorem}\label{prop-w-gamma2-w}
Fix $W>\frac{\gamma^2}{2}$, $W_1\in (0,\frac{\gamma^2}{2})$ and $W_2=\gamma^2-W_1$. Embed a sample from $\QT(W+W_1, W+W_2, 2)$ as $(\bbH, \phi, \infty, 0, 1)$, where the points $(\infty,0,1)$ corresponds to the weights $(W+W_1,W+W_2,2)$. %  {[specify the embedding]}; see Figure~\ref{fig-w-gamma2-w}. 
Then there exists some constant $c = c_{W, W_1}\in (0,\infty)$ and some measure $\mathsf{m}(W;W_1,W_2,2)$ on random curves in $\mathbb{H}$ from $0$ to $\infty$ avoiding 1, such that
\begin{equation}\label{eqn-w-gamma2-w}
	\begin{split}
		\QT(W+W_1, W+W_2, 2)\otimes \mathsf{m}(W;W_1,W_2,2) = c\int_0^\infty \Wd(\Md_2(W;\ell),\QT(W_1,W_2,2;\ell))d\ell.
	\end{split}
\end{equation} 
\end{theorem}

\begin{figure}[ht]
\begin{tabular}{ccc} 
	\includegraphics[width=0.35\textwidth]{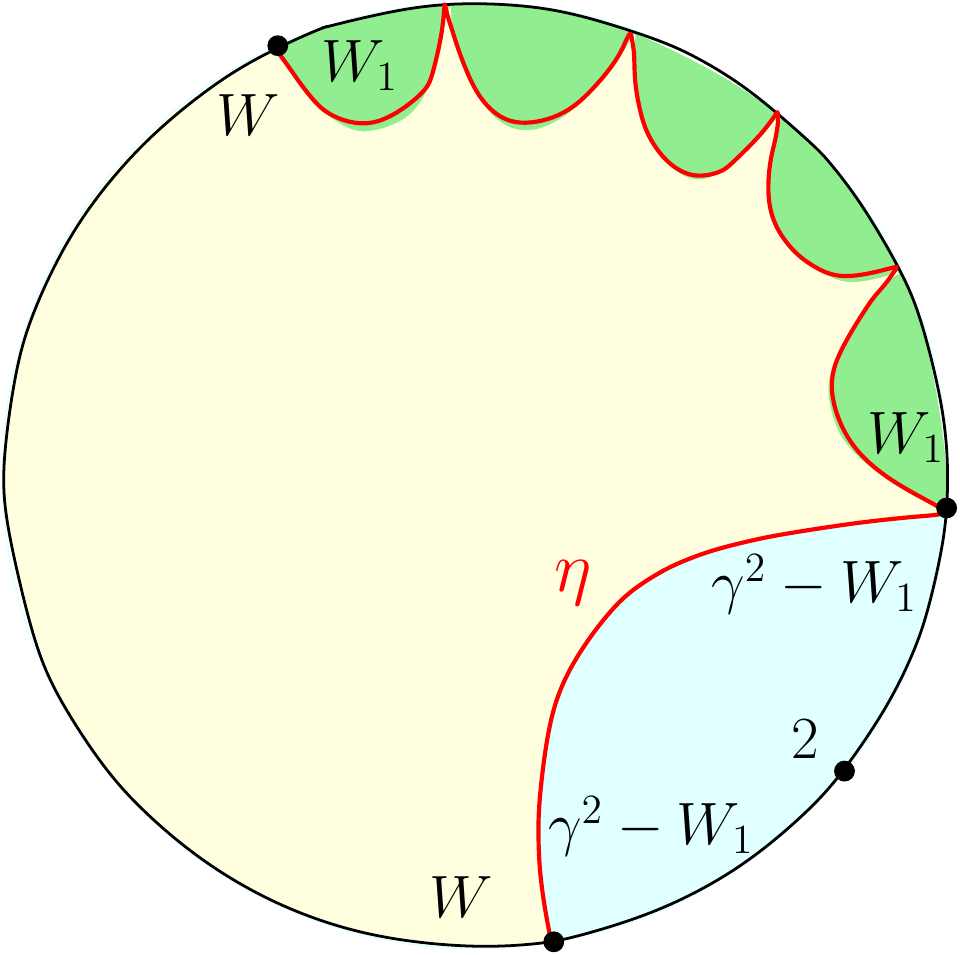}
	& \qquad &
	\includegraphics[width=0.5\textwidth]{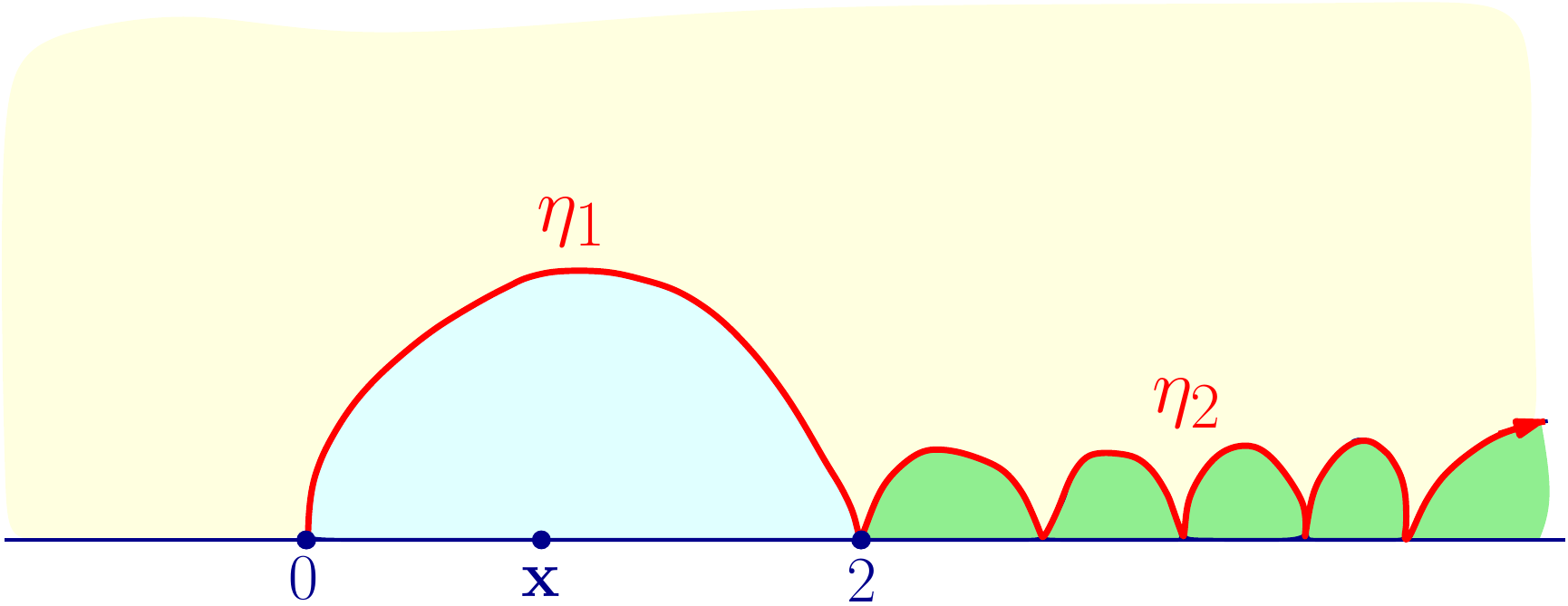}
\end{tabular}
\caption{\textbf{Left:} Setup of Theorem \ref{prop-w-gamma2-w}, where we are proving that cutting a triangle from $\QT(W+W_1, W+\gamma^2-W_1, 2)$ with some independent curve $\eta$ from $\mathsf{m}(W,W_1,\gamma^2-W_1, 2)$ yields the welding of an independent weight $W$ disk and a thin triangle from $\QT(W_1, \gamma^2-W_1, 2)$. \textbf{Right:} Conclusion of Lemma \ref{lm-w-gamma2-w}, embedded as $(\mathbb{H}, \phi,\eta,\infty,0,2)$. We sample a point $\mathbf{x}$ from the quantum length measure on $[0,2]$ and use the scaling $f_\mathbf{x}(z) = \frac{z}{\mathbf{x}}$ to put the added point at 1.  }\label{fig-w-gamma2-w}
\end{figure}

By definition, the quantum surface on the right hand side of \eqref{eqn-w-gamma2-w} consists of three parts: a weight $W$ two-pointed quantum disk, a weight $W_1$ thin quantum  disk, and a three-pointed quantum disk from $\Md_{2,\bullet}(W_2)$. These can be glued together by Proposition \ref{prop-thick-welding}. Parallel to our definition of thin quantum triangles, let $\tilde{\mathcal{M}}_2(W_1)$ be the law of the quantum surface obtained by concatenating a sample from $\Md_2(W_1)\times\Md_2(W_2)$ with $W_2=\gamma^2-W_1$. Then we have the disintegration {on the left boundary length}
\begin{equation}\label{eqn-w-gamma2-w-2}
\tilde{\mathcal{M}}_2(W_1;\ell) = \int_0^r \Md_2(W_1;r)\times\Md_2(W_2;\ell -r)dr.
\end{equation}

\begin{lemma}\label{lm-w-gamma2-w}
In the setting of Theorem \ref{prop-w-gamma2-w}, there exists some constant $c = c_{W, W_1}\in (0,\infty)$ such that 
\begin{equation}\label{eqn-w-gamma2-w-lm}
	\QT(W+W_1, W+\gamma^2-W_1, \gamma^2+2)  \otimes \tilde{\mathsf{m}}(W,W_1)=c\int_0^\infty \Wd(\Md_2(W;\ell), \tilde{\mathcal{M}}_2(W_1;\ell))d\ell
\end{equation}
where  $\tilde{\mathsf{m}}(W,W_1)$ is some law on pairs of the curves describing the two interfaces. % (which are independent of the global field).  {[can drop the parentheses since it's implied by the notation. Also ``global field'' isn't a term we've defined]}
\end{lemma}

\begin{proof}
We start with a triply marked quantum disk  sampled from $\Md_{2,\bullet}(W)$ and recall that $\Md_{2,\bullet}(W)$ is a constant multiple of $\QT(W,W,2)$ (see Remark \ref{remark-qt-3-disk}). By applying Proposition \ref{prop-thick-welding} twice, we can simultaneously glue quantum disks with weight $W_1$ and $\gamma^2-W_1$ to the marked boundary of the $\Md_{2,\bullet}(W)$ quantum disk, with interface having law $\tilde{\mathsf{m}}(W,W_1)
$ and being independent of the surface. That is, if we write $\Md_{2,\bullet}(W;\ell,r)$ for disintegration of the measure $\Md_{2,\bullet}(W)$ over the length of the two boundary arcs containing the third marked point, then
\begin{equation}\label{eqn-w-gamma2-w-1}
	\begin{split}
		&\QT(W+W_1, W+\gamma^2-W_1, \gamma^2+2)  \otimes \tilde{\mathsf{m}}(W,W_1)\\ &= c\int_0^\infty\int_0^\infty \Wd(\Md_{2,\bullet}(W;r,\ell), \Md_2(W_1;r), \Md_2(\gamma^2-W_1;\ell))d\ell dr\\
		&= c\int_0^\infty\int_0^\ell \Wd(\Md_{2,\bullet}(W;r,\ell-r), \Md_2(W_1;r), \Md_2(\gamma^2-W_1;\ell-r))drd\ell. 
	\end{split}
\end{equation}

Now we study the right hand side of \eqref{eqn-w-gamma2-w-1}. By Definition~\ref{three-pointed-disk}, forgetting the marked point of a sample from $\Md_{2,\bullet} (W; r, \ell - r)$ gives a sample from $\Md_2(W; \ell)$. Combining this with~\eqref{eqn-w-gamma2-w-2}, we conclude that the right hand side of~\eqref{eqn-w-gamma2-w-1} equals that of~\eqref{eqn-w-gamma2-w-lm}.
\end{proof}
% 	Consider the process of welding a weight $W'$ quantum disk with right boundary length $\ell$ and a sample from $\tilde{\mathcal{M}}_2(W;\ell)$ together. Using the disintegration \eqref{eqn-w-gamma2-w-2}, this can be achieved by first sampling $r$ from $\mathrm{Leb}_{[0,\ell]}$ and mark the point with distance $r$ to the upper endpoint of the disk, and then simultaneously welding a pair of samples from $\Md_2(W;r)\times\Md_2(\gamma^2-W;\ell -r)$ to the right boundary arc according to quantum length. Moreover, recall that from Definition \ref{three-pointed-disk} and Proposition \ref{prop-m2dot}, once given $r$, after adding a third point onto the right boundary of weight $W'$ quantum disk,  the law of the surface we get is precisely $\Md_{2,\bullet}(W';r,\ell-r)$. Therefore we conclude that  \eqref{eqn-w-gamma2-w-1} is the same as \eqref{eqn-w-gamma2-w-lm}.

% 	 {[Do you think the following alternative is sufficient?] By Definition~\ref{three-pointed-disk}, forgetting the marked point of a sample from $\Md_{2,\bullet} (W'; r, \ell - r)$ gives a sample from $\Md_2(W'; \ell)$. Combining this with~\eqref{eqn-w-gamma2-w-2}, we conclude that the right hand side of~\eqref{eqn-w-gamma2-w-1} equals that of~\eqref{eqn-w-gamma2-w-lm}.}

\begin{proof}[Proof of Theorem \ref{prop-w-gamma2-w}]
We begin with the setting of Lemma \ref{lm-w-gamma2-w}. Embed the quantum triangle from $\QT(W+W_1, W+\gamma^2-W_1, \gamma^2+2)$  as $(\mathbb{H}, \phi,  \infty,0, 2)$. {The law of $\phi$ is} $c\LF_\mathbb{H}^{(\beta_1, \infty),(\beta_2,0)}$ where $\beta_1 = \gamma+\frac{2-W-W_1}{\gamma}$, $\beta_2 = \gamma+\frac{2-W-\gamma^2+W_1}{\gamma}$ and $c = c_{W, W_1}$ is some constant. We emphasize that there is  no $\beta$-insertion at the marked point 2 since $\beta_3 = \gamma+\frac{2-(2+\gamma^2)}{\gamma} = 0$. The interface $\eta_1$ between the weight $W$ quantum disk and the weight $\gamma^2-W_1$ quantum disk is embedded as a simple curve from 0 to 2, and the interface $\eta_2$ between the weight $W$ disk and the weight $W_1$ thin disk is drawn as a boundary hitting curve from 2 to $\infty$. See Figure \ref{fig-w-gamma2-w} for an illustration of the setup. 

We add a fourth point to the field and rescale via the following procedure. First weight the law of $\phi$ by $\nu_\phi([0,2])$ and sample a point $\mathbf{x}$ on $(0,2)$ from the probability measure proportional to the quantum length measure $\nu_\phi|_{[0,2]}$, and then rescale the field and the curves via $f_\mathbf{x}(z) = \frac{z}{\mathbf{x}}$. Let 
\begin{equation}\label{eqn-w-gamma2-w-4}
	\tilde{\phi} = f_x\bullet_\gamma\phi =  \phi\circ f_{\mathbf{x}}^{-1} + Q\log|(f_{\mathbf{x}}^{-1})'| = \phi(\mathbf{x}\cdot)+Q\log \mathbf{x}.
\end{equation}
Then $(\mathbb{H}, \tilde{\phi}, \infty, 0, 1)/{\sim_\gamma} = (\bbH,\phi,\infty,0,\mathbf{x})/{\sim_\gamma}$. Let $\eta$ be the concatenation of the curves $\mathbf{x}^{-1}\eta_1(\cdot)$ and $\mathbf{x}^{-1}\eta_2(\cdot)$, going from 0 to $\infty$. Note that this point $\mathbf{x}$ is added to the weight $\gamma^2-W_1$ disk according to quantum length measure, and again by Proposition \ref{prop-m2dot} this surface has the same law as $c\QT(\gamma^2-W_1,\gamma^2-W_1,2)$. Therefore, by applying the definition of quantum triangles, the law of the curve-decorated surface $(\mathbb{H}, \tilde{\phi}, \eta, \infty, 0, 1)$ is precisely the same as the right hand side of \eqref{eqn-w-gamma2-w}. We are going to prove that $\tilde{\phi}$ has the same law as $\LF_\mathbb{H}^{(\beta_1, \infty),(\beta_2,0), (\gamma,1)}$ and is independent of $\mathbf{x}$, which further implies that $\eta$ is independent of $\tilde{\phi}$ (since $\eta$ is defined via $\eta_1,\eta_2,\mathbf{x}$, which are all independent of $\tilde{\phi}$). This shows that the law of  $(\mathbb{H}, \tilde{\phi}, \eta, \infty, 0, 1)$ is the same as the left hand side of \eqref{eqn-w-gamma2-w}, which concludes the proof. 

Now suppose that $F$ is a bounded, non-negative and continuous function on $H^{-1}(\mathbb{H})$, and $g$ is a compactly  supported non-negative function on $[0,2]$. Let $\phi_\e(x)$ be the circle average of $\phi$ around the semicircle $\{z:|z-x| = \e\}$. By the change of coordinates \eqref{eqn-w-gamma2-w-4}, $(f_x\bullet_\gamma\phi)_{\frac{\e}{x}}(1) = \phi_\e(x)+Q\log x$. Let $Q(d\phi,dx) = \nu_\phi(dx)\LF_\bbH^{(\beta_1,\infty),(\beta_2,0)}(d\phi)$ be the infinite measure on $H^{-1}(\mathbb{H})\times [0,2]$. Then  
\begin{equation}\label{eqn-w-gamma2-w-5}
	\begin{split}
		Q[F(\tilde{\phi})g(\mathbf{x})] &= \lim_{\epsilon\to 0}\int \int_0^2 F(f_x\bullet_\gamma\phi)g(x)\epsilon^{\frac{\gamma^2}{4}}e^{\frac{\gamma}{2}\phi_\epsilon(x)}dx\LF_\bbH^{(\beta_1,\infty),(\beta_2,0)}(d\phi)\\
		&= \lim_{\epsilon\to 0}\int\int_0^2 F(f_x\bullet_\gamma\phi)g(x)(\frac{\epsilon}{x})^{\frac{\gamma^2}{4}}x^{\frac{\gamma^2}{4}}e^{\frac{\gamma}{2}(f_x\bullet_\gamma\phi)_{\frac{\e}{x}}(1)-\frac{\gamma}{2}Q\log x}dx\LF_\bbH^{(\beta_1,\infty),(\beta_2,0)}(d\phi)\\
		&= \lim_{\epsilon\to 0}\int\int_0^2 F(f_x\bullet_\gamma\phi)g(x)\epsilon^{\frac{\gamma^2}{4}}e^{\frac{\gamma}{2}(f_x\bullet_\gamma\phi)_\e(1)}x^{\frac{\gamma^2-2\gamma Q}{4}}dx\LF_\bbH^{(\beta_1,\infty),(\beta_2,0)}(d\phi)\\
		&= \lim_{\epsilon\to 0}\int_0^2\int F(\tilde{\phi})g(x)\epsilon^{\frac{\gamma^2}{4}}e^{\frac{\gamma}{2}\tilde{\phi}_\e(1)}x^{\frac{\gamma^2-2\gamma Q}{4}}[(f_x)_*\LF_\bbH^{(\beta_1,\infty),(\beta_2,0)}](d\tilde{\phi})dx.
	\end{split}
\end{equation}
Here we have used the fact that $\lim_{\epsilon\to 0}\int_0^2 \e^{\frac{\gamma^2}{4}}e^{\frac{\gamma}{2}h_\e(x)}g(x)dx = \int_0^2 g(x)\nu_h(dx)$ in $L^1$ with respect to $P_\bbH$ (see e.g. \cite[Theorem 1.1]{Ber17}).  Meanwhile, by Lemma \ref{lmm-lf-insertion-infty} and Lemma \ref{lmm-lcft-H-conf},  we have
\begin{equation}\label{eqn-w-gamma2-w-6}
	\begin{split}
		(f_x)_*\LF_\bbH^{(\beta_1,\infty),(\beta_2,0)} &= \lim_{r\to+\infty} r^{2\Delta_{\beta_1}}	(f_x)_*\LF_\bbH^{(\beta_1,r),(\beta_2,0)} \\
		&= \lim_{r\to+\infty} r^{2\Delta_{\beta_1}} x^{-\Delta_{\beta_1}-\Delta_{\beta_2}}\LF_\bbH^{(\beta_1,\frac{r}{x}),(\beta_2,0)} = x^{\Delta_{\beta_1}-\Delta_{\beta_2}}\LF_\bbH^{(\beta_1,\infty),(\beta_2,0)}.
	\end{split}
\end{equation}
Since $\frac{\gamma^2-2\gamma Q}{4} = -1$, plugging \eqref{eqn-w-gamma2-w-6} into \eqref{eqn-w-gamma2-w-5} and using Lemma \ref{lmm-lf-insertion-bdry}, we get
\begin{equation}\label{eqn-tildephi-indep-x}
	\begin{split}
		Q[F(\tilde{\phi})g(\mathbf{x})] &= \lim_{\epsilon\to 0}\int_0^2\bigg(\int F(\tilde{\phi})\epsilon^{\frac{\gamma^2}{4}}e^{\frac{\gamma}{2}\tilde{\phi}_\e(1)}\LF_\bbH^{(\beta_1,\infty),(\beta_2,0)}(d\tilde{\phi})\bigg)g(x)x^{-1+\Delta_{\beta_1}-\Delta_{\beta_2}}dx\\
		&= \bigg(\int F(\tilde{\phi}) \LF_\bbH^{(\beta_1,\infty), (\beta_2, 0), (\gamma, 1)}(d\tilde{\phi}) \bigg)\bigg(\int_0^2g(x)x^{-1+\Delta_{\beta_1}-\Delta_{\beta_2}}dx\bigg) .
	\end{split}
\end{equation}
Therefore the law of $\tilde{\phi}$ is $\LF_\bbH^{(\beta_1,\infty), (\beta_2, 0), (\gamma, 1)}$ and is independent of the choice of $\mathbf{x}$, which concludes our proof.

\end{proof}

\subsection{Extension to general weights}\label{sec-general-thin}

In this section, we finish the proof of Proposition \ref{prop-associate-c} by repeatedly applying the change of weight argument in Proposition \ref{prop-change-weight} along with Theorem \ref{prop-w-gamma2-w}. Since most of the proof is identical to that of Proposition \ref{prop-thick-welding}, we will only list the welding pictures and explain by which theorem we can weld surfaces together. 

\begin{proof}[Proof of Proposition \ref{prop-associate-c}]
{By Proposition~\ref{prop-thick-welding}}, we only need to focus on the case when a least one of $W_1$ and $W_2$ is in  $(0,\frac{\gamma^2}{2})$.

\emph{Step 1: $W_1\in (0,\frac{\gamma^2}{2})$, {$W> \frac{\gamma^2}{2}$},  $W_2\ge \gamma^2-W_1$.} Consider the setting on the right panel of Figure \ref{fig-associate-a}. We start from the weight $(W_1, \gamma^2-W_1, \gamma^2-W_1)$ quantum triangle $S_2$ and weld an independent quantum disk $S_1$ from $\Md_2(W)$ to its left boundary and an independent quantum disk $S_3$ from $\Md_2(W_2 - (\gamma^2-W_1))$ to its bottom arc. If we first weld $S_1$ to the left of $S_2$ (by Theorem \ref{prop-w-gamma2-w}) and then $S_3$ to the bottom (by Proposition \ref{prop-thick-welding}) and forget about $\eta_2$,  the resulting law of the curve-decorated surface is the left hand side of \eqref{eqn-qt-change-weight-2}. Meanwhile, we can also start by welding $S_2$ and $S_3$ together, which (by Proposition \ref{prop-3-pt-disk}) leads to the right hand side of \eqref{eqn-qt-change-weight-2}. This justifies the claim. 

\emph{Step 2: $W_1\in (0,\frac{\gamma^2}{2}),\ {W\in(0,\infty)\backslash\{\frac{\gamma^2}{2}\}},\ W_2\ge \gamma^2$.} By Step 1 we can assume $W\in (0,\frac{\gamma^2}{2})$. Consider the welding on the left panel of Figure \ref{fig-associate-b}. By Proposition \ref{prop-3-pt-disk}, we can first weld $S_1$ and $S_2$ together. Then we weld the disk $S_3$ from below (if $W_1+W<\frac{\gamma^2}{2}$ then this is covered by Step 1; otherwise this is from Proposition \ref{prop-thick-welding}). However we can also apply Theorem \ref{prop-w-gamma2-w} and Proposition \ref{prop-change-weight} to glue $S_2$ and $S_3$ together first. Comparing the two procedures (and applying Proposition \ref{prop-change-weight}) yields \eqref{eqn-qt-change-weight-2}. By symmetry we may also swap $W_1$ and $W_2$ and \eqref{eqn-qt-change-weight-2} holds for $W_1\ge \gamma^2,\ W>0,\ W_2\in (0,\frac{\gamma^2}{2})$.

\begin{figure}[ht]
	\begin{tabular}{cc} 
		\includegraphics[width=0.5\textwidth]{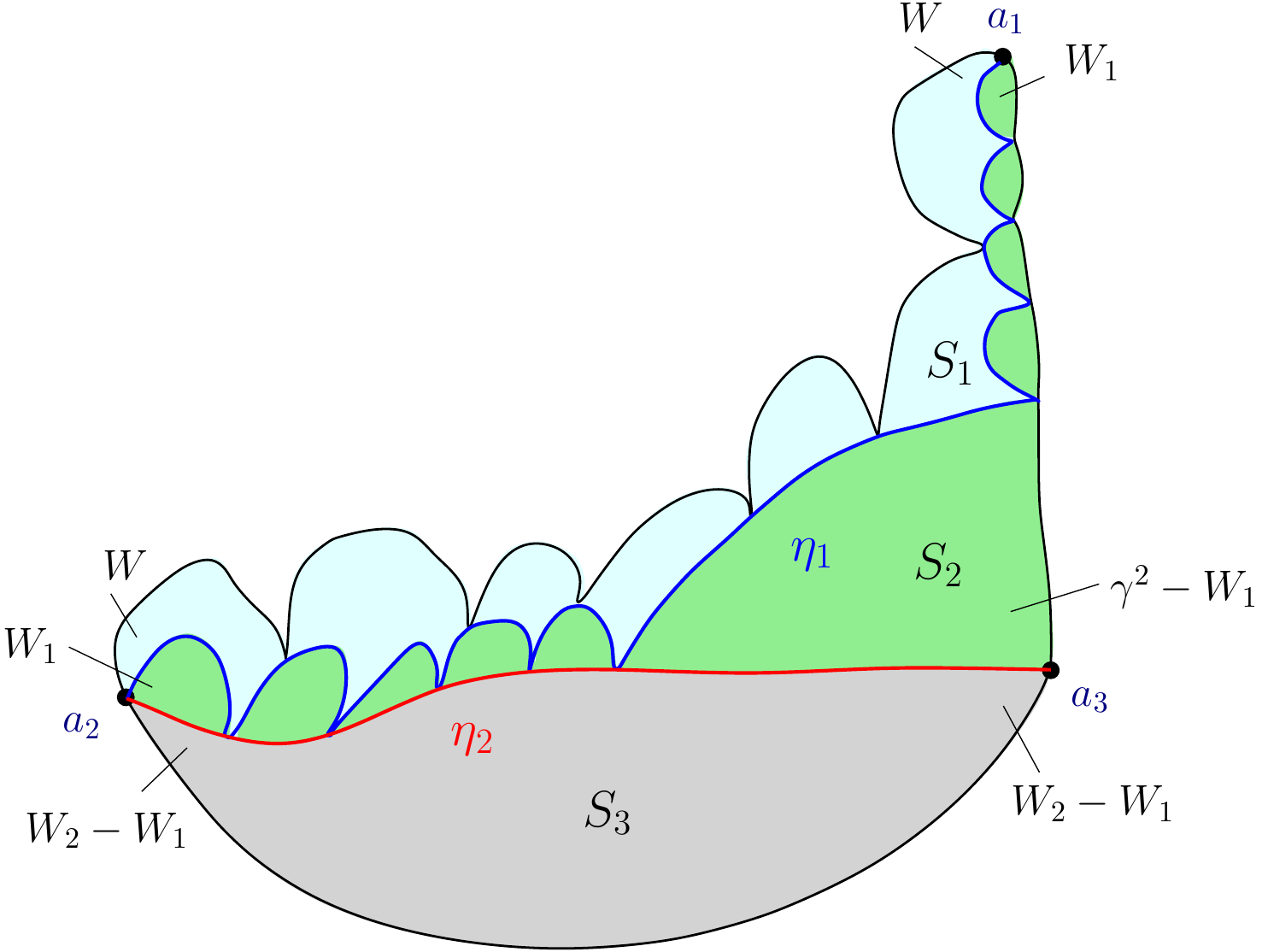}
		& 
		\includegraphics[width=0.35\textwidth]{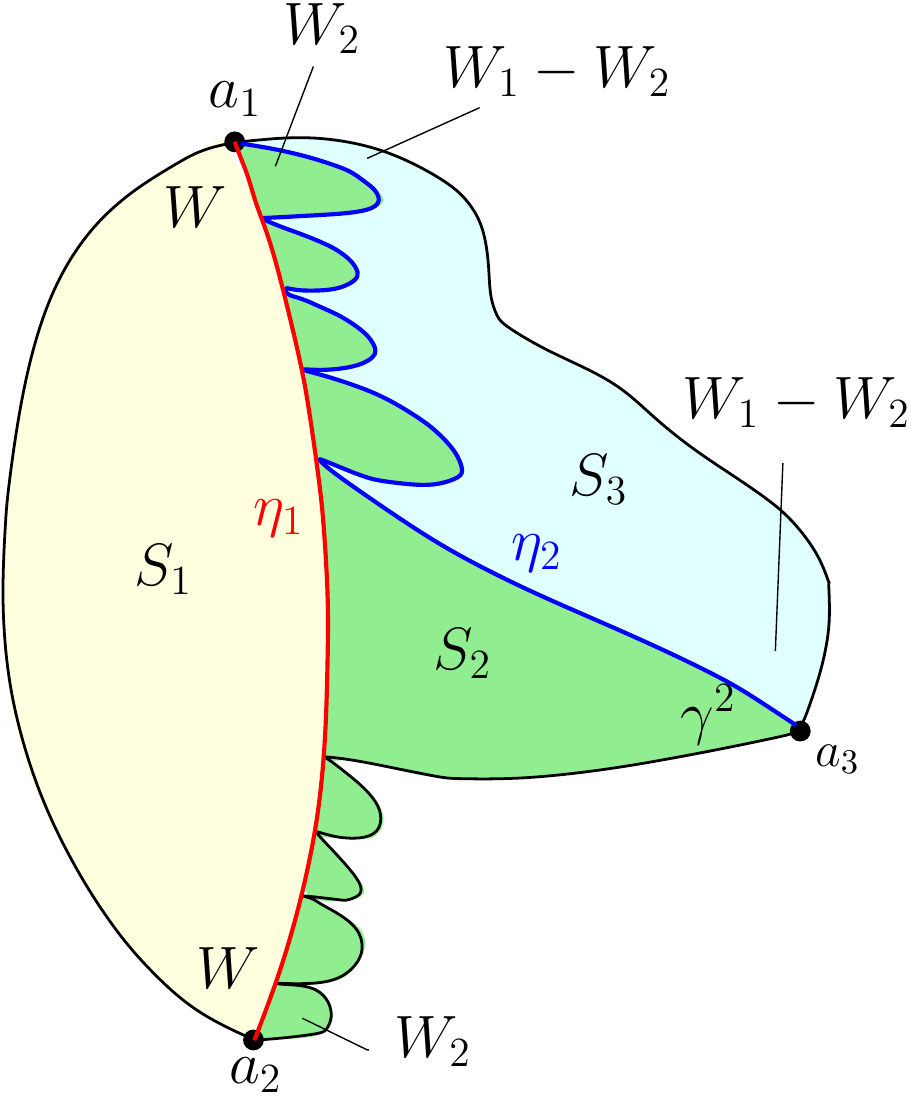}
	\end{tabular}
	\caption{\textbf{Left}: Step 2 of the proof of Proposition \ref{prop-associate-c}. We can weld the three surfaces together by first glue $S_1$ with $S_2$ and then $S_3$ on the bottom. Apply this procedure we get a large triangle of weights $(W+W_1, W+W_2, W_2+\gamma^2-2W_1)$. We can also weld $S_2$ with $S_3$ first and get a triangle of weights $(W_1,W_2,W_2+\gamma^2-2W_1)$. Comparing the two procedures we get \eqref{eqn-qt-change-weight-2} (By Proposition \ref{prop-change-weight}, the weight at vertex $a_3$ can be replaced by any $W_3>0$).   \textbf{Right}: Step 3 of Proposition \ref{prop-associate-c}. We can start by gluing $S_1$ and $S_2$ (Proposition \ref{prop-3-pt-disk}) and then glue $S_3$ to the right (Proposition \ref{prop-thick-welding} and Step 2) to get a triangle of weights $(W+W_1,W+W_2, W_1-W_2+2)$ decorated with independent curves $(\eta_1,\eta_2)$. We can also start with $S_2$ and $S_3$ instead and see that the surface to the right of $\eta_1$ has law $\QT(W_1,W_2, W_1-W_2+2)$. Apply Proposition \ref{prop-change-weight} once more we get the welding equation \eqref{eqn-qt-change-weight-2}.  }\label{fig-associate-b}
\end{figure}

\emph{Step 3. The remaining cases.} % We first deal with the case when $W_1,W_2\le \gamma^2$ and $W>0$  {[Is this case decomposition correct? It looks like we are instead using the range below ($W_1 > W_2$ and $W_2$ thin.)]}. 
{First assume $W\neq\frac{\gamma^2}{2}$.} Without loss of generality suppose $W_2<W_1\le \gamma^2$ and $W_2<\frac{\gamma^2}{2}$ (if $W_1=W_2$ then the claim is straightforward from Proposition \ref{prop-3-pt-disk}). Consider a quantum triangle $S_2$ of weight ($W_2, W_2, \gamma^2$). Again by Proposition \ref{prop-3-pt-disk} we can weld a weight $W$ quantum disk $S_1$ to the left. Then we weld a weight $W_1-W_2$ quantum disk $S_3$ to the right and forget about the interface (if $W+W_2<\frac{\gamma^2}{2}$ we apply Step 2; otherwise we apply Proposition \ref{prop-thick-welding}).   Meanwhile we can apply Step 2 to weld $S_2$ and $S_3$ first. Comparing the two procedures (and change the third weight by Proposition \ref{prop-change-weight}) we obtain \eqref{eqn-qt-change-weight-2}. {Finally if $W=\frac{\gamma^2}{2}$ then we may pick $U\in(0,\frac{\gamma^2}{2})$ and argue as in Step 2 of Proposition \ref{prop-thick-welding} (see the Left panel of Figure~\ref{fig-associate-e} where $S_1$ and $S_2$ are thin quantum disks.)}
\end{proof}

\subsection{The interface law}\label{sec-interface-law}

In this section, we identify the interface law $\mathsf{m}$ in Proposition~\ref{prop-associate-c} as $\wt{\SLE}_\kappa(W-2;W_2-2,W_1-W_2,\alpha)$ with $\alpha  = \frac{W_3+W_2-W_1-2}{4\kappa}(W_3+W_1+2-W_2-\kappa)$ using the SLE curve resampling properties, which completes the proof of {Theorem \ref{thm:M+QTp-rw}} {when none of the weights $W+W_1,W+W_2,W_1,W_2,W_3$ equals $\frac{\gamma^2}2$}. We begin with the direct extension of Theorem \ref{thm:u2v} and work on the case where $W_1\ge W_2>0$, while {the case $W_2>W_1$ is covered via the SLE$_\kappa(\rho_-;\rho_+,\rho_1)$ reversibility in Theorem \ref{thm-sle-reverse}.} {Note that if $W+W_1<\frac{\gamma^2}{2}$ and/or $W+W_2<\frac{\gamma^2}{2}$ then as in Proposition~\ref{prop:ig-flow-descriptions} and the discussion at the beginning of Section~\ref{sec-6-thin-qt}, the $\wt{\SLE}_\kappa(W-2;W_2-2,W_1-W_2,\alpha)$ curve is understood as the concatenation of an $\wt{\SLE}_\kappa(W-2;W_2-2,W_1-W_2,\alpha)$ in the core with independent $\SLE_\kappa(W-2;W_1-2)$ and/or $\SLE_\kappa(W-2;W_2-2)$ in each bead of the weight $W+W_1$ and/or $W+W_2$ thin quantum disk.}

\begin{lemma}\label{lmm-curve-1}
Suppose {$W,W_1>0$ and $W_1,W+W_1\neq\frac{\gamma^2}{2}$}.  Then there exists some constant $c = c_{W,W_1}\in (0,\infty)$ such that
\begin{equation}\label{eqn-lmm-curve-1}
	\QT(W+W_1, W+2, W_1)\otimes \SLE_\kappa(W-2;0,W_1-2) = c\int_0^\infty \Wd(\Md_2(W;\ell),\QT(W_1,2,{W}_1;\ell))d\ell.
\end{equation}
\end{lemma}

%defined in the component of the thick triangle while in the weight $W+W_1$ thin disk part the curve has law $\SLE_\kappa(W-2;W_1-2)$ in each bead.  {[This needs to be more precise]}

\begin{figure}[ht]

\begin{tabular}{ccc} 
	\includegraphics[width=0.4\textwidth]{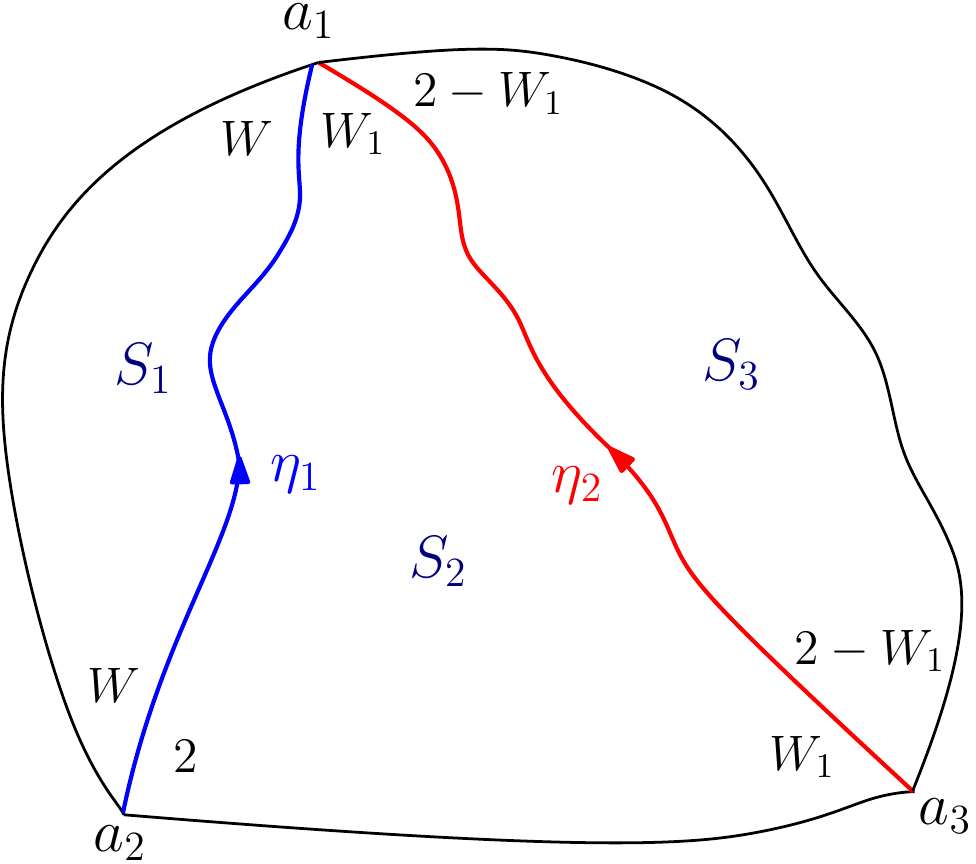}
	& \quad &
	\includegraphics[width=0.4\textwidth]{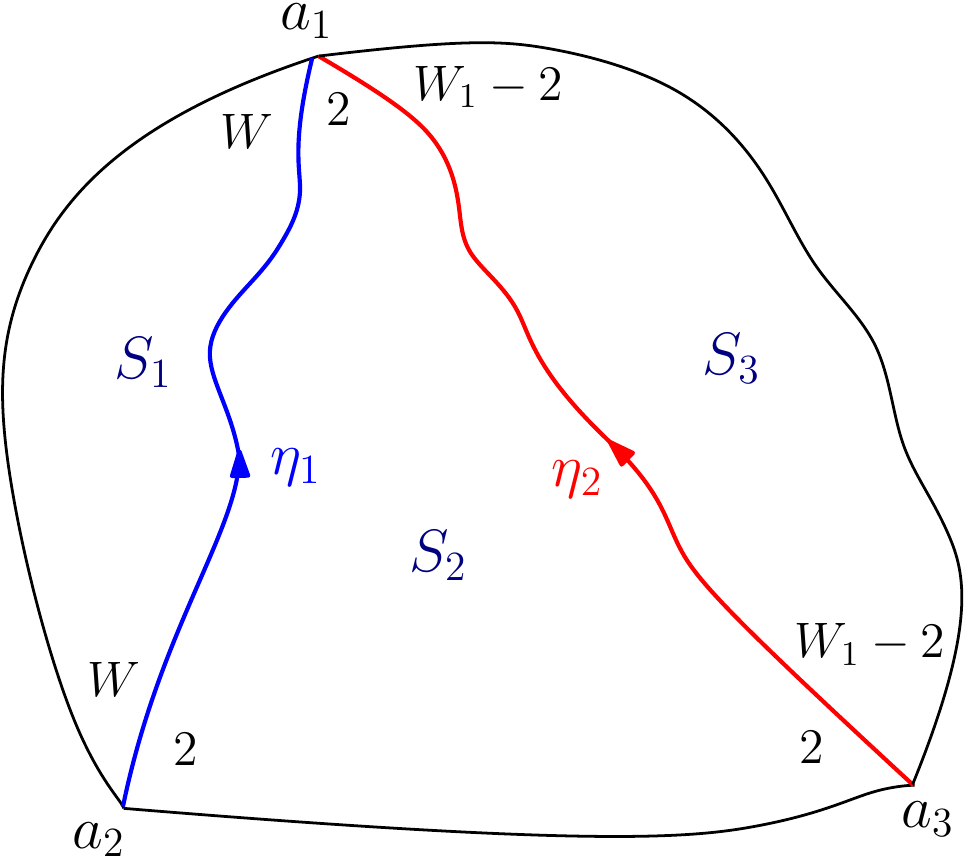}
\end{tabular}
\caption{\textbf{Left}: Suppose $W_1<2$. % {[What if $W_1$ is thin?]} 
	Then we can weld a quantum disk of weight $2-W_1$ to the right side of the quantum  triangle, as in the left panel. Conditioned on the surface $S_1$ and the interface $\eta_1$, Proposition \ref{prop-3-pt-disk} tells us that the interface $\eta_2$ has law $\SLE_\kappa(W_1-2;-W_1)$  from $a_3$ to $a_1$ on the domain to the right of $\eta_1$, while the marginal law of $\eta_1$ is $\SLE_\kappa(W-2)$. This characterizes the law on pairs $(\eta_1,\eta_2)$, while one can verify by 
	Proposition~\ref{prop:ig-flow-descriptions} that the right panel gives a desired coupling. Furthermore by Imaginary geometry the conditional law of $\eta_1$ given $\eta_2$ is $\SLE_\kappa(W-2;0,W_1-2)$, which justifies \eqref{eqn-lmm-curve-1}. \textbf{Right}: For $W_1>2$, similar to the left panel, by Proposition \ref{prop-3-pt-disk} the conditional law of $\eta_1$ given $\eta_2$ is $\SLE_\kappa(W-2)$, while the law of $\eta_2$ given $\eta_1$ is  $\SLE_\kappa(W_1-4)$. Therefore by Proposition \ref{prop:ig-flow-descriptions} the marginal law of $\eta_1$ is $\SLE_\kappa(W-2;0,W_1-2)$. % {[This citation of imaginary geometry isn't great; it assumes the reader has read ig which is uncommon. To discuss with Pu: what are the key statements we need about SLE? To me it looks like 1) the uniqueness statement of Theorem 3.3, and the content of Figure 3? Maybe we can give a citable statement with the content of Figure 3, e.g.\ with something like three equivalent descriptions of a pair of curves: $(\eta_1,\eta_2 \mid \eta_1)$, $(\eta_2, \eta_1\mid \eta_2)$, and $(\eta_1\mid \eta_2, \eta_2 \mid \eta_1)$. Then we wouldn't have to reference an IG diagram]}
}\label{fig-curve-1}
\end{figure}

\begin{proof}
If $W_1 = 2$, then the lemma follows directly from Proposition \ref{prop-3-pt-disk}. For $W_1\neq 2$, see Figure~\ref{fig-curve-1} for an illustration. 
\end{proof}

% \begin{figure}[ht]
% 	\begin{tabular}{ccc} 
	% 		\includegraphics[width=0.4\textwidth]{figures/fig-curve-2.pdf}
	% 		& \quad &
	% 		\includegraphics[width=0.45\textwidth]{figures/fig-ig-2.pdf}
	% 	\end{tabular}
% 	\caption{Suppose $W_1>2$. Then  consider the welding of the three surfaces as in the left panel. Similar to the explanation in Figure \ref{fig-curve-1}, by Proposition \ref{prop-3-pt-disk} the conditional law of $\eta_1$ given $\eta_2$ is $\SLE_\kappa(W-2)$, while the law of $\eta_2$ given $\eta_1$ is  $\SLE_\kappa(W_1-4)$. Therefore by Proposition \ref{prop:ig-flow-descriptions} this determines the joint law $(\eta_1,\eta_2)$, which can be constructed via Imaginary geometry as on the right panel. We can infer that the marginal law of $\eta_1$ is $\SLE_\kappa(W-2;0,W_1-2)$, which gives \eqref{eqn-lmm-curve-1}. }\label{fig-curve-2}
% \end{figure}

% {[Can we just state the proposition as ``Theorem~\ref{thm:M+QTp-rw} holds when $W_1 \geq W_2$ and $\frac{\gamma^2}2 \not \in \{ W + W_1, W+W_2, W_1, W_2, W_3\}$?]}
Now we deal with the case $W_1\ge W_2>0$. Recall the notion of $\widetilde{\SLE}_\kappa(\rho_-;\rho_+,\rho_1;\alpha)$ in \eqref{eqn-sle-CR}.

\begin{figure}[ht]
\begin{tabular}{ccc} 
	\includegraphics[width=0.4\textwidth]{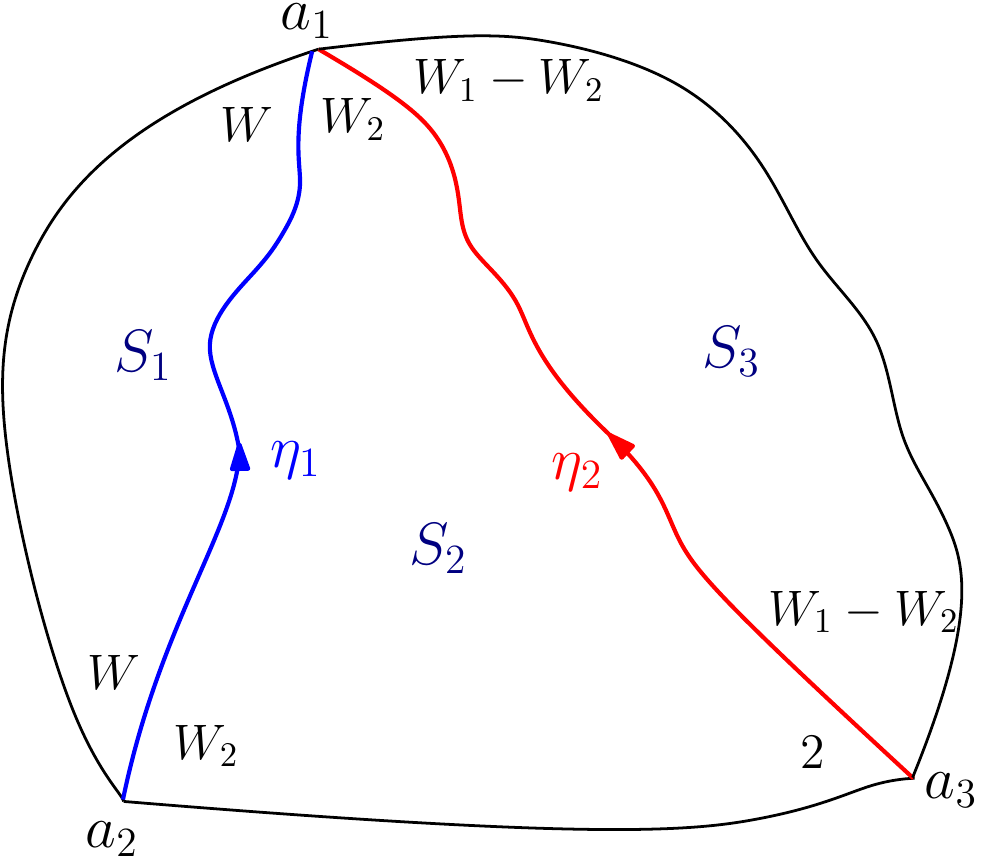}
	& \quad &
	\includegraphics[width=0.4\textwidth]{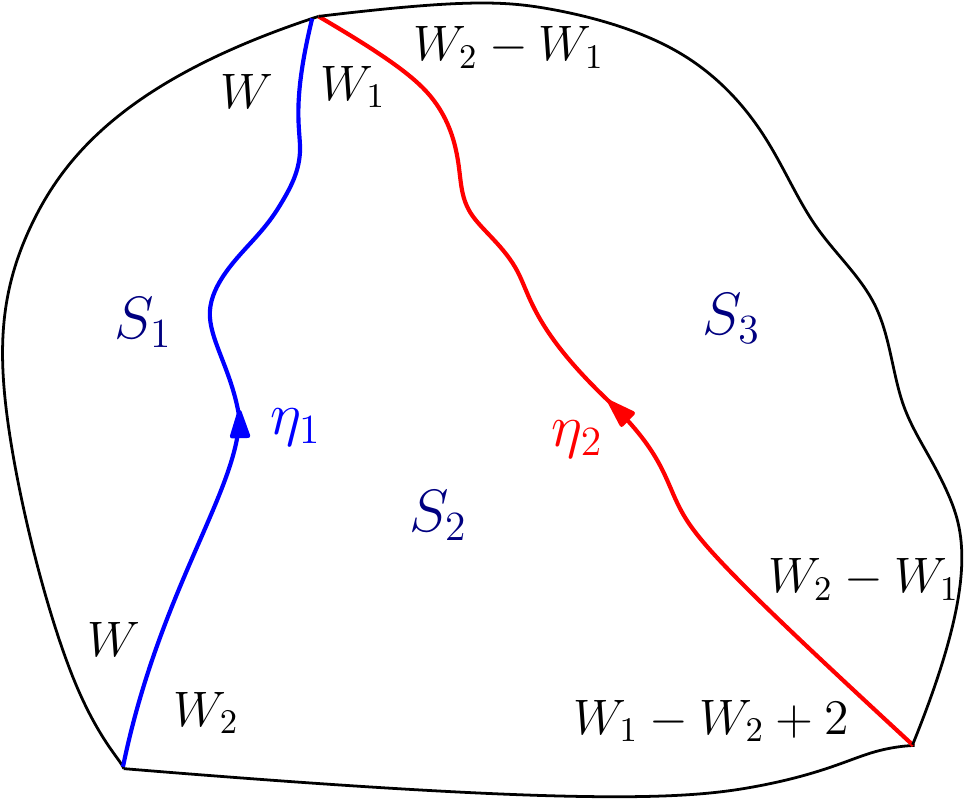}
\end{tabular}
\caption{\textbf{Left:} Suppose $W+W_2\ge\frac{\gamma^2}{2}$. Again consider the welding of the three surfaces as in the left panel. Similar to the explanation in Figure \ref{fig-curve-1}, by Proposition~\ref{prop-3-pt-disk} the conditional law of $\eta_1$ given $\eta_2$ is $\SLE_\kappa(W-2;W_2-2)$, while by Lemma~\ref{lmm-curve-1}, the marginal law of $\eta_2$   is  $\SLE_\kappa(0,W+W_2-2;W_1-W_2-2)$. Therefore by Proposition~\ref{prop:ig-flow-descriptions} we can infer that the marginal law of $\eta_1$ is $\SLE_\kappa(W-2;W_2-2,W_1-W_2)$, which gives \eqref{eq:M+QT2}. \textbf{Right}: Suppose $\max\{W_1,W_2\}\ge 2$ and $|W_1-W_2|<2$. Consider the welding picture in the left panel. By Propositions \ref{prop-3-pt-disk} and \ref{prop-curve-2} we may figure out the joint law of $(\eta_1,\eta_2)$ and therefore recover from Proposition~\ref{prop:ig-flow-descriptions} that the conditional law of $\eta_1$ given $\eta_2$ is $\SLE_\kappa(W-2;W_2-2,W_1-W_2)$.  }\label{fig-curve-3}
\end{figure}

% {[should the following have the condition that weights aren't $\frac{\gamma^2}2$?]}
\begin{proposition}\label{prop-curve-2}
%	Suppose $W,W_1,W_2,W_3>0$ with $W_1\ge W_2$ and  none of the weights $W+W_1,W+W_2,W_1,W_2,W_3$ equals $\frac{\gamma^2}2$. Set 
%	$$\alpha = \frac{W_3+W_2-W_1-2}{4\kappa}(W_3+W_1+2-W_2-\kappa)   .$$
%	 Then there exists some constant $c = c_{W,W_1,W_2}\in (0,\infty)$ such that
%	\begin{equation}\label{eqn-prop-curve-2}
	%	\begin{split}
		%	\QT(W+W_1, W+W_2, W_3)\otimes &\widetilde{\SLE}_\kappa(W-2;W_2-2,W_1-W_2;\alpha) \\
		%	&= c\int_0^\infty \Wd(\Md_2(W;\ell),\QT(W_1,W_2,W_3;\ell))d\ell.
		%	\end{split}
	%	\end{equation}
{Theorem~\ref{thm:M+QTp-rw} holds when $W_1 \geq W_2$ and $\frac{\gamma^2}2 \not \in \{ W + W_1, W+W_2, W_1, W_2, W_3\}$.}
%If $W_1+W<\frac{\gamma^2}{2}$ (resp.\ $W_2+W<\frac{\gamma^2}{2}$), the law of the interface should be understood as {independent} $\SLE_\kappa(W-2;W_1-2)$ (resp.\ $\SLE_\kappa(W-2;W_2-2)$) in each bead of the weight $W+W_1$ (resp.\ $W+W_2$) quantum disk component and {an independent} $\widetilde{\SLE}_\kappa(W-2;W_2-2,W_1-W_2;\alpha)$ in the core.
\end{proposition}

\begin{proof}
Again if $W_1= W_2$ then the conclusion is clear from Proposition \ref{prop-3-pt-disk}. Now we start with the case $W_3 = W_1-W_2+2$ so that $\alpha = 0$ and there is no weighting in the $\SLE$ law. The case when $W+W_2\ge\frac{\gamma^2}{2}$ is explained in Figure \ref{fig-curve-3}. If $W+W_2<\frac{\gamma^2}{2}$, then we may first replace $W$ with $\tilde{W} = 2-W_2$ in Figure \ref{fig-curve-3} and draw an independent curve $\eta_0\sim\SLE_\kappa(-W-W_2;W-2)$ in the weight $\tilde{W}$ disk. Using the same argument one can read off the conditional law of $\eta_1$ given $\eta_0$, which coincides with that in \eqref{eq:M+QT2}. Finally for general $W_3>0$, we apply Proposition \ref{prop-change-weight}. In this setting, $\tilde{\beta}_3 = \gamma+\frac{W_2-W_1}{\gamma}$ and $\beta_3 = \gamma+\frac{2-W_3}{\gamma}$, and we finish the proof by calculating
$$\Delta_{\tilde{\beta}_3}-\Delta_{\beta_3} =   \frac{W_3+W_2-W_1-2}{4\kappa}(W_3+W_1+2-W_2-\kappa) = \alpha. $$ 
\end{proof}
Proposition \ref{prop-curve-2} has discussed the interface law for the case $W_1\ge W_2$. Now if $W_1<W_2$,  by applying Theorem \ref{thm-sle-reverse} and reversing the orientation of the curve, we are now able to finish the proof of Theorem \ref{thm:M+QTp-rw} when none of the weights are $\frac{\gamma^2}{2}$.

\begin{proposition}\label{prop:thm-non-gamma/2}
Theorem \ref{thm:M+QTp-rw} holds when $W_1, W_2, W_3, W+W_1, W+W_2\neq\frac{\gamma^2}{2}$.
\end{proposition}

\begin{proof}%[Proof of Theorem~\ref{thm:M+QTp-rw}]
By Proposition \ref{prop-curve-2}, it remains to work on the case where $W_1<W_2$.  Consider the welding as in the right hand side of \eqref{eq:M+QT2} but with the interface going in the reverse direction. Then by Proposition \ref{prop-curve-2} and left-right symmetry, the law of the interface is $\widetilde{\SLE}_\kappa(W_1-2,W_2-W_1;W-2;\tilde{\alpha})$ where
$$\tilde{\alpha} = \frac{W_3+W_1-W_2-2}{4\kappa}(W_3+W_2+2-W_1-\kappa)   .$$
Note that the conformal radius appeared in the definition \eqref{eqn-sle-CR} is invariant under time reversal and the conformal map $z\mapsto-\frac{1}{z}$, therefore as we reverse the direction and let the interface $\eta$ go from $a_2$ to $a_1$, then by Theorem \ref{thm-sle-reverse} $\eta$ has law $\widetilde{\SLE}_\kappa(W-2;W_2-2,W_1-W_2;\tilde{\alpha}+\frac{(W_2-W_1)(4-\kappa)}{2\kappa})$. (Again if any of $W+W_1$, $W+W_2$ is smaller than $\frac{\gamma^2}{2}$ then in each bead given by thin disk part we apply \cite[Theorem 1.1]{MS16b}.) Therefore we conclude the proof by noticing $\tilde{\alpha}+\frac{(W_2-W_1)(4-\kappa)}{2\kappa} = \alpha$ as given in \eqref{eq:M+QT2}.
\end{proof}

\subsection{Welding of quantum triangles with weight $\frac{\gamma^2}2$}\label{sec-gamma/2}

In this section, we finish the proof of Theorem~\ref{thm:M+QTp-rw}. Using Proposition~\ref{prop:thm-non-gamma/2} and taking a limit, we can allow one or more of $W_1, W_2, W_3$ to be $\frac{\gamma^2}2$  {and require $W > \frac{\gamma^2}2$} (Proposition~\ref{prop-thm-most}). This argument is technical because we need to truncate on suitable events to make the measures finite. Finally we remove the remaining constraint $\frac{\gamma^2}2 \not \in \{ W+W_1, W+W_2\}$ in Proposition~\ref{prop-final-constraints} via gluing with an extra quantum disk.

%We give a proof sketch for Proposition~\ref{prop-thm-most} for $W_1, W_2, W_3 \geq \frac{\gamma^2}2$. Let $(W_i^n)$ be sequences with $W_i^n \downarrow W_i$. By Proposition~\ref{prop:thm-non-gamma/2}, if we sample an appropriate Liouville field and SLE curve $Z^n, \eta^n$, then cutting gives 
\begin{proposition}\label{prop-thm-most}
Theorem~\ref{thm:M+QTp-rw} holds when $W + W_1, W + W_2 \neq \frac{\gamma^2}2$  {and $W > \frac{\gamma^2}2$}.  
\end{proposition}
\begin{proof}
We may assume that $W_3 \geq \frac{\gamma^2}2$ since the $W_3 < \frac{\gamma^2}2$ case follows from applying the result with the weight $\gamma^2 - W_3$ and concatenating with a thin quantum disk of weight $W_3$. 

We first explain the proof when $W_1, W_2 \geq \frac{\gamma^2}2$, then adapt the argument to the general case. 

For each $i$ such that $W_i \neq \frac{\gamma^2}2$, let $(W_i^n)$ be the constant sequence equal to $W_i$. For $i$ such that $W_i = \frac{\gamma^2}2$, let $(W_i^n)$ be a decreasing sequence with limit $W_i$. 
Let $K > 0$ be a parameter we will later send to $\infty$. Let $\rho$ be a probability measure with compact support in $\{ z \in \mathcal S\::\: |z| \leq 1/K\}$ such that $\iint G_\mathcal{S} (z,w) \rho(dw) \rho(dz) < \infty$. 

Define the event $E_K$ for a pair of fields $X$ and $Y$ on $\mathcal S$: \[E_K = \{\nu_X((-\infty, 1)), \nu_X((1,\infty)), |(X, \rho)|, |(Y, \rho)| < K\}.\]
Let $x_1= +\infty, x_2 = -\infty, x_3 = 1$, and let $\{\LF_{\mathcal{S}, \ell}^{(\beta_i^n, x_i)_i}\}$ be the disintegration of $\LF_{\mathcal{S}}^{(\beta_i^n, x_i)_i}$ with respect to the quantum length of $\bbR \times \{ \pi\}$, where $\beta_i^n := Q + \frac\gamma2 - \frac{W_i^n}\gamma$ for $i = 1,2,3$. 
Sample $(X_n, \cD_n)$ from $\int_0^\infty \LF_\mathcal{S, \ell}^{(\beta_i^n, x_i)_i} \times \Md_2(W; \ell) \, d\ell$, and let $Y_n$ be the field such that $\cD_n = (\mathcal S, Y_n, -\infty, +\infty)/{\sim_\gamma}$ and $\nu_{Y_n}(\bbR) = \ell$ with embedding fixed by specifying $\nu_{Y_n}((-\infty, 0)\times \{\pi\}) = \nu_{Y_n}(( 0,\infty)\times \{\pi\})$. {Let $\cL_n$ be the law of $(X_n, Y_n)$  {conditioned on $E_K$}. Similarly, sample $(X, Y)$ in the same way with $W_i^n$ replaced by $W_i$, and let $\cL$ be the law of $(X, Y)$  {conditioned on $E_K$}.} %\xin{Sample $(X,Y)$ in the same way with $W_i^n$ replaced by $W_i$.}
%Let $\cL_n$ \xin{(resp. $\cL$)} be the law of $(X_n, Y_n)$ \xin{(resp. $(X,Y)$)} conditioned on $E_K$. %Let $\cL$ be the law of $(X,Y)$ sampled as  before but with the $W_i^n$ replaced by $W_i$ for $i =1,2,3$.

For a field $Z$ in $\mathcal S$ and a curve $\eta$ from $-\infty$ to $+\infty$ in $\mathcal S$ disjoint from $\partial \mathcal S$, define
\[X(Z, \eta) = f \bullet_\gamma Z, \qquad \qquad  Y (Z, \eta) = g \bullet_\gamma Z\]
where $f$ is the conformal map from the connected component of $\mathcal S \backslash \eta$ below $\eta$ to $\mathcal S$ fixing $-\infty, +\infty$ and $1$, and $g$ is the conformal map from the connected component of $\mathcal S \backslash \eta$ above $\eta$ to $\mathcal S$ sending $(-\infty, +\infty, p) \mapsto (-\infty, +\infty, i\pi)$ where $p \in \bbR \times \{ \pi\}$ is the point such that the $\nu_Z$-lengths of the two components of $(\bbR \times \{\pi\} ) \backslash \{p\}$ are the same. Let $\wt E_K$ be the event $\{(Z, \eta) \: : \:  (X(Z, \eta), Y(Z, \eta)) \in E_K\}$. In other words, 
\[  \wt E_K = \{(Z, \eta) \: : \: \nu_Z((-\infty, 1)), \nu_Z ((1, +\infty)), |(X(Z, \eta), \rho)|, |(Y(Z, \eta) , \rho)| < K \}.\]
Let $\tilde \beta_i^n = Q + \frac\gamma2 - \frac{W_i^n + W}\gamma$ for $i = 1,2$ and let $\tilde \beta_3^n = \beta_3^n$.
Let $\mathcal L_n'$ be the law of a field and curve sampled from $\LF_\mathcal{S}^{(\tilde \beta_i^n, x_i)_i} \times \wt\SLE_\kappa(W-2; W_2^n - 2, W_1^n - W_2^n; \alpha)$ and conditioned on $\wt E_K$, and $\mathcal L'$ the corresponding law when the $W_i^n$ are replaced by $W_i$ for $i= 1,2,3$.  
{We need to show that for a fixed $K$, if we sample $(Z, \eta) $ from $ \mathcal L'$, then the law of $(X(Z, \eta), Y(Z, \eta))$ is $\mathcal L$.}
%For $(X_n, Y_n) \sim \cL_n$, conformally weld $(\mathcal S, X_n, -\infty, +\infty, 1)/{\sim_\gamma}$ and $(\mathcal S, Y_n, -\infty, +\infty, 1)/{\sim_\gamma}$  to obtain a curve-decorated quantum surface, and let $Z_n$ and $\eta_n$ be the field and curve such that $(\mathcal S, Z_n, \eta_n, -\infty, +\infty, 1)$ is an embedding of this quantum surface. 
%By Proposition~\ref{prop:thm-non-gamma/2}, the law of $(Z_n, \eta_n)$ is $\mathcal L'_n$. 

Let $F_\eps = \{Z \:: \: |(Z, \rho)|< 1/\eps \}$ and $G_\eps = \{ \eta \: : \: \mathrm{dist}(1, \eta) > \eps\}$. For fixed $\eps$, as finite measures on the space of curves in $\mathcal S$ (equipped with the Gromov-Hausdorff topology for the two-point compactification of $\mathcal S$) we have  $\lim_{n \to 0} \wt\SLE_\kappa(W-2; W_2^n - 2, W_1^n - W_2^n; \alpha) |_{G_\eps} = \wt\SLE_\kappa(W-2; W_2 - 2, W_1 - W_2; \alpha)|_{G_\eps}$. This and Proposition~\ref{prop-limit-crit} imply that the measure $\cL_n'|_{F_\eps\times G_\eps}$ converges as $n \to \infty$ to $\cL'|_{F_\eps\times G_\eps}$. 

{Proposition~\ref{prop:thm-non-gamma/2} implies that for $(Z_n, \eta_n)\sim \LF_\mathcal{S}^{(\tilde \beta_i^n, x_i)_i} \times \wt\SLE_\kappa(W-2; W_2^n - 2, W_1^n - W_2^n; \alpha)$, the law of $(X(Z_n, \eta_n), Y(Z_n, \eta_n))$ agrees with that of a sample from $C\int_0^\infty \LF_\mathcal{S, \ell}^{(\beta_i^n, x_i)_i} \times \Md_2(W; \ell) \, d\ell$ when the disk is embedded in $\mathcal S$ in the way described above. Since the event $\wt E_K$ for $(Z_n, \eta_n)$ agrees with the event $E_K$ for $(X(Z_n, \eta_n), Y(Z_n, \eta_n))$, conditioning on this event gives the following:}
For $(Z_n, \eta_n) \sim \mathcal L_n'$, the law of $(X(Z_n, \eta_n), Y(Z_n, \eta_n))$ is $\mathcal L_n$. By Proposition~\ref{prop-limit-crit} we have $\mathcal L_n \to \mathcal L$. %{[*Indeed, one may first apply Proposition~\ref{prop:thm-non-gamma/2} for the unconditioned case, and then conditioning on the event $\wt E_K$.*]} 
Combining this with $\mathcal L_n'|_{F_\eps \cap G_\eps} \to \mathcal L'|_{F_\eps \cap G_\eps}$ and $\cL'[F_\eps\times G_\eps] = 1-o_\eps(1)$, we conclude that for $(Z, \eta) \sim \mathcal L'$  the law of $(X(Z, \eta), Y(Z, \eta))$ is within $o_\eps(1)$ in total variation distance of $\mathcal L$. In fact, we see that the law is exactly $\mathcal L$ by sending $\eps \to 0$. Finally, sending $K \to \infty$ gives the desired result for $W_1, W_2 \geq \frac{\gamma^2}2$. 

For the general case where $W_1, W_2$ might not both be thick, the argument is essentially identical, just that in various definitions we would have extra thin quantum disks (corresponding to the weights with $W_i<\frac{\gamma^2}2$). For instance, if $W_1 < \frac{\gamma^2}2 \leq W_2, W_3$, we define $\cL_n$ to be the law of $(X, Y, \cD_1)$ where we sample $(\mathcal T, \mathcal D) \sim \int_0^\infty \QT(W_1^n, W_2^n, W_3^n;\ell)\times \Md_2(W; \ell) \, d\ell$, let $Y$ and $X$ be the fields in $\mathcal S$ from suitably embedding $\cD$ and the core of $\mathcal T$, let $\cD_1$ be the (weight $W_1$) thin quantum  disk at the first vertex of $\mathcal T$, and condition on the event $E_K$ that the two sides of $\mathcal T$ adjacent to the weight $W_3$ vertex have quantum lengths at most $K$ and $|(X, \rho)|, |(Y, \rho)| < K$. We similarly modify the definitions of $\cL, \cL'_n, \cL'$; the arguments are otherwise identical. 
\end{proof}

\begin{proposition}\label{prop-final-constraints}
Theorem~\ref{thm:M+QTp-rw} holds.
\end{proposition}
\begin{proof}
%	(TODO) This follows from Proposition~\ref{prop-thm-most} by soft arguments.  We already know the result holds whenever $W + W_1$ and $W + W_2$ are both not equal to $\frac{\gamma^2}2$. 
See Figure~\ref{fig-soft}. 
We start by sampling 
\begin{equation}\label{eq:pf-thm-most-1}
	( \cD', \cD, \mathcal{T}) \sim \iint_0^\infty \Md_2(2;\ell_1)\times\,\Md_2(W;\ell_1,\ell_2)\times\,\QT(W_1, W_2, W_3;\ell_2) \, d\ell_1 \, d\ell_2.
\end{equation}
By Theorem~\ref{thm-disk-2}, we can weld $\cD'$ to $\cD$ first and then apply Proposition~\ref{prop-thm-most} to weld $\mathcal T$ in. As a consequence, ~\eqref{eq:pf-thm-most-1} is a constant multiple of $\QT(W_1+W+2,W_2+W+2,W_3)\otimes\mathfrak{m}$, where $\mathfrak{m}$ is some measure on the interfaces $(\eta_1,\eta_2)$. If we embed the entire surface as $(\bbH,\infty,0,1)$, then  $
(\eta_1,\eta_2)$ can be produced by (i)  sample $\eta_2$ as a curve from 0 to $\infty$ from $\widetilde{\SLE}_\kappa(W;W_2-2,W_1-W_2;\alpha)$ with $\alpha$ given by~\eqref{eqn-alpha} and force points $0^-;0^+,1$ and (ii) sample $\eta_1$ on the left component of $\bbH\backslash\eta_2$ from the measure $\SLE_\kappa(0;W-2)$ with force points $0^-;0^+$. Then by Lemma~\ref{lm:ig-commute}, we know that a sample $(\eta_1,\eta_2)\sim\mathfrak m$ can also be obtained by (i) sample $\eta_1$ from $\widetilde{\SLE}_\kappa(0;W_2-2,W_1-W_2;\alpha)$ with force points $0^-;0^+,1$ and (ii) sample $\eta_2$ on the right component of $\bbH\backslash\eta_1$ from $\widetilde{\SLE}_\kappa(W-2;W_2-2,W_1-W_2;\alpha)$ with force points $0^-;0^+,1$. Let $\tilde{\mathcal{T}}$ be the {curve-decorated quantum surface given by the} welding of $\cD$ with $\mathcal T$. Then by Proposition~\ref{prop-thm-most}, the law of {$(\cD',\tilde{\mathcal T})$} %$(\cD',\tilde{\mathcal T},\eta_2)$
is a constant times
\begin{equation}\label{eq:law-thm-most-pf}
	\int_0^\infty \Md_2(2;\ell)\times\, \big(\QT(W+W_1,W+W_2,W+W_3;\ell)\otimes \widetilde{\SLE}_\kappa(W-2;W_2-2,W_1-W_2;\alpha)\big)\,d\ell.
\end{equation}
Therefore Theorem~\ref{thm:M+QTp-rw} follows by disintegrating the law~\eqref{eq:law-thm-most-pf} over the right boundary length $\ell$ of the disk $\cD'$.
%we get $\QT(W_1 + W + 2, W_2 + W + 2, W_3)$ decorated by a pair of random curves. From prev result can cut to get $\int_0^\infty (\QT(W_1 + W, W_2 + W, W_3; r)\otimes \wt \SLE) \times \Md_2(2;r)\, dr$. Comparing gives the result. 
\end{proof}

\begin{figure}[ht]
\begin{center}
	\includegraphics[width=\textwidth]{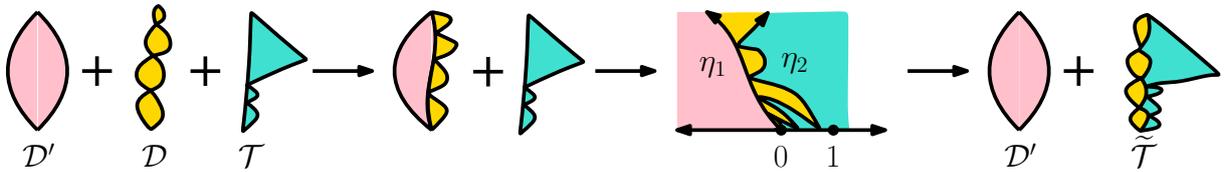}
\end{center}
\caption{ Proposition~\ref{prop-final-constraints} follows from Proposition~\ref{prop-thm-most} and Theorem~\ref{thm-disk-2} by conformally welding quantum surfaces in different orders.
}\label{fig-soft}
\end{figure}

{As a consequence of Theorem~\ref{thm:M+QTp-rw}, we have the following. Recall the notion $\Md_{2,\bullet}(W)$ for $W>0$ in Section~\ref{sec-pre-qt}.
\begin{lemma}\label{lem:QT(W,W,2)}
	For some constant depending only on $W$ and $\gamma$, we have  $\Md_{2,\bullet}(W) = C\QT(W,W,2)$.
\end{lemma}
\begin{proof}
	For $W\neq\frac{\gamma^2}{2}$, the claim follows from ~\cite[Proposition 4.4]{AHS20} and~\cite[Proposition 2.18]{AHS21}. If $W=\frac{\gamma^2}{2}$, then consider the conformal welding $(\cD,\eta)$ of two weight $\frac{\gamma^2}{4}$ quantum disks $\cD_1$ and $\cD_2$ as in Theorem~\ref{thm-disk-2}. We weight the law of $(\cD_1,\cD_2)$ by the right boundary length of $\cD_2$ and sample a marked point on the right boundary of $\cD_2$ according to the quantum length measure. Then it follows from Definition~\ref{three-pointed-disk} that the law of $(\cD,\eta)$ is some constant times $\Md_{2,\bullet}(\frac{\gamma^2}{2})\otimes \SLE_\kappa(\frac{\kappa}{4}-2;\frac{\kappa}{4}-2)$. On the other hand, from the $W\neq\frac{\gamma^2}{2}$ case, we can also view $\cD_2$ as a quantum triangle of weights $\frac{\gamma^2}{4},\frac{\gamma^2}{4},2$. Then by Theorem~\ref{thm:M+QT}, the law of $(\cD,\eta)$ is a constant multiple of $\QT(\frac{\gamma^2}{2},\frac{\gamma^2}{2},2)\otimes \SLE_\kappa(\frac{\kappa}{4}-2;\frac{\kappa}{4}-2)$. This concludes the proof.
\end{proof}
}

\subsection{Proof of Theorem~\ref{thm:disk}}\label{sec-pfthm1.1}
We prove the statement by inductively applying Theorem \ref{thm:M+QT}. When $n=1$, there is only one marked point with the law of $\eta_1$ being $\SLE_\kappa(-\frac{\theta_1\chi}{\lambda}-1;\frac{\theta_1\chi}{\lambda}-1)$, and from Theorem \ref{thm-disk-2} we directly see that the weight 2 disk is the welding of a disk of weight $W_0 = 1-\frac{\theta_1\chi}{\lambda}$ and a disk of weight $W_0 = 1+\frac{\theta_1\chi}{\lambda}$.

Suppose we have proved the statement for the case  with $n$ marked boundary points on the real line. Recall that from \cite[Definition 2.3]{AHS21} a sample from $\mathrm{QD}_{0,n+1}$ can obtained by first sampling $(\mathbb{H}, \phi, \infty)/{\sim_\gamma}$ from $\nu_{\phi}(\bbR)^n\mathrm{QD}_{0,1}$ and then independently sampling the marked points $w_1, ..., w_n$ on $\mathbb{R}$ according to $\nu_{\phi}^\#$. We start by claiming that the following two procedures agree:
\begin{enumerate}
\item Sample $(\mathbb{H}, \phi, w_1, ..., w_{n+1}, \infty)/{\sim_\gamma}$ from $\mathrm{QD}_{0,n+2}$ and let $z_1\le...\le z_{n+1}$ be the reordering of $(w_1, ..., w_{n+1})$. Output $(\mathbb{H}, \phi, z_1, ..., z_{n+1}, \infty)/{\sim_\gamma}$.
\item Sample $(\mathbb{H}, \phi, \tilde{w}_1, ..., \tilde{w}_{n}, \infty)/{\sim_\gamma}$ from $(n+1)s_{n+1}\mathrm{QD}_{0,n+1}$ where $\tilde{z}_1\le...\le \tilde{z}_{n}$ is the reordering of $(\tilde{w}_1, ..., \tilde{w}_{n})$ and $s_{n+1} = \nu_{\phi}((\tilde{z}_n,\infty))$. Then sample $\tilde{z}_{n+1}$ from $(\nu_{\phi}|_{(\tilde{z}_n,\infty)})^\#$. Output $(\mathbb{H}, \phi, \tilde{z}_1, ..., \tilde{z}_{n+1}, \infty)/{\sim_\gamma}$.
\end{enumerate}
Let $\mathcal{L}$ and $\tilde{\mathcal{L}}$ be the corresponding law of the quantum surfaces. To prove the claim, we start with a sample $(\mathbb{H}, \phi, \infty)/\sim_\gamma$ from $\mathrm{QD}_{0,1}$ and let $\ell_i$ (resp.\ $\tilde{\ell}_i$) be the quantum length of $(-\infty, z_i)$ (resp.\ $(-\infty, \tilde{z}_i)$). Then by our definition, for any non-negative functions $f_1, ..., f_{n+1}$ on $\bbR$ and $F$ on $H^{-1}(\mathbb{H})$,  
\begin{equation}
\mathcal{L}\big[F(\phi)f_1(\ell_1)...f_{n+1}(\ell_{n+1})\big] = \int\int_{(0,\nu_{\phi}(\bbR))^{n+1}}(n+1)!\mathbf{1}_{\ell_1\le ...\le \ell_{n+1}}f_1(\ell_1)...f_{n+1}(\ell_{n+1})d\ell_1...d\ell_{n+1} \mathrm{QD}_{0,1}(d\phi).
\end{equation}
On the other hand, 
\begin{equation}
\begin{split}
	\tilde{\mathcal{L}}&\big[F(\phi)f_1(\tilde{\ell}_1)...f_{n+1}(\tilde{\ell}_{n+1})\big] =\\ &=\int\int_{(0,\nu_{\phi}(\bbR))^{n}} n!\mathbf{1}_{\tilde{\ell}_1\le ...\le \tilde{\ell}_{n}}f_1(\tilde{\ell}_1)...f_n(\tilde{\ell}_{n})\big(\int_{\tilde{\ell}_{n}}^{\nu_{\phi}(\bbR)} (n+1)f_{n+1}(\tilde{\ell}_{n+1})d\tilde{\ell}_{n+1}\big)d\tilde{\ell}_1... d\tilde{\ell}_{n}\mathrm{QD}_{0,1}(d\phi)\\
	&= \mathcal{L}\big[F(\phi)f_1(\ell_1)...f_{n+1}(\ell_{n+1})\big],
\end{split}
\end{equation}
which justifies our claim.

From this claim, we may first sample the left $n$ marked points $z_1< ...<z_n$, which produces a disk $(\mathbb{H}, \phi, z_1, ..., z_{n}, \infty)/\sim_\gamma$ from the measure $\mathrm{QD}_{0,n+1}$ weighted by the right most boundary arc. Then by our induction hypothesis, as we grow the $\theta_i$ angle flow lines of the zero boundary GFF, this splits the quantum disk into $n+1$ parts given by  
\begin{equation}
\begin{split}
	\int_{[0,\infty)^{n+1}}s_{n+1}&\Wd\Big(\Md_2(W_0;s_1), \QT(W^1_1, W^2_1, W^3_1; s_1, s_2), \\
	&\cdots, \QT(W^{n-1}_1,W^{n-1}_2,W^{n-1}_3; s_{n-1}, s_{n}), \Md_2(\tilde{W}_n;s_n, s_{n+1})\Big)ds_1\cdots ds_{n+1}. 
\end{split}
\end{equation}
where $W_0 = 1-\frac{\theta_1\chi}{\lambda}$, $W_i^1 = \frac{(\theta_i-\theta_{i+1})\chi}{\lambda}$, $W_i^2 = 1+\frac{\theta_i\chi}{\lambda}$, $W_i^3 = 1-\frac{\theta_i\chi}{\lambda}$ for $i = 1, ..., n-1$ and $\tilde{W}_n = 1+\frac{\theta_n\chi}{\lambda}$. Then from Definition \ref{three-pointed-disk} and Definition \ref{def-m2dot-alpha}, as we add the point $z_{n+1}$ onto $(z_n,\infty)$  according to the quantum length measure, the rightmost surface $D_n$ has law $\Md_{2,\bullet}(\tilde{W}_n)$, which {by Lemma~\ref{lem:QT(W,W,2)}} is a constant times $\QT(\tilde{W}_n, \tilde{W}_n, 2)$. Now conditioned on the points $z_1, ..., z_n$ and $\eta_1, ..., \eta_n$, from \cite[Theorem 1.1]{MS16a},
%Theorem \ref{thm-ig1}  {[please fix this broken reference]} 
the curve $\eta_{n+1}$ has law $\SLE_\kappa(-\frac{\theta_{n+1}\chi}{\lambda}-1, -\frac{\theta_{n}\chi}{\lambda}-1; \frac{\theta_{n+1}\chi}{\lambda}-1) $ from $z_{n+1}$ to $\infty$ within the surface $D_n$. Therefore by Theorems \ref{thm:M+QT}  we know $\eta_{n+1}$ cuts $D_n$ into a triangle of weight $(W_n^1, W_n^2, W_n^3) = (\frac{(\theta_n-\theta_{n+1})\chi}{\lambda}, 1+\frac{\theta_n\chi}{\lambda}, 1-\frac{\theta_n\chi}{\lambda} )$ and a disk of weight $W_{n+1}=\frac{\theta_{n+1}\chi}{\lambda}+1$. This finishes the induction step and concludes the proof.

\qed

\section{Applications to $\SLE$}\label{sec:application}
As an application of our main theorems, in this section we look into several properties of $\SLE_\kappa(\rho_-;\rho_+,\rho_1)$ curves. We comment on the relationship between the SLE reversibility and our conformal welding in Section \ref{sec:application-sle-reverse}, compute the moment of the SLE conformal radius in Section \ref{sec:application-sle-conf-radius}, and finally in Section \ref{sec:sle-commutation} we describe the SLE commutation relation derived from Theorem \ref{thm:M+QTp-rw}. 

\subsection{Comments on $\SLE_\kappa(\rho_-;\rho_+,\rho_1)$ reversibility}\label{sec:application-sle-reverse}

In Section~\ref{sec-sle-reverse}, we proved the $\SLE_\kappa(\rho_-;\rho_+,\rho_1)$ reversibility statement in Theorem \ref{thm-sle-reverse} by extending the results in \cite{Zhan19} via a conformal map composition argument. It served as a key ingredient in the proof of Theorem \ref{thm:M+QTp-rw}. However, for a certain range of weights we can prove Theorem \ref{thm:M+QTp-rw} \emph{independently} of Theorem \ref{thm-sle-reverse}, {just by  reversing the orientation of the curve in \eqref{eq:M+QT} and applying Proposition \ref{prop-change-weight}. We record this proof because this is how we originally reached the statement of  Theorem \ref{thm-sle-reverse}. Moreover, it demonstrates that conformal welding of finite area LQG surfaces is a natural tool for studying the time reversal of SLE curves.}

%We can simply reverse the orientation of the curve in \eqref{eq:M+QT}, {by switching $W_1$ and $W_2$} and apply Proposition \ref{prop-change-weight}. This recovers Theorem \ref{thm-sle-reverse} for some certain cases. 

\begin{proof}[Alternative proof of Theorem \ref{thm:M+QTp-rw} for $\max\{W_1,W_2\}\ge 2$ and $|W_1-W_2|<2$]
{First assume $\frac{\gamma^2}{2}\notin\{W_1,W_2,W+W_1,W+W_2,W_3,W_1-W_2+2\}$.} By Proposition \ref{prop-curve-2} along with the change weight argument Proposition \ref{prop-change-weight}, we may assume that $0<W_1< W_2$, $W_2>2$ and $W_3=W_1-W_2+2$. (If $W_2=2$ we may apply Lemma \ref{lmm-curve-1}.) Since we know the law of the field by Proposition \ref{prop-associate-c}, it remains to identify the law of the interface without applying Theorem~\ref{thm-sle-reverse}.   % {[This isn't precise enough since Proposition~\ref{prop-curve-2} does not give this exact statement. Are we instead using the argument of Proposition~\ref{prop-curve-2}? If so some explanation of precisely what is used is needed]}. 
Consider the setting of right panel of Figure \ref{fig-curve-3} where we start with a quantum triangle of weight $(W+W_1,W+W_2,W_3)$ embedded as $(D,\phi,a_1,a_2,a_3)$ and curves $(\eta_1,\eta_2)$ such that the surfaces $(S_1,S_2,S_3)$ are independent quantum disks and triangles from 
$$\iint_{\bbR_+^2}\Wd(\Md_2(W;\ell_1),\QT(W_1,W_2,W_1-W_2+2;\ell_1,\ell_2),\Md_2({W_2-W_1};\ell_2))\ d\ell_1\ d\ell_2 $$
conditioned on having the same interface length as following from Proposition \ref{prop-associate-c}.  %  {[Should be precise, maybe add one sentence saying ``sample a curve-decorated quantum surface from $\iint...$ and embed it as $(\bbH, \eta_1, \eta_2, 0, 1, \infty)$ sending points ... to ...]}.
Then we know that the marginal law of $\eta_1$ is $\SLE_\kappa(W-2;W_2-2)$, while by Proposition \ref{prop-curve-2} (since $W_1>W_1-W_2+2$) the conditional law of $\eta_2$ given $\eta_1$ is $\SLE_\kappa(W_1-W_2,2-W_2;W_2-W_1-2)$. Therefore we can read off the conditional law of $\eta_1$ given $\eta_2$, which is $\SLE_\kappa(W-2;W_2-2,W_1-W_2)$, and we conclude the proof by reweighting. 

{Finally if $\frac{\gamma^2}{2}\in\{W_1,W_2,W+W_1,W+W_2,W_3,W_1-W_2+2\}$, the result follows from the same limiting argument as in Section~\ref{sec-gamma/2}. }
%	 {[Need to give reference to the place where we obtain the law of the field, and mention that we didn't use Theorem~\ref{thm-sle-reverse}.]}
% 	\begin{figure}[ht]
	% 		\begin{tabular}{ccc} 
		% 			\includegraphics[width=0.4\textwidth]{figures/fig-curve-4.pdf}
		% 			& \quad &
		% 			\includegraphics[width=0.45\textwidth]{figures/fig-ig-4.pdf}
		% 		\end{tabular}
	% 		\caption{Suppose $\max\{W_1,W_2\}\ge 2$ and $|W_1-W_2|<2$. Consider the welding picture in the left panel. By Propositions \ref{prop-3-pt-disk} and \ref{prop-curve-2} we may figure out the joint law of $(\eta_1,\eta_2)$ and therefore recover the conditional law of $\eta_1$ given $\eta_2$ from the imaginary geometry coupling on the right panel. }\label{fig-curve-4}
	% 	\end{figure}
	\end{proof}
	
	From this argument,  we immediately obtain the following case of Theorem \ref{thm-sle-reverse}. 
	\begin{proposition}
Theorem \ref{thm-sle-reverse} holds for  $\max\{\rho_+, \rho_++\rho_1\}\ge 0$ and $|\rho_1|\le 2$.
\end{proposition}
\begin{proof}
{The claim follows immediately by reversing the {direction} of the curve $\eta$ in Theorem~\ref{thm:M+QTp-rw}.}
\end{proof}

\subsection{$\SLE_\kappa(\rho_-;\rho_+,\rho_1)$ conformal radius}\label{sec:application-sle-conf-radius}
In this section, as an application of Theorem \ref{thm:M+QTp-rw}, we shall prove Theorem \ref{thm:sle-conf-radius}. Since the method is almost identical to that in \cite[Section 5]{AHS21}, we will be brief and only list the key steps.  

% We begin with the special case where $\rho_-\in \{0,\frac{\kappa}{2}-2\}$. For $\beta_1, \beta_2, \beta_-, \beta\in\bbR$, set $W_i = \gamma(\gamma-\beta_i)+2$ and $W_- = \gamma(\gamma-\beta_-)+2$.  {[Define $\rho_-, \rho_+, \rho_1$.]}
Recall that by Theorem~\ref{thm:M+QTp-rw}, the weights of the SLE curve are determined by $\rho_-=W-2$, $\rho_+=W_2-2$ and $\rho_1=W_1-W_2$. Define the function $m(\beta_-, \beta_1,\beta_2,\alpha) := \bbE [\psi_{\eta}'(1)^\alpha]$, where $\eta$ is an $\SLE_\kappa(W_--2;W_2-2,W_1-W_2)$ curve {and $\psi_\eta$ is the mapping-out function defined in Section~\ref{subsec-SLE-weighted}}. Recall that $\alpha_0 = \frac{1}{\kappa}(\rho_++2)(\rho_++\rho_1+4-\frac{\kappa}{2}) = \frac{1}{\kappa}W_-(W_1+2-\frac{\kappa}{2})$. To start with, we need the following result on weight 2 and weight $\frac{\gamma^2}{2}$ quantum disks.
\begin{lemma}[Propositions 7.7 and 7.8 of \cite{AHS20}]\label{lem-special-weights}
For $\ell,r>0$,  {there are constants $C_1,C_2$ such that }
\begin{equation}
	|\Md_2(2;l,r)| = C_1(\ell+r)^{-\frac{4}{\gamma^2}-1} \ \ \text{and}\ \ |\Md_2(\frac{\gamma^2}{2};l,r)| = C_2\frac{(\ell r)^{4/\gamma^2-1}}{(\ell^{4/\gamma^2}+r^{4/\gamma^2})^2}.
\end{equation}
\end{lemma}
{The cases $\beta_- \in \{ \gamma, Q\}$ correspond to $W_- \in \{ \frac{\gamma^2}2, 2\}$, and for weight $W_-$ quantum disks Lemma~\ref{lem-special-weights} gives the boundary lengths law. Combining with our conformal welding result, we solve for special values of $m$.}
\begin{lemma}\label{lm-sle-conf-gamma}
For $\beta_1, \beta_2< Q+\frac{\gamma}{2}$ {with $\beta_1,\beta_2\neq Q$}, $\alpha<\alpha_0$, let $\beta$ be either solution to \eqref{eqn-alpha-beta-relation}. Then  we have
\begin{equation}\label{eqn-conf-gamma}
	m(\gamma, \beta_1,\beta_2,\alpha) = \frac{\Gamma(\frac{2}{\gamma}(Q-\frac{\beta_1+\beta_2-\beta}{2}))\Gamma(\frac{2}{\gamma}(2Q-\frac{\beta_1+\beta_2+\beta}{2}) )  }{ \Gamma(\frac{2}{\gamma}(Q+\frac{2}{\gamma}-\beta_1 )  ) \Gamma(\frac{2}{\gamma}(Q+\frac{\gamma}{2}-\beta_2)   )  },
\end{equation}
\begin{equation}\label{eqn-conf-Q}
	m(Q, \beta_1,\beta_2,\alpha) = \frac{\Gamma(\frac{\gamma}{2}(Q-\frac{\beta_1+\beta_2-\beta}{2}))\Gamma(\frac{\gamma}{2}(2Q-\frac{\beta_1+\beta_2+\beta}{2}) )  }{ \Gamma(\frac{\gamma}{2}(Q+\frac{2}{\gamma}-\beta_1 )  ) \Gamma(\frac{\gamma}{2}(Q+\frac{\gamma}{2}-\beta_2)   )  }.
\end{equation}
\end{lemma}

\begin{proof}
% We first emphasize that the reweighting argument in Proposition \ref{prop-change-weight}  and the quantum triangle boundary length law in Proposition \ref{prop-qt-bdry-thick} and Proposition \ref{prop-qt-bdry-thin}  naturally extends to $\beta_3>Q$  {[move this sentence to later when we need to use it.]}. 
%{Without loss of generality assume $\beta_1,\beta_2\neq Q$, as if either $\beta_1$ or $\beta_2$ is equal to $Q$, then we may start with $(\beta_1-\e,\beta_2-\e)$ and apply Lemma~\ref{lm-gamma/2} to send $\e\to0$.}
% {[In the whole proof, since we do not enforce the condition $W_1+2 = W_2 + W_3$, should~\eqref{eq:M+QT} be replaced by~\eqref{eq:M+QT2}? With $\alpha$ defined as a function of $\beta_3$. It could also help to make a reference to Proposition~\ref{prop-change-weight} if we want $\beta_3>Q + \frac\gamma2$ to be allowed. ]}
Let $\widetilde{\QT}(W_1,W_2,\beta_3)$ be the corresponding quantum surface  when the weight $W_3$ vertex is replaced by a $\beta_3$ Liouville field insertion {and the constant $\frac{1}{\gamma(Q-\beta_1)(Q-\beta_2)(Q-\beta_3)}$ is dropped}. {Then by Proposition~\ref{prop-change-weight}, \eqref{eq:M+QT2}} continues to hold with ${\QT}(W_1,W_2,W_3)$ and $\QT(W_1+W,W_2+W,W_3)$ replaced by  $\widetilde{\QT}(W_1,W_2,\beta_3)$ and $\widetilde{\QT}(W_1+W,W_2+W,\beta_3)$ as long as the Seiberg bounds \eqref{eqn-seiberg} holds for $\widetilde{\QT}(W_1,W_2,\beta_3)$. %  {[\eqref{eq:M+QT} isn't defined for e.g.\ the range of $\beta_3$ for which the Seiberg bounds don't hold.]}. 
Let $\beta_3\in (\max\{|2Q-\beta_1-\beta_2|, |\beta_1-\beta_2|\}, 4Q-\beta_1-\beta_2)$. A sample from the left hand side of \eqref{eq:M+QT2} now has left boundary length law $1_{x>0}\mathfrak{C}x^{\frac{\beta_1+\beta_2+\beta_3-2Q}{\gamma}-\frac{4}{\gamma^2}-1}dx$ with {$ \mathfrak{C} = \frac{2}{\gamma}|\bar{H}_{(0,1,0)}^{(\beta_1-\frac{2}{\gamma},\beta_2-\frac{2}{\gamma}, \beta_3 )}|m(\gamma, \beta_1,\beta_2,\alpha_3) $ where $\alpha_3$ is determined by $\beta_3$ via \eqref{eqn-alpha-beta-relation}.} On the other hand, evaluating this using the right hand side of \eqref{eq:M+QT2}, we see that {for some constants $c_{\beta_1,\beta_2}, \tilde c_{\beta_1,\beta_2}$ not depending on $\beta_3$,}
\begin{equation}
	\begin{split}
		\mathfrak{C} &= \frac{2}{\gamma}c_{\beta_1,\beta_2}\int_0^\infty |\Md_2(2;1,\ell)|\cdot |\bar{H}_{(0,1,0)}^{(\beta_1,\beta_2, \beta_3 )}|\ell^{\frac{\beta_1+\beta_2+\beta_3-2Q}{\gamma}-1}d\ell\\
		&= \frac{2}{\gamma}\tilde{c}_{\beta_1,\beta_2}{\Gamma(\frac{\beta_1+\beta_2+\beta_3-2Q}{\gamma})\Gamma(\frac{4}{\gamma^2}+1+\frac{2Q-\beta_1+\beta_2+\beta_3}{\gamma}) }.
	\end{split}
\end{equation}
Using the definition of $\bar{H}_{(0,1,0)}^{(\beta_1,\beta_2, \beta_3 )}$ and the shift relations \eqref{eqn-gamma-shift},  this implies that for some constant $C_{\beta_1,\beta_2}$ not depending on $\beta_3$ we have
\begin{equation}\label{eqn-conf-gamma-a}
	m(\gamma, \beta_1,\beta_2,\alpha_3) = C_{\beta_1,\beta_2}\Gamma(\frac{1}{\gamma}(4Q - \beta_1-\beta_2-\beta_3))\Gamma(\frac{1}{\gamma}(2Q-\beta_1-\beta_2+\beta_3) ).
\end{equation} 
Exactly as in \cite[Section A.2]{AHS21}, the input \cite[Theorem 1.8]{MW17} can be bootstrapped to give  $\bbE[\psi_{\eta}'(1)^\alpha]<\infty$ for any $\alpha<\alpha_0$. By Fubini's theorem and Morera's theorem, it is not hard to observe that $\alpha\mapsto \bbE[\psi_{\eta}'(1)^\alpha]$ is holomorphic on $\{\alpha\in\mathbb{C}: \textup{Re}\ \alpha<\alpha_0\}$. From the uniqueness of holomorphic extensions, we observe that the equation \eqref{eqn-conf-gamma-a} extends to any $\alpha<\alpha_0$. Therefore by setting $\alpha = 0$ (and $\beta_3 = \beta_1-\beta_2+\gamma$), we can solve for the constant $C_{\beta_1,\beta_2}$ as $ 	m(\gamma, \beta_1,\beta_2,0)$ is trivially 1.  {We note, but do not need to use, that this also solves the constants in Theorem~\ref{thm:M+QT} and Theorem~\ref{thm:M+QTp-rw} for the case $W=2$ or $W=\frac{\gamma^2}{2}$.} Substituting this expression of $C_{\beta_1,\beta_2}$ in \eqref{eqn-conf-gamma-a} gives \eqref{eqn-conf-gamma}. By a similar argument one obtains \eqref{eqn-conf-Q}.

\end{proof}

The next step is to establish the shift relations by conformal map composition.
\begin{lemma}
For $\beta_-,\tilde{\beta},\beta_1,\beta_2<Q+\frac{\gamma}{2}$ and $\alpha<0$,  we have
\begin{equation}\label{eqn-lm-shift-a}
	m(\beta_-+\tilde{\beta}-Q-\frac{\gamma}{2},\beta_1,\beta_2,\alpha) = m(\tilde{\beta}, \beta_1+\beta_--\gamma-\frac{2}{\gamma}, \beta_2+\beta_--\gamma-\frac{2}{\gamma}, \alpha)m(\beta_-,\beta_1,\beta_2,\alpha ).
\end{equation}
In particular, if $\beta$ solves \eqref{eqn-alpha-beta-relation},  then
\begin{equation}\label{eqn-sle-shift-a}
	\frac{m(\beta_--\frac{2}{\gamma},\beta_1,\beta_2,\alpha)}{m(\beta_-,\beta_1,\beta_2,\alpha)} = \frac{\Gamma(\frac{2}{\gamma}(2Q+\frac{\gamma-\beta_1-\beta_2-2\beta_-+\beta}{2}) ) \Gamma(\frac{2}{\gamma}(3Q+\frac{\gamma-\beta_1-\beta_2-2\beta_--\beta}{2})  ) }{\Gamma(\frac{2}{\gamma}(3Q-\beta_1-\beta_-) ) \Gamma(\frac{2}{\gamma}(2Q+\gamma-\beta_2-\beta_-))      };
\end{equation}
\begin{equation}\label{eqn-sle-shift-b}
	\frac{m(\beta_--\frac{\gamma}{2},\beta_1,\beta_2,\alpha)}{m(\beta_-,\beta_1,\beta_2,\alpha)} = \frac{\Gamma(\frac{\gamma}{2}(2Q+\frac{\gamma-\beta_1-\beta_2-2\beta_-+\beta}{2}) ) \Gamma(\frac{\gamma}{2}(3Q+\frac{\gamma-\beta_1-\beta_2-2\beta_--\beta}{2})  ) }{ \Gamma(\frac{\gamma}{2}(3Q-\beta_1-\beta_-) ) \Gamma(\frac{\gamma}{2}(2Q+\gamma-\beta_2-\beta_-))      }.
\end{equation}
\end{lemma}
\begin{proof}
Let $\rho_- = \gamma^2-\gamma\beta_-$, $\rho_+ = \gamma^2-\gamma\beta_2$, $\rho_1 = \gamma(\beta_2-\beta_1)$ and $\tilde{\rho} = \gamma^2-\gamma\tilde{\rho}$. Sample an $\SLE_\kappa(\tilde{\rho};\rho_-+\rho_++2, \rho_1  )$ curve $\eta_1$ in $\bbH$ from 0 to $\infty$, and an $\SLE_\kappa({\rho_-};\rho_+, \rho_1  )$ curve $\eta_2$ in the connected component of $\bbH \backslash \eta_1$ with $1$ on its boundary. Let $\psi_{{\eta}_2|{\eta}_1}$ be the conformal map from the right component of $\bbH\backslash\psi_{{\eta}_1}({\eta}_2)$ to $\bbH$ fixing $0,1,\infty$.  As in the proof of Theorem \ref{thm-sle-reverse}, we know $\psi_{\eta_2} = \psi_{\eta_2|\eta_1}\circ\psi_{\eta_1}$ and $\psi_{\eta_2|\eta_1}$ and $\psi_{\eta_1}$ are independent. Moreover, using the imaginary geometry, the marginal law of $\eta_2$ is $\SLE_\kappa({\rho_-+\tilde{\rho}+2};\rho_+, \rho_1  )$. Therefore \eqref{eqn-lm-shift-a} follows from $\bbE[\psi_{\eta_2}'(1)^\alpha] = \bbE[\psi_{\eta_2|\eta_1}'(1)^\alpha]\bbE[\psi_{\eta_1}'(1)^\alpha]$. Equations \eqref{eqn-sle-shift-a} and \eqref{eqn-sle-shift-b} follow by setting $\tilde{\beta} \in \{\gamma, Q\}$ and applying Lemma \ref{lm-sle-conf-gamma}.
\end{proof}

Set $g(\beta_-, \beta_1,\beta_2,\alpha)$ to be the right hand side of \eqref{eqn-sle-conf-radius}, that is
$$g(\beta_-, \beta_1,\beta_2,\alpha) = \frac{F(\beta+\beta_2-\beta_1,\gamma^2,\gamma^2-\gamma\beta_-, \gamma^2-\gamma\beta_2,\gamma(\beta_2-\beta_1))}{F(\gamma,\gamma^2,\gamma^2-\gamma\beta_-, \gamma^2-\gamma\beta_2,\gamma(\beta_2-\beta_1))}. $$
Define $h(\beta_-, \beta_1,\beta_2,\alpha):=\frac{m(\beta_-, \beta_1,\beta_2,\alpha)}{g(\beta_-, \beta_1,\beta_2,\alpha)}$. Using  the argument in \cite[Section A.3]{AHS21}, it is not hard to show that $h$ is meromorphic on $\{\alpha:\textup{Re}\ \alpha <0\}$. By the shift relations \eqref{eqn-gamma-shift}, we see
\begin{equation}\label{eqn-sle-cr-period}
\begin{split}
	&h(\beta_--\frac{2}{\gamma}, \beta_1,\beta_2,\alpha)=h(\beta_-, \beta_1,\beta_2,\alpha) \ \text{for} \ \beta_-,\beta_1,\beta_2<Q+\frac{\gamma}{2};\\
	&h(\beta_--\frac{\gamma}{2}, \beta_1,\beta_2,\alpha)=h(\beta_-, \beta_1,\beta_2,\alpha) \ \text{for} \ \beta_-,\beta_1,\beta_2<Q+\frac{\gamma}{2}.
\end{split}
\end{equation}

\begin{proof}[Proof of Theorem \ref{thm:sle-conf-radius}]
We start with the case where $\alpha<0$.   First suppose $\gamma^2\notin \mathbb{Q}$. Assume $\beta_1,\beta_2\neq Q$. %If $\min\{\beta_1,\beta_2\}>\frac{3\gamma}{2}-\frac{2}{\gamma}$ 
%  {[maybe say a little about why we have this condition? Also, later in this paragraph the conclusion is ``this proves (1.13) for $\beta_1, \beta_2 \neq Q$, but what about this condition? ]}{[**It seems that there is nothing to do with this $\min\{\beta_1,\beta_2\}>\frac{3\gamma}{2}-\frac{2}{\gamma}$ condition and it is removed. ** This condition is required when there was a range bound for $W_1,W_2$ at the very beginning**]}  
{The function $\beta_-\mapsto h(\beta_-, \beta_1,\beta_2,\alpha)$ is well-defined on $(-\infty, Q+\frac{\gamma}{2})$ and is constant} {on a dense subset of $(-\infty, Q+\frac{\gamma}{2})$ by~\eqref{eqn-sle-cr-period}}. Moreover, by Lemma \ref{lm-sle-conf-gamma} we know that $h(\gamma,\beta_1,\beta_2,\alpha) =1 $ and therefore $m(\beta_-, \beta_1,\beta_2,\alpha) = g(\beta_-, \beta_1,\beta_2,\alpha)$ {on a dense subset of $(-\infty, Q+\frac{\gamma}{2})$}. {On the other hand, since $\psi_\eta'(1)>1$, a.s., it follows that $m(\beta_-, \beta_1,\beta_2,\alpha)<1$ whenever $\alpha<0$, and therefore by~\eqref{eqn-lm-shift-a} the function $\beta_-\mapsto m(\beta_-, \beta_1,\beta_2,\alpha)$ is monotone. This proves \eqref{eqn-sle-conf-radius} for $\beta_1,\beta_2\neq Q$, and for $Q\in\{\beta_1,\beta_2\}$, the claim follows by applying~\eqref{eqn-lm-shift-a} to $(Q+\frac{\gamma}{2}-\e,\beta_-,\beta_1+\e,\beta_2+\e)$ for $\e>0$ chosen to be small.}

{Now assume $\gamma^2\in\mathbb{Q}$. By the same SLE continuity argument as in \cite[Lemma A.3]{AHS21}, for $\eta_n\sim\SLE_{\kappa_n}(\rho_-;\rho_+,\rho_1)$ and $\eta\sim\SLE_{\kappa}(\rho_-;\rho_+,\rho_1)$ with $\kappa_n\downarrow\kappa\in(0,4)$, when $\rho_-,\rho_+,\rho_++\rho_1\ge\frac{\kappa}{2}-2$, i.e., the curves are non-boundary hitting, $\psi_{\eta_n}'(1)\to\psi_\eta'(1)$ in probability. This implies $m(\beta_-, \beta_1,\beta_2,\alpha) = g(\beta_-, \beta_1,\beta_2,\alpha)$ for $\beta_-,\beta_1,\beta_2\le Q$ and all $\kappa\in(0,4)$. Then for $\beta_-\le \frac{2}{\gamma}$, $\beta_1,\beta_2<Q+\frac{\gamma}{2}$, $m(\beta_-, \beta_1,\beta_2,\alpha)$ is solved by applying~\eqref{eqn-lm-shift-a} along with~\eqref{eqn-conf-Q}  for the tuple $(Q,\beta_-+\frac{\gamma}{2}, \beta_1,\beta_2)$, and the general  $\beta_-,\beta_1,\beta_2<Q+\frac{\gamma}{2}$ case follows immediately by applying~\eqref{eqn-lm-shift-a} to the tuple $(\beta_-,\frac{2}{\gamma},\beta_1,\beta_2)$.}
%For $\gamma^2\in\mathbb{Q}$, we may extend \eqref{eqn-sle-conf-radius} by the SLE continuity argument as in \cite[Section A.3]{AHS21}. 

Finally, again
by using the holomorphic extension in terms of $\alpha$, \eqref{eqn-sle-conf-radius} extends to the full range $\alpha<\alpha_0$, as desired.
\end{proof}

\subsection{SLE commutation relation}\label{sec:sle-commutation}
In Imaginary geometry theory, the GFF flow line construction neatly characterizes the marginal and conditional laws of interacting $\SLE_\kappa(\underline{\rho})$ curves. On the other hand, we can also read off the interface laws in the conformal welding statement in Theorem \ref{thm:M+QTp-rw}. By considering the different orders of welding quantum disks and triangles, this gives an alternative way of describing the marginal and conditional laws of $\SLE_\kappa(\rho_-;\rho_+,\rho_1)$ curves. Moreover, this also extends to $\widetilde{\SLE} _\kappa(\rho_-;\rho_+,\rho_1;\alpha)$, the $\SLE$ curves weighted by conformal radius.

As a quick application, we prove the following.
\begin{proposition}\label{prop-sle-commutation}
Fix $W,W',W_1,W_2,W_3>0$. % {such that none of $W_1,W_2,W_3,W+W_1,W'+W_1,W+W'+W_1,W+W_2,W'+W_3$ is $\frac{\gamma^2}{2}$}. 
The following two laws on tuples of curves $(\eta_1,\eta_2)$ differs only by a multiplicative constant. Let $\alpha$ be the same as \eqref{eqn-alpha} and $$\alpha' = \frac{W_2+W_3-W_1-2}{4\kappa}({W_2+W_1+2-W_3-\kappa}).$$
\begin{enumerate}
	\item First sample an $\widetilde{\SLE}_\kappa(W-2;W_2-2,W_1-W_2+W';\alpha)$ (with force points $0^-,0^+,1$) curve $\eta_1$ on $\bbH$ from 0 to $\infty$, and then sample an $\widetilde{\SLE}_\kappa(W_3-2,W_1-W_3;W'-2;\alpha')$ (with force points $1^-,0,1^+$) curve $\eta_2$ to the right of $\eta_1$ in $\bbH\backslash \eta_1$.
	\item First sample an $\widetilde{\SLE}_\kappa(W_3-2,W_1-W_3+W;W'-2;\alpha')$ (with force points $1^-,0,1^+$) curve $\eta_2$ on $\bbH$ from 1 to $\infty$, and then sample an $\widetilde{\SLE}_\kappa(W-2;W_2-2,W_1-W_2;\alpha)$ (with force points $0^-,0^+,1$) curve $\eta_1$ on the left component of $\bbH\backslash \eta_2$.
\end{enumerate}
\end{proposition}
\begin{figure}[ht]
\centering
\includegraphics[width=0.35\textwidth]{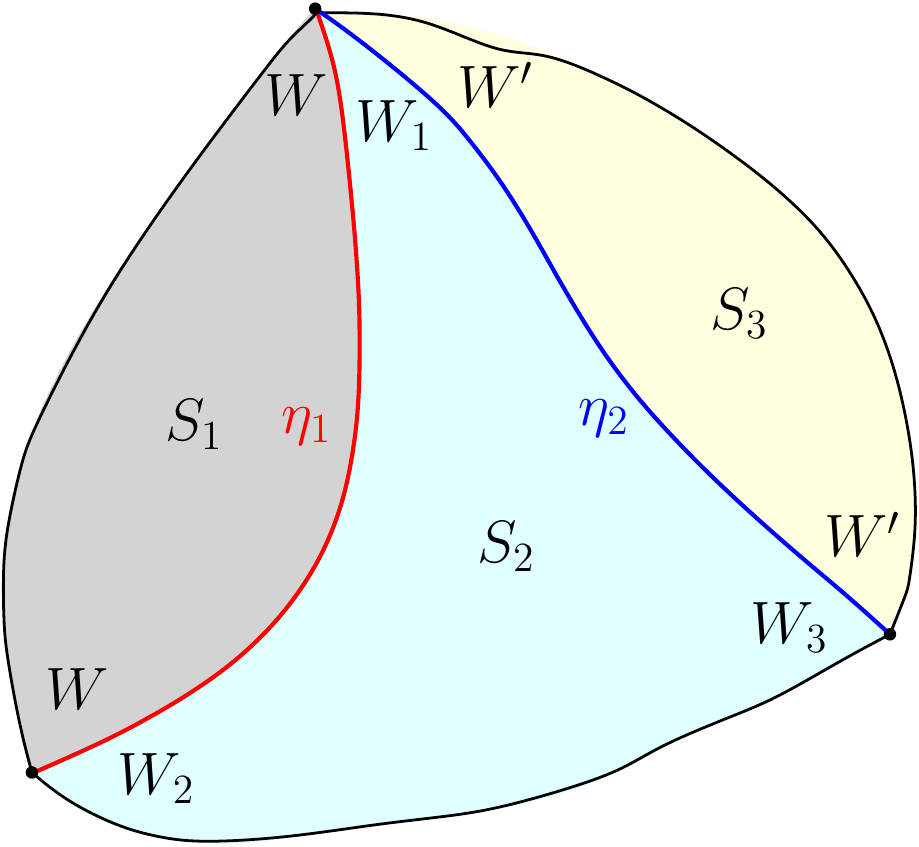}
\caption{Proof of Proposition \ref{prop-sle-commutation}. If we start by welding $S_1$ to the left of $S_2$ first and then $S_3$ to the right, then by Theorem \ref{thm:M+QTp-rw} we know the marginal law of $\eta_2$ and the law of $\eta_1$ given $\eta_2$. If we first weld $S_3$ to the right of $S_2$ and $S_1$ to the left, then we can interpret the marginal law of $\eta_1$ and then the conditional law of $\eta_2$ given $\eta_1$. This justifies Proposition \ref{prop-sle-commutation}.  }\label{fig-sle-commutation}
\end{figure}

\begin{proof}
The proof is again an application of Theorem \ref{thm:M+QTp-rw} and the argument in {Section \ref{sec-interface-law}}. Namely, suppose we are in the setting of Figure \ref{fig-sle-commutation}, where we sample surfaces $(S_1,S_2,S_3)$ from the measure
$$\iint_{\bbR_+^2}\Wd(\Md_2(W;\ell_1),\QT(W_1,W_2,W_3;\ell_1,\ell_2),\Md_2(W';\ell_2))\ d\ell_1\ d\ell_2 $$
and conformally weld them together.
First consider the case where $W+W'+W_1,W+W_2,W'+W_3\ge\frac{\gamma^2}{2}$. We may first weld $S_1$ and $S_2$ together, which implies that given $\eta_2$, the conditional law of $\eta_1$ is   $\widetilde{\SLE}_\kappa(W-2;W_2-2,W_1-W_2;\alpha)$. Then as we weld $S_3$ to the right, we observe that the marginal law of $\eta_2$ is proportional to $\widetilde{\SLE}_\kappa(W_3-2,W_1-W_3+W;W'-2;\alpha')$. This implies the interface law $(\eta_1,\eta_2)$ is a constant multiple of the second law. On the other hand, if we first fix $S_1$ and weld $S_3$ to the right of $S_2$, and then weld $S_1$ to the left, by Theorem \ref{thm:M+QTp-rw}, we know that the conditional law of $\eta_2$ given $\eta_1$ is a constant times  $\widetilde{\SLE}_\kappa(W_3-2,W_1-W_3;W'-2;\alpha')$ and the marginal law of $\eta_1$ is  $\widetilde{\SLE}_\kappa(W-2;W_2-2,W_1-W_2+W';\alpha)$. If any of the vertex in the large triangle is thin, then we may focus on the thick triangle component. This concludes the proof.
\end{proof}

\end{document}